\documentclass[final]{daj}

\usepackage[notcite,notref]{showkeys}  

\usepackage{amsmath,amsfonts,amssymb,amsthm,mathrsfs,lipsum,amsfonts}
\usepackage[english,algoruled,lined,noresetcount,norelsize]{
algorithm2e}

\usepackage{tikz}
\usetikzlibrary{decorations.pathmorphing}
\tikzset{snake it/.style={decorate, decoration={snake,segment length=1.5mm,amplitude=0.5mm}}}
\usetikzlibrary{patterns}

\usepackage{verbatim,textcase}

\usepackage{tabto}
\usepackage{calc}

\usepackage{enumitem}

\newcommand{\cA}{\mathcal{A}}

\def\cH{\mathcal{H}}

\def\cK{\mathcal{K}}
\def\cI{\mathcal{I}}
\def\cJ{\mathcal{J}}

\def\hist{\mathscr{H}}

\setitemize{leftmargin=\parindent,itemsep=2pt,parsep=0pt,topsep=2pt}
\setenumerate{leftmargin=*,itemsep=2pt,parsep=0pt,topsep=2pt}
\setdescription{style=sameline,leftmargin=.7cm,itemsep=2pt,parsep=0pt,topsep=2pt}

\def\itm#1{\rm ({#1})} 
\def\itmit#1{\itm{\it #1\,}} 
\def\rom{\itmit{\roman{*}}} 
\def\abc{\itmit{\alph{*}}}

\def\itmarab#1{\mbox{\itm{{\it #1\,}\arabic{*}\hspace{.05em}}}}
\def\itmarabp#1{\mbox{\itm{{\it #1\,}\arabic{*}'\hspace{.05em}}}}

\newcommand{\By}[2]{\overset{\mbox{\tiny{#1}}}{#2}} 
\newcommand{\ByRef}[2]{   \By{\eqref{#1}}{#2} }

\newcommand{\geBy}[1]{    \By{#1}{\ge} }

\newcommand{\leByRef}[1]{ \ByRef{#1}{\le} } 
\newcommand{\geByRef}[1]{ \ByRef{#1}{\ge} }

\newcommand{\NATS}{\mathbb{N}} 
\newcommand{\cF}{\mathcal{F}}

\renewcommand{\Pr}{\mathbb{P}}
\newcommand{\ev}{\mathcal{E}}

\DeclareMathOperator{\degen}{degen}

\DeclareMathOperator{\bw}{bw}
\DeclareMathOperator{\mindeg}{mindeg}

\newcommand{\Exp}{\mathbb{E}}
\newcommand{\im}{\mathrm{Im}}
\newcommand{\dom}{\mathrm{Dom}}
\newcommand{\RGend}{\mathrm{RGA end}}
\newcommand{\Vmain}{V^{\mathrm{main}}}
\newcommand{\Cmain}{C^{\mathrm{main}}}
\newcommand{\Amain}{A^{\mathrm{main}}}
\newcommand{\Umain}{U^{\mathrm{main}}}
\newcommand{\Xmain}{X^{\mathrm{main}}}
\newcommand{\Vbuf}{V^{\mathrm{buf}}}
\newcommand{\Cbuf}{C^{\mathrm{buf}}}
\newcommand{\Abuf}{A^{\mathrm{buf}}}
\newcommand{\Ubuf}{U^{\mathrm{buf}}}
\newcommand{\Xbuf}{X^{\mathrm{buf}}}

\newcommand{\Vq}{V^{\mathrm{q}}}
\newcommand{\Cq}{C^{\mathrm{q}}}
\newcommand{\Aq}{A^{\mathrm{q}}}
\newcommand{\Uq}{U^{\mathrm{q}}}
\newcommand{\Xq}{X^{\mathrm{q}}}
\newcommand{\Vc}{V^{\mathrm{c}}}
\newcommand{\Cc}{C^{\mathrm{c}}}
\newcommand{\Ac}{A^{\mathrm{c}}}
\newcommand{\Uc}{U^{\mathrm{c}}}
\newcommand{\Xc}{X^{\mathrm{c}}}

\newcommand{\Rbl}{R_{\mathrm{BL}}}
\newcommand{\Rpbl}{R'_{\mathrm{BL}}}
\newcommand{\cXbl}{\mathcal{X}^{\mathrm{BL}}}
\newcommand{\cVbl}{\mathcal{V}^{\mathrm{BL}}}
\newcommand{\Xbl}{X^{\mathrm{BL}}}
\newcommand{\Vbl}{V^{\mathrm{BL}}}
\newcommand{\rbl}{r_{\mathrm{BL}}}
\newcommand{\ronebl}{r_1^{\mathrm{BL}}}
\newcommand{\cIbl}{\mathcal{I}^{\mathrm{BL}}}
\newcommand{\Ibl}{I^{\mathrm{BL}}}
\newcommand{\tcXbl}{\mathcal{\tilde{X}}^{\mathrm{BL}}}
\newcommand{\tXbl}{\tilde{X}^{\mathrm{BL}}}
\newcommand{\alphabl}{\alpha^{\mathrm{BL}}}
\newcommand{\kappabl}{\kappa^{\mathrm{BL}}}
\newcommand{\epsbl}{\eps^{\mathrm{BL}}}
\newcommand{\rhobl}{\rho^{\mathrm{BL}}}
\newcommand{\zetabl}{\zeta^{\mathrm{BL}}}
\newcommand{\DeltaRpbl}{\DeltaRp^{\mathrm{BL}}}
\newcommand{\taubl}{\tau^{\mathrm{BL}}}

\newcommand{\DeltaRp}{\Delta_{R'}}
\newcommand{\comN}{N^*}
\newcommand{\psiRGA}{\psi_{\mathrm{RGA}}}
\newcommand{\psiq}{\psi_{\mathrm{q}}}
\newcommand{\psigood}{\psi_{\mathrm{good}}}

\newcommand{\pitau}{\pi^{\tau}}
\newcommand{\Int}{\mathrm{Int}}
\newcommand{\dist}{\mathrm{dist}}

\newcommand{\RI}{\mathrm{RI}}
\newcommand{\NS}{\mathrm{NS}}
\newcommand{\LNS}{\mathrm{LNS}}
\newcommand{\AG}{\mathrm{CG}}
\newcommand{\CON}{\mathrm{CON}}
\newcommand{\LCON}{\mathrm{LCON}}

\renewcommand{\subset}{\subseteq}
\renewcommand{\supset}{\supseteq}
\def\subsc#1{\textsc{\MakeTextLowercase{#1}}} 
\def\sublem#1{\subsc{L\tiny{\ref{#1}}}}

\newcommand{\eps}{\varepsilon}
\renewcommand{\rho}{\varrho}
\renewcommand{\phi}{\varphi}

\newcommand{\dcup}{\mathbin{\dot\cup}}
\newcommand{\tZ}{\tilde{Z}}

\newcommand{\cE}{\mathcal{E}}
\newcommand{\tX}{\tilde X}

\newcommand{\hatU}{\hat{U}}
\newcommand{\hatA}{\hat{A}^{\mathrm{main}}}

\newcommand{\cV}{{\mathcal{V}}}
\newcommand{\cX}{{\mathcal{X}}}
\newcommand{\tcX}{\tilde{\mathcal{X}}}

\newtheorem{theorem}{Theorem}[chapter]
\newtheorem{lemma}[theorem]{Lemma}
\newtheorem{claim}[theorem]{Claim}
\newtheorem{proposition}[theorem]{Proposition}
\newtheorem{corollary}[theorem]{Corollary}
\newtheorem{fact}[theorem]{Fact}

\newtheorem*{oneRI}{Lemma~\ref{lem:oneRI}}
\newtheorem*{twoRI}{Lemma~\ref{lem:twoRI}}

\theoremstyle{definition}
\newtheorem{definition}[theorem]{Definition}

\newtheorem{problem}[theorem]{Problem}

\theoremstyle{remark}
\newtheorem{remark}[theorem]{Remark}

\newcommand{\oldqed}{}
\def\endofClaim{\hfill\scalebox{.6}{$\Box$}}
\newenvironment{claimproof}[1][Proof]{
  \renewcommand{\oldqed}{\qedsymbol}
  \renewcommand{\qedsymbol}{\endofClaim}
  \begin{proof}[#1]
}{
  \end{proof}
  \renewcommand{\qedsymbol}{\oldqed}
} 

\dajAUTHORdetails{%
  title = {Blow-up Lemmas for Sparse Graphs}, 
  author = {Peter Allen, Julia B\"ottcher, Hi\d{\^{e}}p H\`an, Yoshiharu
    Kohayakawa, and Yury Person},
  plaintextauthor = {Peter Allen, Julia Boettcher, Hiep Han, Yoshiharu
    Kohayakawa, and Yury Person},
  plaintexttitle = {Blow-up Lemmas for Sparse Graphs}, 
  runningtitle = {Blow-up Lemmas for Sparse Graphs}, 
  runningauthor = {P.~Allen, J.~B\"ottcher, H.~H\`an, Y.~Kohayakawa, and Y.~Person},
  copyrightauthor = {P.~Allen, J.~B\"ottcher, H.~H\`an, Y.~Kohayakawa, and Y.~Person},
  keywords = {regularity method, extremal problems in random and pseudorandom graphs}
}   

\dajEDITORdetails{%
   year={2025},
   number={8},
   received={21 March 2019},   
   published={28 August 2025},  
   doi={10.19086/da.143410},       
}

\begin{document}



\begin{frontmatter}[classification=text]

  \title{Blow-up Lemmas for Sparse Graphs\thanks{The cooperation of
      the authors was supported by a joint CAPES-DAAD PROBRAL project
      (Proj.\ no.~430/15, 57143515).  The authors are grateful to
      NUMEC/USP, N\'ucleo de Modelagem Estoc\'astica e Complexidade of
      the University of S\~ao Paulo, for supporting this research.
      This research was partially supported by CAPES (Finance Code
      001).  } } 

\author[pa]{Peter Allen}
\author[jb]{Julia B\"ottcher}
\author[hiep]{Hi\d{\^{e}}p H\`an \thanks{Supported by the FONDECYT Iniciación grant 11150913, ANID Regular grant 1231599, by Millenium Nucleus Information and Coordination in Networks and by ANID Basal Grant CMM FB210005.}
}
\author[yk]{Yoshiharu Kohayakawa \thanks{Supported by FAPESP (2013/03447-6) and CNPq
    (406248/2021-4, 407970/2023-1, 315258/2023-3).}
}
\author[yp]{Yury Person \thanks{Supported by DFG grant PE 2299/1-1.}
}

\begin{dedication}
  Dedicated to Vojt\v ech R\" odl, on the occasion of his 75th birthday.
\end{dedication}

\begin{abstract}
  The blow-up lemma states that a system of super-regular pairs contains all
  bounded degree spanning graphs as subgraphs that embed into a
  corresponding system of complete pairs. This lemma has far-reaching
  applications in extremal combinatorics.

  We prove sparse analogues of the blow-up lemma for subgraphs of random
  and of pseudorandom graphs. Our main results are the following three
  sparse versions of the blow-up lemma: one for embedding spanning graphs
  with maximum degree $\Delta$ in subgraphs of $G(n,p)$ with $p=C(\log
  n/n)^{1/\Delta}$; one for embedding spanning graphs with maximum degree
  $\Delta$ and degeneracy~$D$ in subgraphs of $G(n,p)$ with
  $p=C\big(\log n/n\big)^{1/(2D+1)}$; and one for embedding spanning
  graphs with maximum degree $\Delta$ in $(p,cp^{\max(4,(3\Delta+1)/2)}n)$-bijumbled
  graphs.

  We also consider various applications of these lemmas.
\end{abstract}
\end{frontmatter}

\tableofcontents

\chapter{Introduction, applications, results and proof overview}
\section{Introduction}\label{sec:intro}

Szemer\'edi's \emph{regularity lemma}~\cite{Szemeredi_RL}, originally developed
for the proof of Szemer\'edi's celebrated result on arithmetic
progressions~\cite{Szemeredi_AP}, is one of the most influential tools in
modern Discrete Mathematics. Numerous variants of this lemma, which is an
approximate structure theorem for graphs, have been established for
applications in other areas of mathematics, such as additive number theory,
information theory, and statistical mechanics~\cite{BCO16,COPS15}.

Applications of the regularity lemma include a wealth of results in such
diverse areas as extremal combinatorics, Ramsey theory, property testing,
or discrete geometry.  In such applications the regularity lemma is usually
complemented by the \emph{counting lemma} or the \emph{blow-up lemma}.  The
former allows one to deduce estimates on small substructure counts from the
(finitely sized) approximate structure provided by the regularity lemma.
The latter, on the other hand, is powerful for establishing global
structural properties. More precisely, the blow-up lemma, proved by
Koml\'os, S\'ark\"ozy and Szemer\'edi~\cite{KSS_bl}, permits the embedding
of certain bounded degree spanning graphs.  Alternative proofs can be found
in~\cite{KSS_blalg,RR99,RodRucTar}; for a nice introduction to the blow-up
lemma and explanations about how it is used in applications see the
surveys~\cite{Komlos_survey,KomShoSimSze}.

\smallskip

One limitation of the original regularity lemma is that, because of the error terms, this lemma is suitable only for dense graphs, that is
$n$-vertex graphs with $\Omega(n^2)$ edges. Nonetheless it is desirable to
have equally effective tools at hand for sparse graphs.  The most prominent
example of why such sparse structures are of importance is without doubt
the famous Green-Tao Theorem~\cite{GreenTao} on arithmetic progressions in
the primes, which uses an approximate structure theorem for certain sparse
subsets of the integers.  Moreover, the modern branch of extremal
combinatorics concerned with resilience results (a term coined by Sudakov
and Vu~\cite{SudVu}), which recently received much interest, investigates
such sparse graphs.

Analogues of the regularity lemma that also work in a sparse setting, that
is, for $n$-vertex graphs with $o(n^2)$ edges, do
exist~\cite{Kohayakawa97Szemeredi,Scott}. However, a corresponding counting
lemma simply fails to be true in general (see, e.g., \cite{KLRfour}). This
impediment can be overcome by posing additional restrictions on the graphs
under consideration. The existing counterexamples are known not to occur
in random or certain pseudorandom graphs, and it was a major breakthrough
when recently counting lemmas could finally be established in these
settings: Counting lemmas for subgraphs of random graphs were proved
in~\cite{BalMorSam,ConGowSamSch,SaxTho}, and for subgraphs of pseudorandom
graphs in~\cite{CFZ}.

What was so far missing in this effort to transfer these tools to the
sparse setting was a sparse version of the blow-up lemma. An important step
in this direction was taken in~\cite{ChvRand}, where an embedding lemma for
bounded degree graphs on $cn$ vertices for some very small constant~$c$ in
$n$-vertex subgraphs of random graphs was established. In~\cite{BKT}
it was then shown that methods developed in~\cite{millenium} can be used to
prove a blow-up type result for embedding almost spanning bipartite graphs.
Moreover, in~\cite{BalLeeSam} a sparse embedding lemma for the special case
of spanning triangle factors was proved. Analogues of these partial results
for pseudorandom graphs are not known. But in~\cite{KriSudSza} the
importance of a generalisation of the blow-up lemma to pseudorandom
graphs was acknowledged.

In this paper we provide this missing piece and establish several sparse versions of the
blow-up lemma for random and for pseudorandom graphs. We also discuss a
variety of relatively straightforward applications of these lemmas and
indicate more intricate applications, which will appear elsewhere.

\subsection*{Organisation}

We first motivate our blow-up lemmas in Section~\ref{sec:appl} by
collecting various applications of these lemmas, some of which will be
proved in Chapter~\ref{chap:appproofs} and some of which will be proved in future papers. In Section~\ref{sec:results} we
then provide our blow-up lemmas together with the necessary notation. We also state \emph{regularity inheritance lemmas} which are necessary in applications.
In Section~\ref{sec:proof_overview} we provide an outline of the proofs of the blow-up
lemmas. In Chapters~\ref{chap:toolsetc}--\ref{chap:degen} we give the proofs of our blow-up lemmas. We will describe
the purposes of these various chapters in more detail in the proof outline (Section~\ref{sec:proof_overview}).
We give the proofs of our applications in Chapter~\ref{chap:appproofs}, and finish off with some concluding remarks in Chapter~\ref{chap:concl}.

\subsection*{Notational remarks}

We will routinely omit floor and ceiling signs when they do not
affect the argument. All logarithms are taken to base $2$.

\section{Applications}
\label{sec:appl}

In this section we collect a number of applications of our main results,
establishing new structural properties of random and of pseudorandom
graphs, and improving on a variety of earlier work in this area. We also
present applications in Ramsey theory and the theory of positional games.
We defer proofs of all these results to Chapter~\ref{chap:appproofs}.

Before we provide the results, let us introduce the models of random and
pseudorandom graphs that we use in this paper. The random graph model we
work with is the binomial model $G(n,p)$, where each potential edge is
included in a graph with~$n$ vertices independently with probability
$p=p(n)$, and whose study was pioneered in an influential sequence of
papers by Erd\H{o}s and R\'enyi\footnote{In fact Erd\H{o}s and R\'enyi studied the
  related model $G(n,m)$.} (see, e.g., \cite{Bol01book,JLRbook} for the background).
If $G(n,p)$ has some property with probability tending to~$1$
as~$n$ tends to infinity, we say $G(n,p)$ has this property
\emph{asymptotically almost surely}, abbreviated a.a.s.

The study of pseudorandom graphs was initiated by Thomason~\cite{Tho87},
who asked for a set of easy deterministic properties enjoyed by $G(n,p)$
a.a.s.\ which by themselves imply many of the complex
structural properties we know to hold for $G(n,p)$.
The pseudorandomness notion we use is closely related to the notion
suggested by Thomason, and is among the most widely used ones by now (for example,
the sparse counting lemma in~\cite{CFZ} is developed for this notion as well).
We say a graph $\Gamma$ is \emph{$(p,\beta)$-bijumbled} if for all subsets $X$, $Y\subseteq V(\Gamma)$ we have 
\begin{equation}\label{eq:expander_mixing}
\left|e_\Gamma(X,Y)-p|X||Y|\right|\le \beta\sqrt{|X||Y|}\,,
\end{equation}
where~$e_\Gamma(A,B)$ is the number of pairs in $A\times B$ which form edges in~$\Gamma$. 
The random graph $G(n,p)$ is with
high probability $(p,\beta)$-bijumbled with $\beta=O(\sqrt{pn})$, which
justifies this definition.

Another class of pseudorandom graphs we shall refer to in the applications
are $(n,d,\lambda)$-graphs. These have been studied extensively and are a
special case of bijumbled graphs.
For a graph~$\Gamma$ let $\lambda_1\geq\lambda_2\geq\dots\geq
\lambda_n$ be the eigenvalues of the adjacency matrix of~$\Gamma$. We call
$\lambda(\Gamma):=\max\{|\lambda_2|,|\lambda_n|\}$ the \emph{second eigenvalue}
of~$\Gamma$.  An \emph{$(n,d,\lambda)$-graph}~$\Gamma$ is a $d$-regular graph on $n$
vertices with $\lambda(\Gamma)\leq \lambda$.  The connection between
$(n,d,\lambda)$-graphs and bijumbled graphs is provided by the well-known
\emph{expander mixing lemma} (see, e.g., \cite{AS00}), which states that if
$\Gamma$ is an $(n,d,\lambda)$-graph, then
\begin{equation*}
  \left|e_\Gamma(A,B)-\tfrac{d}{n}|A||B|\right|\le \lambda(\Gamma) \sqrt{|A||B|}
\end{equation*}
for all disjoint subsets $A,B\subset V(\Gamma)$. This implies that $\Gamma$
is $\big(\frac dn,\lambda(\Gamma)\big)$-bijumbled.

\medskip

 \subsection{Universal graphs}
 It is interesting to ask when random, or quasirandom graphs are
 $\cH$-universal for a family of graphs~$\cH$, and recently this class of
 questions enjoyed popularity for various families~$\cH$.  We say $G$
 is \emph{$\cH$-universal} if $H\subseteq G$ for each $H\in\cH$. Two
 particular classes of interest are $\cH(n,\Delta)$, the $n$-vertex graphs
 with maximum degree $\Delta$, and $\cH(n,d,\Delta)$, the $n$-vertex graphs
 with maximum degree $\Delta$ and degeneracy $d$.
  
  An easy corollary of our blow-up lemma
  for random graphs (Lemma~\ref{lem:rg_image}) is that for $\Delta\ge 2$,
  the random graph $G(n,p)$ is a.a.s.\ $\cH(n,\Delta)$-universal if $p\ge
  C\big(\tfrac{\log n}{n}\big)^{1/\Delta}$, reproving a result of
  Dellamonica, Kohayakawa, R\"odl and Ruci\'nski~\cite{DKRR} (for
  $\Delta\ge 3$) and Kim and Lee~\cite{KimLee} (for $\Delta=2$). We omit
  this proof here; it follows along similar lines to that of
  Theorem~\ref{thm:degenuniv} below. Very recently Ferber and Nenadov~\cite{FNuniv} were able to improve on these results, showing $p\ge \big(n^{-1}\log^3 n\big)^{\frac{1}{\Delta-1/2}}$ suffices for $\Delta\ge 3$.

  Turning to $d$-degenerate graphs,
  one would expect that $G(n,p)$ is already $\cH(n,d,\Delta)$-universal for
  some rather smaller $p$ if $d$ is much less than $\Delta$. For $d=1$,
  i.e.\ for bounded degree trees, Montgomery (see~\cite{MontTree})
  announced that one may take $p=C(\Delta)\log^2 n/n$, which is optimal up
  to $\log$ factors. For $d\ge 2$, the best previous result, due to Ferber,
  Nenadov and Peter~\cite{FNP} is that we can take
  $p=\omega\big(\Delta^{12}n^{-1/4d}\log^3 n\big)$. We prove the following
  strengthening\footnote{In fact Ferber, Nenadov and Peter
    proved a universality result in terms of a constraint on the `maximum
    average degree'; in the class of graphs with degeneracy $d$ this
    quantity is between $d$ and $2d$.}, using our blow-up lemma for degenerate graphs (Lemma~\ref{lem:degen}).

  \begin{theorem}\label{thm:degenuniv}
    For each $d\ge 2$, $\Delta\in\NATS$ and each $\gamma>0$ there exists $C$
    such that the random graph $G(n,p)$ is a.a.s.
    \begin{enumerate}[label=\abc]
    \item\label{degenuniv:a} $\cH(n,d,\Delta)$-universal if $p\ge C\big(\tfrac{\log
        n}{n}\big)^{1/(2d+1)}$, and
    \item\label{degenuniv:b} $\cH((1-\gamma)n,d,\Delta)$-universal if $p\ge C\big(\tfrac{\log n}{n}\big)^{1/(2d)}$.
    \end{enumerate}
  \end{theorem}

The almost spanning universality result with $p\ge C(\frac{\log
  n}{n})^{1/2d}$ should be compared to the recent result of Conlon, Ferber,
Nenadov and {\v{S}}kori{\'c}~\cite{CFNS15}, who showed that $G(n,p)$ is
$\cH((1-\gamma)n,\Delta)$-universal when $p=\omega(n^{-1/(\Delta-1)}\log^5
n)$ for $\Delta\ge3 $. Our result is better when $\Delta\ge 2d+1$. 
Conlon and Nenadov (see~\cite[Theorem~3.6]{nenadov16:_ramsey}) showed that the random graph 
$G(n,p)$ is universal for the class of $d$-degenerate graphs on $(1-\gamma)n$ vertices 
for $p\ge \left(\frac{C\log^2n}{n\log\log n}\right)^{1/d}$ improving part~\ref{degenuniv:b} of our 
Theorem~\ref{thm:degenuniv} to the essentially best possible bound on $p$.

Another well-studied question is how big $e(G)$ must be for a
universal graph $G$. Alon and Capalbo showed that the
correct answer is $\Theta\big(n^{2-2/\Delta}\big)$ in~\cite{AC08} for the
the class $\cH(n,\Delta)$, and that if one
further insists that $v(G)=n$ then the correct answer is still
$O\big(n^{2-2/\Delta}\log^{4/\Delta} n\big)$ in~\cite{AlCap} (where the
extra polylog-factor is believed to be
unnecessary). Theorem~\ref{thm:degenuniv} provides sparser
$\cH(n,d,\Delta)$-universal graphs when $d\le \Delta/4$.
  
Finally, we are able to show that sufficiently pseudorandom graphs are
$\cH(n,\Delta)$-universal. This was first suggested as an important
direction of investigation by Krivelevich, Sudakov
and Szab\'o~\cite{KriSudSza}, who established a corresponding result for
$(n,d,\lambda)$-graphs~$\Gamma$ in the special case that one wants to embed
a spanning triangle factor in~$\Gamma$.
In earlier work~\cite{ABHKP14} we were able to improve on their result, showing
that there is $\eps>0$ such that $(p,\eps p^{5/2}n)$-bijumbled $n$-vertex
graphs $G$ with minimum degree $\tfrac12pn$ actually contain the square of
a Hamilton cycle, which implies the theorem of Krivelevich, Sudakov and
Szab\'o with a better bound on $\lambda$. Observe that some minimum degree
condition is necessary in this result since bijumbled graphs may contain
isolated vertices.

For the class $\cH(n,\Delta)$, to the best of our knowledge the only
existing universality result is that obtained by applying the original
blow-up lemma of Koml\'os, S\'ark\"ozy and Szemer\'edi~\cite{KSS_bl}, which
is possible when $\beta\le n^{-\Omega(\Delta^{-2})}$. Using our sparse
blow-up lemma for pseudorandom graphs (Lemma~\ref{lem:psr_main}) we can at
least achieve the correct power of $\Delta$ in the exponent.
  
  \begin{theorem}\label{thm:pseudouniv}
    For each $\Delta\ge 2$ there exists $c>0$ such that for any $p>0$, if
    \[\beta\le cp^{\max(4,3\Delta/2+1/2)}n\,,\] any $n$-vertex
    $(p,\beta)$-bijumbled graph $G$ with $\delta(G)\ge\tfrac12 pn$ is
    $\cH(n,\Delta)$-universal.
  \end{theorem}
  
  We note that Alon and Bourgain~\cite{AlonBourgain} observed that results
  of this type have applications in combinatorial number theory. More
  precisely, they showed that the Cayley sum-graphs $G$ of any
  multiplicative subgroup $U$ of a finite field $\mathbb{F}_q$, where
  $V(G)=U$ and $uv\in E(G)$ whenever $u+v\in U$, is an
  $(n,d,\lambda)$-graph with $n=|U|$ and $\lambda=\sqrt{q}$.  Thus the
  above theorem shows that sufficiently large multiplicative subgroups of
  finite fields contain all bounded-degree additive patterns;
  see~\cite{ABHKP14} or~\cite{AlonBourgain} for the corresponding
  definitions and a more detailed discussion.
  
 \subsection{Partition universality}
 It is also possible to apply our blow-up lemmas to derive certain
 Ramsey-type results.  Given a set $\cH$ of graphs, the graph $G$ is
 \emph{$r$-partition universal} for $\cH$ if in any $r$-colouring of the edge-set
 $E(G)$ there is a colour class which is $\cH$-universal. Kohayakawa,
 R\"odl, Schacht and Szemer\'edi~\cite{ChvRand} showed that for each $r$
 and $\Delta$ there exists $C$ such that if $p\ge C\big(\tfrac{\log
   n}{n}\big)^{1/\Delta}$ then $G(n,p)$ is a.a.s.\ $r$-partition universal
 for $\cH(n,\Delta)$. This result is also an easy corollary of our blow-up
 lemma for random graphs. The proof is, again, omitted, and follows along
 the same lines as that of the following result. Moreover, using our
 blow-up lemma for degenerate graphs, we can provide better
 lower bounds for $r$-partition universality for $\cH(n,d,\Delta)$ if~$d$
 is small compared to~$\Delta$.
  
  \begin{theorem}\label{thm:rpartuniv}
    For each $r,d,\Delta\in\NATS$ there exists $C$ so that if $p\ge
    C\big(\tfrac{\log n}{n}\big)^{1/(2d)}$, then a.a.s.\ $G(Cn,p)$ is
    $r$-partition universal for $\cH(n,d,\Delta)$.
  \end{theorem}
  
  We remark that a result of Conlon and Nenadov
  (see~\cite[Theorem~6.3]{nenadov16:_ramsey}) improves  
  the bound from~\cite{ChvRand} on $p$ to 
   $p\ge C\big(\tfrac{\log n}{n}\big)^{1/(\Delta-1/2)}$ for the class of triangle-free graphs
   of maximum degree $\Delta$ for $\Delta\ge 5$ and is somewhat complementary to our 
   Theorem~\ref{thm:rpartuniv}.\footnote{For newer developments, which appeared
   after the completion of our work, see~\cite{AlBo,ConNenTruj,DraPet}.}
   
  Kohayakawa, R\"odl, Schacht and Szemer\'edi~\cite{ChvRand} also asked if
  a sufficiently pseudorandom graph is $r$-partition universal for
  $\cH(n,\Delta)$. Using our blow-up lemma for pseudorandom graphs
  (Lemma~\ref{lem:psr_main}) we can answer this in the affirmative.

  \begin{theorem}\label{thm:rpartunivdegen}
    For each $r,\Delta\in\NATS$ there exists $C$ such that if $p>0$ and $G$
    is a graph on $Cn$ vertices which is
    $\big(p,\frac{1}{C}p^{\max(4,(3\Delta+1)/2)}n\big)$-bijumbled, then $G$
    is $r$-partition universal for $\cH(n,\Delta)$.
  \end{theorem}
  
  Folkman~\cite{Folkman} showed that for each~$k$ there are $K_{k+1}$-free
  graphs~$G$ which are \emph{Ramsey} for~$K_k$, that is, in any colouring
  of the edges of~$G$ with two colours there is a monochromatic copy of~$K_k$.
  We note that using a proof similar to that of Theorem~\ref{thm:rpartuniv}
  we can show the following
  Folkman-type result (which though can also be proved using the embedding
  lemma from~\cite{ChvRand} instead of our blow-up lemma).

  \begin{theorem}\label{thm:folkmantype}
    For each $r\in\NATS$ and $\Delta\ge 2$, there exists a
    graph which contains no subgraph isomorphic to $K_{2\Delta,2\Delta}$ and is $r$-partition universal for
    $\cH(n,\Delta)$.
  \end{theorem}
  
  We note that the identical method also proves the existence of a $K_{2\Delta}$-free graph which is $r$-partition universal for $\cH(n,\Delta)$. However, one can more easily obtain graphs which are even $K_{\Delta+2}$-free and are $r$-partition universal for $\cH(n,\Delta)$ by starting with a $K_{\Delta+2}$-free graph which is $r$-colour Ramsey for $K_{\Delta+1}$ and blowing it up by $Cn$ for some large $C$. We would like to thank David Conlon for pointing this out to us.
  
  \subsection{Maker-Breaker games}
  The next application of our blow-up lemmas concerns positional games. Let
  us stress that the results we derive have, as such, nothing to
  do with random or pseudorandom graphs, and are in that sense surprising.

  The \emph{Maker-Breaker $H$-game on $K_n$ with bias $b$} is the following
  game. Maker and Breaker take turns to colour the edges of $K_n$ with
  respectively red and blue. In each turn, Maker colours one edge, while
  Breaker colours $b$ edges (and edges may not be recoloured). Maker's aim
  is to create a red copy of $H$, while Breaker's aim is to prevent Maker
  from creating a red $H$. This class of games has been studied
  extensively (see the book of Beck~\cite{Beck08} and a recent survey of
  Krivelevich~\cite{KriSurvey}). If Breaker wins the $H$-game on $K_n$ with
  bias $b$, then obviously Breaker also wins with bias $b+1$; it follows
  that there is a \emph{threshold bias} $b$ for each game which is the
  smallest $b$ such that Breaker wins the $H$-game.
   
  Let us briefly explain how random graphs are connected to these games.
  If either player, playing a randomised strategy, wins a given game with
  positive probability against perfect play from the opponent, then they
  have a deterministic winning strategy (since this is a finite two-person
  game of perfect information without draws). Bednarska and
  {\L}uczak~\cite{BLgame} used this observation to determine the order of
  magnitude of the threshold bias for the $H$-game on $K_n$ for each fixed
  graph $H$. They showed that this bias is $\Theta(p^{-1})$, where $p$ is the threshold
  probability such that $G(n,p)$ contains $H$ robustly, in the sense
  that removing any $\eps(H)pn^2$ edges fails to destroy all copies of
  $H$.

  However, for sequences $H=(H_n)$ with vertex number depending on $n$, much less is known. The critical
  bias for the Hamiltonicity game, with $H=C_n$, was determined only
  recently by Krivelevich~\cite{KriHam}. For more general bounded
  degree graphs it is only known that the critical bias tends to infinity
  as~$n$ tends to
  infinity~\cite{AlonKriSpeSza}\footnote{In~\cite{AlonKriSpeSza} it is only
    explicitly shown that the critical bias is at least two, but
    Krivelevich~\cite{KriPersComm} has explained that the techniques there
    would work for any constant $b$ if $n$ is large enough.}. 
  Using a
  recent result of Ferber, Krivelevich and Naves~\cite{FKN}, which
  informally shows that Maker can a.a.s.\ make a subgraph of $G(n,p)$ with
  minimum degree very close to $pn$, we can show that in fact Maker can win
  the $\cH(n,\Delta)$-universality game on $K_{(1+\delta)n}$ (i.e., Maker's
  graph contains simultaneously each $H\in\cH(n,\Delta)$) against a
  polynomially growing bias, and that this bias can be increased on the
  class $\cH(n,d,\Delta)$ if $d$ is small enough. For the subclasses
  $\cH'(n,\Delta)$ and $\cH'(n,d,\Delta)$ of triangle-free graphs in
  $\cH(n,\Delta)$ and $\cH(n,d,\Delta)$ respectively, we can even win on
  $K_n$.
   
   \begin{theorem}\label{thm:MakerBreaker}
     For each $d,\Delta\in\NATS$ and $\delta>0$ there exists $c>0$ such
     that the following holds.  With bias~$b$ Maker wins
    \begin{enumerate}[label=\abc]
     \item\label{thm:MakerBreaker:a} the $\cH(n,\Delta)$-universality
       game on $K_{(1+\delta)n}$ if $b\le c\big(\tfrac{n}{\log n}\big)^{1/\Delta}$,
     \item\label{thm:MakerBreaker:b} the $\cH'(n,\Delta)$-universality game on $K_n$ if $b\le c\big(\tfrac{n}{\log n}\big)^{1/\Delta}$,
     \item\label{thm:MakerBreaker:c} the $\cH(n,d,\Delta)$-universality game on $K_{(1+\delta)n}$ if $b\le c\big(\tfrac{n}{\log n}\big)^{1/(2d)}$,
     \item\label{thm:MakerBreaker:d} the $\cH'(n,d,\Delta)$-universality game on $K_n$ if $b\le c\big(\tfrac{n}{\log n}\big)^{1/(2d+1)}$.
    \end{enumerate}
  \end{theorem}
   
  We note that the proof of Theorem~\ref{thm:MakerBreaker} uses a
  randomised strategy for Maker which succeeds with high probability
  against any strategy of Breaker, but this strategy does not seem to be
  easy to derandomise. It would be interesting to find explicit
  deterministic Maker strategies for these games, even for substantially
  smaller bias.
   
  In~\cite{ABKNP} we strengthen the above result, showing analogues
  of~\ref{thm:MakerBreaker:b} and~\ref{thm:MakerBreaker:d} for the
  $\cH(n,\Delta)$-universality game and the $\cH(n,d,\Delta)$-universality
  game, respectively. The proof also relies on our blow-up lemmas for
  random graphs, but is more involved.

   \begin{theorem}[Allen, B\"ottcher, Kohayakawa, Naves,
     Person~\cite{ABKNP}]\label{thm:MakerBreakerspan} \mbox{}\\
    For each $d,\Delta\in\NATS$ and $\delta>0$ there exists $c>0$ such that
    the following holds. With bias~$b$ Maker wins
    \begin{enumerate}[label=\abc]
      \item the $\cH(n,\Delta)$-universality game on $K_{n}$ if $b\le c\big(\tfrac{n}{\log n}\big)^{1/\Delta}$,
      \item the $\cH(n,d,\Delta)$-universality game on $K_{n}$ if $b\le
        c\big(\tfrac{n}{\log n}\big)^{1/(2d+1)}$.
      \end{enumerate}
   \end{theorem}
  
  \subsection{Resilience for low-bandwidth graphs}

  Resilience results for random graphs concern the question of how robustly
  the random graph possesses certain properties.  Sudakov and
  Vu~\cite{SudVu} promoted the systematic study of problems of this type,
  which henceforth enjoyed particular popularity in extremal graph theory.
  Our blow-up lemmas are particularly useful for obtaining resilience
  results for spanning or almost spanning subgraphs of random graphs. In
  this section we provide an example.

  The Bollob\'as-Koml\'os conjecture, now the Bandwidth Theorem, proved by
  B\"ottcher, Schacht and Taraz~\cite{BST}, generalises (up to a small
  error term) several embedding results for spanning graphs.
  To state it, we need the concept of \emph{bandwidth}
  $\bw(H)$, defined as the smallest natural $b$ such that there is
  a labelling of the vertices of $H$ by $1,\dots,v(H)$ such that
  $|i-j|\le b$ whenever $ij$ is an edge in that ordering.
   
   \begin{theorem}[Bandwidth Theorem~\cite{BST}]\label{thm:BST} For
     each $\gamma>0$ and $\Delta$ there exists $\beta>0$ and~$n_0$ such
     that for all $n\ge n_0$ the following holds. If $H$ is an $n$-vertex
     graph with maximum degree $\Delta(H)\le\Delta$, chromatic number
     $\chi(H)$, and bandwidth $\bw(H)\le\beta n$, and $G$ is an $n$-vertex
     graph with minimum degree
     $\delta(G)\ge\big(\frac{\chi(H)-1}{\chi(H)}+\gamma\big)n$, then~$H$ is
     a subgraph of~$G$.
   \end{theorem}

   What makes this result valuable is that many important bounded degree
   graph classes have bandwidth~$o(n)$, such as bounded degree
   planar graphs, or more generally any class of bounded degree graphs
   defined by some forbidden minor (see~\cite{BoePruTarWur}).

   B\"ottcher, Kohayakawa and Taraz~\cite{BKT} proved a sparse analogue of
   this theorem for almost-spanning bipartite graphs~$H$, working in
   typical random graphs $G(n,p)$ with $p\ge C\big(\tfrac{\log
     n}{n}\big)^{1/\Delta}$. Huang, Lee and Sudakov~\cite{HLS} on the other
   hand proved an analogue which allows for non-bipartite $H$, but only
   works in random graphs with $p$ constant. Using our blow-up lemma for
   random graphs we can prove the following
   common extension\footnote{Actually Huang, Lee and Sudakov allowed for
     spanning embeddings of some graphs $H$, so that the result here is not
     quite an extension of their theorem, though the result in~\cite{ABET}
     is.} of these results.

   \begin{theorem}\label{thm:sparsebandwidth}
     For each $\gamma>0$ and $\Delta$ there exist $\beta>0$ and $C$ such
     that if $p\ge C\big(\tfrac{\log n}{n}\big)^{1/\Delta}$, then a.a.s.\
     $\Gamma=G(n,p)$ has the following property. If $H$ is a
     $(1-\gamma)n$-vertex graph with maximum degree $\Delta(H)\le\Delta$,
     chromatic number $\chi(H)$, and bandwidth $\bw(H)\le\beta n$, and $G$
     is an $n$-vertex subgraph of $\Gamma$ with minimum degree
     $\delta(G)\ge\big(\tfrac{\chi(H)-1}{\chi(H)}+\gamma\big)pn$, then~$H$
     is a subgraph of~$G$.
   \end{theorem}

   In~\cite{ABET} the following strengthening of this result is proved,
   showing that we can actually have $v(H)=n-Cp^{-2}$, which is optimal, and for some $H$ even $v(H)=n$. 
  
  \begin{theorem}[Allen, B\"ottcher, Ehrenm\"uller and Taraz~\cite{ABET}]
     \label{thm:ABET}
     For each $\gamma >0$ and $\Delta$ there exist $\beta >0$ and $C $ such
     that if $p \ge C\big(\frac{\log n}{n}\big)^{1/\Delta}$, then a.a.s.\
     $\Gamma=G(n,p)$ has the following property.  If $H$ is an $n$-vertex
     graph with maximum degree $\Delta(H)\le\Delta$, chromatic number
     $\chi(H)$, bandwidth $\bw(H)\le\beta n$, such that at least $C
     p^{-2}$ vertices of~$H$ are not contained in
     a triangle, and $G$ is an $n$-vertex subgraph of $\Gamma$ with minimum
     degree $\delta(G)\ge\big(\tfrac{\chi(H)-1}{\chi(H)}+\gamma\big)pn$,
     then~$H$ is a subgraph of~$G$.
   \end{theorem}

   Furthermore a similar resilience statement for sufficiently bijumbled
   graphs, and an improvement when $H$ has small degeneracy, are obtained
   in~\cite{ABET}. The proofs of all these results rely again on our
   blow-up lemmas, but they are significantly harder to obtain than
   Theorem~\ref{thm:sparsebandwidth}, whose proof is a good illustration
   of how one can apply our blow-up lemmas to obtain resilience results.

  \subsection{Robustness of the Bandwidth Theorem}

  Robustness is a measure of how strongly an extremal theorem holds
  (alternative to resilience) and was proposed by Krivelevich, Lee and
  Sudakov~\cite{KLS}. They showed that if $G$ is a \emph{Dirac graph}, that
  is, a graph with $\delta(G)\ge\tfrac12 v(G)$, then there is $C$ such that
  for $p\ge C\tfrac{\log n}{n}$, a.a.s.\ the graph $G_p$ obtained by
  keeping edges of $G$ independently with probability $p$ is
  Hamiltonian. Here we prove a similar extension of the Bandwidth Theorem.

  \begin{theorem}\label{thm:robustbandwidth} 
     For each $\gamma>0$ and $\Delta$ there exist $\beta>0$ and $C$ such
     that if $p\ge C\big(\tfrac{\log n}{n}\big)^{1/\Delta}$, the following
     holds.  If $H$ is an $n$-vertex graph with maximum degree
     $\Delta(H)\le\Delta$, chromatic number $\chi(H)$, and bandwidth
     $\bw(H)\le\beta n$, and $G$ is an $n$-vertex graph with minimum degree
     $\delta(G)\ge\big(\tfrac{\chi(H)-1}{\chi(H)}+\gamma\big)n$, then
     a.a.s.~$H$ is a subgraph of~$G_p$.
  \end{theorem}
   
  We give a fairly detailed sketch proof of this result in
  Section~\ref{sec:robust}. We note that in~\cite{ABEST} we propose a
  strengthening of Theorem~\ref{thm:ABET} which implies that
  under the conditions of
  Theorem~\ref{thm:robustbandwidth}, $G_p$ is in fact a.a.s.\ universal for all $n$-vertex
  graphs $H$ with $\Delta(H)\le\Delta$ and $\bw(H)\le\beta n$.

\section{Statements of the results}
\label{sec:results}

The full versions of our blow-up lemmas are technically complex and so, in
order to make their statement more compact, we will introduce a number of
lengthy definitions. To provide motivation for these definitions we first
state a simplified version of our blow-up lemma for random graphs, which is
useful only in a few applications.

We will then provide the full version of the blow-up lemma for
$G(n,p)$ in Section~\ref{sec:results:Gnp}, and
turn to a blow-up lemma for embedding degenerate graphs into subgraphs of $G(n,p)$
in Section~\ref{sec:results:deg}. In
Section~\ref{sec:results:jumbled} we present the blow-up lemma for
bijumbled graphs. Finally, in Section~\ref{subsec:reginherit} we state two \emph{regularity inheritance} lemmas which are often necessary in applications of our random graph blow-up lemmas.

\subsection{A simplified version}

The dense blow-up lemma is an embedding lemma for super-regular pairs.
Before we can formulate it we need some definitions.  Let $G=(V,E)$ be a
graph and~$A$ and~$B$ be disjoint subsets of~$V$.  We also write $V(G)$ and
$E(G)$ for the set of vertices and edges of~$G$, respectively, and $v(G)$
and $e(G)$ for the sizes of these sets.  For a vertex $v\in V$ we denote by
$N_G(v;A)$ the set of \emph{neighbours} of $v$ in $A$, that is, those vertices
$u\in A$ with $vu\in E$. We write $\deg_G(v;A)$ for its cardinality
$|N_G(v;A)|$.  We write $e_G(A)$ for the number of edges in~$G$ with both
vertices in~$A$ and $e_G(A,B)$ for the number of edges in~$G$ with one
vertex in~$A$ and one in~$B$. Given $X\subset V(G)$, we write $N(X;A):=\bigcup_{v\in X}N_G(v;A)$ for the \emph{joint neighbourhood} in~$A$ of the vertices. We let $\comN_G(X;A):=\bigcap_{v\in X}N_G(v;A)$
denote the \emph{common neighbourhood} in~$A$ of vertices from $X$, and write $\deg_G(X;A)$ for its size $\big|\comN_G(X;A)\big|$. We will often omit the set brackets in $\comN_G\big(\{v_1,\dots,v_\ell\};A\big)$, writing instead $\comN_G(v_1,\dots,v_\ell;A)$. If the graph~$G$ is clear from the context we
sometimes omit it in the subscripts. Furthermore, if we omit the set $A$ we intend $A=V(G)$.

The \emph{density} of the pair $(A,B)$ is $d_G(A,B)=e_G(A,B)/(|A||B|)$.
Let $\eps,d\ge 0$.  The pair $(A,B)$ is called \emph{$(\eps,d)$-regular}
(in $G$) if we have $d(A',B')\ge d-\eps$ for all $A'\subseteq A$ with
$|A'|\ge\eps|A|$ and $B'\subseteq B$ with $|B'|\ge\eps |B|$.

\begin{remark}
  Observe that this differs from the usual definition of $\eps$-regularity
  in that we only require a lower bound on $d(A',B')$. This is sometimes called \emph{dense} in the literature, in particular in~\cite{ChvRand}.
  Note also that with our definition the density of an $(\eps,d)$-regular
  pair $(A,B)$ is only lower-bounded by $d-\eps$. This, though non-standard, has the advantage (over the usual definition where we would require the density of $(A,B)$ to be lower-bounded by $d$) that large subpairs of $(\eps,d)$-regular pairs are $(\eps',d)$-regular (rather than $(\eps',d-\eps')$-regular). We have to use a different regularity parameter, which is unavoidable, but at least we can stick to $d$ for density throughout, and we advocate the use of this definition in future.

  We shall use (sparse versions
  of) this regularity concept whenever we work with random graphs in this
  paper. But for pseudorandom graphs we unfortunately need the usual
  regularity definition with an upper bound on
  $d(A',B')$ as well. Whenever the distinction between the two different concepts
  is essential, which is usually not the case, we will explicitly state which version we are using,
  calling the former `lower-regularity' and the latter `full-regularity'.
  In most of the paper however, we will refer to both versions as
  `regularity'. This will help us to unify much of the proofs of the
  blow-up lemmas for random graphs and for pseudorandom graphs.
  Why it is necessary to use these two different regularity variants is
  explained in Section~\ref{sec:results:jumbled}.
\end{remark}

An $(\eps,d)$-regular pair $(A,B)$ is called
\emph{$(\eps,d)$-super-regular} if for every $u\in A$ we have
$\deg_G(u;B)\ge(d-\eps)|B|$ and for every $v\in B$ we have
$\deg_G(v;A)\ge(d-\eps)|A|$.  The dense blow-up lemma, first proved by Koml\'os, S\'ark\"ozy and Szemer\'edi~\cite{KSS_bl}, then states the following. Let $G$ be a graph formed from a collection of $(\eps,d)$-super-regular pairs with density
$d\gg\eps$, and let $G^*$ be obtained from $G$ by replacing the super-regular pairs with complete bipartite graphs. If $H$ is a graph with maximum degree $\Delta$ which embeds into $G^*$, then $H$ embeds into $G$.

\medskip

This notion of regularity is not meaningful for sparse graphs, because
the definition above implies that \emph{any} pair with $o(n^2)$ edges is
regular (with $d=0$). For obtaining a meaningful sparse version of
regularity and super-regularity we need to relate the density of a pair to the overall
density of the graph under study. This can be obtained by replacing the density in the
definitions above with the $p$-density for a suitable~$p$.

For $0<p<1$, we define the \emph{$p$-density} of the pair $(A,B)$ to be 
$d_{G,p}(A,B)=e_G(A,B)/(p|A||B|)$. In other words, we simply scale the usual density by $p$. Sparse regular pairs are then defined
as follows.  The pair $(A,B)$ is \emph{$(\eps,d,p)$-regular} (in~$G$) if we
have $d_{G,p}(A',B')\ge d-\eps$ for all $A'\subseteq A$ with
$|A'|\ge\eps|A|$ and $B'\subseteq B$ with $|B'|\ge\eps |B|$. Similarly, we
define a pair to be a sparse super-regular pair if it is a sparse regular
pair and satisfies a minimum degree condition.

The setting we
will work with in this subsection is when the graph~$G$ into which we want
to embed is a subgraph of a random graph~$\Gamma$. As mentioned above, for
pseudorandom graphs we will work with a slightly different notion of regularity (see
Section~\ref{sec:results:jumbled}).  The parameter~$p$ will usually be (a constant
factor away from) the density of the ambient graph~$\Gamma$.

\begin{definition}[Sparse super-regularity]
  A pair $(A,B)$ in $G\subset\Gamma$ is called \emph{$(\eps,d,p)$-super-regular} (in~$G$) if it
  is $(\eps,d,p)$-regular and for every $u\in A$ and $v\in B$ we have
  \begin{align}\label{eq:super}  
    \deg_G(u;B) &>(d-\eps) \max\{p |B|, \deg_{\Gamma}(u;B)/2\}\,. \\
    \nonumber \deg_G(v;A) &>(d-\eps) \max\{p |A|, \deg_{\Gamma}(v;A)/2\}\,.
 \end{align}
\end{definition}

We remark that the term $(d-\eps)p|B|$ is a natural lower bound in the
minimum degree condition~\eqref{eq:super} by the following fact, which
easily follows from the definition of regularity (and which we shall use
throughout the paper).

\begin{fact}
  Let~$(A,B)$ be an $(\eps,d,p)$-regular pair. Then less than $\eps|A|$ vertices
  of~$A$ have less than $(d-\eps)p|B|$ neighbours in~$B$.
\end{fact}

Observe though that in~\eqref{eq:super} we have the additional term
$(d-\eps)\deg_{\Gamma}(u;B)/2$, which dominates when~$u$ has an exceptionally
high $\Gamma$-degree into~$B$. The reason why this is necessary will become
clear in a little while after we discuss the need for regularity
inheritance. Note also that it follows from the definition above that in a
super-regular pair $(A,B)$ we have $\deg_\Gamma(u;B)\ge (d-\eps)p|B|$ for
each $u\in A$, since~$G\subset\Gamma$.

Unfortunately, this straightforward extension of the concept of
super-regular pairs to the sparse setting is not on its own enough for a
sparse blow-up lemma. To see this, suppose $A$, $B$ and $C$ are
equally-sized vertex sets, and each pair is super-regular. We would like to
find that there is a spanning triangle factor, but it is possible that for
some $a\in A$ there are no edges between $N_G(a;B)$ and
$N_G(a;C)$. Similarly, if $A$, $B$, $C$, $D$ form a $4$-cycle of
super-regular pairs in $G$, it is nevertheless possible that for some $a\in
A$ there is no $4$-cycle in $G$ using $a$ and one vertex of each of the
other three sets. These problems do not occur for dense
$(\eps,d)$-super-regular pairs for the following reason. In the first example, since $N_G(a;B)$ is
of size at least $(d-\eps)|B|$ and similarly for $N_G(a;C)$, the so-called
slicing lemma guarantees that $\big(N_G(a;B),N_G(a;C)\big)$ is
$(2\eps/d,d)$-regular, that is, it \emph{inherits} regularity from the pair
$(B,C)$ and thus in particular contains edges. A similar argument deals with the second example. The slicing lemma
can be generalised to the sparse setting, and again easily follows from the
definition of regularity.

\begin{lemma}[Slicing lemma]
\label{lem:slice}
  Let $(A,B)$ be an $(\eps,d,p)$-regular pair and $A'\subset A$, $B'\subset
  B$ be sets of sizes $|A'|\ge\alpha|A|$, $|B'|\ge\alpha|B|$. Then
  $(A',B')$ is $(\eps/\alpha,d,p)$-regular.
\end{lemma}

But this unfortunately does not solve the problem indicated above because the size of
$N_G(a;B)$ is only lower bounded by $(d-\eps)p|B|$, which can be tiny
compared to~$B$.

To remedy this, in our blow-up lemmas we will explicitly require that the pair
$\big(N_\Gamma(a;B),C\big)$, and similarly
$\big(N_\Gamma(a;B),N_\Gamma(a;C)\big)$, `inherit' regularity, that is,
they are themselves regular pairs. Since $\deg_G(a;B)$ is a large fraction
of $\deg_\Gamma(a,B)$ by super-regularity, the slicing lemma does show that
regularity of $\big(N_G(a;B),C\big)$ follows from regularity of
$\big(N_\Gamma(a;B),C\big)$, and similarly for
$\big(N_G(a;B),N_G(a;C)\big)$. Note that this is the point where we require
the second term in~\eqref{eq:super}.

Requiring this `inheritance of regularity' is reasonable because it is known that in $(\eps,d,p)$-regular
pairs contained in random graphs~$\Gamma$, for almost all vertices~$a$, the pairs
$\big(N_\Gamma(a;B),C\big)$ and similarly $\big(N_\Gamma(a;B),N_\Gamma(a;C)\big)$ do
inherit sparse regularity in this way. This phenomenon was studied in~\cite{GKRS,KohRod:pairs,ChvRand}. In this paper (see Section~\ref{subsec:reginherit}), we prove this statement with tight bounds on what `almost all vertices' means.

\begin{definition}[Regularity inheritance]
  Let $A$, $B$ and $C$ be vertex sets in
  $G\subset\Gamma$, where~$A$ and~$B$ are disjoint and~$B$ and~$C$
  are disjoint, but we do allow $A=C$. We say that $(A,B,C)$ has
  \emph{one-sided $(\eps,d,p)$-inheritance} if for each $u\in A$ the pair
  $\big(N_\Gamma(u;B),C\big)$ is $(\eps,d,p)$-regular.

  If in addition~$A$ and~$C$ are
  disjoint, then we say that $(A,B,C)$ has \emph{two-sided
    $(\eps,d,p)$-inheritance} if for each $u\in A$ the pair
  $\big(N_\Gamma(u;B),N_\Gamma(u;C)\big)$ is $(\eps,d,p)$-regular.
\end{definition}

In this definition we take neighbourhoods in the ambient
graph~$\Gamma$ instead of~$G$ because this turns out to be easier to handle
in applications (where we need to take care of the vertices in regular
pairs that do not satisfy these inheritance conditions).

The simplified version of our sparse blow-up lemma in random graphs now
states that if we are given an equitable vertex partition
$V_1\dcup\dots\dcup V_r$ of a graph $G\subset\Gamma$ such that all pairs are
super-regular and all triples have one- and two-sided regularity inheritance, then any $r$-partite
graph~$H$ with bounded maximum degree that can be embedded into the complete
$r$-partite graph with partition classes $V_1,\dots,V_r$ can
also be embedded into~$G$. Here, a partition of a
set is called \emph{equitable} if each pair of partition classes differ in
size by at most one.

\begin{lemma}[Simple blow-up lemma for $G(n,p)$]\label{lem:simpBL}
  For all $\Delta\ge 2,r\in\NATS$ and $d>0$ there exist $\eps>0$ and $C$ such that if
 $p\ge C(\log n/n)^{1/\Delta}$, then the random graph $\Gamma=G(n,p)$
 a.a.s.\ has the following property.  
  Let $G\subseteq\Gamma$ and $H$ with $\Delta(H)\le\Delta$ be graphs on~$n$
  vertices. Suppose that~$H$ has an $r$-colouring with colour classes
  $X_1,\ldots,X_{r}$ and that $V(G)=V_1\dcup\cdots\dcup V_{r}$ is an equitable
  partition with $|V_i|=|X_i|$ for all $i\in[r]$ and such that the following conditions hold.
  \begin{enumerate}[label=\rom]
    \item\label{simpBL:super} $(V_i,V_j)$ is
      $(\eps,d,p)$-super-regular in $G$ for each $i,j\in[r]$ with $i\neq j$.
    \item\label{simpBL:one-sided} $(V_i,V_j,V_k)$ has one-sided $(\eps,d,p)$-inheritance
      for each $i,j,k\in[r]$ with $i\neq j$ and $j\neq k$.
    \item\label{simpBL:two-sided}  $(V_i,V_j,V_k)$ has two-sided $(\eps,d,p)$-inheritance
      for each $i,j,k\in[r]$ with $i\neq j$, $j\neq k$, and $k\neq i$.
   \end{enumerate}
   Then $H$ is a subgraph of $G$.
\end{lemma}

Let us briefly comment on the lower bound $p\ge C(\log n/n)^{1/\Delta}$ on
the probability for which our result works. We do not believe that this
bound is in general best possible, though for $\Delta\in\{2,3\}$ it is optimal up to the $\log$-factor.
However, it matches the best known current lower bound~\cite{DKRR} for $p$ such that
$G(n,p)$ is universal for bounded degree spanning~$H$. This
universality result is easily implied by Lemma~\ref{lem:simpBL} with $G=\Gamma$. The problem of improving on~\cite{DKRR} has been prominent in random graph theory for a few years, and seems to be hard. For a more detailed discussion, see Section~\ref{sec:opt:random}.

The restriction $\Delta\ge2$ is necessary for our proof as written. The statement above is true for $\Delta=1$, when in fact we do not require conditions~\ref{simpBL:one-sided} or~\ref{simpBL:two-sided}, but to see this it is easiest to verify Hall's condition in $G$ directly rather than to modify our proof.

What is the difference between this simplified blow-up lemma and our
full-strength blow-up lemma for random graphs, Lemma~\ref{lem:rg_image} (to be introduced below)?
Firstly, in the latter we do not require the partition of~$G$ to be
equitable but allow the partition classes to differ in size by a constant
factor.

Secondly, we do not require all pairs in the partition to be
super-regular (or have regularity inheritance). Instead we will introduce
the concept of a `reduced graph' which encodes where we have super-regular
pairs in our partition. In fact we will even have two reduced graphs $R$
and $R'\subset R$ where the former represents regular pairs and the latter
super-regular pairs, the reason for which will become clear later.

Thirdly,
we do not require two-sided regularity inheritance everywhere in the
partition of~$G$, but only in certain cases where triangles of~$H$ need to
be embedded. This is helpful in some applications, for example in Theorem~\ref{thm:MakerBreaker} we can use the Ferber-Krivelevich-Naves strategy~\cite{FKN} to win Maker-Breaker games with spanning bounded degree triangle-free graphs. This strategy does not allow Maker to win (for example) the spanning triangle factor game, ultimately because of a failure of two-sided regularity inheritance.

Fourthly, the $\eps$ we can choose for our full-scale blow-up lemma does not
depend on~$r$, but only on the maximum degree of the reduced
graph~$R'$. This is a difference also in the dense setting (that is, when
$p=1$) to the blow-up lemma of~\cite{KSS_bl}. In typical applications of
the latter it is necessary to apply this blow-up lemma several times to
embed a spanning~$H$ since this lemma only applies to small subgraphs of
the reduced graph (because $\eps$ depends on~$r$). It is then necessary to
use so-called image restrictions and some manual embedding to connect up
the pieces of $H$, which technically complicates the proofs. Our blow-up
lemma on the other hand avoids this and is formulated with the intention
that in applications it typically only needs to be applied once, which
should make it simpler to use.

Finally, our blow-up lemma permits so-called image restrictions.  Roughly
speaking, image restriction means specifying, for certain vertices of~$H$,
small subsets of~$G$ into which these vertices are to be embedded.  As
explained in the last paragraph we believe that in the dense setting one
usually does not need these image restrictions with our blow-up
lemma. However, in the sparse setting they can be useful, as there may be a
few vertices which do not satisfy the super-regularity or inheritance
conditions required by the blow-up lemma in any sparse-regular partition
of~$G$. Hence we need to embed some $H$-vertices on these vertices `by
hand' before applying the blow-up lemma. These pre-embedded $H$-vertices
then create image restrictions for their neighbours.  Indeed in~\cite{ABET}
precisely this approach is used to prove Theorem~\ref{thm:ABET}.


\subsection{Random graphs}
\label{sec:results:Gnp}

As explained in the previous subsection one of the differences of our blow-up
lemma to the simplified version stated there is that it uses reduced graphs
to specify where edges can be embedded in the graph~$G$ (which is a
subgraph of a random graph~$\Gamma$). We will now first define this
concept and explain what we require of the graph~$H$ that we want to embed
to be `compatible' with such reduced graphs.

The setting in which we work is as follows.  Let $G$ and $H$ be two graphs,
on the same number of vertices, given with partitions
$\cV=\{V_i\}_{i\in[r]}$ and $\cX=\{X_i\}_{i\in[r]}$ of their respective vertex sets.  We call the parts~$V_i$ of~$G$ \emph{clusters}.  We say
that~$\cV$ and~$\cX$ are \emph{size-compatible} if $|V_i|=|X_i|$ for all
$i\in[r]$. Moreover, for $\kappa\geq 1$ we say that $(G,\cV)$ is
\emph{$\kappa$-balanced} if there exists $m\in\NATS$ such that we have
$m\leq |V_i|\leq \kappa m$ for all $i, j\in[r]$ (and thus for all
$X_i\in\cX$ if~$\cV$ and~$\cX$ are size-compatible).
Our goal will be to embed~$H$ into~$G$
respecting these partitions.

As mentioned before, we will have two reduced graphs~$R$ and~$R'\subset R$,
where~$R'$ represents super-regular pairs and~$R$ regular pairs. More
precisely, we require the following properties of~$R$ and~$R'$ and the
partitions~$\cV$ and~$\cX$ of~$G$ and~$H$.

\begin{definition}[Reduced graphs and one-sided inheritance]
  \label{def:RGHpartition}
  Let~$R$ and~$R'$ be graphs on~$r$ vertices.
  \begin{itemize}
  \item $(H,\cX)$ is an \emph{$R$-partition} if each part of $\cX$ is nonempty, and whenever there are edges of $H$ between $X_i$ and $X_j$, the pair $ij$ is an edge of $R$,
  \item $(G,\cV)$ is an \emph{$(\eps,d,p)$-regular $R$-partition} if for each
    edge $ij\in E(R)$ the pair $(V_i,V_j)$ is $(\eps,d,p)$-regular.
  \end{itemize}
  In this case we also say that~$R$ is a \emph{reduced graph} of the partition~$\cV$.
  \begin{itemize}
  \item $(G,\cV)$ is \emph{$(\eps,d,p)$-super-regular on $R'$} if for every
    $ij\in E(R')$ the pair $(V_i,V_j)$ is $(\eps,d,p)$-super-regular.
  \end{itemize}
  Suppose now that $(G,\cV)$ is an $(\eps,d,p)$-regular $R'$-partition.
  \begin{itemize}
  \item $(G,\cV)$ has \emph{one-sided inheritance} on $R'$ 
    if $(V_i,V_j,V_k)$ has one-sided $(\eps,d,p)$-inheritance
    for every $ij, jk\in E(R')$.
  \end{itemize}
  We occasionally also use these concepts when we only work on an induced
  subgraph of~$G$, that is, for a pair $(G,\cV')$
  where~$\cV'$ is a partition of a subset of the vertices of~$G$.
\end{definition}
As in Lemma~\ref{lem:simpBL}, we do require the case $i=k$ in the definition of one-sided inheritance.

It remains to describe where in our partitions we require two-sided
inheritance. For this we first need to define so-called `buffer sets' of
vertices in~$H$, containing `potential buffer vertices'. The purpose of
these buffer sets is that a subset of these vertices, later to be called
`buffer vertices', will be embedded last (for more detailed explanations
see Section~\ref{sec:proof_overview}). For this to work we require the
edges emanating from these vertices and their neighbours to be assigned to
the super-regular pairs given by~$R'$ (and not to other pairs of~$R$). It is only when a potential buffer vertex is contained in a triangle that we
require two-sided inheritance on~$R'$.

The buffer sets can be chosen by the user of the blow-up lemma.
Moreover, we stress that they do not make our blow-up lemma less
powerful in the dense setting than the blow-up lemma of ~\cite{KSS_bl},
because the latter requires all edges of~$H$ to be assigned to
super-regular pairs.


\begin{definition}[Buffer sets and two-sided inheritance]
  \label{def:RGHpartition2}
  Suppose $R'\subset R$ are graphs on $r$ vertices, and $(H,\cX)$ is an
  $R$-partition and $(G,\cV)$ a size-compatible $(\eps,d,p)$-regular $R$-partition.
  We say the family $\tcX=\{\tX_i\}_{i\in[r]}$ of subsets $\tX_i\subset X_i$ is
  an \emph{$(\alpha,R')$-buffer} for $(H,\cX)$ if
 \begin{itemize}
 \item $\tX_i\subset X_i$ and $|\tX_i|\ge\alpha |X_i|$ for all $i\in[r]$  and 
 \item for each $i\in[r]$ and each $x\in\tX_i$, the first and second
   neighbourhood of $x$ \emph{go along $R'$}, that is,
   for each $xy,yz\in E(H)$ with $y\in X_j$ and $z\in X_k$ we have $ij\in E(R')$ and $jk\in E(R')$.
 \end{itemize}
 We also call the vertices in $\tcX$ \emph{potential buffer vertices}.
 Moreover, $(G,\cV)$ has \emph{two-sided inheritance} on $R'$ for $\tcX$ if
 \begin{itemize}
 \item 
   $(V_i,V_j,V_k)$ has two-sided $(\eps,d,p)$-inheritance whenever there is
   a triangle $x_ix_jx_k$ in~$H$ with $x_i\in \tX_i$, $x_j\in X_j$, and
   $x_k\in X_k$.
 \end{itemize}
\end{definition}

We remark that we shall later also occasionally refer to the set of actual
buffer vertices as buffer, when it is clear from the context which set we mean.


Finally, our blow-up lemma allows image restrictions. These generalise the
image restrictions permitted in the dense blow-up lemma. However, in the
sparse setting the necessary conditions become somewhat more involved.  The
idea is as follows. Suppose we wish to embed a graph $H^*$ into a graph $G^*$. Unfortunately $G^*$ does not meet the conditions of our blow-up lemma, typically because regularity inheritance fails. We find a subgraph $G$ of $G^*$ which does meet the conditions of our blow-up lemma, and `pre-embed' some vertices of $H^*$ onto the vertices $V(G^*)\setminus V(G)$. This leaves the induced subgraph $H$ of $H^*$ to embed into $G$.
The image restrictions then originate from these
pre-embedded vertices: If $x\in X_i\subset V(H)$ has neighbours $\{z_1,\dots,z_\ell\}$
in $V(H^*)\setminus V(H)$ which are pre-embedded to
$\{u_1,\dots,u_\ell\}=J_x\subset V(\Gamma)\setminus V(G)$, then $J_x$
restricts the embedding of~$x$ to $I_x=\comN_{G^*}(J_x;V_i)$. In the
following definition we do not explicitly refer to the graphs~$H^*$
and~$G^*$, but only to abstract restricting sets $J_x$, so that we do not
need to include the graphs~$H^*$ and~$G^*$ in our blow-up lemma. For the
same reason we take neighbourhoods in~$\Gamma$ instead of~$G^*$ in this
definition. In addition, to simplify notation, we define an
image restriction set~$I_x$ for each vertex~$x$ of~$H$. For most vertices~$x$,
however, this set is the trivial set $I_x=X_i$ where $X_i$ is the part of
$\cX$ containing~$x$.

\begin{definition}[Image restrictions]\label{def:restrict}  
  Let $R$ be a graph on $r$ vertices, and $(H,\cX)$ be an
  $R$-partition and $(G,\cV)$ a size-compatible $(\eps,d,p)$-regular
  $R$-partition, where $G\subset\Gamma$.
  Let
  $\cI=\{I_x\}_{x\in V(H)}$ be a collection of subsets of $V(G)$, called
  \emph{image restrictions}, and $\cJ=\{J_x\}_{x\in V(H)}$ be a collection of
  subsets of $V(\Gamma)\setminus V(G)$, called \emph{restricting vertices}.
  We say that $\cI$ and $\cJ$ are a
  \emph{$(\rho,\zeta,\Delta,\Delta_J)$-restriction pair} if the
  following properties hold for each $i\in[r]$ and $x\in X_i$.
  \begin{enumerate}[label=\abc]
    \item\label{itm:restrict:Xis} The set $X_i^*\subset X_i$ of  \emph{image restricted}
    vertices in~$X_i$, that is, vertices such that $I_x\neq V_i$, has size $|X_i^*|\leq\rho|X_i|.$ 
    \item\label{itm:restrict:sizeIx} If $x\in X_i^*$, then $I_x\subseteq
    \comN_\Gamma(J_x;V_i)$ is of size at least $\zeta(dp)^{|J_x|}|V_i|$.
    \item\label{itm:restrict:Jx} If $x\in X_i^*$, then $|J_x|+\deg_H(x)\le\Delta$ and 
    if $x\not\in X_i^*$, then $J_x=\emptyset$.  
    \item Each $\Gamma$-vertex appears in at most $\Delta_J$ of the sets of $\cJ$.
    \item\label{itm:restrict:comnbh} We have
    $\big|\comN_\Gamma(J_x;V_i)\big|=(p\pm\eps p)^{|J_x|}|V_i|$.
    \item\label{itm:restrict:reg} If $x\in X_i^*$, for each $xy\in E(H)$ with $y\in X_j$, the pair
    $\big(\comN_\Gamma(J_x;V_i),\comN_\Gamma(J_y;V_j)\big)$ is
    $(\eps,d,p)$-regular in $G$.
 \end{enumerate}
\end{definition}

This definition does indeed generalise the dense image restrictions
of~\cite{KSS_bl}, since~\ref{itm:restrict:Xis} is one of the conditions of~\cite{KSS_bl}, while if $p=1$ and so $\Gamma$ is the complete graph,~\ref{itm:restrict:sizeIx} reduces to the other condition of~\cite{KSS_bl}. We
can set $J_x=\emptyset$ for all $x$, so that~\ref{itm:restrict:Jx}--\ref{itm:restrict:comnbh} become trivial, and~\ref{itm:restrict:reg} follows since $\cX$ is an $R$-partition and $\cV$ is an $(\eps,d,p)$-regular $R$-partition. In the
sparse setting the conditions amount to requiring that pre-embedded vertices
creating image restrictions are embedded on `typical vertices'.

We can now formulate our blow-up lemma for random graphs.

\begin{lemma}[Blow-up lemma for $G(n,p)$]\label{lem:rg_image}
  For all $\Delta\ge 2$, $\DeltaRp$, $\Delta_J$, $\alpha,\zeta, d>0$, $\kappa>1$
  there exist $\eps,\rho>0$ such that for all $r_1$ there is a $C$ such that for
  \[p>C\bigg(\frac{\log n}{n}\bigg)^{1/\Delta}\]
  the random graph $\Gamma=G(n,p)$ a.a.s.\ 
  satisfies the following.
   
  Let $R$ be a graph on $r\le r_1$ vertices and let $R'\subset R$ be a spanning
  subgraph with $\Delta(R')\leq \DeltaRp$.
  Let $H$ and $G\subset \Gamma$ be graphs with $\kappa$-balanced
  size-compatible vertex partitions
  $\cX=\{X_i\}_{i\in[r]}$ and $\cV=\{V_i\}_{i\in[r]}$, respectively, which have
  parts of size at least $m\ge n/(\kappa r_1)$.
  Let $\tcX=\{\tX_i\}_{i\in[r]}$ be a family of subsets of $V(H)$, let
  $\cI=\{I_x\}_{x\in V(H)}$ be a family of image restrictions, and
  $\cJ=\{J_x\}_{x\in  V(H)}$  be a family of restricting vertices.
  Suppose that
  \begin{enumerate}[label=\itmarab{BUL}]
  \item $\Delta(H)\leq \Delta$, $(H,\cX)$ is an $R$-partition, 
    and $\tcX$ is an $(\alpha,R')$-buffer for $(H,\cX)$,
  \item $(G,\cV)$ is an $(\eps,d,p)$-regular $R$-partition, which
    is $(\eps,d,p)$-super-regular on~$R'$, 
    has one-sided inheritance on~$R'$,
    and two-sided inheritance on~$R'$ for $\tcX$,
  \item $\cI$ and $\cJ$ form
    a $(\rho,\zeta,\Delta,\Delta_J)$-restriction pair.
  \end{enumerate}
  Then there is an embedding $\psi\colon V(H)\to V(G)$ such that $\psi(x)\in
  I_x$ for each $x\in H$.
\end{lemma}

\subsection{Degenerate graphs}
\label{sec:results:deg}

We next present a version of our blow-up lemma for random graphs which
allows for smaller edge probabilities in $G(n,p)$ if we want to embed graphs~$H$ whose
maximum degree is much larger than their degeneracy.
The \emph{degeneracy} $\degen(H)$ of a graph~$H$ is the smallest
integer~$\ell$ such that~$H$ is \emph{$\ell$-degenerate}, that is,
each induced subgraph of~$H$ has
minimum degree at most~$\ell$. Equivalently, there is an order of~$V(H)$ such
that each vertex has at most~$\ell$ neighbours before that vertex in the
order. For example, trees have degeneracy~$1$ and planar graphs have
degeneracy at most~$5$. We remark that (a variant of) the degeneracy
determines the exponent in the probability~$p$, but we nevertheless require
the maximum degree of~$H$ to be bounded by a constant~$\Delta$ in the lemma
below. The constant in the probability~$p$ depends on~$\Delta$.  For this
blow-up lemma we use the same notion of regularity, super-regularity,
reduced graphs, inheritance, buffer sets and image restrictions as for
Lemma~\ref{lem:rg_image}.

Given an order $\tau$ on $V(H)$ and a family $\cJ$ of image restricting
vertices, we define 
\[\pitau(x):=|J_x|+\big|\{y\in N_H(x)\colon\tau(y)<\tau(x)\}\big|\,.\]
In other words, $\pitau(x)$ denotes the number of neighbours of $x$ that
come before $x$, including the restricting vertices~$J_x$.

Quantifying the dependence on~$\tau$ of the probability we can work with is somewhat complex.  Firstly, for a vertex~$x$
we distinguish whether it has neighbours~$y$ succeeding~$x$ in the
order~$\tau$, in which case we will need to maintain one-sided
inheritance properties when embedding~$x$, or even two such neighbours~$y$
and~$z$ which form an edge, in which case we may need to maintain
two-sided inheritance. Secondly, we put stricter requirements
on~$x$ if it has some potential buffer vertices in its neighbourhood (since we need
to maintain certain properties to embed the buffer vertices;
see~\ref{ord:Dx} and~\ref{ord:NtX}).  Thirdly,
we put even stricter requirements on vertices~$x$ which are image
restricted or have preceding neighbours `far away' from~$x$ (see
\ref{ord:halfD}). Since this is a very restrictive condition however, we
allow a very small set~$X^e$ of exceptional vertices which are exempted from
this rule.

\begin{definition}[$(D,p,m)$-bounded order]\label{def:Dpm_bdd_order} 
  Let~$H$ be a graph given with buffer sets $\tcX=\{\tX_i\}_{i\in[r]}$ and
  a restriction pair~$\cI=\{I_x\}_{x\in V(H)}$ and~$\cJ=\{J_x\}_{x\in V(H)}$.
  Let~$\tX=\bigcup_{i\in[r]}\tX_i$.
  Let~$\tau$ be an ordering of $V(H)$ and $X^e\subset V(H)$.
  Then~$\tau$ is a \emph{$(D,p,m)$-bounded order} for~$H$, $\tcX$,
  $\cI$ and $\cJ$
  with \emph{exceptional set} $X^e$ if the following conditions are
  satisfied for each $x\in V(H)$.
  \begin{enumerate}[label=\itmarab{ORD}]
   \item\label{ord:Dx} Define \[
     D_x:=\begin{cases}
       D-2 & \text{if there is $yz\in E(H)$ with $y,z\in N_H(x)$ and $\tau(y),\tau(z)>\tau(x)$}\\
       D-1 & \text{else if there is $y\in N_H(x)$ with $\tau(y)>\tau(x)$}\\
       D & \text{otherwise}\,.
     \end{cases}\]
     We have $\pitau(x)\le D_x$, and if $x\in N(\tX)$ even $\pitau(x)\le D_x-1$. Finally, if $x\in\tX$ we have $\deg(x)\le D$.
    \item\label{ord:halfD} One of the following holds:
      \begin{itemize}
        \item $x\in X^e$,
        \item $\pitau(x)\le \frac12 D$,
        \item $x$ is not image restricted and every neighbour~$y$ of~$x$
          with $\tau(y)<\tau(x)$ satisfies $\tau(x)-\tau(y)\le p^{\pitau(x)}m$.
      \end{itemize}
    \item\label{ord:NtX} If $x\in N(\tX)$ then all but at most $D-1-\max_{z\not\in X^e}\pitau(z)$
      neighbours~$y$ of $x$ with $\tau(y)<\tau(x)$ satisfy
      $\tau(x)-\tau(y)\le p^D m$.
 \end{enumerate}
\end{definition}

In order to obtain the best possible value of $p$, our aim is always to find an order $\tau$ which is $(D,p,m)$-bounded and minimises the value of $D$. To give some intuition about what typically is possible, we refer to some of the results of Section~\ref{sec:appl}.

We can typically obtain $D\le 2\degen(H)+1$, taking $\tau$ to be a
degeneracy order for $H$. This gives us $\pitau(x)\le\degen(x)$, so
that~\ref{ord:halfD} and~\ref{ord:NtX} hold trivially (the former with room
to spare, the latter not). Furthermore, most of~\ref{ord:Dx} is trivially
satisfied (though one needs to observe that if $\degen(H)=1$ then $H$
contains no triangle). The only point which is unclear is the restriction
$\deg(x)\le D$ for $x\in\tX$. In practice (see for example the proofs of
Theorems~\ref{thm:degenuniv} and~\ref{thm:MakerBreaker})
one often can obtain this. The reason is that any graph~$H$ contains many
vertices of degree at most $2\degen(H)$, and typically we can find a
partition of $H$ in which these vertices are well-distributed among the
parts.

In the event that we do not require a spanning embedding, but can afford to leave a small fraction of vertices in each part uncovered, we can obtain the slightly stronger $D\le 2\degen(H)$, as in Theorems~\ref{thm:rpartuniv} and~\ref{thm:MakerBreaker}. The reason for this is that we can `pad' an almost-spanning $H$ by adding isolated vertices to obtain a spanning $H'$ to which we apply Lemma~\ref{lem:degen} (see the proof of Theorem~\ref{thm:rpartuniv} for details). We again use a degeneracy order $\tau$, but this time let $\tX$ be the isolated vertices.

When we have a degeneracy order $\tau$ with the extra (bandwidth-type) property that all edges go between vertices very close together in the order, we can typically choose $D\le \degen(H)+3$. The reason is that in this case~\ref{ord:halfD} and~\ref{ord:NtX} are automatically satisfied, and we only need to worry about~\ref{ord:Dx}. Again, we need to be able to choose potential buffer vertices of degree at most $\degen(H)+3$, but in applications this is usually possible. This applies, for example, in the case that $H$ is an $F$-factor.

So far we did not mention the set $X^e$, or image restrictions. In applications often only very few vertices need to be image restricted, and we can typically put all of them in $X^e$. This is, for example, critical to obtaining a resilience result for bounded-degree trees with $p=C\big(\tfrac{\log n}{n}\big)^{1/3}$ in~\cite{ABET}.

In general, the exceptional set $X^e$ gives us the possibility to
specify a small set of vertices in~$H$ for which the value of $\pitau$ does
not have to be bounded by $D/2$. This is for example useful when a few
vertices have been embedded `by hand' before the use of the blow-up lemma,
creating image restrictions. Since we cannot usually select these vertices
according to the degeneracy order of~$H$,  a few image
restricted vertices will have more neighbours embedded `by hand' than
if~$H$ was embedded in the degeneracy order. These vertices should then be
put in~$X^e$.

Moreover, in~\ref{ord:NtX} for each buffer neighbour~$x$ we allow a few (depending on~$D$ and~$\tau$)
exceptions to the rule that embedded neighbours of~$x$ need to be close
to~$x$ in the order. This is useful because vertices~$y$ of exceptionally high
degree often have to come relatively early in the order to
satisfy~\ref{ord:Dx} and~\ref{ord:halfD}, and hence cannot necessarily be
close to~$x$.

We can now state our blow-up lemma for degenerate graphs. The only
difference to the previous blow-up lemma is that we ask for a $(D,p,\eps
n/r_1)$-bounded order of~$H$, that we allow only vertices with degree at
most~$D$ as potential buffer vertices, and that the exponent in the bound
on~$p$ is determined by~$D$. We remark that for some graphs~$H$, for
example $\ell$-regular graphs, which have degeneracy~$\ell$, this bound on~$p$ is
worse than the bound in Lemma~\ref{lem:rg_image}. However, for trees or planar graphs it is often much better.

\begin{lemma}[Blow-up lemma to embed degenerate graphs in $G(n,p)$]
  \label{lem:degen} \mbox{}\\
  For any $\Delta\ge 2$, $\DeltaRp$, $\Delta_J$, $D$, 
  $\alpha,\zeta, d>0$, $\kappa>1$ there exist
  $\eps,\rho>0$ such that for all $r_1$ there is a $C$ such that for
  \[p\ge C\bigg(\frac{\log n}{n}\bigg)^{1/D}\] 
  the random graph
  $\Gamma=G(n,p)$ a.a.s.\ satisfies the following.
   
  Let $R$ be a graph on $r\le r_1$ vertices and let $R'\subset R$ be a
  spanning subgraph with $\Delta(R')\leq \DeltaRp$.  Let $H$ and $G\subset
  \Gamma$ be graphs with $\kappa$-balanced, size-compatible vertex
  partitions $\cX=\{X_i\}_{i\in[r]}$ and $\cV=\{V_i\}_{i\in[r]}$,
  respectively, which have parts of size at least $m\ge n/(\kappa r_1)$.
  Let $\tcX=\{\tX_i\}_{i\in[r]}$ be a family of subsets of $V(H)$, let $\cI=\{I_x\}_{x\in
    V(H)}$ be a family of image restrictions, and $\cJ=\{J_x\}_{x\in V(H)}$
  be a family of restricting vertices.  
  Let $\tau$ be an order of $V(H)$ and $X^e\subset V(H)$ be a set of size
  $|X^e|\le\eps p^{{\max_{x\in X^e}\pitau(x)}}n/r_1$. 
  Suppose that
  \begin{enumerate}[label=\itmarab{DBUL}]
  \item $\Delta(H)\leq \Delta$, $(H,\cX)$ is an $R$-partition, 
    $\tcX$ is an $(\alpha,R')$-buffer for $(H,\cX)$,
  \item $(G,\cV)$ is an $(\eps,d,p)$-regular $R$-partition, which is 
    $(\eps,d,p)$-super-regular on $R'$, 
    has one-sided inheritance on~$R'$,
    and two-sided inheritance on~$R'$ for $\tcX$,
  \item $\cI$ and $\cJ$ form
    a $(\rho,\zeta,\Delta,\Delta_J)$-restriction pair.
  \item\label{dbulcon:4} $\tau$ is a $(D,p,\eps n/r_1)$-bounded order for~$H$, $\tcX$,
    $\cI$, $\cJ$ with
    exceptional set~$X^e$.
  \end{enumerate}
  Then there is an embedding $\psi\colon V(H)\to V(G)$ such that
  $\psi(x)\in I_x$ for each $x\in H$.
\end{lemma}

Let us briefly indicate how this lemma performs in practice. As we see in for example Theorem~\ref{thm:MakerBreaker}, we can use it to obtain spanning embeddings in some situations with $D=2\degen(H)+1$, thus with $p\ge C\big(\tfrac{\log n}{n}\big)^{1/(2\degen(H)+1)}$. In~\cite{ABET} we obtain the promised version of Theorem~\ref{thm:ABET} for degenerate graphs, using the same value of $D$. However, to obtain this result, we have to pre-embed some vertices of $H$, before applying Lemma~\ref{lem:degen}, and we thus have to make some alterations to the degeneracy order to obtain an order $\tau$. The neighbours of the pre-embedded vertices then threaten to destroy $(D,p,m)$-boundedness, and it is critical that we can put them into~$X^e$. 

As a second example, if the degeneracy order (or something close to it) on $H$ happens to have the property that all edges go only a short distance, then conditions~\ref{ord:NtX} and~\ref{ord:halfD} become trivially true, and we can obtain much stronger results. For example, in~\cite{ABET} we give a resilience result for $F$-factors in $G(n,p)$ with $p\ge C\big(\tfrac{\log n}{n}\big)^{1/(\degen(F)+3)}$, that is, using Lemma~\ref{lem:degen} with $D=\degen(F)+3$. In this application we do not need any exceptions, and can set
$X^e=\emptyset$.

\subsection{Bijumbled graphs}
\label{sec:results:jumbled}

Finally, we provide a blow-up lemma for embedding bounded degree graphs
into subgraphs of sufficiently bijumbled graphs. As indicated earlier, in
bijumbled graphs we use a stronger notion of regularity, which in addition
to the lower bound also requires an upper bound on the edge density of
subpairs.

\begin{definition}[Regularity in bijumbled graphs]
  In bijumbled graphs, we say that $(X,Y)$ is \emph{$(\eps,d,p)$-regular} if
  there is a $d'\ge d$ such that for any $X'\subset X$ with
  $|X'|\ge\eps |X|$ and $Y'\subset Y$ with $|Y'|\ge\eps|Y|$, we have
  $d_p(X',Y')= d'\pm\eps$. When we want to be it clear that we are working
  with this regularity concept we also call such a pair
  \emph{$(\eps,d,p)$-fully-regular}.
\end{definition}

The reason why we use lower-regularity in random graphs is that the
regularity inheritance we use in random graphs (see Section~\ref{subsec:RI})
provides only lower-regular pairs. On the other hand, we use
full-regularity in bijumbled graphs because the regularity inheritance we
prove for bijumbled graphs (again, see Section~\ref{subsec:RI}) requires
fully-regular pairs.  Unfortunately, we do not know, in either case, how to
prove a regularity inheritance statement which works with the `other'
version of regularity.


All other parts of our proofs work for both lower-regularity and
full-regularity. Since we would like to use many of these parts for
random graphs as well as for bijumbled graphs, we let regular pairs mean lower-regular pairs
whenever we work in random graphs, and fully-regular
pairs whenever we work in bijumbled graphs.

In particular, super-regularity, reduced graphs, inheritance, buffer sets
and image restrictions for bijumbled graphs are defined exactly as for random
graphs, once the regularity concept there is replaced with the regularity
concept defined here. Our blow-up lemma for bijumbled graphs then has
analogous requirements and conclusions as Lemma~\ref{lem:rg_image}, with
the only exception that the image restrictions allowed here are much
weaker: if~$H$ has maximum degree~$\Delta$ we can only image restrict about
a $p^\Delta$-fraction of the vertices in any given partition class of~$H$,
rather than a small constant fraction.

\begin{lemma}[Blow-up Lemma for bijumbled graphs]\label{lem:psr_main}
  For all $\Delta\ge 2$, $\DeltaRp$, $\Delta_J$, $\alpha,\zeta, d>0$, $\kappa>1$
  there exist $\eps,\rho>0$ such that for all $r_1$ there is a $c>0$ such
  that if $p>0$ and 
  \[\beta\le cp^{\max(4,(3\Delta+1)/2)}n\]
  any $(p,\beta)$-bijumbled graph~$\Gamma$ on $n$ vertices satisfies the following.
   
  Let $R$ be a graph on $r\le r_1$ vertices and let $R'\subset R$ be a spanning
  subgraph with $\Delta(R')\leq \DeltaRp$.
  Let $H$ and $G\subset \Gamma$ be graphs given with $\kappa$-balanced,
  size-compatible vertex partitions 
  $\cX=\{X_i\}_{i\in[r]}$ and $\cV=\{V_i\}_{i\in[r]}$, respectively, which have parts of size at
  least $m\ge n/(\kappa r_1)$. Let 
  $\tcX=\{\tX_i\}_{i\in[r]}$ be a family of subsets of $V(H)$, let
  $\cI=\{I_x\}_{x\in V(H)}$ be a family of image restrictions, and
  $\cJ=\{J_x\}_{x\in  V(H)}$  be a family of restricting vertices.
  Suppose that
  \begin{enumerate}[label=\itmarab{JBUL}]
  \item $\Delta(H)\leq \Delta$, $(H,\cX)$ is an $R$-partition,
    $\tcX$ is an
    $(\alpha,R')$-buffer for $(H,\cX)$,
  \item $(G,\cV)$ is an $(\eps,d,p)$-regular $R$-partition, which is 
    $(\eps,d,p)$-super-regular on $R'$, 
    and has one-sided inheritance on~$R'$,
    and two-sided inheritance on~$R'$ for $\tcX$,
  \item $\cI$ and $\cJ$ form
    a $(\rho p^\Delta,\zeta,\Delta,\Delta_J)$-restriction pair.
  \end{enumerate}
  Then there is an embedding $\psi\colon V(H)\to V(G)$ such that $\psi(x)\in
  I_x$ for each $x\in H$.
\end{lemma}

We remark that we do not believe that the bound on~$\beta$ in this result
is even close to being best possible (this is discussed further in the
concluding remarks, Section~\ref{sec:opt:jumbled}). However, this blow-up lemma
is the first general embedding result which allows the embedding of
spanning structures in bijumbled graphs (see for example
Theorem~\ref{thm:pseudouniv}). Previously such embedding results were known
only for some special subgraphs~$H$, such as Hamilton
cycles~\cite{KriSud}, triangle factors~\cite{KriSudSza}, or powers of
Hamilton cycles~\cite{ABHKP14}.  Moreover, our blow-up lemma allows for
resilience results in pseudorandom graphs. Previously, the only results for
large subgraphs in this direction that we are aware of deal with cycles. Specifically, Sudakov and Vu~\cite{SudVu} found the local resilience of $(n,d,\lambda)$-graphs with respect to Hamiltonicity, Dellamonica, Kohayakawa, Marciniszyn, and Steger~\cite{dellamonica2008} found the local resilience of bijumbled graphs with respect to containing long cycles, and Krivelevich, Lee and Sudakov~\cite{KriLeeSud}, with a stronger bijumbledness requirement, found the local resilience with respect to pancyclicity (containing cycles of all lengths).

\subsection{Inheritance of regularity}
\label{subsec:reginherit}
The final results we would like to highlight are the regularity inheritance lemmas for random graphs we mentioned earlier. These are necessary not only in many applications of our sparse blow-up lemmas, such as~\cite{ABET}, but are also useful in other random graph contexts, for example in~\cite{CTdense,ABKR}

 Both statements rely crucially on the regularity inheritance work of Gerke, Kohayakawa, R\"odl and Steger~\cite{GKRS}, and also use ideas from~\cite{ChvRand}. The first lemma gives an upper bound on the number of vertices which fail to have one-sided lower-regularity inheritance.
 
   \begin{lemma}[One-sided lower-regularity inheritance in $G(n,p)$]
    \label{lem:oneRI}
    For each $\eps',d>0$ there are $\eps_0>0$ and
    $C$ such that for all $0<\eps<\eps_0$ and $0<p<1$, a.a.s.\
    $\Gamma=G(n,p)$ has the following property. 
    Let $G\subset\Gamma$ be a graph and $X,Y$ be disjoint subsets of
    $V(\Gamma)$. 
    If $(X,Y)$ is $(\eps,d,p)$-lower-regular in~$G$ and
    \[|X|\ge C\max\big(p^{-2},p^{-1}\log n\big) \quad\text{and}\quad|Y|\ge Cp^{-1}\log (en/|X|)\,,\] 
    then 
    for at most $Cp^{-1}\log (en/|X|)$ vertices $z\in V(\Gamma)$ the pair
    $\big(N_\Gamma(z;X),Y\big)$ is not $(\eps',d,p)$-lower-regular in $G$.
  \end{lemma}
  
  The corresponding two-sided inheritance lemma is the following.
  
  \begin{lemma}[Two-sided lower-regularity inheritance in $G(n,p)$]
    \label{lem:twoRI}
    For each $\eps',d>0$ there are $\eps_0>0$ and
    $C$ such that for all $0<\eps<\eps_0$ and $0<p<1$, a.a.s.\
    $\Gamma=G(n,p)$ has the following property. 
    Let $G\subset\Gamma$ be a graph and $X,Y$ be disjoint subsets of
    $V(\Gamma)$. 
    If $(X,Y)$ is $(\eps,d,p)$-lower-regular in~$G$ and
    \[|Y|\ge|X|\ge C\max\big(p^{-2},p^{-1}\log n\big)\,,\]
    then there are at most $C\max\big(p^{-2},p^{-1}\log (en/|X|)\big)$ vertices $z\in V(\Gamma)$ such that
    $\big(N_\Gamma(z;X),N_\Gamma(z;Y)\big)$ is not $(\eps',d,p)$-lower-regular in $G$.
  \end{lemma}
  
 These two results are quite similar to~\cite[Proposition~15]{ChvRand}. The
 crucial difference are the bounds on the sizes of~$X$ and~$Y$ as well as
 the bound on the number of vertices~$z$ that do not preserve
 regularity. We remark that the bounds given
 in~\cite[Proposition~15]{ChvRand} are (roughly) equivalent to our bounds
 when~$p$ is as small as possible in our blow-up lemmas, that is
 $p=\Theta\big((\log n/n)^{1/\Delta}\big)$, but for bigger~$p$ our results are stronger.
 
 We would like to stress that these results give sharp bounds, up to the constant $C$, on the number of vertices which may fail to inherit lower-regularity. We prove both theorems, and that they are sharp, in Section~\ref{subsec:RI}.
 
 Finally, we note that corresponding inheritance lemmas (which take as input and give as output fully-regular pairs) for bijumbled graphs were first proved by Conlon, Fox and Zhao~\cite{CFZ}. The improved versions of these lemmas which we state in Section~\ref{subsec:RI} were proved in~\cite{ABSS}.

\section{Proof overview}
\label{sec:proof_overview} 

The following is a high-level overview of the proofs of our blow-up
lemmas. 
In these proofs we merge ideas that were also used in proofs of dense
blow-up lemmas~\cite{BoeKohTarWue,KSS_bl,RR99} and of the sparse embedding lemma
in~\cite{ChvRand} with many new
ingredients.
Of course there are differences between the proofs of
Lemmas~\ref{lem:rg_image}, \ref{lem:degen}, and~\ref{lem:psr_main} but at
the level of this overview we will avoid mentioning most of them. 

In each of our blow-up lemmas we want to embed a graph~$H$, which is given
with a partition~$\cX$ and potential buffer sets~$\tcX$, into a graph
$G\subseteq\Gamma$ with a compatible super-regular partition~$\cV$,
possibly with some image restrictions, and where~$\Gamma$ is a sparse
random or bijumbled graph.  To avoid technical details, we will in this
overview largely ignore the image restrictions; these turn out not to play a large r\^ole in the proofs.  Our embedding strategy
for~$H$ is comprised of different randomised embedding procedures. Before we
can apply them though, we need to prepare the graphs~$H$ and~$G$ by
subpartitioning their partition classes suitably (details of what we need can be found
in Sections~\ref{subsec:partH} and~\ref{subsec:partG}).

We start by subpartitioning the partition classes of~$H$ such that any pair
of vertices in a part of the new partition is at distance at least ten,
which is possible by a trick first used by Alon and F\"uredi~\cite{AlFu}, and in a blow-up lemma setting by R\"odl and Ruci\'nski~\cite{RR99},
that relies on the Hajnal-Szemer\'edi theorem
(Theorem~\ref{thm:HajSze}). Having only distant vertices in a part will
provide sufficient independence of these vertices in our randomised
embedding procedures.  To get a compatible new partition of~$G$, we
subpartition the clusters of~$G$ randomly.

Next, we subpartition each cluster~$V_i$ of $G$ randomly into several parts
\[V_i=\Vmain_i\dcup \Vq_i\dcup \Vc_i\dcup \Vbuf_i\,.\] The first of these
sets is large (of size $(1-3\mu)|V_i|$ for some small $\mu$), while the
remaining three are much smaller (of size $\mu|V_i|$). Because the subpartitionings were performed
randomly, subparts of super-regular pairs (given by~$R'$) maintain
super-regularity. The reason for doing this is that we perform the embedding in stages, and these stages require separate parts.   This idea is also used
in~\cite{BoeKohTarWue}. 

We also partition $X_i$ into a large part $\Xmain_i$ (of size at most
$(1-4\mu)|X_i|$) and a small part $\Xbuf_i$ (of size $4\mu|X_i|$), called
the set of \emph{buffer vertices}. The latter set is required to be a subset of the
potential buffer vertices $\tX_i$ in $X_i$. We also require various extra
properties of $\Xbuf_i$. In the proof of the random graphs
blow-up lemma,
Lemma~\ref{lem:rg_image}, we also include another small subpart~$\Xc_i$ in
this partitioning step. This is due to the fact that in this proof we have
to treat the case that most of the vertices $\tX_i$ are contained in copies
of $K_{\Delta+1}$ specially (in order to obtain the claimed bound on the
probability~$p$). We need the sets~$\Vc_i$ to embed this small subpart.

At this point~$H$ and~$G$ are prepared for the embedding. We now
describe our randomised embedding approach.  Firstly, we make use of a
\emph{random greedy algorithm} (RGA) to embed $\Xmain_i$ into
$\Vmain_i$ (see Section~\ref{sec:rga} for the simplest version used in the proof of Lemma~\ref{lem:rg_image}). An algorithm of this type was also
used by Koml\'os, S\'ark\"ozy and Szemer\'edi~\cite{KSS_bl} to prove the
dense blow-up lemma. The RGA embeds the sets~$\Xmain_i$ vertex by vertex,
in each step avoiding some \emph{bad} vertices of~$G$ and embedding the
current vertex~$x\in\Xmain_i$ into its so-called \emph{candidate set}.  If
$\psi$ is the \emph{partial embedding} of~$H$ into~$G$ that we have
constructed so far, then the candidate set $C(x)\subset V_i$ is the set of
vertices adjacent in $G$ to each $\psi(y)$ such that~$y$ is an embedded
neighbour of~$x$ in~$H$. However, some of the vertices in~$C(x)$ may have been used as images
for other vertices already, so we let $A(x)=C(x)\setminus\im(\psi)$ be
the \emph{available candidate set} for $x$. Obviously, $x$ has to be
embedded into~$A(x)$ in order to obtain an embedding.

In order to succeed with this strategy, we have to maintain certain
properties for the partial embedding $\psi$, which we call \emph{good
  partial embedding} properties (see Section~\ref{sec:gpe}).  We remark
that the properties of a restriction pair ensure that the trivial partial
embedding, which we have at the beginning when no vertices of $H$ are
embedded yet, is a good partial embedding.  This fact means that in the
proofs we usually do not need to distinguish between vertices which are and
are not image restricted (and it also justifies that we ignore image
restrictions in this overview).

Among other properties, such as regularity properties,
in a good partial embedding we require that for each unembedded~$x\in X_i$ the
set~$C(x)$ is large, that is of size $\Omega(d^\ell p^\ell|\Vmain_i|)$
where~$\ell$ is the number of already embedded neighbours of~$x$. In order
to maintain this property we need to avoid certain bad vertices $B(x)$ when
embedding~$x$. More precisely, when we embed~$x$ this leads to a change of
the candidate sets~$C(y)$ of unembedded neighbours~$y$ of~$x$. The set $B(x)$
contains the vertices~$v\in V_i$ that would lead to some $C(y)$ becoming small (or
have some other bad properties that would prevent us from maintaining a
good partial embedding; see Section~\ref{subsec:bad_vertices}). The RGA then
embeds~$x$ \emph{uniformly at random} into $\Vmain_i\cap A(x)\setminus
B(x)$.

That we choose the images randomly helps us in several ways. Most
immediately, we can use it to show that the sets $C(y)$ for unembedded
vertices $y\in\Xmain_i\cup\Xbuf_i$ are `uniformly' distributed over
$\Vmain_i$. In particular, this means that they tend not to be contained
entirely in $\im(\psi)$, which, together with the fact that $B(y)$ is always small, implies that the set $\Vmain_i\cap
A(y)\setminus B(y)$ to which we wish to embed a future $y\in\Xmain_i$ is
usually not small.

Unfortunately though, the RGA will not succeed in embedding every vertex of
$\Xmain_i$, because occasionally we will come across a vertex $x$ such that
$\Vmain_i\cap A(x)\setminus B(x)$ is small. We put such vertices~$x$ in a
\emph{queue} $\Xq_i\subseteq\Xmain_i$. The use of such a queue appears
already in~\cite{KSS_bl}.  The `uniform' distribution of the sets $C(x)$
allows us to show that $\Xq_i$ remains much smaller than the so far
untouched $\Vq_i$. We will show
that this allows us to embed
$\Xq_i$ into $\Vq_i$, maintaining a good partial embedding. We note that in
the proof of the random graphs blow-up lemma, Lemma~\ref{lem:rg_image}, we
perform this embedding after we have embedded all other vertices of
$\Xmain_i$, using a matching strategy similar to that of Kohayakawa, R\"odl,
Schacht and Szemer\'edi~\cite{ChvRand}. In the proofs of the other two
blow-up lemmas, Lemma~\ref{lem:degen} and~\ref{lem:psr_main}, however, we
embed the vertices of $\Xmain_i$ in the given order, again using a
random greedy strategy.
In either case, the underlying idea is to show
that the only reason why we might fail to embed $\Xq_i$ into $\Vq_i$ is
that~$G$ contains a `dense spot', contradicting the fact that our random or
bijumbled $\Gamma$ does not contain such dense spots.

At this point, all vertices of~$\Xmain_i$ are embedded, and it remains to embed
the carefully chosen buffer vertices $\Xbuf_i$. Since any two buffer
vertices are at distance at least $10$ in $H$, in particular all neighbours
of buffer vertices have been embedded, so the candidate set
$C(x)$ of a buffer vertex~$x$ will not change anymore, and it suffices to
find a system of distinct representatives for the available candidate sets $A(x)$, which we achieve
by verifying Hall's condition. Thus, for each set $Y\subset\Xbuf_i$, we
need to show that the union $W$ of the available candidate sets for $y\in
Y$ satisfies $|W|\ge|Y|$. We separate three cases: $Y$ is either a small
fraction of $\Xbuf_i$, or most of $\Xbuf_i$, or somewhere
intermediate. In the first case we use essentially the same argument as
when embedding the queue vertices. In the last case we show that the
`uniform' distribution of the sets $C(x)$ for $x\in\Xbuf_i$ implies
that~$W$ is most of $V_i\setminus\im(\psi)$. 

To handle the remaining case when $Y$ contains most or all of $\Xbuf_i$, we show
that the RGA gives us an extra property: for every vertex $v\in V_i$ there
are many vertices $x\in\Xbuf_i$ such that $v$ is a candidate for $x$.  This is the point in the proof where we use the
super-regularity and inheritance properties that $G$ satisfies on the
subgraph $R'$ of $R$.  We also remark that in
the case of the random graphs blow-up lemma, Lemma~\ref{lem:rg_image}, we
cannot in general establish this property: When most of the vertices
in~$\tX_i$ are contained in copies of $K_{\Delta+1}$ we obtain the claimed
feature only for most vertices instead of all vertices $v\in V_i$. Hence in
this case, to recover the special property, we perform an additional
embedding stage, using the sets~$\Xc_i$ and~$\Vc_i$, to fix these
\emph{buffer defects} (see Section~\ref{sec:fixbuffer}). This additional embedding stage also uses the super-regularity and inheritance properties of $R'$. We do not use these properties elsewhere in the proof.  Finally, using
again that~$\Gamma$ and hence~$G$ does not have `dense spots' we can show
that the above described extra property implies $|W|\ge|Y|$ as desired.

Summarising, our blow-up lemma proofs contain three main embedding procedures:
the random greedy algorithm, the queue embedding, and the embedding of the
buffer vertices. In the proof of Lemma~\ref{lem:rg_image} we also perform
an additional embedding procedure for fixing buffer defects.

\medskip

One important difference between the proofs of Lemmas~\ref{lem:rg_image} and~\ref{lem:psr_main}, and of Lemma~\ref{lem:degen}, is that in the former we choose the order of embedding vertices in the RGA, whereas in the latter an order is given to us. We change it only by moving the vertices $\Xbuf_i$ to the end of the order.

\chapter{Tools, preparation, and setup}
\label{chap:toolsetc}

In this chapter we lay the foundations for the proofs of our blow-up lemmas.
We start in Section~\ref{sec:Tools} by providing various tools we need: a
collection of concentration inequalities, a suitable variant of the
Hajnal-Szemer\'edi theorem on equitable partitions of graphs, and regularity
inheritance lemmas in random and in bijumbled graphs. In
Section~\ref{sec:pseudo} we then establish various deterministic properties
that the random graph or a bijumbled graph enjoys and which we use in our
proofs. In Section~\ref{sec:gensetup} we describe the general setup that is
used in the proofs of our three blow-up lemmas. In
Section~\ref{sec:RGAlemmas} finally we collect a series of lemmas which are
useful for analysing the various random greedy algorithms our proofs use.

\section{Tools}
  \label{sec:Tools}
  
  \subsection{Probabilistic inequalities}
  We need the following forms of Chernoff's inequality and of the hypergeometric
  inequality.
  Recall that if $X$ is hypergeometrically distributed with parameters $N$, $n$ and $m$, then $\Exp X=mn/N$.

  \begin{theorem}[Corollary~2.4, Theorems~2.8 and~2.10 
   from~\cite{JLRbook}]
   \label{thm:chernoff}
  Suppose
  $X$ is a random variable which is either the sum of a collection of independent Bernoulli random variables, or is hypergeometrically distributed.
  Then we have for $\delta\in(0,3/2)$
 \[\Pr\big(X>(1+\delta)\Exp  X\big)<e^{-\delta^2\Exp X/3}\quad\text{and}\quad
 \Pr\big(X<(1-\delta)\Exp  X\big)<e^{-\delta^2\Exp X/3}
 \,.\]
  We also have, for any $t\ge 6\Exp X$, 
  \[\Pr\big(X\ge \Exp X+t)\le e^{-t}\,.\]
  \end{theorem}

  We will often want to bound above the sum of a sequence of Bernoulli
  random variables $Y_1,\ldots,Y_n$ coming from some process which are
  \emph{not} independent, but which have the following \emph{sequentially
    dependent} structure. Suppose $\hist_0,\dots,\hist_n$ are increasing
  `histories' of the process, that is, information on all the random
  choices made in the process up to given increasing times in the
  process. Suppose that for each $1\le i\le n$, the value of the random variable $Y_i$
  is determined by $\hist_i$, and that we have functions
  $p_i=p_i(\hist_{i-1})$ such that $\Exp[Y_i|\hist_{i-1}]\le p_i$ holds
  \emph{almost surely}, that is, with probability~$1$. If $\sum_{i=1}^np_i$ is almost surely bounded above by
  $x$, then the following lemma claims the same upper tail bound on
  $\sum_{i=1}^nY_i$ holds as we would get from Theorem~\ref{thm:chernoff}
  if the $Y_i$ were independent and the sum of their expectations were $x$,
  and also gives the same lower tail bound as Theorem~\ref{thm:chernoff} under similar conditions.
  
  It is convenient to phrase this lemma in terms of a sequence of
  partitions $\cF_0,\ldots,\cF_n$, each refining the previous, of a
  probability space $\Omega$. For the connection to processes and
  histories, observe that any finite stochastic process is associated with
  the finite probability space of all possible outcomes, with the
  probability measure coming from the process. The possible histories of
  the process up to any given time $t$ naturally give a partition of this
  probability space, and two histories up to an earlier and later time give
  two partitions, the first refined by the second. Thus the following lemma
  indeed applies to sequentially dependent random variables as described
  above.
  
  We note that this lemma could be phrased in terms of a filtration and
  random variables measurable with respect to elements of the filtration,
  which might be more familiar to readers with a background in probability.
  That we do not use this notation here is purely to avoid defining these
  concepts. Furthermore, we remark that the lemma is essentially a
  super/submartingale inequality, proved in the standard way.\footnote{We
    would like to thank Oliver Riordan and Ori Gurel-Gurevich for pointing
    this out to us.} However, we did not find this particular inequality in
  the literature, so give a proof from first principles here.

  \begin{lemma}[Sequential dependence lemma]\label{lem:coupling} 
    Let $\Omega$ be a finite probability space, and $\cF_0,\dots,\cF_n$ be
    partitions of $\Omega$, with $\cF_{i-1}$ refined by $\cF_i$ for each
    $i\in[n]$. For each $i\in[n]$ let $Y_i$ be a Bernoulli random
    variable on $\Omega$ which is constant on each part of $\cF_i$, and let
    $p_i$ be a real-valued random variable on $\Omega$ which is constant on
    each part of $\cF_{i-1}$. Let~$x$ be a real number, $\delta\in(0,3/2)$, and
    $X=Y_1+\dots+Y_n$.
    \begin{enumerate}[label=\abc]
      \item\label{coupling:a}
        If $\sum_{i=1}^n p_i\le x$ holds almost surely, and $\Exp [Y_i|\cF_{i-1}]\le
        p_i$ holds almost surely for all $i\in[n]$, then 
        \[\Pr\big(X>(1+\delta) x \big)<e^{-\delta^2 x/3}\,.\]
      \item\label{coupling:b}
        If $\sum_{i=1}^n p_i\ge x$ holds almost surely, and $\Exp [Y_i|\cF_{i-1}]\ge
        p_i$ holds almost surely for all $i\in[n]$, then
        \[\Pr\big(X<(1-\delta)x \big)<e^{-\delta^2 x/3}\,.\]
    \end{enumerate}
  \end{lemma}
  \begin{proof}
    We start with~\ref{coupling:a}. We shall first show by induction on $n$
    that if $\sum_{i=1}^n p_i\le x$ holds almost surely and $\Exp [Y_i|\cF_{i-1}]\le
        p_i$ holds almost surely for all $i\in[n]$, then we have the following bound on the
    moment generating function of~$X$ for every $u\ge 0$:
    \begin{equation}\label{coupling:moment}
      \Exp e^{uX}=\Exp\prod_{i=1}^n e^{uY_i}\le \Big(1-\frac{x}{n}+\frac{x}{n}e^{u}\Big)^n\,.
    \end{equation}
    For $n=1$ we have $\Pr(Y_1=1|\cF_0)\le p_1\le x$ almost surely, hence
    $\Pr(Y_1=1)\le x$ almost surely. We conclude that indeed
    \[\Exp e^{uY_1}=1-\Pr(Y_1=1)+e^u\Pr(Y_1=1)\le 1-x+xe^u\,,\]
    because $1-t+te^u$ is non-decreasing in~$t$.

    For $n\ge 2$, we shall use for each $F\in\cF_1$ with $\Pr(F)>0$, 
    the
    induction hypothesis applied to the $n-1$ random variables
    $p_2|F,\dots,p_n|F$ and the $n-1$ random variables
    $Y_2|F,\dots,Y_n|F$. This is possible because,
    since $\Pr(F)>0$, we have $\sum_{i=2}^np_i|F\le x-p_1|F$ almost
    surely and $\Exp[Y_i|F,\cF_{i-1}]\le p_i|F$
    almost surely for each $2\le i\le n$.
    We conclude by induction that
    \begin{equation}
      \label{eq:coupling:ind}
      \begin{split}
      \Exp\prod_{i=1}^ne^{uY_i}
      &=\sum_{F\in\cF_1}\Pr(F)e^{uY_1|F}\Exp\Big[\prod_{i=2}^ne^{uY_i}\Big|F\Big]\\
      &\le\sum_{F\in\cF_1}\Pr(F)e^{uY_1|F}\bigg(1-\frac{x-p_1|F}{n-1}+\frac{x-p_1|F}{n-1}e^u\bigg)^{n-1}\,.
    \end{split}\end{equation}
    Further, because~$Y_1$ is constant on each part~$F$ of~$\cF_1$ we have
    for each $F_0\in\cF_0$ that
    \begin{multline*}
      \sum_{F\in\cF_1,F\subset F_0} \Pr(F) e^{uY_1|F} 
      =\Pr(F_0) \Exp[e^{uY_1}|F_0] \\
      =\Pr(F_0)\big(1-\Pr(Y_1=1|F_0)+e^u\Pr(Y_1=1|F_0)\big)
      \le\Pr(F_0)(1-p_1|F_0+e^up_1|F_0)\,,
    \end{multline*}
    where the inequality uses that $1-t+te^u$ is non-decreasing in~$t$. 
    Because~$p_1$ is constant on each part~$F_0$ of~$\cF_0$ it thus follows from~\eqref{eq:coupling:ind} that
    \begin{multline}\label{eq:moment:est}
        \Exp\prod_{i=1}^ne^{uY_i}
        \le\sum_{F_0\in\cF_0}\sum_{F\in\cF_1,F\subset F_0}
        \Pr(F)e^{uY_1|F}\Big(1-\frac{x-p_1|F_0}{n-1}+\frac{x-p_1|F_0}{n-1}e^u\Big)^{n-1}\\
        \le\sum_{F_0\in\cF_0}\Pr(F_0)\big(1-p_1|F_0+e^up_1|F_0\big)\Big(1-\frac{x-p_1|F_0}{n-1}+\frac{x-p_1|F_0}{n-1}e^u\Big)^{n-1}\,.
    \end{multline}
    Since $f(t)=\ln(1-t+te^u)$ is a concave function in $t$ it follows from
    Jensen's inequality that
    \[f(t)+(n-1)f\Big(\frac{x-t}{n-1}\Big)\le n\cdot
    f\Big(\frac1n\Big(t+(n-1)\frac{x-t}{n-1}\Big)\Big)=n\cdot f\Big(\frac
    xn\Big)\,,\]
    for any $t\in[0,1]$.  Substituting this, with $t=p_1|F_0$ for each
    $F_0\in\cF_0$, into~\eqref{eq:moment:est}
    yields~\eqref{coupling:moment}, as desired.

    We are now in a position to apply Bernstein's method of applying
    Markov's inequality to the moment generating function. Since
    $e^{uz}$ is strictly increasing in~$z$ for each $u>0$,
    we have
    \begin{equation}\begin{split}\label{eq:coupling:upper}
      \Pr\big(X>(1+\delta)x\big)=\Pr\big(e^{uX}>e^{(1+\delta)ux}\big)&\le \big(\Exp e^{uX}\big) e^{-(1+\delta)ux}\\
      &\leByRef{coupling:moment} \big(1-\tfrac{x}{n}+\tfrac{x}{n}e^u\big)^ne^{-(1+\delta)ux}\,,
    \end{split}\end{equation}
    if $u>0$,
    where the first inequality is Markov's inequality. We may assume
    $(1+\delta)x<n$, since otherwise $\Pr\big(X>(1+\delta)x\big)=0$ and the lemma
    statement is trivial. Thus we can choose $u$ such that
    $e^u=\tfrac{(1+\delta)(n-x)}{n-(1+\delta)x}$, plug that into the right
    hand side of~\eqref{eq:coupling:upper} and obtain
    \begin{multline*}
      \Pr\big(X>(1+\delta)x\big)\le
      \left(1+ \tfrac{\delta x}{n-(1+\delta)x} \right)^n \left(\tfrac{(1+\delta)(n-x)}{n-(1+\delta)x}\right)^{-(1+\delta)x}\\
      =\left(1+ \tfrac{\delta x}{n-(1+\delta)x} \right)^{n-(1+\delta)x}
         \left(\left(1+ \tfrac{\delta x}{n-(1+\delta)x} \right)^{-1} \left(\tfrac{(1+\delta)(n-x)}{n-(1+\delta)x}\right)\right)^{-(1+\delta)x}\\
      \le e^{\delta x}(1+\delta)^{-(1+\delta)x}\le e^{-\delta^2 x/3}\,,
    \end{multline*}   
    where the last inequality uses
    $(1+\delta)\ln(1+\delta)-\delta\ge\delta^2/3$. This proves~\ref{coupling:a}.
   
    The proof of~\ref{coupling:b} is similar. By the analogous calculation, we have
    \[\Exp e^{u(n-X)}\le\big(1-\tfrac{n-x}{n}+\tfrac{n-x}{n}e^{u}\big)^n\,,\]
    and again Bernstein's method, this time choosing
    $e^u=\tfrac{n-x+\delta x}{(n-x)(1-\delta)}$, gives the desired result
    for $\delta\in(0,1)$:
    \begin{multline*}
      \Pr\left(X<(1-\delta)x\right)=\Pr\left(n-X>n-x+\delta
        x\right)=\Pr\left(e^{u(n-X)}>e^{u(n-x+\delta x)}\right)\\
      \begin{aligned}
      & \le\Exp[e^{u(n-X)}] e^{-u(n-x+\delta x)}\le\big(1-\tfrac{n-x}{n}+\tfrac{n-x}{n}e^{u}\big)^ne^{-u(n-x+\delta x)}\\
      &=\left(1+\tfrac{n-x}{n}\left(\tfrac{\delta
            n}{(n-x)(1-\delta)}\right)\right)^{n} \left(\tfrac{n-x+\delta
          x}{(n-x)(1-\delta)}\right)^{-(n-x+\delta x)}\\
      & = (1-\delta)^{-n}\left(\tfrac{n-x+\delta x}{(n-x)(1-\delta)}\right)^{-(n-x+\delta x)} 
       =\left(1-\delta\right)^{-(1-\delta)x} \left(\tfrac{n-x+\delta 
          x}{n-x}\right)^{-(n-x+\delta x)} \\
      &\le\left(1-\delta\right)^{-(1-\delta)x} \left(1+\tfrac{\delta
          x}{n-x}\right)^{-(n-x)}
      \le\left(1-\delta\right)^{-(1-\delta)x} e^{-\delta x}
      \le e^{-\delta^2 x/3}\,.
    \end{aligned}
    \end{multline*}   
    Note that in this case we only need to consider $\delta\in(0,1)$ as the
    probability of $X<0$ is zero, so this gives~\ref{coupling:b}.
  \end{proof}
  
 \subsection{Equitable partitions}

 The Hajnal-Szemer\'edi theorem states that any graph~$F$ has a
 $k$-colouring with equitable colour classes for every $k\ge(\Delta(F)+1)$.

  \begin{theorem}[Hajnal-Szemer\'edi~\cite{HajSze70}]\label{thm:HajSze} 
    Given any graph $F$ and $k\ge\Delta(F)+1$, there is an equitable
    partition $V(F)=V_1\dcup\cdots\dcup V_k$ such that each part is
    independent.
  \end{theorem}
  
  For preparing the graph~$H$ for the embedding in the proofs of our
  blow-up lemmas, we require an equitable partitioning result similar to
  the Hajnal-Szemer\'edi theorem (which we shall apply to an auxiliary
  graph~$F$ defined for each part of the given partition of~$H$).  The
  difference is that we want to specify a subset~$X$ of the vertices of~$F$
  and obtain an equitable partition of~$F$ that also equitably
  partitions~$X$ (this subset will be the buffer set). Since for us it is
  not essential to obtain a sharp bound on the number of parts required,
  this result is not difficult to deduce from the Hajnal-Szemer\'edi theorem.

\begin{lemma}\label{lem:nicepartition}
  Given any graph $F$ and a subset $X$ of $V(F)$, there is an equitable
  partition $V(F)=V_1\dcup\dots\dcup V_{8\Delta(F)}$ such that each part is
  independent and the sets $V_i\cap X$ form an equitable partition of $X$.
\end{lemma}
\begin{proof}
  We may assume without loss of generality that $|X|\le v(F)/2$, since
  otherwise we could replace $X$ with $V(F)\setminus X$.  We now first use
  the Hajnal-Szemer\'edi theorem to get an equitable partition of $F[X]$
  into $8\Delta(F)$ independent parts.  We proceed by adding the remaining
  vertices of $F$ to these parts, maintaining their independence, to obtain a
  partition $V_1\dcup\dots\dcup V_{8\Delta(F)}$ of $V(F)$.  Among all
  possible such partitions we choose the one which is as equitable as
  possible. Observe that in this process it is always possible to put every
  vertex into some part: for any vertex, there are at most $\Delta(F)$
  parts to which it cannot be added, so at least $7\Delta(F)$ parts exist
  to which it can be added to form an independent set.
 
  We are now done if there exist no pair of parts $V_i,V_j$ whose sizes
  differ by two or more. Suppose $V_1$ is the smallest part.  Observe that
  there are at most $\Delta(F)|V_1|$ vertices in $V(F)\setminus X$ which
  have a neighbour in~$V_1$.  Since $8\Delta(F)|V_1|\le v(F)$ it follows that there are at least $\frac38v(F)$
  vertices in $V(F)\setminus X$ which have no neighbour in~$V_1$.

  Assume for contradiction that there are parts $V_i$ with
  $|V_i|\ge|V_1|+2$. Observe that for each such~$i$ every vertex of
  $V_i\setminus X$ is adjacent to some vertex of $V_1$, or we would be able
  to move a vertex from $V_i\setminus X$ to $V_1$ and obtain a more
  equitable partition.  Hence the at least $\frac38v(F)$ vertices in
  $V(F)\setminus X$ which have no neighbour in~$V_1$ are all contained in
  parts of size at most $|V_1|+1\le 2|V_1|$, and thus lie in at least
  $\tfrac32\Delta(F)$ different parts. Now fix a set $V_i$ with
  $|V_i|\ge |V_1|+2$, and let $v$ be any vertex of $V_i\setminus X$ (which
  exists because we started with an equitable partition of $X$). Since $v$
  has at most $\Delta(F)<\tfrac32\Delta(F)$ neighbours, there is a set
  $V_k$ which contains no vertex adjacent to $v$, and which contains a
  vertex $w$, not in $X$, with no neighbours in $V_1$. We replace $V_i$
  with $V_i\setminus\{v\}$, $V_k$ with $V_k\cup\{v\}\setminus \{w\}$, and
  $V_1$ with $V_1\cup \{w\}$. The result is a more equitable partition, a
  contradiction.
\end{proof} 

  \subsection{Regularity inheritance}\label{subsec:RI}

  In our blow-up lemmas we require certain regularity inheritance
  properties. In this subsection we justify that we can obtain these in
  subgraphs of random or bijumbled graphs. More precisely, we provide
  inheritance lemmas stating that most vertices in a (sparse) regular
  partition satisfy these inheritance properties. This is important for the
  correctness of our embedding procedures proving the blow-up lemmas.

  Recall that, according to our definition, regularity in random graphs and
  regularity in bijumbled graphs are different things: lower-regularity in
  the former and full-regularity in the later. Recall also that the sole
  reason for this difference is that we can only establish regularity
  inheritance for the respective version, that is, that the lemmas provided
  in this subsection require lower-regularity when they concern random graphs
  and full-regularity when they concern bijumbled graphs. Accordingly, in
  this subsection we shall make the two different regularity concepts
  explicit and talk about lower-regularity and full-regularity for clarity.

  For subgraphs~$G$ of bijumbled graphs~$\Gamma$ we can simply rely on the
  following inheritance lemmas from~\cite{ABSS}. These lemmas consider
  three disjoint vertex sets $X,Y,Z$ such that~$(X,Y)$ forms a
  fully-regular pair. The first lemma is a one-sided regularity inheritance
  statement and states that for most vertices $z\in Z$ the pair
  $\big(N_\Gamma(z;X),Y\big)$ inherits full-regularity.

  \begin{lemma}[One-sided inheritance in bijumbled graphs~{\cite[Lemma~3]{ABSS}}]\label{lem:jrione}
    \mbox{} \\
    For each $\eps',d>0$ there are $\eps_0,c_0>0$ such that for all
    $0<\eps\le\eps_0$ and $0<p<1$ the
    following holds.
    Let $G\subset \Gamma$ be graphs and $X,Y,Z$ be disjoint subsets of
    $V(\Gamma)$.  If $(X,Y)$ is $(\eps,d,p)$-fully-regular in~$G$ and $\Gamma$ is
    \[\Big(p,c_0p^{2}(\log_2\tfrac{1}{p})^{-1/2}\sqrt{\max\big(|X||Y|,|X||Z|\big)}\Big)\text{-bijumbled}\]
    then for at most $\eps'|Z|$
    vertices~$z$ of~$Z$ the pair $\big(N_\Gamma(z;X),Y\big)$ is not
    $(\eps',d,p)$-fully-regular in~$G$.
  \end{lemma}
  
  The second lemma requires stronger bijumbledness for establishing two-sided
  regularity inheritance.  It states that for most $z\in Z$ the
  pair $\big(N_\Gamma(z;X),N_\Gamma(z;Y)\big)$ inherits full-regularity.

  \begin{lemma}[{Two-sided inheritance in bijumbled graphs~\cite[Lemma~4]{ABSS}}]\label{lem:jritwo}
    \mbox{} \\
    For each $\eps',d>0$ there are $\eps_0,c_0>0$ such that for all
    $0<\eps<\eps_0$ and $0<p<1$ the following holds.
    Let $G\subset \Gamma$ be graphs and $X,Y,Z$ be disjoint subsets of
    $V(\Gamma)$.  If 
    $(X,Y)$ is $(\eps,d,p)$-fully-regular in~$G$ and $\Gamma$ is
    \[\Big(p,c_0p^3\sqrt{\max\big(|X||Y|,|X||Z|,|Y||Z|\big)}\Big)\text{-bijumbled}\]
    then 
    for at most $\eps'|Z|$ vertices~$z$ of~$Z$ the pair
    $\big(N_\Gamma(z;X),N_\Gamma(z;Y)\big)$ is not $(\eps',d,p)$-fully-regular in~$G$.
  \end{lemma}

\medskip

 The remainder of this subsection is devoted to proving the inheritance lemmas for random graphs, which we stated in Section~\ref{subsec:reginherit}, and showing their optimality. For the reader's convenience, we restate these lemmas here, beginning with the one-sided inheritance lemma.
 
 \begin{oneRI}  For each $0<\eps',d$ there are $\eps_0>0$ and
    $C$ such that for all $0<\eps<\eps_0$ and $0<p<1$, a.a.s.\
    $\Gamma=G(n,p)$ has the following property. 
    Let $G\subset\Gamma$ be a graph and $X,Y$ be disjoint subsets of
    $V(\Gamma)$. 
    If $(X,Y)$ is $(\eps,d,p)$-lower-regular in~$G$ and
    \[|X|\ge C\max\big(p^{-2},p^{-1}\log n\big) \quad\text{and}\quad|Y|\ge Cp^{-1}\log (en/|X|)\,,\] 
    then 
    for at most $Cp^{-1}\log (en/|X|)$ vertices $z\in V(\Gamma)$ the pair
    $\big(N_\Gamma(z;X),Y\big)$ is not $(\eps',d,p)$-lower-regular in $G$.
 \end{oneRI}
 The corresponding two-sided inheritance lemma is:
 \begin{twoRI}
    For each $0<\eps',d$ there are $\eps_0>0$ and
    $C$ such that for all $0<\eps<\eps_0$ and $0<p<1$, a.a.s.\
    $\Gamma=G(n,p)$ has the following property. 
    Let $G\subset\Gamma$ be a graph and $X,Y$ be disjoint subsets of
    $V(\Gamma)$. 
    If $(X,Y)$ is $(\eps,d,p)$-lower-regular in~$G$ and
    \[|Y|\ge|X|\ge C\max\big(p^{-2},p^{-1}\log n\big)\,,\]
    then there are at most $C\max\big(p^{-2},p^{-1}\log (en/|X|)\big)$ vertices $z\in V(\Gamma)$ such that
    $\big(N_\Gamma(z;X),N_\Gamma(z;Y)\big)$ is not $(\eps',d,p)$-lower-regular in $G$.
  \end{twoRI}
  
 Note that the regularity inheritance lemmas in random graphs, Lemmas~\ref{lem:oneRI} and~\ref{lem:twoRI} are stated in a different form to the
 ones for bijumbled graphs above, not requiring any bound on~$p$, but rather
 having~$p$ appear in the estimates for the sizes of~$X$ and~$Y$ as well as
 the number of vertices not preserving regularity. We state them in this form
 because this seems most suitable when we want to use them in
 applications of our blow-up lemma. For bijumbled graphs on the other hand we
 decided not to switch to this form since we would obtain more complicated
 conditions in this case (bounding the products of set sizes such as
 $|X||Y|$, instead of the sets themselves). 
 
  In the proofs of Lemmas~\ref{lem:oneRI} and~\ref{lem:twoRI} we follow the
  approach of~\cite{ChvRand}, and rely on the following result of Gerke,
  Kohayakawa, R\"odl and Steger~\cite{GKRS} stating that the vast majority
  of subpairs $(X',Y)$ of a lower-regular pair $(X,Y)$ in any graph inherit
  lower-regularity.  Unfortunately this result becomes false if we replace
  lower-regularity by full-regularity, and it is precisely this reason why
  we need to work with lower-regularity in random graphs.

  \begin{theorem}[Theorem 3.6 from~\cite{GKRS}]\label{thm:OneSideInherit}
    For any $d$, $\beta$, $\eps'>0$ there exist
    $\eps_0>0$ and~$C$ such that for
    any $0<\eps<\eps_0$ and $0<p<1$, if $(X,Y)$ is an $(\eps,d,p)$-lower-regular
    pair in a graph $G$, then the number of sets $X'\subseteq X$ with $|X'|=w\ge
    C/p$ such that $(X',Y)$ is an $(\eps',d,p)$-lower-regular pair in~$G$ is at least
    $(1-\beta^w)\binom{|X|}{w}$.
  \end{theorem}
  
  The following is an analogous statement for subpairs $(X',Y')$ of $(X,Y)$
  and is an immediate consequence of Theorem~\ref{thm:OneSideInherit}.
    
  \begin{corollary}[Corollary 3.8 from~\cite{GKRS}]\label{cor:TwoSideInherit}
    For any $d$, $\beta$, $\eps'>0$ there exist
    $\eps_0>0$ and $C$ such that for
    any $0<\eps<\eps_0$ and $0<p<1$, if $(X,Y)$ is an $(\eps,d,p)$-lower-regular
    pair in a graph $G$, then the number of pairs $X'\subseteq X$ and
    $Y'\subseteq Y$ with $|X'|=w_1\ge C/p$ and $|Y'|=w_2\ge C/p$ such that
    $(X',Y')$ is an $(\eps',d,p)$-lower-regular pair in~$G$ is at least
    $(1-\beta^{\min(w_1,w_2)})\binom{|X|}{w_1}\binom{|Y|}{w_2}$.
  \end{corollary}
   
  The proof idea for Lemma~\ref{lem:oneRI} is the
  following. First, we will show that if $(X,Y)$ is a counterexample, that
  is, a lower-regular pair not satisfying the conclusion of the lemma,
  then there are $X'\subset X$ and $Y'\subset Y$ which are only slightly smaller than $X$ and $Y$ respectively, and a set $Z'$ of vertices outside $X'\cup Y'$ whose neighbourhoods in $X'$ all have size about $p|X'|$ but which nevertheless do not form regular pairs with $Y'$. We then find $Y^*\subseteq Y'$ which has the same properties as $Y'$, but which has size $\tfrac12 Cp^{-1}\log(en/|X|)$. We finally argue that this last structure is unlikely to exist in $G(n,p)$, completing the proof. It may not be obvious what we gain by reducing $(X',Y')$ to $(X',Y^*)$; what we gain is that there are not too many edges of $\Gamma$ between $X'$ and $Y^*$, and hence not too many possible lower-regular subgraphs of $\Gamma$. This allows us to take a union bound over all possible choices, which would fail if we attempted it with $(X',Y')$.
  
  The proof of Lemma~\ref{lem:twoRI} is similar but somewhat more complicated, and we will sketch it after proving Lemma~\ref{lem:oneRI}. Before embarking on these proofs, we need three preliminary lemmas.
  
  The first lemma assumes we are given a lower-regular pair together with a
  collection of~$\ell$  subpairs which do not inherit lower-regularity, and asserts that
  we can scale down the sizes of this pair and the subpairs in many ways. For Lemma~\ref{lem:oneRI} we only need one way; for Lemma~\ref{lem:twoRI} we will need the full power of this lemma.
  
  \begin{lemma}\label{lem:shrinkirreg}
    For each $d,\eps_2,\delta >0$ there exist $\eps_1$ and $C$ such that
    for all $0<\eps<\eps_1$ the following holds. 
    Let $(X,Y)$ be an $(\eps,d,p)$-lower-regular pair in a graph~$G$, and suppose
    $|X|\ge m\ge Cp^{-1-i}$ for $i\in\{0,1\}$.
    Further, let $X_j\subset X$ and $Y_j\subset Y$ with $j\in[\ell]$ form~$\ell$
    subpairs of $(X,Y)$ for some $\ell\in[m]$, with $|X_j|\ge\tfrac{1}{100}p^i|X|$ for each $j$, such that $(X_j,Y_j)$ is not $(\delta,d,p)$-lower-regular for any $j\in[\ell]$.

    Then there are at least $\tfrac12\binom{|X|}{m}$ choices of $X^*\subset X$ with $|X^*|=m$ such that $(X^*,Y)$ is $(\eps_2,d,p)$-lower-regular, such that $|X_j\cap X^*|=\big(1\pm\tfrac12\big) m|X_j|/|X|$ for each $j\in[\ell]$, and such that for each $j\in[\ell]$, the pair $(X_j\cap X^*,Y_j)$ is not $(\delta^2/10,d,p)$-lower-regular.
  \end{lemma}
  \begin{proof}
    Given $d,\eps_2>0$, let $\eps_1$ and $C'$ be returned by
    Theorem~\ref{thm:OneSideInherit} with input $d, \tfrac14,\eps_2$. Given $\delta>0$, 
    let $C\ge C'$ be such that $3(x^2/C)\exp(-\delta_1^2 x/3600)<\tfrac14$ for
    each $x\ge 1$.
    
    Given $0<\eps<\eps_1$, and given an
    $(\eps,d,p)$-lower-regular pair $(X,Y)$ with $ |X|\ge m\ge Cp^{-1-i}$,
    let $X^*\subseteq X$ be chosen uniformly at random from the $m$-sized
    subsets of $X$.  By Theorem~\ref{thm:OneSideInherit}, with probability
    at least $\tfrac34$, the pair $(X^*,Y)$ is $(\eps_2,d,p)$-lower-regular.
    
    Given any pair $(X_j,Y_j)$, since the pair is not $(\delta,d,p)$-regular, there exist subsets $X'_j$ and $Y'_j$ of $X_j$ and $Y_j$ respectively, of sizes $\delta|X_j|$ and $\delta|Y_j|$ respectively, such that $d_p(X'_j,Y'_j)\le d-\delta$. There thus exists a
    subset $X''_j$ of $X'_j$ consisting of $\delta^2|X_j|/3$ vertices
    each of whose degree into $Y'_j$ is smaller than $(d-\delta/2)p|Y'_j|$,
    as otherwise we would have
    \[ e(X'_j\setminus X''_j,Y'_j)\ge (\delta-\delta^2/3)|X_j|(d-\delta/2)p|Y'_j|> (d-\delta)p|X'_j||Y'_j| >e(X'_j,Y'_j),\]
     a contradiction.
     Let $X^*_j=X''_j\cap X^*$, then $d_p(X^*_j,Y'_j)<d-\delta/2$, by
     definition of $X''_j$. The quantity $|X^*_j|$ is hypergeometrically
     distributed with mean at least $\delta^2 p^i m/300$ since
     $|X''_j|=\delta^2|X_j|/3$ and $|X_j|\ge \tfrac1{100} p^i|X|$. By
     Theorem~\ref{thm:chernoff}, the probability that
     $|X^*_j|<\delta^2|X_j|m/(6|X|)$ is at most $\exp(-\delta^2 p^i m/3600)$. Similarly, the quantity $|X_j\cap X^*|$ is hypergeometrically distributed with mean at least $\delta^2 p^i m/300$, so by Theorem~\ref{thm:chernoff} the probability that $|X^*\cap X_j|\neq\big(1\pm\tfrac12\big)|X_j|m/|X|$ is at most $2\exp(-\delta^2 p^i m/3600)$.
    
     By choice of $m$, we have $(p^i m)^2\ge Cp^{-1}p^i m\ge Cm$, so
     setting $x=p^i m\ge 1$, by choice of $C$ we have
     $3(x^2/C)\exp(-\delta^2 x/3600)<\tfrac14$. Thus taking a union bound, we
     conclude that with probability at least
    \[1-\tfrac14-3m\exp(-\delta^2 p^{i} m/3600)>\tfrac12\]
    none of the above bad events occurs. In particular, for at least $\tfrac12\binom{|X|}{m}$ choices of $X^*$, for each $j\in[\ell]$ the pair $(X_j\cap X^*,Y_j)$ is not $(\delta^2/10,d,p)$-regular, giving the conclusion of the lemma.
  \end{proof}
  
  The next easy lemma shows that adding a few vertices to an
  $(\eps,d,p)$-lower-regular pair does not destroy regularity. We note that
  the corresponding result for an $(\eps,d,p)$-fully-regular pair would
  require knowing that $G$ does not contain dense spots.
  
  \begin{lemma}\label{lem:sticking}
  Let $0<\eps<\tfrac{1}{10}$. Let $G$ be a graph and let $U',V'\subset V(G)$ be disjoint sets such that $(U',V')$ is
  $(\eps,d,p)$-lower-regular in $G$.  If $U\supset U'$ with
  $|U|\le\big(1+\tfrac{1}{10}\eps^3\big)|U'|$ and $V\supset V'$ with 
  $|V|\le\big(1+\tfrac{1}{10}\eps^3\big)|V'|$ are disjoint, then $(U,V)$ is
  $(2\eps,d,p)$-lower-regular in $G$.
\end{lemma}
\begin{proof}  
  Let $X$ and $Y$ be arbitrary subsets of $U$ and $V$ with
  $|X|\ge 2\eps|U|,|Y|\ge2\eps|V|$. We need to show that
  $e_G(X,Y)\ge(d- 2\eps)p|X||Y|$. Since $(U',V')$ is
  $(\eps,d,p)$-lower-regular in $G$, and since
  $2\eps-\tfrac{1}{10}\eps^3>\eps$, we have
  \begin{multline*}
   e_G(X\cap U',Y\cap V')\ge(d-\eps)p|X\cap U'||Y\cap V'|\\
   \ge (d-\eps)p\big(|X|-\tfrac{1}{10}\eps^3|U'|\big)\big(|Y|-\tfrac{1}{10}\eps^3|V'|\big)
   \ge (d-\eps)(1-\tfrac{1}{10}\eps^2)^2p|X||Y|\\
   \ge(d-2\eps)p|X||Y|\,,
  \end{multline*}
  as desired.
\end{proof}

Finally, we need a statement about the edge distribution in $G(n,p)$.

\begin{lemma}\label{cl:inh:dist}
     For any $p\ge n^{-1/2}$, a.a.s.\ $\Gamma=G(n,p)$ has the following
    property.
    \begin{enumerate}[label=($\star$)]
     \item\label{inh:star} For any disjoint $W,Z\subset V(\Gamma)$ with $|W|\le|Z|$ and $|W|\ge 48p^{-1}\log (en/|Z|)$, we have
     $e_\Gamma(W,Z)=\big(1\pm\tfrac12\big)p|W||Z|$,
     and for any $Z\subset V(\Gamma)$ we
     have
     $e_\Gamma(Z)\le 4p|Z|^2+2|Z|\log n$.
    \end{enumerate}
   \end{lemma}
   \begin{proof}
   Indeed, using Theorem~\ref{thm:chernoff} and
    a union bound over the choices of $W$ and $Z$, 
    we sum over all possible cardinalities of $W$ and $Z$ to obtain that 
    the failure probability of the first assertion is at most
    \begin{align*}
      \sum_{|W|,|Z|} & \binom{n}{|W|}\binom{n}{|Z|}\cdot
      2 \exp\Big(-\frac{p|W||Z|}{12}\Big) \\
      \le & \sum_{|W|,|Z|}2\cdot \big(\tfrac{en}{|Z|}\big)^{2|Z|}\exp\big(-4|Z|\log\tfrac{en}{|Z|}\big)\\
       \le & 2n^2\exp\big(-n^{-1/2}\big)\,,
    \end{align*}
 where  the first inequality uses $|W|,|Z|\ge 48p^{-1}\log (en/|Z|)$, and the second our assumption $p\ge n^{-1/2}$. This tends to zero as $n$ tends to infinity, as desired. Similarly, since $4p|Z|^2$ is greater than seven times the expected number of edges in $Z$, the failure probability of
    the second assertion is at most
    \[\sum_{|Z|}n^{|Z|}\cdot e^{-2|Z|\log n}= \sum_{|Z|} n^{-|Z|}\,,\]
    which tends to zero as $n$ tends to infinity.
    \end{proof}

We now prove our one-sided regularity inheritance lemma, Lemma~\ref{lem:oneRI}. 

\begin{proof}[Proof of Lemma~\ref{lem:oneRI}]
  Given $0<\eps',d$, we set $\delta=\eps'/2$. We obtain $\eps_2$ and $C_1$ from Theorem~\ref{thm:OneSideInherit} for input $d$, $\beta=e^{100}$ and $\delta^2/10$. Now we obtain $\eps_1$ and $C_2$ from Lemma~\ref{lem:shrinkirreg} for input $d$, $\eps_2$ and $\delta$. We set $\eps_0=\tfrac12\eps_1$. We set $C'=1000C_1C_2(\eps')^{-3}$ and $C=1000(\eps')^{-3}C'$. We may suppose that
  $p\ge n^{-1/2}$, since otherwise the lemma holds vacuously as it requires $|X|>n$.   
    
    We now show that, provided $\Gamma$ has the property~\ref{inh:star} of Lemma~\ref{cl:inh:dist}, given a counterexample $(X,Y)$ and subgraph $G$ of $\Gamma$ to the conclusion of Lemma~\ref{lem:oneRI} we can find $X'\subset X$, $Y'\subset Y$ and a set $Z'$ such that $X'$, $Y'$ and $Z'$ are disjoint, such that $(X',Y')$ is $(\eps_1,d,p)$-regular, and such that each $z\in Z'$ has $\deg_\Gamma(z;X')$ between $\tfrac12p|X'|$ and $\tfrac32p|X'|$, yet the pair $\big(N_\Gamma(z;X'),Y'\big)$ is not $(\delta,d,p)$-regular in $G$.
    
    To that end, suppose that $(X,Y)$ and a subgraph $G$ of $\Gamma$ forms a counterexample. 
    Recall that we have $|X|\ge C\max(p^{-2},p^{-1}\log n)$ and $|Y|\ge Cp^{-1}\log(en/|X|)$. 
    Let $Z$ be a set of $C'p^{-1}\log(en/|X|)$ vertices witnessing this counterexample, i.e.\ a set of vertices whose 
    neighbourhoods in $X$ do not form $(\eps',d,p)$-regular pairs with $Y$ in $G$. 
    By~\ref{inh:star}, we have $e_\Gamma(Z)\le 4p|Z|^2+2|Z|\log n$, so in particular for at least half of the vertices $z\in Z$ we have $\deg_\Gamma(z;Z)\le 16p|Z|+8\log n\le 20C'\log n$. 
    Now let $X'=X\setminus Z$, and $Y'=Y\setminus Z$. 
    By choice of $C'$, we have $|X'|\ge \big(1-\tfrac{1}{100}(\eps')^3\big)|X|$, and $|Y'|\ge \big(1-\tfrac{1}{100}(\eps')^3\big)|Y|$. 
    Furthermore, by~\ref{inh:star}, there are at most $100p^{-1}\log(en/|X'|)$ vertices $z\in Z$ such that 
    $\deg_\Gamma(z;X')\neq\big(1\pm\tfrac12\big)p|X'|$. 
    Let $Z'$ consist of $\tfrac14C'p^{-1}\log(en/|X|)$ vertices in $Z$ with at most $20C'\log n$ neighbours in $Z$ and 
    $\big(1\pm\tfrac12\big)p|X'|$ neighbours in $X'$, which we can do by choice of $C'$. 
    We claim that $(X',Y',Z')$ form our desired smaller counterexample. Indeed, by Lemma~\ref{lem:slice}, $(X',Y')$ is $(\eps_1,d,p)$-regular in $G$. 
    Furthermore, by Lemma~\ref{lem:sticking} and choice of $C$, if $\big(N_\Gamma(z;X'),Y')$ were $(\delta,d,p)$-regular then 
    $\big(N_\Gamma(z;X),Y)$ would be $(\eps',d,p)$-regular, in contradiction to the definition of $Z$. Thus $\big(N_\Gamma(z;X'),Y')$ is not $(\delta,d,p)$-regular for any $z\in Z'$.
    
    We now use Lemma~\ref{lem:shrinkirreg} to obtain $Y^*$. We apply Lemma~\ref{lem:shrinkirreg} to the regular pair $(Y',X')$, with input $d$, $\eps_2$ and $\delta$, and with $i=0$ and $m=C'p^{-1}\log(en/|X|)$. We let $\ell=|Z'|$ and the sets $(X_j,Y_j)_{j\in[\ell]}$ be the pairs $(Y',N_\Gamma(z;X'))$. The result is that there exist choices of $Y^*\subset Y'$ with $|Y^*|=m$ such that $(X',Y^*)$ is $(\eps_2,d,p)$-regular in $G$, yet for each $z\in Z'$ the pair $\big(N_\Gamma(z;X'),Y^*\big)$ is not $(\delta^2/10,d,p)$-regular in $G$. We fix one such $Y^*$.
    
    We have now shown that if for some $\Gamma$ we have~\ref{inh:star} but the conclusion of Lemma~\ref{lem:oneRI} fails, then there exist $(X',Y^*,Z')$ and a graph $G$ on $(X',Y^*)$ with the properties above. Our aim now is to show that this \emph{bad object} is unlikely to exist in $G(n,p)$. To that end, fix any $X'\subseteq [n]$ of size at least $\tfrac12C\max(p^{-2},p^{-1}\log n)$. Fix now disjoint sets $Y^*,Z'\subseteq [n]\setminus X'$ of sizes $m$ and $\tfrac14m$ respectively. We reveal the edges of $G(n,p)$ between $X'$ and $Y^*$, and let $G$ be any fixed $(\eps_2,d,p)$-regular subgraph of these edges. Next, we reveal the edges of $\Gamma=G(n,p)$ between $Z'$ and $X'$. We would like to upper bound the probability that we obtain $\deg_\Gamma(z;X')=\big(1\pm\tfrac12\big)p|X'|$ for each $z\in Z'$, and $\big(N_\Gamma(z;X'),Y^*\big)$ is not $(\delta^2/10,d,p)$-regular for any $z\in Z'$. Since the edges from each $z\in Z'$ to $X'$ are independent, it is enough to estimate the probability that $\big(N_\Gamma(z;X'),Y^*\big)$ fails to be $(\delta^2/10,d,p)$-regular in $G$, conditioning on $\deg_\Gamma(z;X')$. By Theorem~\ref{thm:OneSideInherit}, this probability is at most $\beta^{\deg_\Gamma(z;X')}\le\beta^{p|X'|/2}$. We conclude that the probability that a given $X'$, $Y^*$, $Z'$ and $G$ form a bad object is at most $\beta^{p|X'||Z'|/2}$. We now take the union bound over choices of $X'$ (including the choice of $|X'|$), $Y^*$ and $Z'$ (whose sizes are fixed given $|X'|$), and over the choices of $G\subset\Gamma$. Note that because~\ref{inh:star} holds, the number of edges in $\Gamma$ between $X'$ and $Y^*$ is at most $3p|X'||Y^*|/2$, so the number of choices of $G\subset\Gamma$ is at most $2^{3p|X'||Y^*|/2}$. Finally, using the facts that $\tfrac12|X|\le|X'|\le|X|$, that $|Y^*|=m$ and $|Z'|=m/4$, and $|X'|\ge m$, we conclude that the probability that~\ref{inh:star} holds but $G(n,p)$ contains a bad object is at most
    \begin{align*}
     &\sum_{|X'|}\binom{n}{|X'|}\binom{n}{|Y^*|}\binom{n}{|Z'|}2^{3p|X'||Y^*|/2}\beta^{p|X'||Z'|/2}\\
     \le & \sum_{|X'|}\big(\tfrac{en}{|X'|}\big)^{6|X|'}2^{3p|X'|m/2}\beta^{p|X'|m/8}\\
     \le & \sum_{|X'|}\exp\big(12|X'|\log\tfrac{en}{|X'|}+\tfrac32C'|X'|\log\tfrac{en}{|X'|}-12.5 C'|X'|\log\tfrac{en}{|X'|}\big)\\
     \le & \sum_{|X'|}e^{-|X'|}
    \end{align*}
    where the final sum tends to zero as $n$ tends to infinity since $|X'|\ge\tfrac12Cp^{-1}\log n$, which tends to infinity with $n$. Since by Claim~\ref{cl:inh:dist}, $G(n,p)$ has~\ref{inh:star} a.a.s., we conclude that a.a.s.\ $\Gamma=G(n,p)$ satisfies the conclusion of Lemma~\ref{lem:oneRI}.
\end{proof}

Our proof of Lemma~\ref{lem:twoRI} is closely related to the above proof. Again, we begin by assuming that a counterexample $(X,Y)$ and $G$ exists, with $|X|\le|Y|$, and deducing the existence of disjoint sets $X'$, $Y'$ and $Z'$. This time we will obtain $|X'|=|Y'|$ to be only slightly smaller than $|X|$. Again, we will have that each vertex of $Z'$ has roughly the expected size neighbourhood in each of $X'$ and $Y'$ but these neighbourhoods fail to induce regular pairs in $G$. Now, however, we have to work harder to show that this structure is unlikely to exist in $G(n,p)$. In order to make the union bound over choices of $G$ work, we have to apply Lemma~\ref{lem:shrinkirreg} twice to reduce $X'$ to $X^*$ and $Y'$ to $Y^*$, and we have to choose $|X^*|=|Y^*|=m$ of comparable size to $Z'$. Unfortunately now, if $p^{-1}\log(en/|X|)$ is too small the obvious union bound over choices of $X^*$, $Y^*$ and $Z'$ fails. Thus our \emph{bad object} is not one triple $(X^*,Y^*,Z')$ together with $G$, but that $(X',Y',Z')$ contains a collection of $\tfrac14\binom{|X'|}{m}\binom{|Y'|}{m}$ such triples $(X^*,Y^*,Z')$, which we obtain using the full strength of Lemma~\ref{lem:shrinkirreg}. We will show that this bad object is sufficiently unlikely to exist to make the union bound over choices of $X'$, $Y'$ and $Z'$ succeed.

\begin{proof}[Proof of Lemma~\ref{lem:twoRI}]
 Given $0<\eps',d$, we set $\delta_1=\tfrac12\eps'$ and for each $2\le j\le 5$
 we set $\delta_j=\delta_{j-1}^2/10$. We obtain $\eps_5$ from
 Corollary~\ref{cor:TwoSideInherit} for input $d$, $\beta=e^{-100}$ and
 $\delta_5$. For each $j=4,\dots,1$ we obtain $\eps_j$ and $C_{j+1}$ from
 Lemma~\ref{lem:shrinkirreg} for input $d$, $\eps_{j+1}$, $\delta_j$ and
 $i=1$. We set $\eps_0=\tfrac12\eps_1$. Finally, we let
 $C'=1000C_1C_2C_3C_4C_5(\eps')^{-3}$, and $C=1000(\eps')^{-3}C'$. Given $p>0$
 and $|X|$, let
 \[m=C'\max\big(p^{-2},p^{-1}\log(en/|X|)\big)\,.\]
 We will always apply Lemma~\ref{lem:shrinkirreg} and
 Corollary~\ref{cor:TwoSideInherit} with input pairs of sets of sizes at least
 $m$, and since $m\ge C'p^{-2}$, by choice of $C'$ this is indeed large enough
 for these two lemmas. This is the only place where we need $m\ge C'p^{-2}$ in
 the proof; our probabilistic calculations will always need $m\ge C'p^{-1}\log
 (en/|X|)$.
 
 Again, we may suppose that $p\ge n^{-1/2}$ since otherwise the lemma statement holds vacuously. We may also assume $|Y|\ge|X|$. As before, we assume that $\Gamma$ satisfies~\ref{inh:star} but that $(X,Y)$ and the subgraph $G$ of $\Gamma$ is a counterexample to the conclusion of Lemma~\ref{lem:twoRI}. Our aim is to find disjoint sets $X'$, $Y'$ and $Z'$ which witness failure of regularity inheritance. As before, we let $Z$ consist of $C'\max\big(p^{-2},p^{-1}\log(en/|X|)\big)$ vertices $z$ such that $\big(N_\Gamma(z;X);N_\Gamma(z;Y)\big)$ fails to be $(\eps',d,p)$-regular in $G$. We let $X''=X\setminus Z$, and $Y''=Y\setminus Z$. As before, using~\ref{inh:star} at least $|Z|/2$ vertices $z\in Z$ have $\deg_\Gamma(z;Z)\le 20C'\max(p^{-1},\log n)$, and at most $200p^{-1}\log(en/|X''|)$ vertices of $Z$ either fail $\deg_\Gamma(z;X'')=\big(1\pm\tfrac12\big)p|X''|$ or $\deg_\Gamma(z;Y'')=\big(1\pm\tfrac12\big)p|Y''|$. Again, by choice of $C$ we can find a set $Z'$ of $\tfrac14C'\max\big(p^{-2},p^{-1}\log(en/|X|)\big)=\tfrac14m$ vertices $z\in Z$ which have $\deg_\Gamma(z;Z)\le 20C'\max(p^{-1},\log n)$, and $\deg_\Gamma(z;X'')=\big(1\pm\tfrac12\big)p|X''|$, and $\deg_\Gamma(z;Y'')=\big(1\pm\tfrac12\big)p|Y''|$. Again, by Lemma~\ref{lem:slice} the pair $(X'',Y'')$ is $(\eps_1,d,p)$-regular in $G$, and again by Lemma~\ref{lem:sticking} for each $z\in Z'$ the pair $\big(N_\Gamma(z;X''),N_\Gamma(z;Y'')\big)$ is not $(\delta_1,d,p)$-regular in $G$.
 
 We now choose $k$ to be the largest integer such that $km\le|X''|,|Y''|$. By choice of $C$ we have $km\ge|X|/2$. We apply Lemma~\ref{lem:shrinkirreg} twice. First we apply it with input $d$, $\eps_2$, $\delta_1$ and $i=1$ to $(X'',Y'')$ to obtain $X'\subset X''$ with $|X'|=km$. Second, we apply it with input $d$, $\eps_3$, $\delta_2$ and $i=1$ to $(X',Y'')$ to obtain $Y'\subset Y''$ with $|Y'|=km$. By construction, the pair $(X',Y')$ is $(\eps_3,d,p)$-regular in $G$, for each $z\in Z'$ we have $\deg_\Gamma(z;X'),\deg_\Gamma(z;Y')=\big(1\pm\tfrac12\big)km$, and for each $z\in Z'$ the pair $\big(N_\Gamma(z;X'),N_\Gamma(z;Y')\big)$ is not $(\delta_3,d,p)$-regular in $G$.
 
 We now show that this triple $(X',Y',Z')$ contains the claimed bad object. Again, we need to apply Lemma~\ref{lem:shrinkirreg}. This time, we first apply it with input $d$, $\eps_4$, $\delta_3$ and $i=1$ to $(X',Y')$ to obtain $\tfrac12\binom{km}{m}$ choices of $X^*$, each of size $m$. For each such choice of $X^*$, we apply it with input $d$, $\eps_5$, $\delta_4$ and $i=1$ to $(X^*,Y')$ to obtain $\tfrac12\binom{km}{m}$ choices of $Y^*$, each of size $m$. In total, we find $\tfrac14\binom{km}{m}^2$ choices of $(X^*,Y^*)$ which is an $(\eps_5,d,p)$-lower-regular pair in $G$ with parts of size $m$, such that for each $z\in Z'$ we have $\deg_\Gamma(z;X^*),\deg_\Gamma(z;Y^*)=\big(1\pm\tfrac12\big)^2pm$ and $\big(N_\Gamma(z;X^*),N_\Gamma(z;Y^*)\big)$ is not $(\delta_5,d,p)$-lower-regular in $G$. We call such a pair $(X^*,Y^*)$ a \emph{bad pair} with respect to $Z'$, and the triple $(X',Y',Z')$ containing $\tfrac14\binom{km}{m}^2$ bad pairs the \emph{bad object}.
 
 We now show that a.a.s.\ $G(n,p)$ does not contain any such bad object. To that end, with  $k$ given we fix disjoint sets $X',Y',Z'$ of size $|Z'|=m$ and $|X'|=|Y'|=km$. Suppose that for some graph $G$ this triple forms a bad object. Picking a uniform random partition $X'=X^*_1\cup\dots\cup X^*_k$ and $Y'=Y^*_1\cup\dots\cup Y^*_k$, in expectation at least $k/4$ of the pairs $(X^*_j,Y^*_j)$ are bad pairs. In particular, letting $T$ be the number of pairs of partitions of $X'$ and $Y'$ into sets of size $m$, we see that at least $T/8$ pairs of partitions contain $k/8$ or more bad pairs.
 
 On the other hand, we can estimate the probability that there exists a graph $G\subset G(n,p)$ on $(X',Y',Z')$  such that the uniform random partition contains at least $k/8$ bad pairs in $G(n,p)$. We expose the edges of $\Gamma=G(n,p)$ between each $X^*_j$ and $Y^*_j$, and fix any $(\eps_5,d,p)$-lower-regular subgraphs to obtain $G$. Now we expose the edges from $Z'$ to each of $X'$ and $Y'$. 
 When exposing the edges for a $z\in Z'$ we first reveal its degrees. Note that if the degree of $z$  to some $X^*_j$ or $Y^*_j$ is not $(1\pm\tfrac12)^2 pm$, then by definition $(X^*_j,Y^*_j)$ is not a bad pair with respect to~$Z'$.
Conditioned on the event that all degrees are within $(1\pm\tfrac12)^2 pm$
 Corollary~\ref{cor:TwoSideInherit} implies that the probability that for any given $z\in Z'$ and $j\in[k]$ the pair $\big(N_\Gamma(z;X^*_j),N_\Gamma(z;Y^*_j)\big)$ is not $(\delta_5,d,p)$-regular is at most $\beta^{pm/4}$. Since these events are independent for different $z\in Z'$, the probability that any given $(X^*_j,Y^*_j)$ forms a bad pair with respect to $Z'$ is at most $\beta^{pm^2/16}$. If $\Gamma$ has property~\ref{inh:star}, the probability that there exists some graph $G\subset \Gamma$ such that $(X^*_j,Y^*_j)$ forms a bad pair in $G$ with respect to $Z'$ is thus at most $2^{3pm^2/2}\beta^{pm^2/16}$, which by choice of $\beta$ is at most $\beta^{pm^2/32}$.
 
 Thus the probability that for at least $k/8$ of these pairs there is a graph $G$ in which they are bad with respect to $Z'$ is
 \begin{align*}
  &\sum_{\ell=k/8}^k\binom{k}{\ell}\beta^{\ell pm^2/16}\big(1-\beta^{pm^2/32}\big)^{k-\ell}\\
  \le & k\cdot\binom{k}{k/8}\beta^{kpm^2/1000}\le 2^k\beta^{pkm^2/1000}
 \end{align*}
 where the first inequality is by choice of $\beta$. By linearity of expectation, the expected number of pairs of partitions of $X'$ and $Y'$ in $G(n,p)$ for which there is a graph $G$ such that at least $k/8$ pairs are bad with respect to $Z'$ is at most $2^k\beta^{pkm^2/1000}\cdot T$. Applying Markov's inequality, we see that the probability that $(X',Y',Z')$ can be a bad object with respect to any graph $G$ is at most $8\cdot 2^k\beta^{pkm^2/2000}$. Finally, taking the union bound over choices of $k$, $X'$, $Y'$ and $Z'$ we see that the probability of $G(n,p)$ containing a bad object is at most
 \begin{align*}
  &\sum_{k=1}^{n/m}\binom{n}{km}^2\binom{n}{m}\cdot 8\cdot 2^k\beta^{pkm^2/2000}\le 8 \sum_{k=1}^{n/m}\big(\tfrac{n}{km}\big)^{3km} 2^k\beta^{pkm^2/2000}\\
  \le &8\sum_{|X|=\log n}^{n} \exp\big(3|X|\log(en/|X|)+|X|/m+(\log\beta)\cdot C'|X|\log(en/|X|)/2000\big)\\
  \le & \sum_{|X|=\log n}^{n}\exp\big(-|X|\log(en/|X|)\big)
 \end{align*}
 where the third inequality is since $|X|/2\le km\le |X|$ and $|X|\ge\log n$, and the fourth by choice of $\beta$ and $C'$. This tends to zero as $n$ tends to infinity, which together with Lemma~\ref{cl:inh:dist} shows that a.a.s.\ $G(n,p)$ does not contain any bad object. This completes the proof of Lemma~\ref{lem:twoRI}.
\end{proof}

 We conclude this section by sketching constructions which show that the number of vertices failing to inherit lower-regularity in each of Lemmas~\ref{lem:oneRI} and~\ref{lem:twoRI} are sharp for all $n^{-1/2}\le p\le\tfrac12$ and all $C\max(p^{-2},p^{-1}\log n)\le |X|\le n/4$. 
 
 First, for Lemma~\ref{lem:oneRI}, given $C\max(p^{-2},p^{-1}\log n)\le m\le n/100$ we fix a set $Z$ of $\big\lfloor \tfrac12p^{-1}\log (en/m)\big\rfloor$ vertices. If $|Z|=0$, then trivially there is a set $X$ of $m$ vertices with no neighbours in $Z$. If $|Z|\ge 1$, then the probability that a vertex outside $Z$ has no neighbours in $Z$ is
 \[(1-p)^{|Z|}\ge e^{-2p\cdot p^{-1}\log(en/m)/4}=\big(\tfrac{m}{en}\big)^{1/2}\]
 and so the expected number of vertices outside $Z$ with no neighbours in $Z$ is at least $\tfrac{99n}{100}(m/en)^{1/2}\ge 2m$. By Theorem~\ref{thm:chernoff}, since $m$ tends to infinity with $n$, a.a.s.\ there is a set $X$ of $m$ vertices with no neighbours in $Z$. Furthermore, for each $\eps>0$ a.a.s.\ the pair $(X,[n]\setminus X)$ is $(\eps,\tfrac12,p)$-lower-regular, but trivially for each vertex $z\in Z$ the neighbourhood of $z$ in $X$ does not form a regular pair with $[n]\setminus X$. 
 
 For Lemma~\ref{lem:twoRI}, the same construction shows we need $|Z|\ge cp^{-1}\log(en/|X|)$. Finally, for each $\eps>0$ there exists $p_\eps>0$ such that for $0<p<p_\eps$, Huang, Lee and Sudakov~\cite[Proposition~6.3]{HLS} construct a subgraph $G$ of $\Gamma$ with minimum degree $(1-\eps)pn$ in which $\Omega(p^{-2})$ vertices are not in triangles. This construction in particular shows that $\Omega(p^{-2})$ vertices can fail to inherit regularity in Lemma~\ref{lem:twoRI}.
\section{Deterministic properties of the ambient graph}\label{sec:pseudo}

In this section we introduce the deterministic properties which we require
of our ambient graphs~$\Gamma$.
We then also prove that random and bijumbled
graphs have (a subset of) these properties, in Lemma~\ref{lem:det_Gnp} and~Lemma~\ref{lem:det_jumbled},
respectively.
These are exactly the properties which we shall use in our blow-up lemma proofs.


\medskip

Our first deterministic property  asserts that most vertices of $\Gamma$ have close to the
expected degree into any given reasonably large subset of vertices. This
property will be used for both random and bijumbled graphs.

\begin{definition}[Neighbourhood size property
  $\NS(\eps,T,\Delta)$]\label{def:NS}
 Given $\eps>0$ and integers $T$ and $\Delta$, we say that the graph $\Gamma$
  has property $\NS(\eps,T,\Delta)$ if the following is true for some $p$
  such that $\Gamma$ has density $(1\pm \eps)p$.
  For any set $W$ of size at least
  $\eps p^{\Delta-1} v(\Gamma)/T^2$, there are at most
  $\eps p^{\Delta-1 }v(\Gamma)/T^2$ vertices $v$ outside $W$ such that $\deg_\Gamma(v;W)\neq
  (1\pm \eps)p|W|$.
\end{definition}

Our next property concerns one-sided and two-sided regularity inheritance.
Based on the regularity inheritance lemmas established in Section~\ref{subsec:RI}
we shall also establish this property for random as well as bijumbled graphs.
Recall our convention though, that in random graphs regular pairs are
lower-regular pairs and in bijumbled graphs they are fully-regular pairs.

\begin{definition}[Regularity inheritance property
  $\RI(\eps,(\eps_{a,b}),\eps',d,T,\Delta)$]\label{def:RI} \mbox{}\\
  Given $0<\eps<\eps'<d$ and integers $T$, $\Delta$, we say that the graph $\Gamma$
  has property $\RI(\eps,(\eps_{a,b}),\eps',d,T,\Delta)$ if we have $\eps_{a,b}\in[\eps,\eps']$ for each $0\le a,b\le\Delta-1$
  such that the following holds for some $p$ such that $\Gamma$ has density $(1\pm \eps)p$.

  If $0\le a\le \Delta-2$ and $0\le b\le\Delta-1$, if $G\subset\Gamma$ and if
  $X,Y\subset V(\Gamma)$ are disjoint sets with
  \[|X|\ge\eps' p^{\Delta-2} \cdot \frac{v(\Gamma)}{T^2} 
    \qquad\text{and}\qquad
    |Y|\ge\eps' p^{\Delta-1}\cdot \frac{v(\Gamma)}{T^2}
  \]
  and if $(X,Y)$ is $(\eps_{a,b},d,p)$-regular pair in $G$, then
  $\big(N_{\Gamma}(v; X),Y\big)$ is $(\eps_{a+1,b},d,p)$-regular in~$G$ for
  all but at most $\eps p^{\Delta-1} v(\Gamma)/T^2$ vertices
  $v\in V(\Gamma)\setminus(X\cup Y)$. Furthermore, if additionally
  $b\le \Delta-2$ and
  \[|Y|\ge \eps'p^{\Delta-2}\cdot\frac{v(\Gamma)}{T^2}\,,\]
  then the pair $\big(N_{\Gamma}(v;X),N_{\Gamma}(v;Y)\big)$ is
  $(\eps_{a+1,b+1},d,p)$-regular in $G$ for all but at most
  $\eps p^{\Delta-2} v(\Gamma)/T^2$ vertices $v\in V(\Gamma)\setminus(X\cup Y)$.
\end{definition}



The next property concerns the count of certain stars in~$\Gamma$, and we
shall call it congestion property (following~\cite{ChvRand}, where a very
similar property was used). We only establish this property for random
graphs. For bijumbled graphs it does not hold (with any reasonable choice of
parameters).
Given a graph $\Gamma$, a set $U\subset V(\Gamma)$ and a collection
$\mathcal{F}$ of pairwise disjoint $\ell$-sets in $V(\Gamma)$, we define the
\emph{congestion graph} $\AG(\Gamma,U,\mathcal{F})$ to be the bipartite graph with
vertex sets $U$ and $\mathcal{F}$ with $uF$ an
edge of $\AG(\Gamma,U,\mathcal{F})$ if $u\in U$ is a common neighbour in $\Gamma$
of the vertices in $F\in\mathcal{F}$.

\begin{definition}[Congestion property $\CON(\rho,T,\Delta)$]
  \label{def:con} Given $\rho>0$ and integers $T$ and $\Delta$,
  we say that $\Gamma$ has property $\CON(\rho,T,\Delta)$ if the following
  statement is true for some $p$ such that $\Gamma$ has the density $(1\pm\rho)p$.
  For each $1\le \ell\le\Delta$, each $U\subset V(\Gamma)$ and each
  collection $\mathcal{F}$ of pairwise disjoint $\ell$-sets in $V(\Gamma)\setminus U$, if we have
  $|U|\le |\mathcal{F}|\le \rho v(\Gamma)$, then
  \[e\big(\AG(\Gamma,U,\mathcal{F})\big)\le
  7p^\ell|U||\mathcal{F}|+\rho p^\ell\cdot \frac{v(\Gamma)}{T}|\mathcal{F}|\,.\]
\end{definition}



The congestion condition will help us to verify Hall's condition (on some
linearly sized set) in order to
embed many vertices at a time. However, in the proof of our
blow-up lemma for degenerate graphs,
Lemma~\ref{lem:degen}, we cannot use this strategy. Instead we will embed the vertices one by one in the
given order. For this purpose we need to have a `local' version of the
congestion property, which we will want to apply to sets~$U$ in the common
$\Gamma$-neighbourhood of already embedded vertices. We remark that for
such small sets~$U$ the bound on the number of edges in the congestion
graph given by the congestion property becomes trivial. The following local
congestion property is designed to give a useful bound in this case.
Again, we shall only establish this property for random graphs.

\begin{definition}[Local congestion property $\LCON(\eps,T,\Delta)$]
\label{def:lcon} Given $\eps>0$ and integers $T$ and $\Delta$,
  we say that $\Gamma$ has property $\LCON(\eps,T,\Delta)$ if the following
  statement is true for some $p$ such that $\Gamma$ has density $(1\pm\eps)p$.
  For each $i,\ell\ge 1$ with $i+\ell\le\Delta$, each $U\subset V(\Gamma)$
  of size at least $\eps p^i v(\Gamma)/T^2$ and each
  collection~$\mathcal{F}$ of pairwise disjoint $\ell$-sets in $V(\Gamma)\setminus U$ we have
  \[e\big(\AG(\Gamma,U,\mathcal{F})\big)\le
  7p^{\ell}|U|\max\big(\eps|U|,|\mathcal{F}|\big)\,.\]
\end{definition}


Since for bijumbled graphs we cannot use the congestion property or
local congestion property we introduce the following final deterministic
property which we shall use there instead. This property is a strengthening
of the neighbourhood size property, which also takes smaller sets~$W$ 
in which we measure neighbourhoods into consideration,
and further requires the number of
exceptional vertices outside~$W$ to be smaller. We call it `lopsided' to distinguish it from property $\NS(\eps,T,\Delta)$, in which the maximum number of failing vertices and the minimum size of the set into which neighbourhoods are taken are the same.

\begin{definition}[Lopsided neighbourhood size property $\LNS(\eps,T,\Delta)$]\label{def:LNS}
  Given $\eps>0$ and integers $T$ and $\Delta$, we say that $\Gamma$
  has property $\LNS(\eps,T,\Delta)$ if the following is true for some $p$
  such that $\Gamma$ has density $(1\pm \eps)p$.
  For each $0\le j\le \Delta-1$ and any set $W$ of size at least
  $\eps p^{\Delta+j} v(\Gamma)/T^2$, there are at most
  $\eps p^{2\Delta-j-1} v(\Gamma)/T^2$ vertices $v$ outside $W$ such that $\deg_\Gamma(v;W)\neq
  (1\pm \eps)p|W|$.
\end{definition}
  
  \subsection{Random graphs}
  \label{subsec:GnpPseudo}
  In this subsection we will show that $G(n,p)$ has the neighbourhood size
  property, regularity inheritance property, congestion property and local
  congestion property by proving the following lemma

  \begin{lemma}[Deterministic properties of $G(n,p)$]\label{lem:det_Gnp}
    For every $\Delta\ge 2$ and $d,\eps'>0$ there exist $\eps>0$ and
    $\eps_{a,b}>0$ for each $0\le a,b\le\Delta-1$ such that for every $T$
    and $\rho>0$ there exists $C>0$ such that if
    $p\ge C(\log n/n)^{1/\Delta}$ then $G(n,p)$ a.a.s.\ has
    \begin{enumerate}[label=\abc]
      \item\label{det_gnp:randns}
        $\NS(\eps,T,\Delta)$,
      \item\label{det_gnp:randri}
        $\RI(\eps,(\eps_{a,b}),\eps',d,T,\Delta)$,
      \item\label{det_gnp:randcon}
        $\CON(\rho,T,\Delta)$,
      \item\label{det_gnp:randlcon}  $\LCON(\eps,T,\Delta)$.
    \end{enumerate}
  \end{lemma}

  The proof of part~\ref{det_gnp:randns} is standard. Part~\ref{det_gnp:randri}
  follows from the regularity inheritance lemmas, Lemmas~\ref{lem:oneRI} and~\ref{lem:twoRI}.
  The proof of part~\ref{det_gnp:randcon} follows arguments
  from~\cite{ChvRand} (but is slightly different), and
  part~\ref{det_gnp:randlcon} is proved similarly. For completeness we give the
  details. 


  \begin{proof}[Proof of Lemma~\ref{lem:det_Gnp}]
    Given $\Delta\ge 1$ and $d,\eps'>0$, we assume without loss of
    generality that
    $(1-\eps')^\Delta>1/2$. We choose $\eps_{a,b}$ for each $0\le a\le\Delta-1$ and $0\le b\le\Delta-1$ as follows. We set
    $\eps_{\Delta-1,\Delta-1}=\eps'/2$ and we define the other $\eps_{a,b}$ inductively.  
    For each $a$ and $b$, we require that $\eps_{a,b}$ is smaller than 
    the $\eps_0$ returned by Lemma~\ref{lem:oneRI} with input
    $\eps_{a+1,b}/2$ and $d$ (provided $a<\Delta-1$), and that returned with
    input $\eps_{a,b+1}/2$ and $d$ (provided $b<\Delta-1$), and also than the $\eps_0$
    returned by Lemma~\ref{lem:twoRI} with input $\eps_{a+1,b+1}/2$ and $d$
    (provided $a<\Delta-1,b<\Delta-1$). Let~$\eps$ be the minimum of the
    $\eps_{a,b}$. Note that we then have $\eps_{a,b}=\eps_{b,a}$ for each $a,b$.
    Given $T$, let $C_1$ be the maximum of all of the constants $C$
    returned by the above applications of Lemmas~\ref{lem:oneRI}
    and~\ref{lem:twoRI}. Further, given~$\rho$ set
    \[C=\frac{100T^2C_1\Delta}{\eps^4\rho}\,.\] 
    Now let $p\ge C(\log n/n)^{1/\Delta}$. Let
    $\Gamma=G(n,p)$. By the Chernoff bound in Theorem~\ref{thm:chernoff}, since $p>1/n$, a.a.s.\
    $\Gamma$ has density $\big(1\pm \min(\eps,\rho)\big)p$.  In the following we condition on this
    occurring.

    \smallskip

    Proof of~\ref{det_gnp:randns}:
    Using the Chernoff bound again, if $X$ and $Y$ are any
    disjoint sets of size at least $6\eps^{-2}p^{-1}\log n$, then the probability
    that $e(X,Y)\neq (1\pm \eps)p|X||Y|$ is at most $2\exp(-\eps^2p|X||Y|/3)$. It follows
    that the probability that there exist such $X$ and $Y$ is at most
    \begin{align*}
      \sum_{|X|,|Y|\ge 6\eps^{-2}p^{-1}\log n}&n^{|X|+|Y|}\cdot
      2 \exp\Big(-\frac{\eps^2p|X||Y|}{3}\Big) \\
      \le & \sum_{|X|,|Y|}2\cdot 2^{2\max(|X|,|Y|)\log
      n}\exp\big(-2\max(|X|,|Y|)\log n\big)\\
       \le & 2n^2\big(2^2\exp(-2)\big)^{\log^2 n}\,,
    \end{align*}
    where the first inequality uses $\min(|X|,|Y|)\ge 6\eps^{-2}p^{-1}\log n$.
    It follows that a.a.s.\ any such pair of sets in~$\Gamma$ has density $(1\pm
    \eps)p$. We condition on this in the following.
    
    Now suppose that $X$ is any set of size at least
    $\eps p^{\Delta-1}n/T^2$ in $V(\Gamma)$, and $Y$ is the set of
    vertices outside $X$ with fewer than $(1-\eps)p|X|$ neighbours in $X$. Since
    the density of $(X,Y)$ is less than $(1-\eps)p$, and since by choice of $C$ we
    have $\eps p^{\Delta-1}n/T^2>6\eps^{-2}p^{-1}\log n$, we conclude
    that $|Y|<6\eps^{-2}p^{-1}\log n<\eps p^{\Delta-1}n/(2T^2)$ (again by choice of
    $C$). The same argument bounds the number of vertices outside $X$ with more
    than $(1+\eps)p|X|$ neighbours in $X$, so $\Gamma$ has $\NS(\eps,T,\Delta)$ as
    desired.

    \smallskip

    Proof of~\ref{det_gnp:randri}: In this proof we shall apply the regularity
    inheritance lemmas,
    Lemmas~\ref{lem:oneRI} and~\ref{lem:twoRI}, less than $3\Delta^2$
    times. Hence we can assume that the less than $3\Delta^2$ properties
    asserted to hold a.a.s.\ for $G(n,p)$ by these applications of the
    lemmas hold simultaneously a.a.s\ for $G(n,p)$. Note that whenever we apply these lemmas, we have $|X|>e$ and hence $\log(en/|X|)\le\log n$. We will not need to use the sharper bound taking into account the size of $X$ given in Lemmas~\ref{lem:oneRI} and~\ref{lem:twoRI}.
    
    Given $0\le a\le\Delta-2$ and $0\le b\le\Delta-1$, suppose that $(X,Y)$ is
    an $(\eps_{a,b},d,p)$-lower-regular pair in a subgraph $G$ of $\Gamma$ with
    $|X|\ge \eps' p^{\Delta-2}n/T^2$
    and $|Y|\ge \eps'p^{\Delta-1}n/T^2$.
    Let $Z$ be the set of vertices outside $X\cup Y$ such that $(N_\Gamma(z; X),Y)$ is
    not $(\eps_{a+1,b},d,p)$-lower-regular in $G$. By choice of $p$, 
    we have $\eps' p^{\Delta-2}n/T^2\ge C_1p^{-2}\log n$ and
    $\eps'p^{\Delta-1}n/T^2\ge C_1p^{-1}\log n$. Therefore we can apply
    Lemma~\ref{lem:oneRI} to conclude that 
    \[|Z|\le C_1p^{-1}\log
    n\le\tfrac{\eps}{2T^2}Cp^{-1}\log n<\eps p^{\Delta-1}n/T^2\,,\]
    as desired.
    
    Similarly, given $0\le a\le\Delta-2$ and $0\le b\le\Delta-2$, suppose that
    $(X,Y)$ is an $(\eps_{a,b},d,p)$-lower-regular pair in a subgraph $G$ of $\Gamma$ with
    $|X|\ge \eps'p^{\Delta-2}n/T^2$ and $|Y|\ge \eps' p^{\Delta-2}n/T^2$.
    Let $Z$ be the set of vertices outside $X\cup Y$ such that $(N_\Gamma(z;X),N_\Gamma(z; Y))$ is not $(\eps_{a+1,b+1},d,p)$-lower-regular in $G$. By
    choice of $p$, we have $\eps' p^{\Delta-2}n/T^2>C_1 p^{-2}\log n$. 
    Therefore  we can apply Lemma~\ref{lem:twoRI} to conclude that 
    \[
    |Z|\le C_1\max(p^{-2},p^{-1}\log n)\le\tfrac{\eps}{2T^2}C\max(p^{-2},p^{-1}\log n)<\eps p^{\Delta-2}n/T^2\,,
    \]
    again as desired.

    \smallskip
   
    Proof of~\ref{det_gnp:randcon}:
    Given $1\le i\le\Delta$, a set $U\subset V(\Gamma)$, and a family~
    $\mathcal{F}$ of pairwise
    disjoint $\ell$-sets in $V(\Gamma)\setminus U$ with
    $|U|\le|\mathcal{F}|\le\rho|V(\Gamma)|$, the graph $\AG(\Gamma,U,\mathcal{F})$
    is a random bipartite graph with edge probability $p^\ell$ and
    with parts $U$ and $\mathcal{F}$. So the expected number 
    of edges of $\AG(\Gamma,U,\mathcal{F})$ is $p^\ell|U||\cF|$. By the
    Chernoff bound in
    Theorem~\ref{thm:chernoff} the probability that
    \[\big|e\big(\AG(\Gamma,U,\mathcal{F})\big)\big|> 7p^\ell|U||\mathcal{F}|+\rho p^\ell
    n|\mathcal{F}|/T\] is at most
    \[\exp\big(-\rho p^\ell n |\mathcal{F}|/T\big)\,.\]
    If $|\mathcal{F}|=m\le\rho n$, then since $|U|\le|\mathcal{F}|$ we have $|U|\le
    m$. Taking a union bound over~$\ell$, over~$m$, over the
    at most $n^{m+1}$ choices for $U$ and its size, and over the at most $n^{\Delta m}$
    choices for $\mathcal{F}$, we see that the probability of failure of
    $\CON(\rho,T,\Delta)$ is at most
    \begin{align*}
      &\sum_{\ell=1}^\Delta\sum_{m=1}^{\rho n}n^{m+1} n^{\Delta
      m}\exp\big(-\rho p^\ell n m/T\big) \le \Delta\sum_{m=1}^{\rho n}\exp\big(2\Delta
      m\log n-\rho p^\Delta n m/T\big)\\
      \le&\Delta\sum_{m=1}^{\rho n}\exp\big(-(\Delta+3)m
      \log n \big)<\Delta n\exp\big(-(\Delta+2)\log n\big)<1/n\,,
    \end{align*}
    where the second inequality uses $p^\Delta n\ge C\log n$ and the choice of
    $C$. 
    We conclude that $G(n,p)$ a.a.s.\ has $\CON(\rho,T,\Delta)$ as desired.

    \smallskip
    Proof of~\ref{det_gnp:randlcon}:
    Given integers $i,\ell\ge 1$ such that $i+\ell\le\Delta$, a set $U\subset V(\Gamma)$ of size at least $\eps p^i n/T^2$, and a family of pairwise
    disjoint $\ell$-sets $\mathcal{F}$ in $V(\Gamma)\setminus U$, the graph $\AG(\Gamma,U,\mathcal{F})$
    is a random bipartite graph with edge probability $p^\ell$ and
    parts~$U$ and~$\mathcal{F}$. The expected number of edges it contains is therefore $p^\ell|U||\mathcal{F}|$.
    
    We separate two cases. If $|\mathcal{F}|\le\eps|U|$, then we have $\Exp e\big(\AG(\Gamma,U,\mathcal{F})\big)\le \eps p^\ell|U|^2$, so by Theorem~\ref{thm:chernoff} the probability of $e\big(\AG(\Gamma,U,\mathcal{F})\big)\ge 7\eps p^\ell|U|^2$ is at most
    \[\exp\big(-6\eps p^\ell|U|^2\big)\le \exp\big(-6\eps^2 p^{i+\ell}n|U|/T^2\big)\le \exp(-3\Delta|U|\log n)\,.\]
    Taking the union bound over the at most $\Delta^2$ choices of $i$ and $\ell$, at most $n^2$ choices of $|U|$ and $|\mathcal{F}|\le|U|$, the at most $n^{|U|}$ choices of $U$ and the at most $n^{\Delta|U|}$ choices of $\mathcal{F}$, we see that the probability that there exists such a choice with $e\big(\AG(\Gamma,U,\mathcal{F})\big)\ge 7\eps p^\ell|U|^2$ is at most
    \[\Delta^2n^2n^{(\Delta+1)|U|}\exp(-3\Delta|U|\log n)\]
    which tends to zero as $n$ tends to infinity.
    
    Next we consider the case $|\mathcal{F}|\ge\eps |U|$. Again by Theorem~\ref{thm:chernoff} the probability of $e\big(\AG(\Gamma,U,\mathcal{F})\big)\ge 7p^\ell|U||\mathcal{F}|$ is at most
    \[\exp\big(-6p^\ell|U||\mathcal{F}|\big)\le \exp\big(-6\eps p^{i+\ell}n|\mathcal{F}|/T^2\big)\le \exp(-3\eps^{-1}\Delta|\mathcal{F}|\log n)\,.\]
    Taking the union bound over the at most $\Delta^2$ choices of $i$ and $\ell$, at most $n^2$ choices of $|U|$ and $|\mathcal{F}|$, the at most $n^{|\mathcal{F}|/\eps}$ choices of $U$ and the at most $n^{\Delta|\mathcal{F}|}$ choices of $\mathcal{F}$, we see that the probability that there exists such a choice with $e\big(\AG(\Gamma,U,\mathcal{F})\big)\ge 7p^\ell|U||\mathcal{F}|$ is at most
    \[\Delta^2n^2n^{(\Delta+1/\eps)|\mathcal{F}|}\exp(-3\eps^{-1}\Delta|\mathcal{F}|\log n)\]
    which tends to zero as $n$ tends to infinity. 
    We conclude that $G(n,p)$ a.a.s.\ has $\LCON(\eps,T,\Delta)$ as desired.
  \end{proof}

\subsection{Bijumbled graphs}

 In this subsection we will verify the neighbourhood size property, the
 regularity inheritance property, and the lopsided neighbourhood size
 property for bijumbled graphs, as stated in the following lemma.

 \begin{lemma}[Deterministic properties of bijumbled graphs]
   \label{lem:det_jumbled}
   For each $\Delta\in\NATS$ and $d,\eps'>0$
   there exist $\eps_{a,b}>0$ for $0\le a,b\le\Delta$ and $\eps>0$
   such that for each $T\in\NATS$ there is $c>0$ such that if
   $p>0$ then any $(p,\beta)$-bijumbled graph on~$n$ vertices has
   \begin{enumerate}[label=\abc]
     \TabPositions{.4\textwidth}
     \item \label{det_jumbled:pseudns} 
       $\NS(\eps,T,\Delta+1)$ \tab
       if $\beta\le c p^{\Delta+1} n$,
     \item \label{det_jumbled:pseudri} 
       $\RI(\eps,(\eps_{a,b}),\eps',d,T,\Delta+1)$ \tab 
       if $\beta\le cp^{\Delta+2}n$,
     \item \label{det_jumbled:pseudlns}
       $\LNS(\eps,T,\Delta)$ \tab
       if $\beta\le c p^{(3\Delta+1)/2}n$.
   \end{enumerate}
 \end{lemma}

 Let us briefly justify why we cannot use the congestion property in
 bijumbled graphs. Indeed, \emph{blowing up} a bijumbled graph by a factor
 $\Delta$---that is, replacing vertices with independent $\Delta$-sets and
 edges with complete bipartite graphs---degrades the bijumbledness only
 slightly, but in this blow up the congestion condition fails badly: If we
 choose $\ell=\Delta$ and as $\Delta$-tuples the blow ups of vertices, then
 the congestion property fails by a factor $p^{\Delta-1}$.

 In the proof of Lemma~\ref{lem:det_jumbled}, Part~\ref{det_jumbled:pseudri} is a
 consequence of the regularity inheritance lemmas, Lemmas~\ref{lem:jrione}
 and~\ref{lem:jritwo}.  Parts~\ref{det_jumbled:pseudns} and~\ref{det_jumbled:pseudlns} use
 the following easy consequences of the definition of bijumbled graphs.

 \begin{proposition}\label{prop:jumbled}
   Let~$\Gamma$ be a $(p,\beta)$-bijumbled graph on~$n$ vertices and let
   $W\subset V(\Gamma)$.
   \begin{enumerate}[label=\abc]
     \item\label{prop:jumbled:a} If $\eps\ge2\beta\frac1{pn}$, then $e(\Gamma)=(1\pm\eps)p\binom{n}{2}$.
     \item\label{prop:jumbled:b} $\deg_\Gamma(v;W)\neq(1\pm\eps)p|W|$
       for at most $2 \beta^2/(\eps^{2} p^{2}|W|)$ vertices $v\in
       V(\Gamma)\setminus W$.
   \end{enumerate}
 \end{proposition}
\begin{proof}
From the definition of jumbledness~\eqref{eq:expander_mixing} we obtain 
\[2e(\Gamma)=e_{\Gamma}(V,V)=pn^2\pm\beta n=pn^2\pm \eps pn^2/2\] 
 and~\ref{prop:jumbled:a} follows.

  For~\ref{prop:jumbled:b} let $U$ be the set of
  $v\in V(\Gamma)\setminus W$ with $\deg_\Gamma(v;W)>(1+\eps)p|W|$. We have
  \[(1+\eps)p|U||W|<e(U,W)\le p|U||W|+\beta\sqrt{|U||W|}\]
  and hence $ |U|<\beta^2/(\eps^{2} p^{2}|W|)$. Similarly, the number of
  $v\in V(\Gamma)\setminus W$ with $\deg_\Gamma(v;W)<(1-\eps)p|W|$
  is also at most $\beta^2/(\eps^{2} p^{2}|W|)$ and~\ref{prop:jumbled:b} follows.
\end{proof}

We can now prove Lemma~\ref{lem:det_jumbled}.

\begin{proof}[Proof of Lemma~\ref{lem:det_jumbled}]
    Given $\Delta\ge 2$ and $d,\eps'>0$, we assume
    $(1-\eps')^\Delta>1/2$. We choose $\eps_{a,b}$ for each $0\le a\le\Delta$ and 
    $0\le b\le\Delta$ as follows. We set
    $\eps_{\Delta,\Delta}=\eps'/2$ and we define the other $\eps_{a,b}$ inductively.  
    For each $a$ and $b$, we require that $\eps_{a,b}$ is smaller than 
    the $\eps_0$ returned by Lemma~\ref{lem:jrione} with input
    $\eps_{a+1,b}/2$ and $d$ (provided $a<\Delta$), and that returned with
    input $\eps_{a,b+1}/2$ and $d$ (provided $b<\Delta$), and the $\eps_0$
    returned by Lemma~\ref{lem:jritwo} with input $\eps_{a+1,b+1}/2$ and $d$
    (provided $a<\Delta,b<\Delta$). Let $\eps$ be the minimum of the
    $\eps_{a,b}$. Note that we then have $\eps_{a,b}=\eps_{b,a}$ for each $a,b$. 
    Let $c'$ be the minimum of all the constants~$c_0$ returned by the applications of  
    Lemmas~\ref{lem:jrione} and~\ref{lem:jritwo}.
    Given $T$, let \[c=\tfrac12\eps^2\eps'c'T^{-2}\,.\]
    Let~$\Gamma$ be a $(p,\beta)$-bijumbled graph on~$n$ vertices with
    $\beta\le c p^{\Delta+1}n$.
    By Proposition~\ref{prop:jumbled}\ref{prop:jumbled:a} the
    graph~$\Gamma$ has density $(1\pm\eps)p$.
 
    \smallskip

    Proof of~\ref{det_jumbled:pseudns}: Let $W\subset V(\Gamma)$ be a set with
    $|W|\ge \eps p^{\Delta}n/T^2$. By
    Proposition~\ref{prop:jumbled}\ref{prop:jumbled:b} the number of
    vertices $v\in V(\Gamma)\setminus W$ such that
    $\deg(v;W)\neq(1\pm\eps)p|W|$ is at most
    \begin{equation*}
      \frac{2\beta^2}{\eps^2p^2\cdot \eps p^{\Delta}nT^{-2}}
      \le\frac{2c^2T^2p^{\Delta}n}{\eps^3}\le\eps p^{\Delta}\frac{n}{T^2}\,,
    \end{equation*}
    and hence $\Gamma$ has $\NS(\eps,T,\Delta+1)$.

  \smallskip

   Proof of~\ref{det_jumbled:pseudlns}: Assume $\beta\le c p^{(3\Delta+1)/2}n$.
   Let $0\le j\le\Delta$ and $W\subset V(\Gamma)$ be a set with
   $|W|\ge \eps p^{\Delta+j}n/T^2$. By 
    Proposition~\ref{prop:jumbled}\ref{prop:jumbled:b} the number of
    vertices $v\in V(\Gamma)\setminus W$ such that
    $\deg(v;W)\neq(1\pm\eps)p|W|$ is at most
    \begin{equation*}
      \frac{2\beta^2}{\eps^2p^2\cdot \eps p^{\Delta+j}nT^{-2}}
      \le\frac{2c^2T^2p^{2\Delta-j-1}n}{\eps^3}\le\eps p^{2\Delta-j-1}\frac{n}{T^2}\,,
    \end{equation*}
    and hence $\Gamma$ has $\LNS(\eps,T,\Delta)$.
  
   \smallskip
 
    Proof of~\ref{det_jumbled:pseudri}:
    We shall prove that $\Gamma$ has
    $\RI(\eps,(\eps_{a,b}),\eps',d,T,\Delta+1)$ if $\beta\le cp^{\Delta+2}n$ by contradiction.
    Given $0\le a\le\Delta-1$ and $0\le b\le\Delta$, suppose that $(X,Y)$ is
    an $(\eps_{a,b},d,p)$-fully-regular pair in $G\subset\Gamma$ with
    $|X|\ge \eps'p^{\Delta-1}n/T^2$ and $|Y|\ge \eps'p^{\Delta}n/T^2$.
    Let~$Z$ be the set of vertices $z\in V(\Gamma)\setminus(X\cup Y)$ such that $(N_\Gamma(z; X), Y)$ is
    not $(\eps_{a+1,b},d,p)$-fully-regular in $G$. Assume for contradiction
    that $|Z|\ge\eps p^{\Delta}n/T^2$.
    Then we have $\max\big(|X||Y|,|X||Z|\big)\ge\eps\eps'p^{2\Delta-1}n^2/T^4$. 
    Therefore, by choice of $c$ and since $(\log_2\tfrac{1}{p})^{-1}>p$, we see that
    \[
    \beta\le
   cp^{\Delta+2}n\le c'\eps\eps' p^{\Delta+2} n/T^2< c'p^2(\log_2
   \tfrac{1}{p})^{-1/2}\sqrt{\max\big(|X||Y|,|X||Z|\big)}\,,
    \]
    and thus, $\Gamma$ is bijumbled enough to apply Lemma~\ref{lem:jrione}, 
    whose statement contradicts the assumption on~$Z$.
    
    Similarly, given $0\le a,b\le\Delta-1$, suppose that $(X,Y)$ is an
    $(\eps_{a,b},d,p)$-fully-regular pair in $G\subset\Gamma$
    with $|X|\ge \eps'p^{\Delta-1}n/T^2$ and $|Y|\ge
    \eps'p^{\Delta-1}n/T^2$.  Let $Z$ be the set of vertices $z\in
    V(\Gamma)\setminus(X\cup Y)$ such that $(N_\Gamma(z;X),N_\Gamma(z; Y))$ is not $(\eps_{a+1,b+1},d,p)$-fully-regular in
    $G$. Again, assuming for contradiction that $|Z|\ge\eps p^{\Delta-1}n/T^2$, we have
    \[\max\big(|X||Y|,|Y||Z|, |X||Z|\big)\ge\eps\eps'p^{2\Delta-2}n^2/T^4,\] 
    so that by choice of $c$, we see that 
    \[
    \beta\le
    cp^{\Delta+2}n\le c'\eps' p^{\Delta+2} n/T^2< 
    c'p^3 \sqrt{\max\big(|X||Y|,|Y||Z|, |X||Z|\big)}\,,
   \]
   hence the conclusion of Lemma~\ref{lem:jritwo} contradicts the
   assumption on~$Z$.
\end{proof}

\section{The general setup}	\label{sec:gensetup}

This section constitutes the first step towards the proofs of our blow-up
lemmas. More or less, one can think of this section as showing that, for each of our blow-up lemmas, it suffices to prove a corresponding blow-up lemma in which several extra conditions are required. This reduction, as mentioned in the proof overview (Section~\ref{sec:proof_overview}) involves refining the partitions $\cXbl$ of $H$ and $\cVbl$ of $G$ given by the user of the blow-up lemma to obtain new partitions $\cX$ of $H$ and $\cV$ of $G$, which in turn entails replacing the reduced graphs $\Rbl$ and $\Rpbl$ with new graphs $R$ and $R'$, and altering various constants mentioned in the blow-up lemmas. It is convenient to keep using the letter choices seen in our blow-up lemmas for these altered partitions, graphs and constants in the proofs to come, so in this section we use the suffix $\mathrm{BL}$ for objects mentioned in one of our blow-up lemmas which we will alter for the proof, following the style of the previous sentence.

For convenience, we shall collect the various conditions on the graphs~$H$
and~$G$ that we assume in our blow-up lemmas and those extra conditions
that we derive in this section in a \emph{general setup} (see
Definition~\ref{def:general_setup}) which we shall use in many forthcoming
lemmas. Moreover, we will define the concept of \emph{good partial
  embedding} (see Definition~\ref{def:GPE}), which is a collection of
invariants we maintain in our proofs when embedding~$H$.

This section is structured as follows.  In
Section~\ref{sec:setup_constants} we collect and describe, mainly as
reference for the reader, the various constants that appear in our blow-up
lemma proofs. Sections~\ref{subsec:partH} and~\ref{subsec:partG} then list
the properties we require of the refined partitions of~$H$ and~$G$,
respectively. In Section~\ref{subsec:obtain} we prove that partitions with
these desired properties can be obtained from partitions supplied to our
blow-up lemmas. In Section~\ref{sec:gpe} we define some further key
concepts used in our embedding procedures, such as candidate sets,
available candidate sets and good partial embeddings, which then allows us
to formally define our general setup.  Finally, in
Section~\ref{subsec:bad_vertices}, we define the notion of bad vertices
which will, as explained in the proof overview, be vertices avoided in our
embedding process, and prove Lemma~\ref{lem:fewbad} which states that most
vertices in candidate sets are not bad.

\subsection{Constants}\label{sec:setup_constants}

In the following chapters, sections and lemmas we will use various constants.
These are listed below, with their meaning. Firstly, we have
the following constants which are chosen by the user of our blow-up
lemmas.

\begin{description}
\item[$\Delta$] is the maximum degree of $H$.
\item[$\Delta_J$] is the maximum number of times a vertex of $\Gamma$ is a
  restricting vertex, that is, appears in a set $J_x$ for $x\in V(H)$.
\item[$D$] is a constant governed by the degeneracy of~$H$, which is only
  used in the degenerate graphs blow-up lemma.
\item[$d$] is the minimum $p$-density of regular pairs.
\end{description}

\smallskip
\noindent
The user of our blow-up lemmas further chooses the constants $\DeltaRpbl$, $\alphabl$, $\zetabl$, $\kappabl$ and $\ronebl$ (the last after being given $\epsbl$ and $\rhobl$). However, our preprocessing in Lemma~\ref{lem:matchreduce} changes the values of these constants (and the reduced graphs $R$ and $R'$), giving the following constants which we use in the rest of the paper. 

\smallskip

\begin{description}
\item[$\DeltaRp$] is the maximum degree of the reduced graph $R'$, which captures
  super-regular pairs.
\item[$\alpha$] is the fraction of a part required to be potential buffer vertices.
\item[$\zeta$] is the minimum relative size of any image restriction set~$I_x$.
\item[$\kappa$] is the balancing factor (greater than~$1$), bounding
  $|V_i|/|V_j|$ for any pair of clusters of~$G$ (and hence parts of~$H$).
\item[$r_1$] is the upper bound on the number of clusters.
\end{description}

\smallskip

\noindent
The blow-up lemmas guarantee the existence of the following constants.  

\smallskip

\begin{description}
\item[$\epsbl$] is the regularity we require in the regular partition
  provided by the user of the blow-up lemma.
\item[$\rhobl$] is the fraction of vertices in each part of the partition of $H$ supplied by the user which may be
  image restricted.
\item[$C$] is large, only appears in the random graphs blow-up lemmas and is the
  constant factor in the bound on the probability~$p$.
\item[$c$] is small, only appears in the bijumbled graphs blow-up lemma and is the
  constant factor in the bound on the bijumbledness error term~$\beta$.
\end{description}

\smallskip

\noindent
Furthermore, the most important additional constants appearing in the
proofs of the blow-up lemmas are the following.

\smallskip

\begin{description}
\item[$\mu$] is the fraction of each cluster of $G$ contained in each of
  the ``small'' sets of the partition of~$G$, thus $|\Vq_i|=|\Vc_i|=|\Vbuf_i|=\mu|V_i|$ for each $i$.
\item[$\vartheta$] comes into the maximum fraction $\rho p^{\vartheta}$ of each part of $H$ which may be image restricted. It is equal to either zero (proving Lemmas~\ref{lem:rg_image} and~\ref{lem:degen}) or $\Delta$ (proving Lemma~\ref{lem:psr_main}).
\item[$\rho$] appears in the maximum fraction $\rho p^{\vartheta}$ of each part of $H$ which may be image restricted. We also use~$\rho$ for a second purpose (which does not conflict with the first one): it is the fraction of
  non-image restricted vertices in~$X_i$ which may enter the queue
  during the RGA without causing the RGA to fail. Hence in total at most $2\rho|X_i|$ vertices of~$X_i$ may enter
  the queue during the RGA.
\item[$\eps'$] is the global worst case regularity appearing in the proof
  of the blow-up lemmas. That is, whenever we use the fact that a pair is
  $(\cdot,d,p)$-regular in $G$, it will be at worst
  $(\eps',d,p)$-regular. 
\item[$\eps_{a,b}$] is the worst case regularity in the proofs between
  underlying restriction sets of adjacent vertices in $H$ with respectively
  $a$ and $b$ previously embedded neighbours (see Section~\ref{sec:gpe} for
  definitions of these terms). 
\item[$\eps$] is the initial regularity after pre-processing, and also controls the deterministic properties we require of $\Gamma$.
\end{description}

\smallskip

\noindent
In order to make our proofs work it is enough to have constants in the size order
\begin{align*}
  0 &< \eps\ll\eps_{a,b} \le\eps'
  \ll\rho \ll\mu
  \ll \alpha,\zeta,d,\Delta^{-1},\DeltaRp^{-1},\Delta_J^{-1},D^{-1},\kappa^{-1} \\
  \text{and}\quad 0&<c,C^{-1} \ll r_1^{-1},\eps 
\end{align*}
where by $x\ll y$ we mean that there is a non-decreasing function
$f\colon(0,1]\to(0,1]$ such that our proof works if $0<x\le f(y)$.  Observe
that the constants on the right hand sides are (effectively) chosen by the user of the
blow-up lemma. We remark
that one can then safely read this paper assuming that, for example, any
function of~$\rho$ appearing in the proofs tending to zero
with~$\rho$ is much smaller than any function of~$\mu$ appearing in the
proofs tending to zero with~$\mu$. However, for the convenience of
the reader wishing to verify the proofs, we specify our constants (more or less)
explicitly in each of the following results. We have made no attempt to optimise the values we give.

At the beginning of the proof of each of our blow-up lemmas we give explicit values for these various constants. In the proofs of the auxiliary lemmas, we begin with `we require' various inequalities to hold. By this we mean that if these inequalities hold then the proof of the auxiliary lemma goes through, and it is easy to check that the required inequalities hold for our explicit values.

\subsection{The partition of \texorpdfstring{$H$}{H}}
\label{subsec:partH}

We shall refine the partition $\cXbl$ of~$H$ given
as input to one of our blow-up lemmas to obtain a partition
$\cX=\{X_i\}_{i\in[r]}$ of $V(H)$ with some additional properties, which we
shall need in our proofs.  We shall also refine the given partition of~$G$
(see Section~\ref{subsec:partG} for the properties we require from this
refined partition of~$G$). In this subsection we will define the properties
that we require from the refined partition of~$H$. To put this definition
into context (and provide some explanations on how it fits to the
refinement of the partition of~$G$ in the subsequent subsection) we first need
some explanations concerning our strategy in the proofs of the blow-up
lemmas.


We shall (in a series of steps, some of which are performed in
this preprocessing step, and some of which are performed on later parts of
the proof) construct subsets $\Xmain_i$, $\Xq_i$, $\Xc_i$ and
$\Xbuf_i$ of~$X_i$.
We define
\begin{equation*}
 \Xmain=\bigcup_{i\in[r]}\Xmain_i \,,\quad
 \Xq=\bigcup_{i\in[r]}\Xq_i \,,\quad
 \Xc=\bigcup_{i\in[r]}\Xc_i \,,\quad
 \Xbuf=\bigcup_{i\in[r]}\Xbuf_i \,.
\end{equation*}
Here, $\Xbuf_i\subset\tX_i$ contains the \emph{buffer vertices}, which will
be chosen from the set of potential buffer vertices (see Section~\ref{subsec:obtain}).
The set~$\Xc$ is only used in the
proof of the random graphs blow-up lemma, Lemma~\ref{lem:rg_image}; in the
other proofs we set $\Xc=\emptyset$. The vertices in~$\Xc$ shall form the
\emph{reserved cliques} which we will use to fix so-called \emph{buffer
  defects} (see Section~\ref{sec:fixbuffer} for more details). We choose this set of vertices at the beginning of the proof of Lemma~\ref{lem:rg_image}. Once we
chose~$\Xbuf$ and~$\Xc$, we set
$\Xmain_i:=X_i\setminus(\Xbuf_i\cup \Xc)$. These sets will have sizes
$|\Xbuf_i|=4\mu|X_i|$ and $|\Xc_i|\le\mu|X_i|$ and hence
$(1-5\mu)|X_i|\le|\Xmain_i|\le(1-4\mu)|X_i|$.  The set $\Xq\subset\Xmain$
will form the \emph{queue} and only gets chosen during the embedding, in
the random greedy algorithm; it will also be of size at most $\mu|X_i|$.
We remark that in the following definition we shall only refer to the
sets~$X_i$ and~$\Xbuf_i$. We only chose to mention the sets~$\Xq_i$ and~$\Xc_i$
here as well because we need them to motivate the refinement of the
partition of~$G$ in the next subsection.

The properties that we require of our refined partition and the buffer
vertices are collected in the following
definition. We remark that the reduced graphs~$R$ and~$R'$ used for this
refined partition are blow-ups of the reduced graphs $\Rbl$ and $\Rpbl$ provided to the
respective blow-up lemma. These properties bound the number of buffer vertices per part, and state that the first and second
neighbourhoods of buffer vertices lie along edges of $R'$.
They also assert that any pair of vertices
within one part~$X_i$, but also any pair of buffer vertices, and any pair
consisting of a buffer vertex and an image restricted vertex, are far apart. Note that this implies that $H[X_i,X_j]$ forms a matching for any
pair of parts~$X_i$ and~$X_j$, and it also implies that buffer vertices may
not be image restricted. We further require that all vertices in
any one~$\Xbuf_i$ have the same degree, and distinguish between clique buffers
and non-clique buffers. The latter will be important for the selection of
reserved cliques in the proof of Lemma~\ref{lem:rg_image} (Section~\ref{sec:fixbuffer}).

\begin{definition}[Good $H$-partition]
  We say that a partition
  $\cX=\{X_i\}_{i\in[r]}$ of $V(H)$ is a \emph{good $H$-partition} for
  reduced graphs $R'\subset R$ on vertex sets $[r]$, 
  with buffer~$\tcX=\{\tX_i\}_{i\in[r]}$, buffer vertices
  $\Xbuf\subset\bigcup_{i\in[r]}\tX_i$ and image restricting vertices
  $\cJ=\{J_x\}_{x\in  V(H)}$,
  if
  the following
  conditions are satisfied for each $i,j,k\in[r]$.
  \begin{enumerate}[label=\itmarab{H}]
    \item\label{H:partition} $(H,\cX)$ is an $R$-partition and $\tcX$ is an
      $(\alpha,R')$-buffer for $(H,\cX)$.
    \item\label{PtH:dist} $\dist_H(x,y)\ge
      10$ and $J_x\cap J_{y}=\emptyset$ for each $x,y\in X_i$ with $x\neq y$.
    \item\label{BUF:sizebuf} $|\Xbuf_i|=4\mu|X_i|$, 
      and $|\{x\in X_i\colon x\in N(\Xbuf)\}|\le 4\kappa\DeltaRp\mu|X_i|$.
    \item\label{BUF:dist} $\dist_H(x,y)\ge 5$ for each $x,y\in\Xbuf$ with
      $x\neq y$. 
    \item\label{BUF:last} $\dist_H(x,y)\ge 3$ for each $x\in\Xbuf$ and each image restricted $y$.
    \item\label{BUF:deg} All vertices of $\Xbuf_i$ have degree $b$ for some $0\le b\le\Delta$.
      We then call $\Xbuf_i$ a \emph{degree-$b$ buffer}.
    \item\label{BUF:cliqueb} Either all or none of the vertices of $\Xbuf_i$ are in copies of
    $K_{\Delta+1}$. We then call $\Xbuf_i$ a \emph{clique buffer} or \emph{non-clique
      buffer}, respectively.
    \item\label{BUF:existcliques} If $\Xbuf_i$ is a clique buffer, then at
      least $\tfrac{1}{2\Delta+4}\alpha|X_i|$ vertices of $\tX_i$ are in copies
      of $K_{\Delta+1}$ which do not contain either vertices of $\Xbuf$ or
      image restricted vertices.
 \end{enumerate}
\end{definition}

Note that in this definition we do need to know about image restricting vertices (since they play a r\^ole in~\ref{PtH:dist}), which are in $\Gamma$, and we do need to know which vertices of $H$ are image restricted (since they play a r\^ole in~\ref{BUF:last}), but we do not need to know anything about $G$.

\subsection{The partition of \texorpdfstring{$G$}{G}}\label{subsec:partG}

We next define the properties we require from our refined partition
of~$G$. Firstly, we shall subdivide each cluster of the partition $\cVbl$ of~$G$, which is given as input to one of our blow-up lemmas, into several new
clusters to obtain a new partition $\cV$ matching the refined partition of~$H$. Further, as also mentioned
in the proof overview, we shall partition each new cluster~$V_i$ into parts
$\Vmain_i$, $\Vq_i$, $\Vc_i$ and $\Vbuf_i$.
Again, we set
\begin{equation*}
 \Vmain=\bigcup_{i\in[r]}\Vmain_i \,,\quad
 \Vq=\bigcup_{i\in[r]}\Vq_i \,,\quad
 \Vc=\bigcup_{i\in[r]}\Vc_i \,,\quad
 \Vbuf=\bigcup_{i\in[r]}\Vbuf_i \,.
\end{equation*}
As previously indicated, each vertex $x\in V(H)$ will be embedded into the cluster $V_i$
such that $x\in X_i$. In the rest of the paper, we often refer to $V_i$ simply by
$V(x)$, to $\Vmain_i$ by $\Vmain(x)$, and so on, to avoid the use of
indices whenever these are not important.

As explained before, most of the vertices from $\Vmain$ will be used to
embed most vertices of~$\Xmain$, while some vertices from $\Vq$ will be
used to embed the remainder, which constitutes the queue~$\Xq$. We also have the set $\Vc$, the \emph{clique
reservoir}, that we will use in the proof of the random graphs blow-up
lemma to embed some of the vertices in~$\Xc$. Some vertices of each~$\Xc_i$ may need
to be used to fix so-called buffer defects, and these could be embedded
anywhere in $V_i$. All other vertices of each~$\Xc_i$ get embedded to~$\Vc_i$.
Finally, the vertices from
$\Vbuf$ together with all remaining vertices of $\Vmain\cup\Vq\cup\Vc$,
that is, vertices that were not
used when embedding $\Xmain\cup\Xq\cup\Xc$,
will be used for embedding the buffer vertices~$\Xbuf$. 

The properties we require of a good $G$-partition are as follows. It must be a regular $R$-partition, and furthermore the one-sided and two-sided inheritance properties of $\cVbl$ on $\Rpbl$ must be transferred to $\cV$ on $R'$. The super-regularity properties of $\cVbl$ on $\Rpbl$ must be transferred to give minimum degree conditions for each of the parts $\Vmain_i$, $\Vq_i$ and $\Vc_i$. Finally, we need to create new image restriction sets $\cI$ from the supplied $\cIbl$ which with the (unchanged) $\cJ$ form a restriction pair.

\begin{definition}[Good $G$-partition]
  Given a good $H$-partition of~$V(H)$ for reduced graphs $R'\subset R$ on
  vertex sets $[r]$, with buffer~$\tcX=\{\tX_i\}_{i\in[r]}$ and buffer vertices
 $\Xbuf\subset\bigcup_{i\in[r]}\tX_i$, with
  image restrictions $\cI=\{I_x\}_{x\in V(H)}$, and restricting vertices
  $\cJ=\{J_x\}_{x\in V(H)}$, we say that a partition
  $\cV=\{V_i\}_{i\in[r]}$ of $V(G)$ with a partition
  $V_i=\Vmain_i\dcup\Vq_i\dcup\Vc_i\dcup \Vbuf_i$ of each cluster
  is a \emph{good $G$-partition} for
  $R'\subset R$, if
  the following conditions are satisfied.
  \begin{enumerate}[label=\itmarab{G}]
  \item\label{G:sizes}
    For each $i$ we have $|\Vmain_i|=(1-3\mu)|V_i|$ and $|\Vq_i|=|\Vc_i|=|\Vbuf_i|=\mu|V_i|$.
  \item\label{G:inh} 
    $(G,\cV)$ is an $(\eps,d,p)$-regular $R$-partition, 
    which
    has one-sided inheritance on~$R'$,
    and two-sided inheritance on~$R'$ for~$\tcX$.
  \item\label{G:deg}
    For each $ij\in R'$ and $v\in V_i$, we have
    \begin{align*}
     \deg_G(v;\Vmain_j)&\ge(1-3\mu)(d-\eps)\max\big(p|V_j|,\deg_\Gamma(v;V_j)/2\big)\,,\\
     \deg_G(v;\Vq_j)&\ge\mu(d-\eps)\max\big(p|V_j|,\deg_\Gamma(v;V_j)/2\big)\,,\\
     \deg_G(v;\Vc_j)&\ge\mu(d-\eps)\max\big(p|V_j|,\deg_\Gamma(v;V_j)/2\big) \,.
    \end{align*}
  \item\label{G:imres}
    For each $x\in X_i$, we have
    \begin{align*}
     |I_x\cap \Vmain_i|&\ge(1-\eps)(1-3\mu)|I_x|\,,\\
     |I_x\cap\Vq_i|\,,\,|I_x\cap \Vc_i|\,,\,|I_x\cap\Vbuf_i|&\ge(1-\eps)\mu|I_x|\,,\\
     \big|\comN(J_x;\Vmain_i)\big|&=\big(1\pm\eps\big)(1-3\mu)p^{|J_x|}|V_i|\,,\text{ and}\\
     \big|\comN(J_x;\Vq_i)\big|\,,\,\big|\comN(J_x;\Vc_i)\big|\,,\,\big|\comN(J_x;\Vbuf_i)\big|&=\big(1\pm\eps\big)\mu p^{|J_x|}|V_i|\,.
    \end{align*}
  \item\label{G:restr} 
    $\cI$ and~$\cJ$ form a $(\rho p^\vartheta,\zeta,\Delta,\Delta_J)$-restriction pair
    for the partitions~$\cX$ and~$\cV$.
  \item\label{PtH:image} For each $i$, the number of image restricted vertices in $X_i$ is
    at most $\rho p^\vartheta|X_i|$.
\end{enumerate}
\end{definition}

Note that the sizes of the sets $\Vmain_i$, $\Vq_i$, $\Vc_i$, $\Vbuf_i$
are \emph{not} the same as the sizes of the sets $\Xmain_i$, $\Xq_i$,
$\Xc_i$ and $\Xbuf_i$.  For example we will have
$|\Xmain_i|=(1-4\mu)|X_i|<(1-3\mu)|V_i|=|\Vmain_i|$, and
$|\Xq_i|,|\Xc_i|\le 2\rho|X_i|$, which is much smaller than $\mu
|V_i|$. Thus, the subsets $\Xmain_i$, $\Xq_i$, $\Xc_i$ of $X_i$ will be
significantly smaller than the corresponding sets in $V_i$, while the set
$\Xbuf_i$ is larger than $\Vbuf_i$.
We also note that we do require two-sided inheritance for $\tcX$ in~\ref{G:inh}, not just for $\Xbuf$. In the proof of Lemma~\ref{lem:rg_image}, when dealing with buffer defects, we will make use of this stronger statement.

\subsection{Obtaining a good \texorpdfstring{$H$}{H}-partition and a good \texorpdfstring{$G$}{G}-partition}\label{subsec:obtain}

The following lemma proves that we can obtain good partitions of~$H$
and~$G$ from the partitions provided to our blow-up lemmas. The proof of
this lemma is straightforward, though not short since several conditions
must be checked. Briefly, the idea is that we will draw an auxiliary graph
$F_i$ on each part $\Xbl_i$ of $\cXbl$ with edges joining pairs of vertices
which are at distance less than $10$ or share an image restricting
vertex. We will apply our modification of the Hajnal-Szemer\'edi theorem
(Lemma~\ref{lem:nicepartition}) to each of these graphs together with the
set $\tX_i$ to obtain a partition $\cX$ and reduced graphs $R$ and $R'$
satisfying the first three conditions of a good $H$-partition, and then
construct the sets $\Xbuf_i$ greedily to obtain the remaining
properties. We will then randomly refine the partition $\cVbl$ to obtain a
matching partition $\cV$, and further randomly split each part $V_i$ of
$\cV$ into $\Vmain_i$, $\Vq_i$, $\Vc_i$ and $\Vbuf_i$. We will use the
concentration inequalities in Theorem~\ref{thm:chernoff} to obtain
concentration for various set sizes in the randomly chosen parts, which in
particular show that with high probability we obtain the desired good
$G$-partition. We stress that although $p$ with a lower bound reminiscent
of our random graph blow-up lemmas makes an appearance in the lemma
statement, this does not mean we are going to assume $\Gamma$ is a random
graph. We will simply need to know that the quantity $p^bn$ is large
compared to $\log n$.

\begin{lemma}[Good partitions lemma]\label{lem:matchreduce}
  For all positive integers $\Delta$, $\DeltaRp$, $\Delta_J$, $r_1$, $\vartheta$ and
  $b$, all $\kappa>2$, all $\alpha,\zeta,d,\eps>0$, and all sufficiently small $\mu,\rho>0$, there exists
  $C$ such that the following holds whenever $p\ge C\big(\tfrac{\log
    n}{n}\big)^{1/b}$ and $n>C$. Let
  $\delta=\tfrac{1}{8}(\Delta+\Delta_J)^{-10}$ and $r\le r_1$.
   
  Let $\Rbl$ be a graph on $\rbl=\delta r$ vertices and let $\Rpbl\subset \Rbl$ be a spanning
  subgraph with $\Delta(\Rpbl)\leq \DeltaRpbl=\delta\DeltaRp$.
  Let $H$ and $G\subset\Gamma$ be graphs with $\tfrac12\kappa$-balanced
  size-compatible vertex partitions 
  $\cXbl=\{\Xbl_i\}_{i\in[\rbl]}$ and $\cVbl=\{\Vbl_i\}_{i\in[\rbl]}$,
  respectively, which have parts of size at least $2n/(\kappa\delta r_1)$.
  Let $\tcXbl=\{\tXbl_i\}_{i\in[\rbl]}$ be a family of subsets of $V(H)$,
  $\cIbl=\{\Ibl_x\}_{x\in V(H)}$ be a family of image restrictions, and
  $\cJ=\{J_x\}_{x\in  V(H)}$  be a family of restricting vertices.
 Suppose that 
  \begin{enumerate}[label=\rom]
  \item $\Delta(H)\le\Delta$, $(H,\cXbl)$ is an $\Rbl$-partition,
    and $\tcXbl=\{\tXbl_i\}_{i\in[\rbl]}$ is a
    $(2\alpha,\Rpbl)$-buffer for $(H,\cX)$,
  \item $(G,\cVbl)$ is a $\big(\tfrac{1}{2}\delta\eps,d,p\big)$-regular $\Rbl$-partition, which is 
    $\big(\tfrac{1}{2}\delta\eps,d,p\big)$-super-regular on $\Rpbl$, 
    has one-sided inheritance on~$R'$, and two-sided inheritance on~$R'$ for~$\tcXbl$,
  \item\label{matchreduce:imgres} $\cIbl$ and $\cJ$ form
    a $\big(\tfrac{1}{2}\delta\rho p^\vartheta,2\zeta,\Delta,\Delta_J\big)$-restriction pair ,
    and $|J_x|\le b$ for each $x\in V(H)$.
 \end{enumerate}
 Then there is a graph $R$ on $r$ vertices and a spanning subgraph
 $R'\subset R$ with $\Delta(R')\le\DeltaRp$, together with
 $\kappa$-balanced size-compatible partitions $\cX=\{X_i\}_{i\in[r]}$
 of~$H$ and $\cV=\{V_i\}_{i\in[r]}$ of~$G$, which have parts of size at
 least $n/(\kappa r_1)$, a family $\tcX=\{\tX_i\}_{i\in[r]}$ of potential
 buffer vertices, a family $\cI=\{I_x\}_{x\in V(H)}$ of image restrictions,
 subsets $\Xbuf_i\subset\tX_i$ for each $i\in[r]$, and a partition
 $V_i=\Vmain_i\dcup\Vq_i\dcup\Vc_i\dcup\Vbuf_i$ for each $i\in[r]$,
 which give a good $H$-partition and a corresponding good $G$-partition.
\end{lemma}
\begin{proof}
 Let $\delta=\tfrac18(\Delta+\Delta_J)^{-10}$.
 We set $C=10^6\delta^{-1}\mu^{-1}\eps^{-2}d^{-\Delta}\zeta^{-1}\kappa r_1$.
 We require
 \[\mu<\frac{\alpha}{10000\kappa\DeltaRp^4\Delta^{10}(\Delta+2)}\qquad\text{and}\qquad \rho\le\mu\,.\]

First we refine the vertex partition $\cXbl$  of $H$ to obtain properties~\ref{H:partition} and~\ref{PtH:dist}. 
 For each $i\in[\rbl]$ we define a graph $F_i$ on the vertex set $\Xbl_i$ by putting an edge between $x$ and $x'$ whenever either $J_x\cap J_{x'}\neq\emptyset$ or $x$ and $x'$ are at distance less than $10$ in $H$. Observe that $\Delta(F_i)\le(\Delta+\Delta_J)^{10}$. We now apply Lemma~\ref{lem:nicepartition} to $F_i$ with the set $\tXbl_i$. This returns an equitable partition of the vertices $\Xbl_i$ into independent sets in $F_i$ which also partitions $\tXbl_i$ equitably. We do this for each $i\in[\rbl]$ to obtain a partition $\cX$ of $V(H)$ into $8(\Delta+\Delta_J)^{10}\rbl=\delta^{-1}\rbl=r$ parts and a family $\tcX=\{\tX_i\}_{i\in[r]}$. We let $R$ be the graph obtained from $\Rbl$ by replacing each vertex with an independent set on $8(\Delta+\Delta_J)^{10}=\delta^{-1}$ vertices, and each edge with a complete bipartite graph between the corresponding sets. We obtain $R'$ similarly from $\Rpbl$. Thus $(H,\cX)$ is an $R$-partition, satisfying~\ref{PtH:dist} by construction. Furthermore, we have $\Delta(R')=\delta^{-1}\Delta(\Rpbl)\le\DeltaRp$.
 
 Since each $\Xbl_i$ with $i\in[\rbl]$ has size at least $2n/(\kappa\delta r_1)$ and is equipartitioned into $\delta^{-1}$ parts, we see that, because $n$ is chosen sufficiently large, we have $|X_i|\ge n/(\kappa r_1)$ for each $i\in[r]$. Given $i,j\in[r]$ let $i',j'\in[\rbl]$ be such that $X_i\subset\Xbl_{i'}$ and $X_j\subset\Xbl_{j'}$. Then we have
 \[|X_i|\le\delta|\Xbl_{i'}|+1\le\tfrac{\kappa}{2}\delta|\Xbl_{j'}|+1\le \tfrac{\kappa}{2}(|X_j|+1)+1\le\kappa|X_j|\]
 where the final inequality is since $|X_j|\ge n/(\kappa r_1)$ and $n$ is sufficiently large. Thus $\cX$ is $\kappa$-balanced. 
 
 Given $i\in[r]$ let $i'\in[\rbl]$ be such that $X_i\subset\Xbl_{i'}$. Then we have
 \[|\tX_i|\ge \delta|\tXbl_{i'}|-1\ge 2\alpha\delta|\Xbl_{i'}|-1\ge 2\alpha(|X_i|-1)-1\ge\alpha|X_i|\]
 where again the final inequality is since $n$ is sufficiently large. Thus $\tcX$ forms an $(\alpha,R')$-buffer for $(H,\cX)$.
 
 Before we verify the remaining good $H$-partition properties, it is convenient to check~\ref{PtH:image}. By assumption~\ref{matchreduce:imgres} of the lemma, at most $\tfrac12\delta\rho p^\vartheta|\Xbl_{i'}|$ vertices of any given $\Xbl_{i'}$ are image restricted. Thus for any $i$ such that $X_{i}\subset \Xbl_{i'}$, at most 
\[\tfrac12\delta\rho p^\vartheta|\Xbl_{i'}|\le\tfrac12\rho p^\vartheta\big(|X_{i}|+1\big)\le\rho p^\vartheta|X_{i}|\]
vertices of $X_{i}$ are image restricted, as required for~\ref{PtH:image}.
 
 We proceed with properties~\ref{BUF:sizebuf}--\ref{BUF:existcliques} by choosing buffer vertices $\Xbuf_i$ from the potential buffer vertices $\tX_i$ for each $i\in[r]$ sequentially. For a given $i\in[r]$, we first decide on the degree of the buffer and whether or not it is a clique buffer. We split the vertices of $\tX_i$ into $\Delta+2$ subsets, one for each of the possible degrees in $\{0,\dots,\Delta-1\}$, one for vertices of degree $\Delta$ not in copies of $K_{\Delta+1}$, and one for vertices in copies of $K_{\Delta+1}$. We take a largest one $S_i$ of these sets, and will choose $\Xbuf_i$ within it. We do this greedily, one vertex at a time, always picking a vertex at distance at least five from previously chosen buffer vertices, and at distance at least three from any image restricted vertices. We now justify that it is possible to do this, that is, that we can pick the desired $4\mu|X_i|$ vertices in $S_i$ without running out of vertices in $S_i$. 
 
 Since all neighbours (respectively, second neighbours) in $H$ of vertices in $\tX_i$ are in clusters $V_j$ with $ij\in R'$ (respectively, $V_k$ with $ij,jk\in R'$), by~\ref{PtH:image}, verified above, it follows that the number of image restricted vertices at distance $2$ or less to a vertex of $\tX_i$ is at most $(\DeltaRp^2+\DeltaRp+1)\rho\kappa|X_i|$. Each of these vertices has at most one neighbour or second neighbour in $\tX'_i$, since the distance between any two vertices of $X_i$ is at least ten. Similarly, if $x$ is in some $\tX_k$ and is at distance four or less from some vertex of $\tX_i$, then the path between $x$ and $\tX_i$ follows edges of $R'$, and there is only one vertex of $\tX_i$ which is at distance four or less from $x$. Thus the total number of vertices of $\tX_i$ which are at distance four or less from some vertex of $\Xbuf$ is at most \[(1+\DeltaRp+\DeltaRp^2+\DeltaRp^3+\DeltaRp^4)\cdot 4\mu\kappa|X_i|\,.\]
 Since $|S_i|\ge\tfrac{1}{\Delta+2}|\tX_i|$, it follows that at each step the number of vertices in $S_i\subset \tX_i$ we have to choose from is at least
 \[\tfrac{1}{\Delta+2}\alpha|X_i|-3\DeltaRp^2\rho\kappa|X_i|-20\DeltaRp^4\mu\kappa|X_i|\,,\]
 which is greater than $\tfrac{1}{2\Delta+4}\alpha|X_i|$ by choice of $\mu$ and $\rho$. Thus we can always choose the desired vertices of $\Xbuf_i$, giving~\ref{BUF:dist}--\ref{BUF:cliqueb}. In the event that $\Xbuf_i$ is a clique buffer, we have at the end at least $\tfrac{1}{2\Delta+4}\alpha|X_i|$ vertices of $\tX_i$ remaining which are in copies of $K_{\Delta+1}$ (by definition of $S_i$) that do not contain either vertices in $\Xbuf$ or image restricted vertices, giving~\ref{BUF:existcliques}.
 
  Finally, neighbours of $\Xbuf$ in $X_i$ must be in parts $X_j$ such that $ij\in R'$. Since any two vertices in $X_i$ are at distance at least ten, each such vertex has at most one neighbour in $X_i$ and so at most $\DeltaRp4\mu\kappa|X_i|$ vertices of $X_i$ are in $N(\Xbuf)$, completing~\ref{BUF:sizebuf}. This establishes that we have a good $H$-partition.
  
  We now construct a corresponding good $G$-partition. The construction is simple: for each $\Vbl_i\in\cVbl$ we choose, uniformly at random, an equitable partition into $\delta^{-1}$ parts, which parts we associate (arbitrarily) to subparts of $\Xbl_i$ of corresponding size, obtaining $\cV$. Note that this is possible since $|\Xbl_i|=|\Vbl_i|$. We then, for each $i\in[r]$, choose uniformly at random a partition of $V_i$ into one set of size $(1-3\mu)|V_i|$ and three of size $\mu|V_i|$, which form $\Vmain_i$, $\Vq_i$, $\Vc_i$ and $\Vbuf_i$ respectively. For each $x\in V_i$, we let $I_x=\Ibl_x\cap V_i$.
  
  By construction $\cX$ and $\cV$ are size-compatible, so what remains is to show that $\cV$ is with positive probability a good $G$-partition. In fact, we will show this holds asymptotically almost surely. Recall that we have already established~\ref{PtH:image}, and by construction~\ref{G:sizes} holds, so that the remaining conditions are~\ref{G:inh}--\ref{G:restr}.

  Observe that the sets $V_i$, $\Vmain_i$, $\Vq_i$, $\Vc_i$ and $\Vbuf_i$ are all distributed as the uniform random set of that size within $\Vbl_{i'}$, where $V_i\subset \Vbl_{i'}$, though they are of course not independent. It follows that for any set $U\subset \Vbl_{i'}$, the sizes of each of the sets $V_i\cap U$, $\Vmain_i\cap U$, $\Vq_i\cap U$, $\Vc_i\cap U$ and $\Vbuf_i\cap U$ are hypergeometrically distributed, each with (since $n$ is sufficiently large) expectation at least $\tfrac12\mu\delta|U|$. In particular, if $|U|\ge 10^3\eps^{-2}\log n$, then the probability that any given one of these sizes is not within a $\big(1\pm\tfrac{\eps}{10}\big)$-factor of its expectation is by Theorem~\ref{thm:chernoff} at most
  \[2e^{-\eps^2|U|/300}\le n^{-2}\,.\]
  
  We now give various sets $U$ to which we will apply this observation, and verify that each is sufficiently large.  
  For each $ij\in R'$ and each $v\in V_j$, by construction of $R'$ and since $(G,\cVbl)$ is super-regular on $\Rpbl$, we have
  \begin{align*}
   \deg_\Gamma(v;\Vbl_{i'})&\ge \deg_G(v;\Vbl_{i'})\\
   &\ge (d-\eps)\max\big\{p|\Vbl_{i'}|,\tfrac12\deg_\Gamma(v;\Vbl_{i'})\big\}\ge10^3\eps^{-2}\log n\,,
  \end{align*}
  where the final inequality holds by choice of $p$ and $C$. We can thus take $U$ to be $N_\Gamma(v;\Vbl_{i'})$ or $N_G(v;\Vbl_{i'})$.
  
  For each $x\in X_i$, where $X_i\subset\Xbl_{i'}$, since $\big(\cIbl,\cJ\big)$ form a $\big(\tfrac12\delta\rho,2\zeta,\Delta,\Delta_J\big)$-restriction pair, we have
  \[\big|\comN_\Gamma(J_x;\Vbl_{i'})\big|\ge |\Ibl_x|\ge\zeta(dp)^{|J_x|}|\Vbl_i|\ge 10^3\eps^{-2}\log n\,,\]
  where the final inequality comes from the assumption $|J_x|\le b$ and our choice of $p$ and $C$. We can thus also take $U$ to be $\comN_\Gamma(J_x;\Vbl_{i'})$ or $I_x$.
  
  In total, we have at most $5r$ randomly chosen sets in which we are interested, and at most $4n$ choices of $U$ with which we intersect the random sets. Taking the union bound, we see that a.a.s.\ each of these intersections has size within a $\big(1\pm\tfrac{\eps}{10}\big)$-factor of its expectation. In particular, there exists a partition $\cV$ in which each of these intersections is within a $\big(1\pm\tfrac{\eps}{10}\big)$-factor of its expectation. We fix such a partition, and prove that it is the desired good $G$-partition.
  
 We begin with~\ref{G:inh}. Suppose $ij\in R$. Then $V_i$ and $V_j$ come from two parts $\Vbl_{i'}$ and $\Vbl_{j'}$ of $\cVbl$ which form a $\big(\tfrac{1}{2}\delta\eps,d,p\big)$-regular pair, and we have $|V_i|>\tfrac12\delta|\Vbl_{i'}|$ and $|V_j|>\tfrac1{2}\delta|\Vbl_{j'}|$. It follows from the definition that $(V_i,V_j)$ is $(\eps,d,p)$-regular, so $(G,\cV)$ forms an $(\eps,d,p)$-regular $R$-partition. Given now $ij,jk\in R'$, with $V_j\subset\Vbl_j$ and $V_k\subset\Vbl_k$, and any $v\in V_i$, we have by construction of $\cV$ that $|\big|N_\Gamma(v;V_j)\big|\ge\tfrac12\delta\big|N_\Gamma(v;V_{j'})\big|$ and that $|V_k|\ge\tfrac12\delta|V_{k'}|$. Since the pair $\big(N_\Gamma(v;V_{j'}),V_{k'}\big)$ is $\big(\tfrac12\delta\eps,d,p\big)$-regular in $G$, it follows from the definition that the pair $\big(N_\Gamma(v;V_{j}),V_{k}\big)$ is $(\eps,d,p)$-regular in $G$, so that we have one-sided inheritance on $R'$. A similar argument shows that if also $jk\in R'$ and there is a triangle of $H$ with one vertex in each of $\tX_i$, $X_j$ and $X_k$, then also $\big(N_\Gamma(v;V_{j}),N_\Gamma(v;V_{k})\big)$ is $(\eps,d,p)$-regular in $G$, so that we have two-sided inheritance on $R'$ for $\tcX$.
 
 Both of~\ref{G:deg} and~\ref{G:imres} follow directly from the construction of $\cV$ together with, respectively, the $\big(\tfrac12\delta\eps,d,p\big)$-super-regularity of $(G,\cVbl)$ on $\Rpbl$ and that $(\cIbl,\cJ)$ forms a $\big(\tfrac12\delta\rho p^\vartheta,2\zeta,\Delta,\Delta_J\big)$-restriction pair. The same holds for everything except~\ref{itm:restrict:reg} in the definition of a restriction pair for~\ref{G:restr}, since~\ref{PtH:image} was already established. Given an image restricted $x\in X_i\subset\Xbl_{i'}$ and an edge $xy\in E(H)$ with $Y\in X_j\subset\Xbl_{j'}$, the pair $\big(\comN_\Gamma(J_x;V_{i'}),\comN_\Gamma(J_y;V_{j'})\big)$ is $\big(\tfrac12\delta\eps,d,p\big)$-regular by assumption. Since $\big|\comN_\Gamma(J_x;V_i)\big|\ge\tfrac12\delta\big|\comN_\Gamma(J_x;V_{i'})\big|$, and similarly for $y$, it follows from the definition that $\big(\comN_\Gamma(J_x;V_{i}),\comN_\Gamma(J_y;V_{j})\big)$ is $(\eps,d,p)$-regular, completing the verification of~\ref{G:restr}.
\end{proof}

With this we are ready to define the general setup which is common to
our proofs and encapsulates the definitions of good embeddings and the
constant choices above. This encapsulation is convenient because we will
use this general setup in many of the lemmas to come.

\begin{definition}[General Setup]\label{def:general_setup}
  When we say that \emph{we assume the General Setup} we assert that
  constant choices satisfying the conditions of
  Section~\ref{sec:setup_constants} have been made, that there exist
  $\kappa$-balanced size-compatible partitions $\cX$ and $\cV$ of $V(H)$
  and $V(G)$ respectively whose parts are of size at least
  $\tfrac{n}{\kappa r_1}$, together with reduced graphs $R$ and $R'$ with
  $\Delta(R')\le\DeltaRp$, an image restriction pair $(\cI,\cJ)$, an
  $(\alpha,R')$-buffer $\tcX$, buffer vertices $\Xbuf$, and sets $\Vmain$,
  $\Vq$, $\Vc$ and $\Vbuf$, that the partition of $V(H)$ is a good
  $H$-partition, and that the partition of $G$ is a good $G$-partition.
\end{definition}

\subsection{Underlying restrictions, candidates and good partial embeddings}\label{sec:gpe}

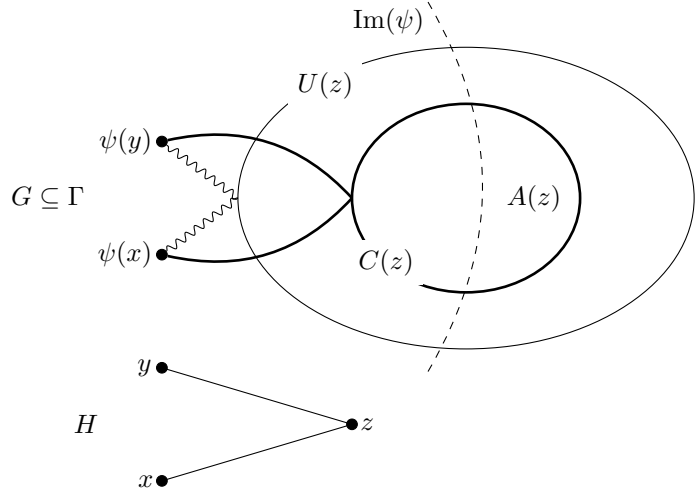
\begin{figure}[t]
   \begin{tikzpicture}
  \node at (-1,0.75) {$H$} ;
  \node at (-1.5,3.75) {$G\subset\Gamma$} ;
  \tikzstyle{every node}=[draw,circle,fill=black,minimum size=4pt,
                            inner sep=0pt]
  \draw (0,0) node (x) [label=left:$x$] {}
  -- ++(2.5,0.75) node (z) [label=right:$z$] {}
  -- ++(-2.5,0.75) node (y) [label=left:$y$] {};
  
  \node at (0,3) (px) [label=left:$\psi(x)$] {} ;
  \node at (0,4.5) (py) [label=left:$\psi(y)$] {} ;
  \tikzstyle{every node}=[draw=none,fill=none]
  \node at (1,3.75) (Gaz) {} ;
  \path[draw] (px) edge [snake it] node {} (Gaz.center)
                       (Gaz.center) edge [snake it] node {} (py) ;
  \draw (4,3.75) ellipse (3cm and 2cm) ;
  \draw[line width=1pt] (4,3.75) ellipse (1.5cm and 1.25cm) ;
  \node at (2.5,3.75) (Gz) {} ;
  \path[draw] (px) edge [line width=1pt,bend right] node {} (Gz.center)
              (Gz.center) edge [line width=1pt,bend right] node {} (py) ;
  \path[draw] (3.5,1.45) edge [bend right,dashed] node {} (3.5,6.35) ;
  \node at (3.7,6.1) (imp) [label=left:$\im(\psi)$] {} ;
  \node[fill=white] at (2.15,5.25) (Uz) {$U(z)$} ;
  \node[fill=white] at (2.95,2.9) (Cz) {$C(z)$} ;
  \node at (5.5,3.75) (Az) [label=left:$A(z)$] {} ;
 \end{tikzpicture}
\caption{Underlying restriction sets, candidate sets, and available
  candidate sets for an example in which $x$ and $y$ are embedded and their common neighbour $z$ is not.
The thick lines are edges of $G$, and the wavy lines of $\Gamma$.}
\label{fig:UCA}
\end{figure}

In this section we provide definitions of a number of additional sets that
will play important roles in our embedding algorithms. We also state which
invariants are maintained throughout the embedding, and encapsulate this in
the definition of good partial embeddings.

Suppose $\psi$ is a \emph{partial embedding} of $H$ into $G$, that is, an injective
graph homomorphism with domain $\dom(\psi)\subseteq V(H)$.
We call a vertex $x\in H$ \emph{embedded} if $x\in\dom(\psi)$, and otherwise
\emph{unembedded}. Of course this is with respect to $\psi$, but this will
always be clear from the context and we will not in future mention it. The set
of embedded neighbours of $x$ is
\[\Pi(x):=\big\{\psi(y):xy\in E(H),y\in\dom(\psi)\big\}\]
and we define \[\pi(x):=\big|\Pi(x)\big|\quad \text{and}\quad
\pi^*(x)=\big|\Pi(x)\big|+\big|J_x\big|.\]
Thus, $\pi^*(x)$ counts the number of already embedded neighbours of $x$ by
the partial embedding $\psi$ plus the `neighbours' from $J_x$ (if any).

Given an unembedded vertex $x\in X_i$, we define the \emph{underlying
restriction set} $U(x)$, the \emph{candidate set} $C(x)$, and the
\emph{available candidate set} $A(x)$ of $x$ by
\begin{align*}
 U(x)&:=V(x)\cap \comN_\Gamma\big(\Pi(x)\cup
 J_x\big)\,,\\
 C(x)&:=I_x\cap \comN_G\big(\Pi(x)\big)\,,\text{ and}\\
 A(x)&:=C(x)\setminus\im(\psi)\,,\text{ respectively.}
\end{align*}
For an illustration see Figure~\ref{fig:UCA}.

Recall that $G$ is a subgraph of $\Gamma$ (being either a random or a
bijumbled graph), which is typically far from being complete. The
underlying set signifies thus the vertices that could in principle be used
for embedding (if $G=\Gamma$), whereas the candidate sets denote the
vertices that are possible \emph{in} $G$, and the available candidate sets
only consider those vertices from the candidate sets which have not been
previously taken as images by the vertices embedded earlier (in
$\im(\psi)$). We also refer to vertices from $C(x)$ as candidate vertices
for $x$.  Further, we define 
\begin{align*}
\Umain(x)&:=U(x)\cap\Vmain(x)\,, &
\Uq(x)&:=U(x)\cap\Vq(x)\,, \\
\Uc(x)&:=U(x)\cap\Vc(x)\,, &
\Ubuf(x)&:=U(x)\cap\Vbuf(x)\,,
\end{align*}
and similarly for the respective subsets of the
candidate set~$C(x)$ and the available candidate set~$A(x)$ of $x$. For an illustration see
Figure~\ref{fig:mainqcbuf}.

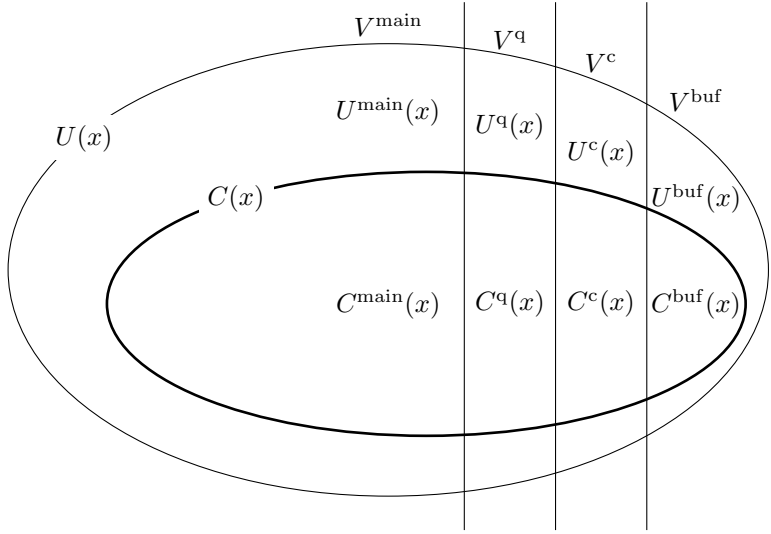
\begin{figure}[t]
 \begin{tikzpicture}
  \draw (0,0.45) ellipse (5cm and 3cm) ;
  \draw[line width=1pt] (0.5,0) ellipse (4.2cm and 1.75cm) ;
  \path[draw] (1,-3) edge node {} (1,4) ;
  \path[draw] (2.2,-3) edge node {} (2.2,4) ;
  \path[draw] (3.4,-3) edge node {} (3.4,4) ;
  \node at (0,3.7) {$\Vmain$} ;
  \node at (1.6,3.5) {$\Vq$} ;
  \node at (2.8,3.2) {$\Vc$};
  \node at (4.05,2.7) {$\Vbuf$} ;
  \node at (0,2.55) {$\Umain(x)$} ;
  \node at (1.6,2.35) {$\Uq(x)$} ;
  \node at (2.8,2) {$\Uc(x)$} ;
  \node at (4.05,1.4) {$\Ubuf(x)$} ;
  \node at (0,0) {$\Cmain(x)$} ;
  \node at (1.6,0) {$\Cq(x)$} ;
  \node at (2.8,0) {$\Cc(x)$} ;
  \node at (4.05,0) {$\Cbuf(x)$} ;
  \node[fill=white] at (-2,1.4) {$C(x)$} ;
  \node[fill=white] at (-4,2.2) {$U(x)$} ;
 \end{tikzpicture}
\caption{Partitioning underlying restriction sets and candidate sets into
  one big and three small parts}
\label{fig:mainqcbuf}
\end{figure}

Further we define the \emph{underlying
restriction graph} between $Y\subseteq X_i$ and
$Z\subseteq V_i$ to be the bipartite graph with parts $Y$ and $Z$ and edges $yz$
whenever $z\in U(y)$. We define similarly the \emph{candidate
graph} between $Y\subseteq X_i$ and
$Z\subseteq V_i$ to be the bipartite graph with parts $Y$ and $Z$ and edges $yz$
whenever $z\in C(y)$. If $x\in V(H)$ and $v\in C(x)$ we say~$v$ is a
\emph{candidate} for $x$.

Now we can introduce the concept of good partial embeddings.
The intention of this definition is to fix properties which easily allow to
prove that a good partial embedding is `locally'
good: If $x$ is unembedded, then almost all of $C(x)$ consists of vertices $v$
such that if $v$ is not in $\im(\psi)$, then $\psi\cup\{x\to v\}$ is a good
partial embedding (this assertion is made formal in Lemma~\ref{lem:fewbad}). 
For an illustration see Figure~\ref{fig:GPE}.

\begin{definition}[Good partial embedding]\label{def:GPE}
  We call $\psi$ a \emph{good partial embedding} if the following conditions hold.
  \begin{enumerate}[label=\itmarab{GPE}]
  \item\label{GPE:rightplace} For each $x\in\dom(\psi)$ we have $\psi(x)\in
    I_x$.
  \item\label{GPE:sizeU} For each unembedded $x\in X_i$ we have 
    \begin{align*}
      \big|U(x)\big|=(p\pm\eps p)^{\pi^*(x)}\big|V(x)\big|\,,\quad
      \big|\Umain(x)\big|&=(1-3\mu)(p\pm\eps p)^{\pi^*(x)}\big|V(x)\big|\,,\quad \text{and}\\
      \quad \big|\Uq(x)\big|\,,\,\big|\Uc(x)\big|\,,\,\big|\Ubuf(x)\big|&=\mu(p\pm\eps p)^{\pi^*(x)}\big|V(x)\big|\,.
    \end{align*}
  \item\label{GPE:sizeC} For each unembedded $x$ we have
    \begin{align*}
      \big|C(x)\big|&\geq(1-\eps')(dp-\eps' p)^{\pi(x)}\big|I_x\big|\,,\\
      \big|\Cmain(x)\big|&\geq(1-\eps')(1-3\mu)(dp-\eps' p)^{\pi(x)}\big|I_x\big|\,,\quad\text{and}\\
      \big|\Cq(x)\big|\,,\,\big|\Cc(x)\big|\,,\,\big|\Cbuf(x)\big|&\geq(1-\eps')\mu(dp-\eps' p)^{\pi(x)}\big|I_x\big|\,.
    \end{align*}
  \item\label{GPE:Ureg} For each $xy\in E(H)$ with $x$ and $y$ unembedded,
    the pair $\big(U(x),U(y)\big)$ is $(\eps_{\pi^*(x),\pi^*(y)},d,p)$-regular in
    $G$.
  \end{enumerate}
\end{definition}

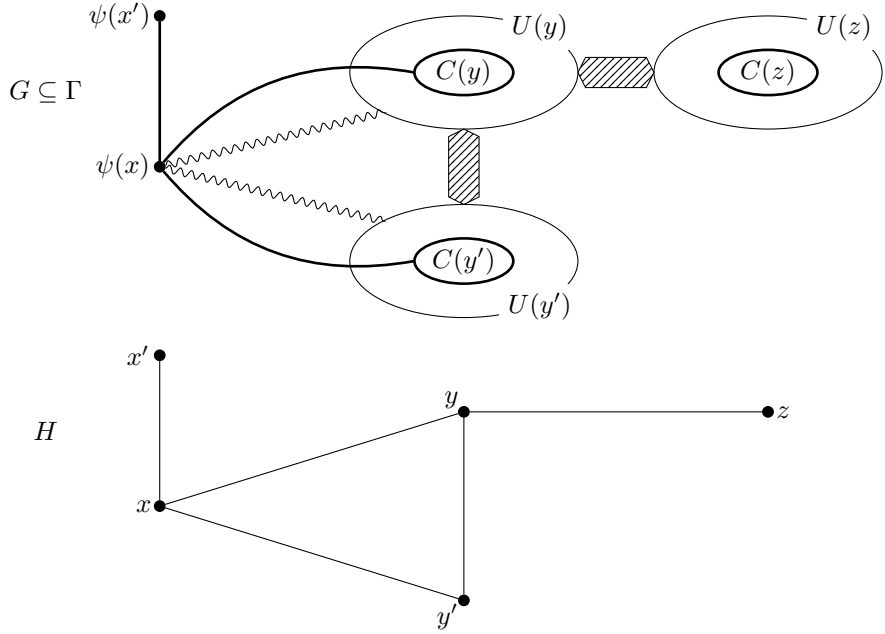
\begin{figure}[t]
  \begin{tikzpicture}
   \tikzstyle{every node}=[draw,circle,fill=black,minimum size=4pt,
                            inner sep=0pt]
   \node at (-3,-0.5) (x) [label=left:$x$] {} ;
   \node at (-3,1.5) (xp) [label=left:$x'$] {} ;
   \node at (1,0.75) (y) [label=above left:$y$] {} ;
   \node at (1,-1.75) (yp) [label=below left:$y'$] {} ;
   \node at (5,0.75) (z) [label=right:$z$] {} ;
   \draw (xp)--(x)--(yp)--(y) ;
   \draw (x)--(y)--(z) ;
   
   \node at (-3,4) (psx) [label=left:$\psi(x)$] {} ;
   \node at (-3,6) (psxp) [label=left:$\psi(x')$] {} ;
   \tikzstyle{every node}=[draw=none,fill=none]
   \draw[line width=1pt] (psx)--(psxp) ;
   
   \draw (1,2.75) ellipse (1.5cm and 0.75cm) ;
   \draw[line width=1pt] (1,2.75) ellipse (0.65cm and 0.3cm) ;
   \path[draw] (psx) edge [snake it] node {} (1-1.418*0.75,2.75+1.418*0.375) ;
   \path[draw] (psx) edge [line width=1pt,bend right] node {} (0.35,2.75) ;
   
   \draw (1,5.25) ellipse (1.5cm and 0.75cm) ;
   \draw[line width=1pt] (1,5.25) ellipse (0.65cm and 0.3cm) ;
   \path[draw] (psx) edge [snake it] node {} (1-1.418*0.75,5.25-1.418*0.375) ;
   \path[draw] (psx) edge [line width=1pt,bend left] node {} (0.35,5.25) ;
   
   \draw (5,5.25) ellipse (1.5cm and 0.75cm) ;
   \draw[line width=1pt] (5,5.25) ellipse (0.65cm and 0.3cm) ;
   
   \filldraw[pattern=north east lines] (2.5,5.25)--(2.6,5.45)--(3.4,5.45)--(3.5,5.25)--(3.4,5.05)--(2.6,5.05)--(2.5,5.25) ;
   \filldraw[pattern=north east lines] (1,3.5)--(1.2,3.6)--(1.2,4.4)--(1,4.5)--(0.8,4.4)--(0.8,3.6)--(1,3.5) ;

   \node at (-4.5,0.5) {$H$} ;
   \node at (-4.5,5) {$G\subset\Gamma$} ;
   \node[fill=white] at (2,2.15) {$U(y')$} ;
   \node[fill=white] at (2,5.85) {$U(y)$} ;
   \node[fill=white] at (6,5.85) {$U(z)$} ;
   \node at (1,2.75) {$C(y')$} ;
   \node at (1,5.25) {$C(y)$} ;
   \node at (5,5.25) {$C(z)$} ;
  \end{tikzpicture}
\caption{Good partial embeddings form embeddings when restricted to the
  mapped vertices $x,x'$ and guarantee sizes and regularity for the
  underlying restriction sets of unmapped vertices $y,y',z$. The thick lines are edges of $G$, the wavy lines of $\Gamma$, and the hatched areas represent regular pairs in $G$.}
\label{fig:GPE}
\end{figure}

It is important to observe that the trivial partial embedding, in which no
vertices are embedded, is not automatically a good partial embedding: image
restricted vertices might destroy any
of~\ref{GPE:sizeU}--\ref{GPE:Ureg}. However, it follows from
straightforward checking of the definitions that if we are provided with a
good $H$-partition and a corresponding good $G$-partition, then the trivial
partial embedding is a good partial embedding.

\begin{fact}
  If we assume the General Setup for~$H$ and~$G$ then the trivial embedding
  is a good partial embedding of~$H$ into~$G$.
\end{fact}

Finally, in much of the rest of this paper we will want to consider not just one partial embedding, but a sequence $\psi_0,\dots$ of them. We will want to refer to sets $\Amain(x)$, $\Cmain(x)$ and so on, and quantities $\pi(x)$ and so on, with reference to each of these. We will do this by following the convention that a subscript $t$ attached to any of these means it is with reference to $\psi_t$. Thus we write $\Amain_t(x)$
  for the set $\Amain(x)$ with respect to $\psi_t$, and similarly
  $\Cmain_t(x)$, $\Pi_t(x)$, $\pi_t(x)$, $\pi^*_t(x)$, $U_t(x)$, $C_t(x)$,
  $A_t(x)$ and so on. 

\subsection{Bad vertices}\label{subsec:bad_vertices}
As mentioned in the proof overview, we will embed vertices of $H$ into $G$
one at a time, choosing each vertex in a way that allows for the future
continuation of the embedding. This `allowing for future continuation'
comes in two parts, `local' conditions which broadly say that we embed a
vertex of $H$ in a way which allows for the embedding of other vertices
nearby in $H$, and `global' conditions which say that the embedding of many
(distant) vertices of $H$ do not tend to cause problems for any given
vertex. We handle the latter probabilistically, and this constitutes the
main work of this paper.  The purpose of this subsection is to define
precisely what we require locally, by specifying which vertices of $G$ are
bad for $x\in V(H)$, that is, not suitable for embedding $x$ to, and to
show that this is always a small set of vertices (in comparison to the
set $A(x)$). In fact, the definition of badness also depends on a set
$Q\subset V(H)$. In the random greedy algorithm (outlined in the
proof overview, Section~\ref{sec:proof_overview}), $Q$ will be the current
queue, while at other times $Q$ will simply be  $V(H)$.

  \begin{definition}[Bad vertices, badness condition]
   \label{def:bad_vertices}
   Let $\psi$  be a good partial embedding and $Q\subset V(H)$ be a set of
   unembedded vertices. The vertex $v$ is called \emph{bad for $x$ with
     respect to $\psi$ and $Q$} if the extension $\psi\cup\{x\to v\}$ is
   not a good partial embedding or there is an unembedded neighbour $y$ of
   $x$ with $y\not\in Q$ such that
   \begin{equation}\label{eq:bad_vertices}
     \deg_G\big(v;\Amain(y)\big)<(d-\eps')p|\Amain(y)|\,.
   \end{equation}
   We stress that in this equation we take $\Amain(y)$ with respect to the
   embedding $\psi$ (and not the extended embedding).
   When $\psi$ and $Q$ are clear from the context (as they always will be) we let $B(x)$ be the set of vertices in $C(x)$ which are bad for $x$ with respect to $\psi$ and $Q$.
   When we have a sequence of embeddings $\psi_0,\dots$ and a sequence of sets
   $Q_0,\dots$ then we write, again, $B_t(x)$ for the set of bad
   vertices for $x$ with respect to $\psi_t$ and $Q_t$. 
   We will also refer to~\eqref{eq:bad_vertices} as the \emph{badness condition}.
  \end{definition}

  The following lemma provides control over the bad vertices with respect
  to a good partial embedding. We remark that the definition of the
  value~$D$ in this lemma is somewhat complex, because we want to apply
  this lemma in the proofs of all three of our blow-up lemmas, in
  particular the blow-up lemma for degenerate graphs.

  \begin{lemma}[Few bad vertices]
   \label{lem:fewbad}
   We assume the General Setup. Let $\psi$ be a good partial embedding and let $Q\subset V(H)$ be a set of vertices such that for each unembedded $x\in V(H)\setminus Q$ we
   have
   \[|\Amain(x)|\ge\tfrac12\mu(d-\eps')^{\pi^*(x)}p^{\pi^*(x)}|\Vmain(x)|\,.\]
   Given an unembedded $x\in V(H)$, let~$D$ be such that for all unembedded
   $y,z\in V(H)$ we have
   \begin{enumerate}[label=\rom]
    \item\label{itm:fewbad:x} $D\ge\pi^*(x)+1$,
    \item\label{itm:fewbad:xy} If $xy\in E(H)$ then $D\ge\pi^*(y)+1$,
    \item\label{itm:fewbad:xyyz} If $xy,yz\in E(H)$ then $D\ge\pi^*(y)+2,\pi^*(z)+1$,
    \item\label{itm:fewbad:xyyzxz} If $xy,yz,xz\in E(H)$ then $D\ge\pi^*(x)+2,\pi^*(y)+2,\pi^*(z)+2$.
   \end{enumerate}
   Then the following holds.
   \begin{enumerate}[label=\abc]
  \item\label{fewbad:a} If all neighbours of $x$ are embedded then no
     vertex in $A(x)$ is bad for $x$ with respect to
     $\psi$ and $Q$.
   \item\label{fewbad:b} Suppose that $\Gamma$ satisfies $\NS(\eps,r_1,D)$
     and $\RI(\eps,(\eps_{a,b}),\eps',d,r_1,D)$. Then at most
     $20\Delta^2\eps'p^{\pi^*(x)}|V(x)|$ vertices of $C(x)$ are bad for~$x$ with respect to~$\psi$ and~$Q$.
  \end{enumerate}
  \end{lemma}
\begin{proof}
 We require
 \[\mu\le\tfrac16\,,\qquad \eps'\le\frac{\mu d^\Delta \zeta}{1000\kappa 4^{\Delta}}\qquad\text{and}\qquad \eps\le\frac{d\eps'}{\kappa}\,.\]
 For~\ref{fewbad:a}, observe that all the conditions for a vertex $v$ to be in $B(x)$ are trivially false.
 
 For part~\ref{fewbad:b}, we need to consider all of the possible reasons
 $v$ could be bad for $x$ with respect to $\psi$ and $Q$.  
 It could be that
 there is some unembedded neighbour $y\in V(H)\setminus Q$ of $x$ such
 that the badness condition~\eqref{eq:bad_vertices} holds.
 The pair $\big(U(x),U(y)\big)$ is $(\eps',d,p)$-regular in~$G$
 by~\ref{GPE:Ureg}.  
 By assumption and because~\ref{GPE:sizeU} holds for~$y$ with respect to~$\psi$
 we have
 \begin{equation*}\begin{split}
   |\Amain(y)|
   &\ge\frac12\mu(d-\eps')^{\pi^*(y)}p^{\pi^*(y)}|\Vmain(y)|
   =\frac12\mu(1-3\mu)(d-\eps')^{\pi^*(y)}p^{\pi^*(y)}|V(y)| \\
   &\ge\frac12\mu(1-3\mu)(d-\eps')^{\pi^*(y)}p^{\pi^*(y)}\frac{|U(y)|}{(1+\eps)^{\pi^*(y)}p^{\pi^*(y)}}
   \ge\eps'|U(y)|\,.
 \end{split}\end{equation*}
 So the badness condition~\eqref{eq:bad_vertices}
 holds for at most $\eps'|U(x)|\le 2\eps'p^{\pi^*(x)}|V(x)|$
 vertices~$v$ of $C(x)$, where we use~\ref{GPE:sizeU} and $(1+\eps)^\Delta\le 2$.
  Since $x$ has at most
 $\Delta$ unembedded neighbours, in total there are at most
 $2\Delta\eps'p^{\pi^*(x)}|V(x)|$ vertices~$v$ of $C(x)$ such
 that the badness condition~\eqref{eq:bad_vertices} holds for some unembedded neighbour of~$x$.

 It could be that $\psi\cup\{x\to v\}$
 is not a good partial embedding. Since $v\in C(x)\subseteq
 I_x$,~\ref{GPE:rightplace} cannot fail. 
 We next consider~\ref{GPE:sizeC}. Since $\psi$ is a good partial 
 embedding, \ref{GPE:sizeC} cannot fail for any vertex~$y$ which is not a
 neighbour of~$x$ (because the candidate sets of these vertices do not change).
 So let~$y$ be any neighbour of~$x$ in~$H$.
 Recall that
 $\big(U(x),U(y)\big)$ is $(\eps',d,p)$-regular in~$G$ by~\ref{GPE:Ureg}.
 We have
 \begin{equation*}\begin{split}
   |C(y)| &\geBy{\ref{GPE:sizeC}} (1-\eps')(dp-\eps'p)^{\pi(y)}|I_{y}|
   \geBy{\ref{G:restr}}(1-\eps')(dp-\eps'p)^{\pi^*(y)}\zeta|V(y)| \\
   &\geBy{\ref{GPE:sizeU}}(1-\eps')(d-\eps')^\Delta(1+\eps')^{-\Delta}\zeta|U(y)|
   \ge\eps'|U(y)|\,.
 \end{split}\end{equation*}
 Since $C(x)\subset U(x)$, there are at most $\eps'|U(x)|\le
 2\eps'p^{\pi^*(x)}|V_i|$ vertices $v\in C(x)$ such that $|C(y)\cap
 N_G(v)|<(d-\eps')p|C(y)|$. Since
 $(d-\eps')p|C(y)|\ge(1-\eps')(dp-\eps'p)^{\pi(y)+1}|I(y)|$, there are
 at most $2\eps'p^{\pi^*(x)}|V(x)|$ vertices $v\in C(x)$ such that the first
 condition in~\ref{GPE:sizeC} fails for~$y$ with respect to
 $\psi\cup\{x\to v\}$. We can analogously argue for
 each of the other four conditions. Since~$x$ has at most~$\Delta$
 unembedded neighbours, we conclude that for all but at most $10\Delta\eps'p^{\pi^*(x)}|V_i|$
 vertices~$v\in C(x)$ we have that~\ref{GPE:sizeC} is satisfied with
 respect to $\psi\cup\{x\to v\}$.

 Now we turn to~\ref{GPE:sizeU}. Again, it suffices to consider
 neighbours~$y$ of~$x$.  As~\ref{GPE:sizeU} holds with respect
 to~$\psi$, we see that $|U(y)|\ge(p-\eps p)^{\pi^*(y)}|V(y)|\ge\eps
 p^{D-1}n/r_1$ since $|V(y)|\ge n/(\kappa r_1)$ and by
 assumption~\ref{itm:fewbad:xy}.  It follows
 that we can use $\NS(\eps,r_1,D)$ and assumption~\ref{itm:fewbad:x} to
 conclude that all but at most $\eps p^{D-1}n/r_1^2\le \eps'
 p^{\pi^*(x)}|V(x)|$ vertices $v\in C(x)$ are such that $|U(y)\cap
 N_\Gamma(v)|=(1\pm\eps)p|U(y)|=(p\pm\eps p)^{\pi^*(y)+1}|V(y)|$, again using~\ref{GPE:sizeU} for~$y$ with respect to~$\psi$.  Again, we
 can argue analogously for the other four conditions, and so in total for
 all but at most $5\Delta\eps' p^{\pi^*(x)}|V(x)|$ vertices $v\in C(x)$ we
 have that~\ref{GPE:sizeU} is satisfied with respect to $\psi\cup\{x\to
 v\}$.

 It remains to consider \ref{GPE:Ureg}. Again, we only need to consider
 edges $yz\in E(H)$ between unembedded vertices such that $y\in N_H(x)$ or
 $z\in N_H(x)$, or both. Consider first the case that only $y\in
 N_H(x)$. The case that only $z\in N_H(x)$ follows
 analogously. Since~\ref{GPE:sizeU} is satisfied with respect to $\psi$ we
 have $|U(y)|\ge(p-\eps p)^{\pi^*(y)}|V(y)|$ and $|U(z)|\ge(p-\eps
 p)^{\pi^*(z)}|V(z)|$, which together with assumption~\ref{itm:fewbad:xyyz}
 implies $|U(y)|\ge(p-\eps p)^{D-2}n/(\kappa
 r_1)\ge\eps'p^{D-2}n/r_1^2$ and
 $|U(z)|\ge\eps'p^{D-1}n/r_1^2$.  Hence, since $\Gamma$ satisfies
 $\RI(\eps,(\eps_{a,b}),\eps',d,r_1,D)$ and the pair $\big(U(y),U(z)\big)$
 is $(\eps_{\pi^*(y),\pi^*(z)},d,p)$-regular in~$G$ because \ref{GPE:Ureg}
 is satisfied with respect to~$\psi$, we get that there are at most
 $\eps p^{D-1}n/r_1^2\le\eps'p^{\pi^*(x)}|V(x)|$ vertices $v\in C(x)$ such that $\big(N_\Gamma(v)\cap
 U(y), U(z)\big)$ is not
 $(\eps_{\pi^*(y)+1,\pi^*(z)},d,p)$-regular in $G$. Analogously, in the second
 case that $y,z\in N_H(x)$ we can use assumption~\ref{itm:fewbad:xyyzxz} to
 conclude that there are at most $\eps p^{D-2}n/r_1^2\le\eps'p^{\pi^*(x)}|V_i|$ vertices $v\in C(x)$
 such that $\big(N_\Gamma(v)\cap U(y), N_\Gamma(v)\cap U(z)\big)$ is not
 $(\eps_{\pi^*(y)+1,\pi^*(z)+1},d,p)$-regular in $G$. So, in total there are at
 most $\Delta^2 \eps'p^{\pi^*(x)}|V(x)|$ vertices $v\in C(x)$ such
 that~\ref{GPE:Ureg} is not satisfied with respect to $\psi\cup\{x\to v\}$.

Summing up, there are at most
\begin{multline*}
  2\Delta\eps'p^{\pi^*(x)}|V(x)|+10\Delta\eps'p^{\pi^*(x)}|V(x)|+5\Delta\eps'
  p^{\pi^*(x)}|V(x)|+\Delta^2 \eps'p^{\pi^*(x)}|V(x)| \\
  \le 20\Delta^2\eps'p^{\pi^*(x)}|V(x)|
\end{multline*}
vertices $v\in C(x)$ such
that~$v$ is bad for~$x$ with respect to~$\psi$ and~$Q$.
\end{proof}

\section{The behaviour of random greedy algorithms}
  \label{sec:RGAlemmas}

  The proofs of our three blow-up lemmas each use, as crucial ingredient, a
  slightly different random greedy algorithm (RGA). In this section we
  collect four different lemmas which will be useful in conjunction with
  these RGAs.

  Roughly speaking, each of our RGAs has the property that we embed
  vertices sequentially, at each time choosing an image of the next vertex
  uniformly at random in a not-too-small subset of its candidate vertices,
  and this, together with the General Setup (see Definition~\ref{def:general_setup}), 
  is all we need to know about
  our RGAs for the purposes of this section. Nevertheless, the reader who
  wishes to see a concrete example of an RGA before reading the following
  lemmas should look at Algorithm~\ref{alg:RGA} in Section~\ref{sec:rga},
  which is the simplest of our RGAs. The other two RGAs are
  Algorithm~\ref{alg:RGA:psr} in Section~\ref{sec:pseudoRGA} and
  Algorithm~\ref{alg:RGA:deg} in Section~\ref{sec:degenRGA}.
	 
  Our first lemma (Lemma~\ref{lem:tau}) constructs the vertex ordering
  which we shall use in the RGAs for the random graphs blow-up lemma
  (Lemma~\ref{lem:rg_image}) and the blow-up lemma for bijumbled graphs
  (Lemma~\ref{lem:psr_main}).  The point of the ordering we use is that we
  need to analyse the way neighbours of (non-clique) buffer vertices are
  embedded during the RGA. Putting them first in the order, with each
  neighbourhood coming as an interval in the order, makes this
  possible. Even then, for Lemma~\ref{lem:rg_image} we need to insist that
  in degree-$\Delta$ buffers the final two neighbours of a buffer vertex
  are not adjacent, thus their embeddings do not affect each other
  much. 

  In the proof of our blow-up lemma for degenerate graphs
  (Lemma~\ref{lem:degen}) we cannot apply this lemma, because we have to
  work with the order we are supplied with.  Ultimately, the fact that we
  cannot expect these nice conditions from this supplied order is the
  reason why we cannot obtain stronger bounds on~$p$ in that lemma.

  \begin{lemma}[A good vertex order for an RGA]
    \label{lem:tau} Let $H$ be a graph with $\Delta(H)\le\Delta$ and $V(H)=\Xbuf\dcup\Xmain\dcup\Xc$
    such that $N(\Xbuf)\subset \Xmain$ and each pair of vertices in $\Xbuf$
    has distance at least ten.
    Then there exists an ordering $\tau$ of $\Xmain$ with the
    following properties.
    \begin{enumerate}[label=\abc]
      \item\label{cond:seg} For all $x\in \Xmain\setminus N(\Xbuf)$ and $y\in N(\Xbuf)$ we
        have $\tau(x)>\tau(y)$.
      \item\label{cond:noedge} For all $x\in \Xbuf$ we can enumerate $N(x)$ as $y_1,\dots,y_b$
        such that $\tau(y_{j+1})=\tau(y_j)+1$ for all
        $j\in[b-1]$. Moreover, if $\deg_H(x)=\Delta$ and $x$ is not in a copy of $K_{\Delta+1}$ then
        $y_{\Delta-1}y_\Delta\not\in E(H)$.
       \item\label{cond:ncl_ord} The neighbours of non-clique buffer vertices come before the neighbours of the 
       clique-buffer vertices.
    \end{enumerate}
  \end{lemma}

  Note that in this
  lemma we do not need to assume the General Setup, but we
  nevertheless use names for the parts of~$H$ corresponding to those from
  the General Setup.
  
 \begin{proof}
   We separate the vertices in $\Xbuf$ into two classes, those not in
   copies of $K_{\Delta+1}$ and those in copies of $K_{\Delta+1}$. We take
   an enumeration $\Xbuf=\{x_1,\ldots,x_{|\Xbuf|}\}$ with the vertices not
   in copies of $K_{\Delta+1}$ coming first. We now create the ordering
   $\tau$ as follows. We start with the empty ordering. For each $i$ in
   succession such that $x_i\in\Xbuf$ is not in a copy of $K_{\Delta+1}$,
   we append to $\tau$ the vertices $N_H(x_i)$ in some order such that if
   $\deg_H(x_i)=\Delta$ then the last two vertices in $\tau$ of $N_H(x_i)$
   are not adjacent. Note that if $\deg_H(x)=\Delta$ then there is a pair of non-adjacent vertices in
   $N_H(x_i)$ because $x_i$ is not in a copy of $K_{\Delta+1}$. Next, for
   each $i$ in succession such that $x_i\in\Xbuf$ is in a copy of
   $K_{\Delta+1}$ we append to $\tau$ the vertices $N_H(x_i)$ in an
   arbitrary order. Finally, we append to $\tau$ any remaining vertices of
   $\Xmain$ in an arbitrary order.
 \end{proof}	 
	 
 The next lemma states that provided we embed vertices uniformly at random
 into not-too-small sets, the candidate sets of unembedded vertices are
 likely to be distributed uniformly. In this lemma, as will be the case for
 much of the rest of the paper, we have a sequence $\psi_0,\dots$ of good
 partial embeddings, and we want to talk about the various sets and
 quantities defined in Section~\ref{sec:gpe}, such as candidate sets, with
 respect to each of these good partial embeddings. As mentioned in
 Section~\ref{sec:gpe}, we will follow the convention that (for example)
 $C_t(x)$ is the candidate set of $x$ with respect to $\psi_t$. Recall also, that 
 the set $X_i^*\subset X_i$ denotes image restricted
 vertices in~$X_i$ (see Definition~\ref{def:restrict}).

\begin{lemma}[Uniform distribution of candidate sets]\label{lem:rga:welldistr}
  We assume the General Setup. Suppose that for some integer~$T$ we have a
  sequence $\psi_0,\psi_1,\dots,\psi_T$ of good partial embeddings,
  where $\psi_0$ is the trivial partial embedding
  and each $\psi_t$ is obtained from $\psi_{t-1}$ by embedding some
  vertex $x$ in $V(H)\setminus\dom(\psi_{t-1})$ to a uniform random
  vertex from a subset~$S$ of $C_{t-1}(x)$ with 
  $|S|\ge\tfrac{1}{10}\mu\zeta(dp)^{\pi_{t-1}^*(x)}|V(x)|$.
	
  Then the following holds with probability at least
  $1-2^{-n/(\kappa r_1)}r_1$.  For every $i\in[r]$ and every set
  $W\subseteq V_i$ of size at least $\rho |V_i|$, the number of
  vertices $x\in X_i\setminus X^*_i$ (i.e.\ vertices which are not
  image restricted) such that there exists $t=t(x)$ when $x$ is
  unembedded and we have
  \begin{equation}\label{eq:rga:welldistr}
    |C_t(x)\cap W|<\big(dp-\eps'p\big)^{\pi^*_t(x)}|W|
  \end{equation}
  is at most $\rho |X_i|$.
\end{lemma}

This lemma corresponds to the Main Lemma of~\cite{KSS_bl}, and its proof is
similar in spirit. The proof of this lemma exploits the fact that when a
vertex $y$ is embedded, condition~\eqref{eq:rga:welldistr} might become
true for at most one vertex from $X_i\setminus X^*_i$ (since
by~\ref{PtH:dist} the vertices in $X_i$ have distance at least
$10$). Moreover, by the condition of always embedding into not too small
subsets and by the regularity property~\ref{GPE:Ureg}, the vertex $y$ makes
an unlucky choice with a very small probability, so that many such choices
are extremely unlikely.

\begin{proof}[Proof of Lemma~\ref{lem:rga:welldistr}]
   We require
   \[\eps'<\min\big(4^{-\Delta}d^\Delta\rho,\tfrac{\mu\zeta d^{\Delta}}{20\Delta}2^{-4/\rho}\big)\quad\text{and}\quad 2^{-n/\kappa r_1}r_1<1\,.\]     
   
     Fix a set $W\subseteq V_i$ of size at least $\rho|V_i|$, and a set $X'\subseteq X_i\setminus X^*_i$ of size $\rho|X_i|$. We aim to show that the probability of the following event is at most $2^{-4|X_i|}$. For each $x\in X'$ there is a $t=t(x)$ such that $x$ satisfies~\eqref{eq:rga:welldistr}. The desired result then follows by taking the union bound over all $i$, subsets $W$ of $V_i$ and $X'$ of $X_i$.
     
     If $x$ is any vertex of $X'$, then we have $C_0(x)=V_i$ since $x$ is not image restricted, thus 
     $|C_0(x)\cap W|=|W|$ and hence~\eqref{eq:rga:welldistr} is false for $x$. 
     If there is a $t(x)$ such that $x$ satisfies~\eqref{eq:rga:welldistr}, then we can fix $t$ to be the smallest integer such that~\eqref{eq:rga:welldistr} is true for $t+1$ and $x$. 
 Since the candidate set of $x$ changes
 only when a neighbour of $x$ is embedded, it follows that  
 the vertex $y$ that is embedded to create $\psi_{t+1}$ from $\psi_t$ is a neighbour of $x$ in $H$ and thus $\pi^*_{t+1}(x)=\pi^*_{t}(x)+1$. 
 Moreover, since equation~\eqref{eq:rga:welldistr} becomes true for $x$, 
the vertex~$y$ is embedded to a vertex $w$ such that 
     \begin{equation}\label{eq:welldist:bad}
       \deg_G\big(w;C_{t}(x)\cap W\big)<(d-\eps')p\big|C_{t}(x)\cap W\big|\,,
     \end{equation}
     as otherwise we would still have 
     \[|C_{t+1}(x)\cap W|\ge ((d-\eps')p\big)|C_{t}(x)\cap W|\ge ((d-\eps')p)^{\pi^*_{t+1}(x)}|W|\,.\]
 
 Since $\psi_{t}$ is a good partial embedding, by~\ref{GPE:Ureg} the pair 
  $\big(U_{t}(x),U_{t}(y)\big)$ is an $(\eps',d,p)$-regular pair in $G$, and by~\ref{GPE:sizeU} we have
   $\big|U_{t}(x)\big|=(p\pm\eps'p)^{\pi^*_{t}(x)}|V_i|$, and $\big|U_{t}(y)\big|=(p\pm\eps'p)^{\pi^*_{t}(y)}|V(y)|$. Since~\eqref{eq:rga:welldistr} is false for $x$ at time $t$, 
   we have $\big|C_{t}(x)\cap W\big|\ge \big((d-\eps')p\big)^{\pi^*_{t}(x)}|W|$, and so
     \[  
   \big|C_{t}(x)\cap W\big|\ge \left(\tfrac{d-\eps'}{1+\eps'}\right)^{\pi^*_{t}(x)}\rho |U_{t}(x)|\ge\eps' |U_{t}(x)|\,,
   \]
   by the requirements on the constants at the beginning of the proof.
   We conclude by $(\eps',d,p)$-regularity of $\big(U_{t}(x),U_{t}(y)\big)$ that at most $\eps'\big|U_{t}(y)\big|\le 2\eps'p^{\pi^*_{t}(y)}|V(y)|$ vertices $w$ of $U_{t}(y)$ satisfy~\eqref{eq:welldist:bad}.
   
   Since $\psi_{t+1}$ is created by embedding $y$ uniformly at random into a subset of $C_t(y)\subset U_t(y)$ of size at least $\tfrac{1}{10}\mu\zeta(dp)^{\pi^*_t(y)}|V(y)|$, the probability of embedding $y$ to a vertex $w$ satisfying~\eqref{eq:welldist:bad}, conditioning on any history, is at most
   \begin{equation}\label{eq:bad_probability}
   \frac{2\eps'p^{\pi^*_{t}(y)}|V(y)|}{\tfrac{1}{10}\mu\zeta(dp)^{\pi^*_t(y)}|V(y)|}
  \le 20\eps'\mu^{-1}\zeta^{-1}d^{-\Delta}\,.
   \end{equation}
 Next we argue that the probability that for each $x\in X'$ there is a first time $t=t(x)$ such that $x$ satisfies~\eqref{eq:rga:welldistr} is at most
   \[\big(20\Delta\eps'\mu^{-1}\zeta^{-1}d^{-\Delta}\big)^{\rho|X_i|}\le 2^{-4|X_i|}\,.\] 
 Let us denote this event by $\cE_{X'}$. Observe that $\cE_{X'}$ is split into $\Delta^{|X'|}$ events (since $\Delta(H)\le\Delta$) by specifying for each $x\in X'$ a neighbour $y_x$ of $x$ whose embedding occurs at time $t(x)$. Let $(y_x)_{x\in X'}$ be any such assignment, and $\cE_{X',(y_x)}$ be the corresponding event. We aim to bound $\Pr(\cE_{X',(y_x)})$.
 
 By~\ref{PtH:dist} the vertices of $X_i$ are at distance at least ten in $H$, so the vertices $(y_x)_{x\in X'}$ are distinct. The corresponding conditional probabilities thus multiply, and we have
\[\Pr(\cE_{X',(y_x)})\le\big(20\eps'\mu^{-1}\zeta^{-1}d^{-\Delta}\big)^{|X'|}\]
by~\eqref{eq:bad_probability}. Applying the union bound over the events $\cE_{X',(y_x)}$ we conclude
\[\Pr(\cE_{X'})\le \big(20\Delta\eps'\mu^{-1}\zeta^{-1}d^{-\Delta}\big)^{\rho|X_i|}\le 2^{-4|X_i|}\,.\]
   where the final inequality is by choice of $\eps'$. 
   
   Taking the union bound over the at most $2^{|V_i|}=2^{|X_i|}$ choices of $W$ in $V_i$ and $2^{|X_i|}$ choices of $X'$ in $X_i$, we see that the probability that, for any fixed $i$, there exist subsets $W$ of $V_i$ and $X'$ of $X_i$, of sizes at least $\rho|V_i|$ and $\rho|X_i|$ respectively, such that each vertex $x$ of $X'$ satisfies~\eqref{eq:rga:welldistr} at some time $t$, is at most $2^{2|X_i|}2^{-4|X_i|}=2^{-2|X_i|}$. Now we have $|X_i|\ge n/(\kappa r_1)$, and $n\ge n_0$ where $n_0$ is chosen large enough such that $2^{-n_0/(\kappa r_1)}r_1<1$. Thus taking the union bound over the at most $r_1$ choices of $i$, we conclude that the probability that there exists $i$, and a subset $W$ of $V_i$ such that there are $\rho|X_i|$ vertices $x$ of $X_i\setminus X^*_i$ each of which satisfies~\eqref{eq:rga:welldistr} at some time~$t$, is at most $2^{-n/(\kappa r_1)}r_1$ as desired.
   \end{proof} 
   
   The next lemma, Lemma~\ref{lem:large_nbs}, shows that, again provided we embed vertices uniformly
   into not-too-small sets, we do not tend to cover vertex neighbourhoods
   in $G$ disproportionately fast. Specifically, if $ij\in E(R')$ then,
   by~\ref{G:deg}, each $v\in V_i$ has a large $G$-neighbourhood in
   $V_j$. At some time $T$ when only a small but linear fraction of each
   part of $H$ has been embedded, it is very likely that less than half of
   this $G$-neighbourhood is in the image $\im(\psi_T)$ of the current
   partial embedding. One should think of this as: early on in the
   embedding process, the minimum degree conditions~\ref{G:deg} provided by
   super-regularity are preserved. This fact is trivially satisfied in the dense blow-up 
   lemma~\cite{KSS_bl}, but it is far from trivial in our case.
   
   We will apply this lemma in the proofs of the random graphs blow-up
   lemma (Lemma~\ref{lem:rg_image}) as well as the jumbled graphs blow-up
   lemma (Lemma~\ref{lem:psr_main}). For the proof of the degenerate graphs
   blow-up lemma, on the other hand, we cannot use it; but we will replace
   it by a suitable modification of the ideas in its proof.

   The idea behind Lemma~\ref{lem:large_nbs} is as follows. As discussed in
   the proof overview (Section~\ref{sec:proof_overview}), we will need our
   RGAs to guarantee that each $v\in V_i$ is a candidate for many vertices
   $x\in\Xbuf_i$ in order to complete the embedding. This means we need it
   to be not too unlikely that $N_H(x)$ is embedded to $N_G(v)$ for any
   given $x\in\Xbuf_i$, and as a first step to showing this, it is
   necessary to show that we have not covered $N_G(v)$ with embedded
   vertices before we get around to embedding $N_H(x)$. In the proofs of
   the random graphs blow-up lemma and the bijumbled blow-up lemma, we show
   this by embedding $N(\Xbuf)$ first and applying
   Lemma~\ref{lem:large_nbs}. 
   
   \begin{lemma}[Preservation of super-regularity]\label{lem:large_nbs}
     We assume the General Setup. Suppose that $\Gamma$ has
     $\NS(\eps,r_1,B)$. Suppose that for some $T$ we have a sequence
     $\psi_0,\psi_1,\dots,\psi_T$ of good partial embeddings, where
     $\psi_0$ is the trivial partial embedding and each $\psi_t$ is
     obtained from $\psi_{t-1}$ by embedding one vertex $x_t\in V(H)$ to a
     uniform random vertex from a subset~$S$ of $C_{t-1}(x_t)$ with
     $|S|\ge\tfrac{1}{10}(dp)^{\pi_{t-1}^*(x_t)}|V(x_t)|$. Suppose furthermore
     that for each $t$ the vertex $x_t$ has at most $B-2$ neighbours in
     $\dom(\psi_{t-1})$, and that for each $i\in[r]$ we have
     $\big|\dom(\psi_T)\cap X_i\big|\le 8\mu\kappa\DeltaRp|X_i|$.
    
    Then with probability at least $1-\exp(-\eps p n/r_1)$, for each
    $i\in[r]$, each vertex $v\in V_i$, and each $j$ such that 
 $ij\in E(R')$
 we have $\big|N_{G}(v;\Vmain_j)\setminus \im(\psi_T)\big|\ge \tfrac12\deg_{G}(v;V_j)$.
\end{lemma}

In the proof of the bijumbled blow-up lemma we set $B=\Delta+1$ and will show
that this is enough for Lemma~\ref{lem:large_nbs} to handle all of
$N(\Xbuf)$. In the proof of the random graphs blow-up lemma, by contrast, we set
$B=\Delta$, which turns out to be good enough to handle buffers of degree
up to and including $\Delta$ which are not clique buffers. This is one of
the reasons why we have to handle clique buffers differently.

The proof of Lemma~\ref{lem:large_nbs} involves estimating the probability
of embedding $x_t$ to $N_G(v)$, conditioned on $\psi_{t-1}$. We show that
either this probability is small, or that a previously embedded neighbour
of $x_t$ was embedded oddly, which is guaranteed to be a low-probability
event. In either case, for $x_t$ to be embedded to $N_G(v)$ an unlikely
event must occur. For $N_G(v)$ to be substantially filled up, many of these
events, which are sequentially dependent, have to occur. The sequential
dependence lemma, Lemma~\ref{lem:coupling}, then shows this is unlikely
enough to take a union bound over all choices of $v$.

\begin{proof}
 We require
 \[\mu<\tfrac{d^{\Delta}}{1320\kappa\DeltaRp (\Delta+1)}\,,\quad\eps<\tfrac{\mu d}{10\kappa}\quad\text{and}\quad r_1 n\exp(-\eps p n/r_1)<1\,.\]

Fix $ij\in E(R')$ and a vertex $v\in V_i$. We first estimate the conditional probability that $x_t\in X_j\cap\dom(\psi_T)$ is embedded to $N_G(v)\cap\Vmain (x_t)$, given the history. Because $x_t$ is embedded uniformly at random to a subset of $C_{t-1}(x_t)$ of size at least $\tfrac{1}{10}(dp)^{\pi^*_{t-1}(x_t)}|V_j|$, this probability is at most
\begin{equation}\label{eq:large_nbs:prob}
 \frac{\big|\Umain_{t-1}(x_t)\cap N_G(v)\big|}{\tfrac{1}{10}(dp)^{\pi^*_{t-1}(x_t)}|V_j|}\,,
\end{equation}
which we will show to be at most $20d^{-\Delta}\tfrac{\deg_G(v;V_j)}{|V_j|}$.

This clearly holds when $\big|\Umain_{t-1}(x_t)\cap N_G(v)\big|\le (p+\eps p)^{\pi^*_{t-1}(x_t)}\deg_G(v;V_j)$ 
and we therefore assume the opposite.
In this case the estimate on the conditional
probability~\eqref{eq:large_nbs:prob} that $x_t$ occupies a vertex from
$N_G(v)$ could be as great as $1$, but we will show that this is an unlikely event. Specifically, for this  to occur
there must be a first neighbour $y_t$ of $x_t$ in $H$ which is
\emph{embedded oddly} at time $t'<t$, that is, is such that $\big|\Umain_{t'-1}(x_t)\cap N_G(v)\big|\le (p+\eps p)^{\pi^*_{t'-1}(x_t)}\deg_G(v;V_j)$ but (with $\pi^*_{t'}(x_t)=\pi^*_{t'-1}(x_t)+1$)
\begin{equation}\label{eq:large_candidate_NG}
\deg_\Gamma\big(\psi_{t'}(y_t);\Umain_{t'-1}(x_t)\cap N_G(v)\big)>(p+\eps p)^{\pi^*_{t'-1}(x_t)+1}\deg_G(v;V_j)\,.
\end{equation}
Let $W$ be a superset of $\Umain_{t'-1}(x_t)\cap N_G(v)$ of size $(p+\eps p)^{\pi^*_{t'-1}(x_t)}\deg_G(v;V_j)>\mu(d-\eps)p^{B-1}|V_j|$, where the inequality uses the fact that $\pi^*_{t'-1}(x_t)\le B-2$ and~\ref{G:deg}.
By property $\NS(\eps,r_1,B)$, which $\Gamma$ satisfies, we see that there are at most $\eps p^{B-1} n/r_1^2$ vertices $Z$ of $\Gamma$ which have this many neighbours in $W$ and thus satisfy~\eqref{eq:large_candidate_NG}. 
 Now $\pi^*_{t'-1}(y)\le B-2$ by the conditions of the lemma. We conclude that the probability of embedding $y$ to a vertex of $Z$, conditioning on the history, is at most
\begin{equation}\label{eq:occupy_prob_two}
\tfrac{\eps p^{B-1}n/r_1^2}{\tfrac{1}{10}(dp)^{B-2}|V(y)|}\le 10\eps d^{-\Delta}\kappa p<10d^{-\Delta+1}p\leq 20d^{-\Delta}\tfrac{\deg_G(v;V_j)}{|V_j|}\,,
\end{equation}
where we use $|V(y)|\ge n/(\kappa r_1)$ in the first and the choice of $\eps$ in the second and~\ref{G:deg} in the third inequality.

We define a sequence of Bernoulli random variables $Y_1,\ldots,Y_T$ as follows. Given the embedding $\psi_{t-1}$, if $x_t$ is a neighbour of a vertex in $X_j\cap\dom(\psi_T)$ none of whose previous neighbours were oddly embedded, and $x_t$ is oddly embedded, we set $Y_t=1$. If $x_t$ is in $X_j$, none of its previous neighbours were oddly embedded, and $x_t$ is embedded to $N_G(v;V_j)$, we set $Y_t=1$. Otherwise we set $Y_t=0$. By assumption, the total number of $Y_t$ which are not deterministically zero (that is, are in $X_j\cap\dom(\psi_T)$ or are neighbours of such a vertex) is at most $8\mu\kappa\DeltaRp(\Delta+1)|X_i|$. Observe furthermore that for all of the $Y_t$ which are not deterministically zero, we have just shown in~\eqref{eq:occupy_prob_two} that $Y_t$ is one with probability at most $20d^{-\Delta}\tfrac{\deg_G(v;\Vmain_j)}{|V_j|}$ conditioning on the history up to, but not including, embedding $x_t$. This history determines $Y_{t-1}$, so that the $Y_1,\dots,Y_T$ are sequentially dependent. It follows that we can apply Lemma~\ref{lem:coupling}, with 
\[x=8\mu\kappa\DeltaRp(\Delta+1)|X_j|\cdot20d^{-\Delta}\tfrac{\deg_G(v;V_j)}{|V_j|}=160\mu\kappa d^{-\Delta}\DeltaRp(\Delta+1)\deg_G(v;V_j)\]
and $\delta=1$, to show that
\[
 \Pr\big(Y_1+\dots+Y_T\ge 2x\big)\le 2\exp\big(-\tfrac{x}{3}\big)
 <\exp\big(-\tfrac{2\eps p n}{r_1}\big)\,,
\]
where the final inequality is by choice of $\eps$ and since $\deg_G(v;V_j)\ge (1-3\mu)(d-\eps)p|V_j|$ by~\ref{G:deg}. By choice of $\mu$, we conclude that with probability at most $\exp(-2\eps p n/r_1)$ we have $Y_1+\dots+Y_T\ge\tfrac14\deg_G(v;V_j)$. Now observe that for any vertex to be embedded to $N_G(v;\Vmain_j)$ one of these variables $Y_1,\ldots,Y_T$ must be one, and since vertices of $X_j$ are at distance at least ten in $H$ (by~\ref{PtH:dist}), no $Y_t$ can be responsible for two different vertices of $X_j$ being embedded to $N_G(v;\Vmain_j)$. Thus $Y_1+\dots+Y_T$ is an upper bound for $\im(\psi_T)\cap N_G(v;\Vmain_j)$. Taking a union bound over the choices of $j$ and of $v$, and using~\ref{G:deg}, we see that with probability at least $1-r_1 n\exp(-2\eps p n/r_1)>1-\exp(-\eps p n/r_1)$ the statement of the lemma holds.
 \end{proof}
 
 Our final lemma in this section complements the above lemma, showing that
 provided $v\in V_i$ does not have $N_G(v)$ more than half covered by
 $\im(\psi)$ early in the embedding, it is reasonably likely that $N_H(x)$
 is embedded to $N_G(v)$ for any given $x\in\Xbuf_i$. As with the previous
 lemma, it contains a parameter $B$ which will be either set to $\Delta$
 (in the proof of the random graphs blow-up lemma) or to $\Delta+1$ (in the proof
 of the bijumbled graphs blow-up lemma), and again in the former case the consequence
 is that we cannot use it to deal with clique buffers (that is, when
 $N_H(x)$ is a copy of $K_\Delta$). For the degenerate graphs blow-up lemma
 we again need to replace this lemma by suitable similar ideas.
 
\begin{lemma}[Buffer neighbour embedding]\label{lem:nonclique_buffer}
  We assume the General Setup. Let
  $B\in\{\Delta,\Delta+1\}$ and assume that $\Gamma$ has $\NS(\eps,r_1,B)$
  and $\RI(\eps,(\eps_{a,b}),\eps',d,r_1,B)$.  Given $v\in V_i$ and
  $x\in\Xbuf_i$, 
  let $y_1,\ldots,y_b$ be the neighbours
  of $x$, and suppose that if $b=B$ then $y_{b-1}y_b$ is not an edge of
  $H$.
  Let $\psi_0$ be a good partial embedding in which
  no vertex at distance two or less from $x$ is embedded, and let
  $Q_0\subset\Xmain$. Suppose further that
  $\big|N_G(v;\Vmain_j)\setminus\im(\psi_0)\big|\ge\tfrac12\deg_G(v;V_j)$
  for each $j$ such that $ij\in R'$. 
 
  Let $\psi_1,\ldots,\psi_b$ be good partial embeddings and
  $Q_1,\ldots,Q_b$ subsets of $\Xmain$, such
  that for each $t=1,\dots,b$ the embedding $\psi_t$ is obtained
  from $\psi_{t-1}$ by embedding $y_t$ uniformly at random into its
  available candidates minus the vertices which are bad for~$y_t$, that is, into
  $\Amain_{t-1}(y_t)\setminus B_{t-1}(y_t)$, and such that
  we have
 \[Q_t=Q_{t-1}\cup\big\{z\in \Xmain\setminus\dom(\psi_t)\colon \big|\Amain_t(z)\big|<\tfrac12
        \mu(d-\eps')^{\pi^*_{t}(z)}p^{\pi^*_{t}(z)}|\Vmain(z)|\big\}\,.\] 
  Then with probability at least $\big(d^{\Delta}p/10\big)^b$, we have $\psi_b\big(N_H(x)\big)\subset N_G(v)$.
\end{lemma}

The proof of this lemma is quite long, but much of it is `bookkeeping'.
Briefly, the idea is as
follows. In $\psi_0$, no neighbour of any $y_j$ is embedded and hence
each $y_j$ has candidate set $C_0(y_j)=V(y_j)$. We know that
$N_G\big(v;V(y_1)\big)$ is not covered by $\im(\psi_0)$, so if we
choose a uniform random vertex of $A_0(y_1)$ the probability of
choosing a member of $N_G(v)$ is at least $dp/4$. However, the random
greedy algorithm does not choose a random vertex of $A_0(y_1)$, but
rather of $A_0(y_1)\setminus B_0(y_1)$. Thus we have to show that
$B_0(y_1)$ does not cover too much of $N_G(v)$, which we can do using
properties~\ref{G:inh} and~\ref{G:deg} of the partition of $G$,
and $\NS(\eps,r_1,B)$ and $\RI(\eps,(\eps_{a,b}),\eps',d,r_1,B)$ which
we assume of $\Gamma$, in much the same way as the proof of
Lemma~\ref{lem:fewbad}. In addition we have to show that some extra
properties are preserved which allow us to analyse the embedding of
$y_2,\dots,y_b$, which we can show is likely in a similar way. We will
refer to a good partial embedding with these extra properties as a
buffer-friendly partial embedding. We conclude that the probability of
embedding $y_j$ to $N_G(v)$ and maintaining a buffer-friendly partial embedding
is at least $d^{\Delta}p/10$, conditioning on the history, for each
$j$. The conditional probabilities multiply, giving the desired
result.
 
\begin{proof}
 We require
 \[\mu\le\tfrac16\,,\qquad \eps<(\eps')^2,\qquad\eps'\le \frac{\mu d^\Delta \zeta}{1000\kappa\Delta^2}\qquad\text{and}\qquad (d-\eps')^{10\Delta}>\tfrac12 d^{10\Delta}\,.\]

  For each $0\le t\le b$ and each $y_j$, $j\in[b]$, we define
  $\hatA_t(y_j):=\Amain_t(y_j)\cap N_G(v)$ and 
  for any vertex $x'\in X$ we set $\hatU_t(x'):=U_t(x')\cap N_\Gamma(v)$. 
  
Observe that since $y_t$ is embedded to a vertex of $C_{t-1}(y_t)\setminus
B_{t-1}(y_t)$ for each $t=1,\dots,b$ we automatically maintain the property
that $\psi_t$ is a good partial embedding for each $t$ (see
Definition~\ref{def:bad_vertices} of bad vertices). We now formulate five
additional conditions on $\psi_t$ for $t=0,\dots,b$, which allow us to show
a lower bound on the desired probability inductively. If $\psi_t$ satisfies
these conditions we say that it is a \emph{buffer-friendly partial embedding}.
\begin{enumerate}[label=\itmarab{BPE}]
    \item\label{C:past} We have that $\{\psi_t(y_1),\dots, \psi_t(y_t)\}
      \subset N_G(v)$,    
    \item\label{C:largeU} $|\hatU_{t}(y_k)|= (p\pm\eps' p)^{\pi_{t}^*(y_k)} \deg_{\Gamma}(v;V(y_k))$ for $k=t+1,\ldots,b$,
    \item\label{C:largeA} $|\hatA_{t}(y_k)|\ge \tfrac12(dp-\eps'p)^{\pi_{t}^*(y_k)} \deg_G(v;V(y_k))$ for $k=t+1,\ldots,b$,
    \item\label{C:twoside} for unembedded $y_k$ and $y_\ell$ with 
     $y_k y_\ell\in E(H)$ we have $(\hatU_{t}(y_k),\hatU_{t}(y_\ell))$ is 
     $(\eps_{\pi_{t}^*(y_k),\pi_{t}^*(y_\ell)},d,p)$-regular in $G$,
     \item\label{C:oneside} for unembedded $y_k$ and $z$ with 
	     $y_k z\in E(H)$ the pair  $(\hatU_{t}(y_k),U_{t}(z))$ is 
     $(\eps_{\pi_{t}^*(y_k),\pi_{t}^*(z)},d,p)$-regular in $G$.     
\end{enumerate}

Observe that these conditions are all satisfied for $\psi_0$, i.e.\ $\psi_0$ is a buffer-friendly partial embedding. Indeed,~\ref{C:past} is vacuously satisfied. \ref{C:largeU} holds by 
definition of $\hatU$ and since no neighbours of any $y_k$ are embedded. \ref{C:largeA} holds by the definition of $\hatA$, the assumption on
$\deg_G\big(v;V(x_k)\setminus\im(\psi_0)\big)$ of the lemma, and since no neighbours of any $y_k$ are embedded. If $y_k$ and $y_\ell$ are adjacent in $H$ then $xy_ky_\ell$ is a triangle in $H$ with $x\in\Xbuf_i\subset\tX_i$, so by~\ref{G:inh} we have~\ref{C:twoside}. Finally, again since $x\in\Xbuf_i\subset\tX_i$, by~\ref{G:inh} we have~\ref{C:oneside}.

For each $t=0,\dots,b-1$ we let $P_t$ be the set of \emph{unfriendly} vertices
$u$ in $\hatA_t(y_{t+1})$, that is, the vertices that are such that if $\psi_{t+1}(y_{t+1})=u$ then $\psi_{t+1}$ is not a buffer-friendly partial embedding.

We now show that, for any $t=1,\dots,b$, given a buffer-friendly partial embedding $\psi_{t-1}$, the probability that $\psi_t$ is a buffer-friendly partial embedding is at least 
$d^\Delta p/10$. This clearly yields, by multiplying the conditional probabilities,
 the lower bound of $(d^\Delta p/10)^b$ claimed 
in the statement of the lemma.
For the analysis it will suffice to show that only a tiny proportion of vertices 
from $\hatA_{t-1}(y_t)$ are in $B_{t-1}(y_t)$ or $P_{t-1}$. 

 So assume that $\psi_{t-1}$ is a buffer-friendly partial embedding.
 We first give some lower bounds on set sizes. Since $\psi_{t-1}$ is a good partial embedding
 we have for an unembedded vertex $z$ that
\begin{equation}\label{eq:buffer:largeUC}
 |\Uq_{t-1}(z)|\ge\mu(p-\eps p)^{\pi_{t-1}^*(z)}\big|V(z)\big|\ge \eps'p^{\pi^*_{t-1}(z)}n/r_1
\end{equation}
  where the first inequality is by~\ref{GPE:sizeU} 
  and the second because $|V(z)|\ge n/(\kappa r_1)$ and by choice of $\eps'$. The same lower bound also holds for $U_{t-1}(z), C_{t-1}(z)$ et cetera since these sets are all by~\ref{GPE:sizeU} and~\ref{GPE:sizeC} at least as large.
Since in addition $\psi_{t-1}$ is a buffer-friendly partial embedding, we
have for any $y_k$ with $k\ge t$ that
\begin{equation}\label{eq:buffer:largehat}
 |\hatU_{t-1}(y_k)|\ge|\hatA_{t-1}(y_k)|\ge\tfrac12(dp-\eps'p)^{\pi^*_{t-1}(y_k)}\deg_G\big(v;V(y_k)\big)\ge\eps'p^{\pi^*_{t-1}(y_k)+1}n/r_1
\end{equation}
  where the second inequality is by~\ref{C:largeA} and the third is by~\ref{G:deg} and choice of $\eps'$. 
  We also have for $k\ge t$, using~\ref{C:largeU},~\ref{C:largeA},~\ref{GPE:sizeU},~\ref{G:deg} and the choice of $\eps$ and $\eps'$ that
\begin{equation}\label{eq:buffer:relhat}
 |\hatA_{t-1}(y_k)|\ge \max\big(\tfrac14d^{\Delta}p|U_{t-1}(y_k)|\,,\,\tfrac14d^{\Delta}|\hatU_{t-1}(y_k)|\big)\ge\eps'|\hatU_{t-1}(y_k)|\,.
\end{equation}

We are next going to estimate the number of bad vertices in $\hatA_{t-1}(y_t)$.

\begin{claim}\label{cl:buffer:bad}
  $|\hatA_{t-1}(y_t)\cap B_{t-1}(y_t)| \le 30\Delta^2 d^{-\Delta}\eps' |\hatA_{t-1}(y_t)|$.
\end{claim}
\begin{claimproof}
  We need to estimate the number of vertices $w$ from $\hatA_{t-1}(y_t)$ 
 which are bad with respect to $\psi_{t-1}$ and $Q_{t-1}$. For that we will consider 
 all of the possible reasons.
   
 First assume that  there is some unembedded 
   neighbour $z\in V(H)\setminus Q_{t-1}$ of $y_t$ such 
 that the badness condition
 \begin{equation}\label{eq:bad_vts_ncbl}
  \deg_G\big(w;\Amain(z)\big)<(d-\eps')p|\Amain(z)|
 \end{equation}
 holds for $w$ and $z$. The pair 
 $\big(\hatU_{t-1}(y_t),U_{t-1}(z)\big)$ is $\big(\eps_{\pi_{t-1}^*(y_t),\pi_{t-1}^*(z)},d,p\big)$-regular 
in $G$ by~\ref{C:oneside}. 
Because $z\not\in Q_{t-1}$ we have that 
\begin{align*}
|\Amain_{t-1}(z)|&\ge \tfrac{1}{2}\mu(d-\eps')^{\pi^*_{t-1}(z)}p^{\pi^*_{t-1}(z)}|\Vmain(z)|\\
&\ge \tfrac{1}{2}\mu(1-3\mu)(d-\eps')^{\pi^*_{t-1}(z)}p^{\pi^*_{t-1}(z)}|V(z)|\ge \eps'|U_{t-1}(z)|\,,
\end{align*}
where the first inequality is by  definition of $Q_{t-1}$ in the statement of the lemma, 
 the second inequality is by~\ref{GPE:sizeU} and the last by choice of $\eps'$ and by~\ref{GPE:sizeU}.
Therefore the badness condition~\eqref{eq:bad_vts_ncbl} holds 
for at most  
\[
\eps_{\pi_{t-1}^*(y_t),\pi_{t-1}^*(z)} |\hatU_{t-1}(y_t)|\leByRef{eq:buffer:relhat} 4d^{-\Delta}\eps' |\hatA_{t-1}(y_t)|
\]
vertices 
 $w$ of $\hatA_{t-1}(y_t)$. 
 Since $y_t$ may have at most $\Delta$ unembedded neighbours, 
 in total there are at most $4\Delta d^{-\Delta}\eps' |\hatA_{t-1}(y_t)|$ vertices $w$ of $\hatA_{t-1}(y_t)$ such 
 that the badness condition~\eqref{eq:bad_vts_ncbl} holds for some unembedded neighbour of $y_t$.
 
Next we need to estimate the number of vertices $w$ from $\hatA_{t-1}(y_t)$ such that 
$\psi_{t-1}\cup\{y_t\to w\}$ is not a good partial embedding (i.e.\ does not satisfy properties~\ref{GPE:rightplace}--\ref{GPE:Ureg}). First observe that since $w\in \hatA_{t-1}(y_t)\subseteq C_{t-1}(y_t)\subseteq I_{y_t}$, 
\ref{GPE:rightplace} cannot fail. 

Next we turn to~\ref{GPE:sizeU}. It is  sufficient 
to consider unembedded neighbours $z$ of $y_t$, so $\pi^*_{t-1}(z)\le\Delta-1$. As the bound in~\eqref{eq:buffer:largeUC} also holds for $U_{t-1}(z)$ we have $|U_{t-1}(z)|\ge\eps' p^{\Delta-1}n/r_1$.
Therefore, by the neighbourhood size property $\NS(\eps,r_1,B)$, all but at most $\eps p^{B-1}n/r_1^2$ vertices $w$ from $\hatA_{t-1}(y_t)$ 
are such that $|U_{t-1}(z)\cap N_{\Gamma}(w)|=(1\pm\eps)p|U_{t-1}(z)|=(p\pm\eps p)^{\tau^*_{t-1}(z)+1}|V(z)|$, referring to \ref{GPE:sizeU} for $z$ with respect to $\psi_{t-1}$. 
Further we have $\pi^*_{t-1}(y_t)\le B-2$ since $y_t$ has at most $y_1,\dots,y_{B-2}$ as embedded neighbours (if $b=B$ then by assumption $y_by_{b-1}$ is not an edge of $H$), so since $\eps<(\eps')^2$ we have
 \[\eps p^{B-1}n/r_1^2\leByRef{eq:buffer:largehat} \eps'|\hatA_{t-1}(y_t)|\,.\]
 We  can argue analogously for the other four conditions of~\ref{GPE:sizeU} obtaining 
  in total  that for
 all but at most $5\Delta\eps' |\hatA_{t-1}(y_t)|$ vertices $w$ from $\hatA_{t-1}(y_t)$, 
 \ref{GPE:sizeU} is satisfied with respect to $\psi_{t-1}\cup\{y_t\to w\}$.

Now we consider~\ref{GPE:sizeC}. Let $z$ be any unembedded neighbour of $y_t$. Because $z$ is at distance at most $2$ from $x\in\Xbuf$, by~\ref{BUF:last} $z$ is not image restricted, so $I_z=V(z)$. Since 
 \ref{GPE:sizeU} and~\ref{GPE:sizeC} hold for $\psi_{t-1}$ we have:
 \begin{equation*}
|C_{t-1}(z)|\ge (1-\eps')(dp-\eps'p)^{\pi^*_{t-1}(z)}|V(z)|\ge\eps'|U_{t-1}(z)|.
 \end{equation*}
By~\ref{C:oneside},  $(\hatU_{t-1}(y_t),U_{t-1}(z))$ is  
$(\eps_{\pi_{t-1}^*(y_t),\pi_{t-1}^*(z)},d,p)$-regular in $G$. 
So, there are at most
\[\eps_{\pi_{t-1}^*(y_t),\pi_{t-1}^*(z)}|\hatU_{t-1}(y_t)|\leByRef{eq:buffer:relhat}4d^{-\Delta}\eps'|\hatA_{t-1}(y_t)|\]
vertices $w$ in 
$\hatU_{t-1}(y_t)$ such that $|C_{t-1}(z)\cap N_G(w)|< (d-\eps')p|C_{t-1}(z)|$.
 Similarly we argue for each of the other four conditions. 
 Thus, since $y_t$ has at most $\Delta$ unembedded neighbours, in total for all but at most 
 $20\Delta d^{-\Delta}\eps' |\hatA_{t-1}(y_t)|$ vertices $w\in \hatA_{t-1}(y_t)$ we have that~\ref{GPE:sizeC} 
 is satisfied with respect to $\psi_{t-1}\cup\{y_t\to w\}$. 

 It remains to consider \ref{GPE:Ureg}. Again, we only need to consider
 edges $zz'\in E(H)$ between unembedded vertices such that $z\in N_H(y_t)$. Consider first the case that $z'\not\in
 N_H(y_t)$. In this case we have $\pi^*(z)\le \Delta-2$, $\pi^*(z')\le \Delta-1$ and $\pi^*_{t-1}(y_t)\le\Delta-2$ (since both $x$ and $z$ are unembedded neighbours of $y_t$). Since $\Gamma$ satisfies
 property $\RI(\eps,(\eps_{a,b}),\eps',d,r_1,B)$, using~\eqref{eq:buffer:largeUC} and the fact that by~\ref{GPE:Ureg} the pair $\big(U_{t-1}(z),U_{t-1}(z')\big)$
 is $(\eps_{\pi_{t-1}^*(z),\pi_{t-1}^*(z')},d,p)$-regular in~$G$, we see that there are at most
 \[
 \eps p^{\Delta-1}n/r_1^2\leByRef{eq:buffer:largehat}  \eps'|\hatA_{t-1}(y_t)| 
 \] 
 vertices $w\in \hatA_{t-1}(y_t)$ such that the pair 
 $\big(N_\Gamma(w;U_{t-1}(z)), U_{t-1}(z')\big)$ is not\\
 $(\eps_{\pi_{t-1}^*(z)+1,\pi_{t-1}^*(z')},d,p)$-regular in $G$. 
 
 In the case that 
  $z,z'\in N_H(y_t)$ we have $\pi_{t-1}^*(z')\le\Delta-2$. We also have $\pi^*_{t-1}(y_t)\le B-3$. This requires a little explanation. Observe that $y_t$ has unembedded neighbours $z$ and $z'$, so $\pi^*_{t-1}(y_t)\le\Delta-2$, and if $B=\Delta+1$ then we are done. If $B=\Delta$, then observe that the embedded neighbours of $y_t$ are contained in $\{y_1,\ldots,y_{t-1}\}$. It follows that if $t\le\Delta-2$, we have $\pi^*_{t-1}(y_t)\le\Delta-3$ as desired. If $t=\Delta$, then $x$, $z$ and $z'$ are distinct (since the only unembedded neighbour of $x$ is $y_t$) and we again have $\pi^*_{t-1}(y_t)\le\Delta-3$. It remains to consider the case $t=\Delta-1$. Again we are done if $x$, $z$ and $z'$ are distinct. If, however, (without loss of generality) we have $x=z$, then $z'$ is an unembedded neighbour of $x$ which is not $y_t$; in other words, we have $z'=y_\Delta$. But we assumed that $y_{\Delta-1}y_\Delta$ is not an edge of $H$, contradicting the assumption $y_tz'\in H$. We have thus justified $\pi^*_{t-1}(y_t)\le B-3$, so by $\RI(\eps,(\eps_{a,b}),\eps',d,r_1,B)$ we conclude that there are at most 
 \[
 \eps p^{B-2}n/r_1^2\leByRef{eq:buffer:largehat}  \eps'|\hatA_{t-1}(y_t)| 
 \] 
 vertices $w\in \hatA_{t-1}(y_t)$
 such that the pair $\big(N_\Gamma(w; U_{t-1}(z)), N_\Gamma(w; U_{t-1}(z'))\big)$ is not
 $(\eps_{\pi_{t-1}^*(z)+1,\pi_{t-1}^*(z')+1},d,p)$-regular. So, in total there are at
 most $\Delta^2 \eps' |\hatA_{t-1}(y_t)|$ vertices $w\in \hatA_{t-1}(y_t)$ such
 that~\ref{GPE:Ureg} is not satisfied with respect to
 $\psi_{t-1}\cup\{y_t\to w\}$.

 Summing up, we conclude that
 \begin{equation*}\begin{split}
   |\hatA_{t-1}(y_t)\cap B_{t-1}(y_t)|
   &\le
   4\Delta d^{-\Delta}\eps' |\hatA_{t-1}(y_t)|+5\Delta\eps' |\hatA_{t-1}(y_t)|+20\Delta d^{-\Delta}\eps' |\hatA_{t-1}(y_t)|
   +\Delta^2 \eps' |\hatA_{t-1}(y_t)| \\
   &\le 30\Delta^2 d^{-\Delta}\eps' |\hatA_{t-1}(y_t)|
 \end{split}\end{equation*}
 as required.
\end{claimproof}
 
Next we will estimate the number of unfriendly vertices in $\hatA_{t-1}(y_t)$.

\begin{claim}\label{cl:buffer:P}
  $|P_{t-1}|\le 10\Delta^2d^{-\Delta}\eps' |\hatA_{t-1}(y_t)|$.
\end{claim}
\begin{claimproof}
  Observe that by definition, since $\psi_{t-1}$ is a buffer-friendly partial embedding, for each $w\in\hatA_{t-1}(y_t)$ the embedding $\psi_{t-1}\cup\{y_t\to w\}$ has~\ref{C:past}. For $t=b$ the remaining properties~\ref{C:largeU}--\ref{C:oneside} are trivial, so we from now on assume $t\le b-1$.

For~\ref{C:largeU}, let $k\in\{t+1,\ldots,b\}$ be such that $y_ty_k\in E(H)$. Observe that $\pi^*_{t-1}(y_t)\le t-1\le b-2\le B-2$. 
Therefore,  by~$\NS(\eps,r_1,B)$, by~\eqref{eq:buffer:largehat} and~\eqref{eq:buffer:relhat}, for all but at most $\eps p^{B-1}n/r_1^2\le  \eps'|\hatA_{t-1}(y_t)|$ vertices $w$ we have 
     $|N_\Gamma(w; \hatU_{t-1}(y_k))|= (1\pm\eps')p|\hatU_{t-1}(y_k)|$. 

  To estimate the number of vertices $w$ that do not preserve~\ref{C:largeA}, we only need to consider those 
  $y_k$ with $k\ge t+1$ such that $y_ty_k\in E(H)$. We 
  use that $\big(\hatU_{t-1}(y_t),\hatU_{t-1}(y_k)\big)$ is $(\eps',d,p)$-regular in $G$ by~\ref{C:twoside}. 
Further, by~\eqref{eq:buffer:relhat} we have $|\hatA_{t-1}(y_t)|\ge \eps'|\hatU_{t-1}(y_t)|$, so there are at most
\[\eps'|\hatU_{t-1}(y_t)|\leByRef{eq:buffer:relhat} 4d^{-\Delta}\eps'|\hatA_{t-1}(y_t)|\]
 vertices $w$ in 
$\hatA_{t-1}(y_t)$ such that $\deg_G(w;\hatA_{t-1}(y_k))<(d-\eps')p|\hatA_{t-1}(y_k)|$. 

  Next we consider~\ref{C:oneside}. Let $k\in\{t+1,\dots,b\}$, 
  and $z\in V(H)$ be such that $y_kz\in E(H)$. By~\ref{C:oneside} 
  the pair $\big(\hatU_{t-1}(y_k),U_{t-1}(z)\big)$ is $(\eps_{\pi^*_{t-1}(y_k),\pi^*_{t-1}(z)},d,p)$-regular in $G$. 
  There are  three cases to consider: $y_ty_k\in E(H)$, $y_tz\in E(H)$, and both. 
  
  In the first case (i.e.\ $y_ty_k\in E(H)$, $y_tz\not\in E(H)$), we have $t<B-1$ because $y_{B-1}y_B$ is not an edge of $H$, 
  so we conclude $\pi^*_{t-1}(y_t),\pi^*_{t-1}(y_k)\le t-1\le B-3$, and $\pi^*_{t-1}(z)\le\Delta-1$. 
  By~\eqref{eq:buffer:largeUC} and~\eqref{eq:buffer:largehat},
  and~$\RI(\eps,(\eps_{a,b}),\eps',d,r_1,B)$, we see that for all vertices $w$ of $\hatA_{t-1}(y_t)$ but at most 
  \[\eps p^{B-1}n/r_1^2\leByRef{eq:buffer:largehat} \eps' |\hatA_{t-1}(y_t)|\,,\]
  the pair $\big(N_\Gamma(w;\hatU_{t-1}(y_k)),U_{t-1}(z)\big)$ is $(\eps_{\pi^*_{t-1}(y_k)+1,\pi^*_{t-1}(z)},d,p)$-regular in $G$. 
  
  In the second case (i.e.\ $y_ty_k\not\in E(H)$, $y_tz\in E(H)$), we have 
  $\pi^*_{t-1}(y_t),\pi^*_{t-1}(y_k)\le t-1\le B-2$, and $\pi^*_{t-1}(z)\le t-1\le\Delta-2$ and similarly, by~$\RI(\eps,(\eps_{a,b}),\eps',d,r_1,B)$,  there are not more than $ \eps' |\hatA_{t-1}(y_t)|$ vertices $w$ of $\hatA_{t-1}(y_t)$ for which the pair
  $\big(\hatU_{t-1}(y_k),N_\Gamma(w; U_{t-1}(z))\big)$ is not $(\eps_{\pi^*_{t-1}(y_k),\pi^*_{t-1}(z)+1},d,p)$-regular in $G$. 
  
  In the final case (i.e.\ $y_ty_k\in E(H)$, $y_tz\in E(H)$), again we have $t<B-1$, hence $\pi^*_{t-1}(y_t),\pi^*_{t-1}(y_k)\le B-3$, and $\pi^*_{t-1}(z)\le\Delta-2$, and again all but at most $\eps' |\hatA_{t-1}(y_t)|$ vertices $w$ of $\hatA_{t-1}(y_t)$ are such that $\big(N_\Gamma(w; \hatU_{t-1}(y_k)),N_\Gamma(w; U_{t-1}(z))\big)$ is $(\eps_{\pi^*_{t-1}(y_k)+1,\pi^*_{t-1}(z)+1},d,p)$-regular in $G$.
  
  In total, we see that for all but at most $\Delta^2\eps'|\hatA_{t-1}(y_t)|$ vertices $w$ of $\hatA_{t-1}(y_t)$, the partial embedding $\psi_{t-1}\cup\{y_t\to w\}$ has~\ref{C:oneside}.

Finally, we handle~\ref{C:twoside}. Let $k,\ell\in\{t+1,\dots,b\}$ be such that 
$y_ky_\ell\in E(H)$. We have
\[\pi^*_{t-1}(y_t),\pi_{t-1}^*(y_k),\pi_{t-1}^*(y_\ell)\le t-1\le B-3\,.\]
By~\ref{C:twoside}, the pair $\big(\hatU_{t-1}(y_k),\hatU_{t-1}(y_\ell)\big)$ is 
     $(\eps_{\pi_{t-1}^*(y_k),\pi_{t-1}^*(y_\ell)},d,p)$-regular in $G$. Without loss of generality we may assume $y_ty_k\in E(H)$, and again there are two cases to consider depending on whether $y_ty_\ell\in E(H)$ or not. As before, using~\eqref{eq:buffer:largehat}, by $\RI(\eps,(\eps_{a,b}),\eps',d,r_1,B)$, at most
     \[\eps p^{B-2} n/r_1^2\leByRef{eq:buffer:largehat} \eps'|\hatA_{t-1}(y_t)|\]
     vertices $w\in\hatA_{t-1}(y_t)$ are such that 
     $\big(N_\Gamma(w;\hatU_{t-1}(y_k)),N_\Gamma(w;\hatU_{t-1}(y_\ell))\big)$ is not 
     $(\eps_{\pi_{t-1}^*(y_k)+1,\pi_{t-1}^*(y_\ell)+1},d,p)$-regular in $G$. The other case follows similarly, and we conclude that for all but at most $\Delta^2 \eps'|\hatA_{t-1}(y_t)|$ vertices $w$ of $\hatA_{t-1}(y_t)$, the partial embedding~$\psi_{t-1}\cup\{y_t\to w\}$ has~\ref{C:twoside}.

In total we get $|P_{t-1}|\le\Delta\eps'|\hatA_{t-1}(y_t)|+4\Delta
d^{-\Delta}\eps'|\hatA_{t-1}(y_t)|+\Delta^2\eps'|\hatA_{t-1}(y_t)|+\Delta^2\eps'|\hatA_{t-1}(y_t)|$
as claimed.
\end{claimproof}

We conclude from Claim~\ref{cl:buffer:bad} and Claim~\ref{cl:buffer:P} that
the number of vertices $w\in \hatA_{t-1}(y_t)$ such
that~$w$ is bad for~$y_t$ with respect to~$\psi_{t-1}$ and~$Q_{t-1}$ or
$w\in P_{t-1}$ is at most $40\Delta^2 d^{-\Delta} \eps' |\hatA_{t-1}(y_t)|$.

Now we can estimate the probability that $\psi_t$ is a buffer-friendly partial embedding, conditioning on the history and on $\psi_{t-1}$ being a buffer-friendly partial embedding. This event occurs if $y_t$ is embedded to a member of $\hatA_{t-1}(y_t)\setminus\big(B_{t-1}(y_t)\cup P_{t-1}\big)$. The number of such vertices is at least
\[\big|\hatA_{t-1}(y_t)\big|-40\Delta^2 d^{-\Delta} \eps' |\hatA_{t-1}(y_t)|\ge\tfrac12\big|\hatA_{t-1}(y_t)\big|\geByRef{eq:buffer:largehat}\tfrac14(dp-\eps'p)^{\pi^*_{t-1}(y_t)+1}|V(y_t)|\,,\]
while $y_t$ is embedded into a set of size at most $U_{t-1}(y_t)$, which by~\ref{GPE:sizeU} has size at most $(p+\eps p)^{\pi^*_{t-1}(y_t)}|V(y_t)|$. We see that the desired conditional probability is at least
\[\frac{\tfrac14(dp-\eps'p)^{\pi^*_{t-1}(y_t)+1}|V(y_t)|}{(p+\eps p)^{\pi^*_{t-1}(y_t)}|V(y_t)|}\ge \frac{d^\Delta p}{10}\,.\]
The statement of the lemma follows since the conditional probabilities multiply.\end{proof}

\chapter{Proof of the blow-up lemma for random graphs}
\label{chap:random}

In this chapter we provide the proof of Lemma~\ref{lem:rg_image}. This
proof relies on four main lemmas: The RGA Lemma (Lemma~\ref{lem:rga}), the
queue embedding lemma (Lemma~\ref{lem:queue}), the buffer defect lemma
(Lemma~\ref{lem:deletebad}, and the buffer embedding lemma
(Lemma~\ref{lem:completematch}). These lemmas and the actual proof of
Lemma~\ref{lem:rg_image} are provided in Section~\ref{sec:BUL_Gnp}. The RGA
Lemma is then proved in Section~\ref{sec:rga}, the queue embedding lemma in
Section~\ref{sec:queue}, and the buffer defect lemma in
Section~\ref{sec:fixbuffer}. The proof of Lemma~\ref{lem:completematch} is
short and thus given already in Section~\ref{sec:BUL_Gnp}.

\section{Main lemmas and the proof of the blow-up lemma}
\label{sec:BUL_Gnp}

We divide the proof of Lemma~\ref{lem:rg_image} into four lemmas, which
correspond to our four different embedding stages described in the proof
overview in Section~\ref{sec:proof_overview}. At the end of this section we
show how these four lemmas together with 
Lemma~\ref{lem:matchreduce} (good partitions lemma) imply
Lemma~\ref{lem:rg_image}.

The first of our lemmas encapsulates the outcome after applying the
random greedy algorithm (RGA) which tries to embed~$\Xmain$
into~$\Vmain$. How this RGA operates is explained in the proof of
this lemma in Section~\ref{sec:rga}.  The lemma claims the existence of a
good partial embedding with certain deterministic properties which we
require and which the RGA with high probability produces. Recall that the
RGA constructs the queues~$\Xq_i$.

\begin{lemma}[RGA lemma]\label{lem:rga}
  We assume the General Setup. Suppose that $\Gamma$ has
  $\NS(\eps,r_1,\Delta)$ and
  $\RI(\eps,(\eps_{a,b}),\eps',d,r_1,\Delta)$. Then there is a good
  partial embedding $\psiRGA$ of~$H$ into~$G$ with the following
  properties. For each~$i$, let $\Xq_i:=\Xmain_i\setminus\dom(\psiRGA)$. 
  Then the
  following hold for each~$i$. Let $b$ be such that $\Xbuf_i$ is a degree-$b$
  buffer.
  \begin{enumerate}[label=\itmarab{RGA}]
    \item\label{rga:emb:nobufqueue} All neighbours of all buffer vertices
    are embedded by $\psiRGA$.
    \item\label{rga:emb:place} Every vertex in $X_i\cap\dom(\psiRGA)$ is
    embedded to $\Vmain_i$ by $\psiRGA$.
    \item\label{rga:emb:smallqueue} We have $|\Xq_i|\le2\rho|X_i|$.
    \item\label{rga:emb:main} For every set $W\subset V_i$ of size at least
    $\rho|V_i|$, there are at most $\rho|X_i|$
    vertices in $\Xbuf_i$ with fewer than $(dp)^b|W|/2$ candidates in $W$.
    \item\label{rga:emb:nobad} If $\Xbuf_i$ is not a clique buffer, then every vertex in $V_i$ is a
    candidate for at least $\mu\big(d^\Delta p/100\big)^b|X_i|$ vertices of $\Xbuf_i$.
  \end{enumerate}
\end{lemma}

We shall next extend the good partial embedding provided by the RGA lemma
and embed the queue vertices into the sets $\Vq_i$. That this is possible
is stated by the following lemma. As mentioned in the proof overview, the
proof of this lemma relies on~$\Gamma$ not having a dense spot, which the
congestion condition $\CON(\rho,r_1,\Delta)$ guarantees.

\begin{lemma}[Queue embedding lemma]\label{lem:queue}
  We assume the General Setup. Suppose that $\Gamma$ has $\NS(\eps,r_1,\Delta)$, $\RI(\eps,(\eps_{a,b}),\eps',d,r_1,\Delta)$ and $\CON(\rho,r_1,\Delta)$.
  Let $\psi$ be a good partial embedding whose image is disjoint from the sets
  $\Vq_i$, and suppose that for each $i$ we have a set $\Xq_i\subset X_i$ of
  size at most $2\rho|X_i|$. Then there is a good partial embedding $\psiq$
  extending $\psi$ such that
  \[\dom(\psiq)\setminus\dom(\psi)=\Xq\text{ and } 
  \im(\psiq)\setminus\im(\psi)\subset\Vq\,.\]
\end{lemma}

Note that in the RGA lemma, Lemma~\ref{lem:rga}, we do not have that all vertices~$v$ in $V_i$ are candidates for many
vertices of $\Xbuf_i$ if $\Xbuf_i$ is a clique buffer (and in fact our RGA
may fail to produce this property). In this case we say that~$v$ has
a \emph{buffer defect} for~$\Xbuf_i$.
It follows from~\ref{rga:emb:main}, however, that only at most $\rho|V_i|$
vertices of $V_i$ can have a buffer defect for~$\Xbuf_i$.
Our next lemma, the buffer defect lemma, deals with fixing these buffer defects.

In order to apply this lemma we need to alter the embedding strategy
described so far slightly.
Indeed, before we begin the embedding, we will reserve some copies of
$K_{\Delta+1}$ in~$H$, each of which has a vertex in $\tX_i$. We call these the
\emph{reserved cliques} and denote the family of reserved cliques for~$\tX_i$
by $\cK_i$. We put the vertices
of all reserved cliques into the set $\Xc$, which is disjoint from $\Xmain$ and
$\Xbuf$. Then we apply the RGA lemma and the queue embedding lemma to
obtain a good partial embedding of all of~$H$ but the reserved
cliques~$\Xc$ and the buffer vertices~$\Xbuf$.
Next, we use the buffer defect lemma to embed the reserved
cliques (and, for technical reasons, also
some vertices of~$\Xbuf_i$) on vertices with buffer defects, to obtain a good partial embedding $\psigood$ in which only
the vertices of $\Xbuf$ remain unembedded and without any buffer defects
left.
Note that in the following lemma $\Abuf(x)=A(x)\cap\Vbuf(x)=C(x)\cap\Vbuf(x)\setminus\im(\psi)$.

\begin{lemma}[Buffer defect lemma]\label{lem:deletebad}
  We assume the General Setup. Suppose that~$\Gamma$ has
  $\NS(\eps,r_1,\Delta)$, $\RI(\eps,(\eps_{a,b}),\eps',d,r_1,\Delta)$ and
  $\CON(\rho,r_1,\Delta)$. Let~$\psi$ be a good partial embedding whose
  image is disjoint from the
  sets~$\Vc_i$ and whose domain contains $N(\Xbuf)$. Suppose that for each~$i$ we are given a family~$\cK_i$ of
  reserved cliques $K_{\Delta+1}$, not embedded by~$\psi$, such that for each $K\in\cK_i$ and $x,y\in K$ we have $x\in X_j$ and
  $y\in X_k$ with $jk\in R'$, and that either 
  \begin{enumerate}[label=\abc] 
  \item  $\Xbuf_i$ is a degree-$b$ non-clique buffer with $\cK_i=\emptyset$,
    and all vertices in $V_i$ are candidate for at least $\mu\big(d^\Delta p/100\big)^b|X_i|$
    vertices in $\Xbuf_i$, or
  \item $\Xbuf_i$ is a clique buffer with $|\cK_i|=2\rho|X_i|$, and 
    all but at most $\rho|V_i|$ vertices of
    $V_i$ are candidate for at least $\mu(dp)^\Delta|X_i|/2$ vertices in
    $\Xbuf_i$.
  \end{enumerate} 
  Let $\Xc_i$ be the set of vertices 
  of~$X_i$ contained in some reserved clique from $\cK_i$ and assume that
  $|\Xc_i|\le 2\kappa(\DeltaRp+1)\rho|X_i|$.
  
  Then there is a good partial
  embedding $\psigood$ extending $\psi$ such that for each~$i$ the
  following hold. Let $b$ be such that $\Xbuf_i$ is a degree-$b$
  buffer.
 \begin{enumerate}[label=\itmarab{BD}]
    \item\label{FIN:place} $\Xc_i\subseteq\dom(\psigood)$.
    \item\label{FIN:bigA} For each $x\in\Xbuf_i\setminus\dom(\psigood)$ we have
    $|\Abuf(x)|\ge \mu(dp)^b|V_i|/4$.
    \item\label{FIN:ManyCand} Each $v\in V_i\setminus\im(\psigood)$ is a candidate
    for at least $\mu\big(d^\Delta p/100\big)^b|X_i|$ vertices in $\Xbuf_i\setminus\dom(\psigood)$.
  \end{enumerate}
\end{lemma}

It remains to embed the buffer vertices, which is performed with the help
of the following lemma. 

\begin{lemma}[Buffer embedding lemma]\label{lem:completematch}
  We assume the General Setup. Suppose that $\Gamma$ has
  $\CON(\rho,r_1,\Delta)$. Suppose that $\psi$ is a good partial embedding,
  in which all the vertices $N(\Xbuf)$ are embedded, and that for each~$i$
  we have a subset $X'_i$ of $\Xbuf_i$, and a subset $V'_i$ of
  $V_i\setminus\im(\psi)$ with $|X'_i|=|V'_i|$. Suppose $\Xbuf_i$ is a
  degree-$b$ buffer, and for some $\delta>0$ and $\rho\le \frac{\delta\mu d^b}{100\kappa}$ we have
 \begin{enumerate}[label=\itmarab{BUF}]
  \item\label{cpm:deg} for each $x\in X'_i$ we have
    $|C(x)\cap V'_i|\ge \mu(dp)^b|V_i|/4$,
  \item\label{cpm:pseud} for each $W\subset V'_i$ with $|W|\ge\rho|V_i|$, there are at most $\rho|X_i|$ vertices in $X'_i$ which do not have a candidate in $W$, and
  \item\label{cpm:cands} each $v\in V'_i$ is a candidate
    for at least $\delta p^b|X_i|$ vertices in $X'_i$.
 \end{enumerate}
 Then there is a good partial embedding $\psi'$ extending $\psi$ such that
 $\dom(\psi')=\dom(\psi)\cup \bigcup_i X'_i$ and $\im(\psi')=\im(\psi)\cup\bigcup_{i} V'_i$.
\end{lemma}

The proof of this lemma mainly consists of a straightforward check of
Hall's condition and we provide it right away.
We remark that we will use this lemma again in the proof of the degenerate
graphs blow-up lemma, which
is why we do not give an explicit constant in place of~$\delta$ in~\ref{cpm:cands}.

\begin{proof}[Proof of Lemma~\ref{lem:completematch}]
  Since $\Xbuf$ is independent in H (and the sets $V'_i$ are disjoint in~$G$), we can embed the vertices of~$X'_i$
  for each~$i$ in succession, without affecting the vertices in any
  other~$X'_{i}$. So we fix~$i$ and will show that~$X'_i$ can be embedded into~$V'_i$.

  Let $Y\subseteq X'_i$ be non-empty, and let $W$ be the set of vertices in
  $V'_i$ which are a candidate for some member of $Y$. We wish to verify
  Hall's condition, i.e.\ show that $|W|\ge|Y|$. We separate three cases.
 
  First, suppose $0<|Y|\le \rho|X_i|$. If $|W|<|Y|$, then we can take a
  subset $Y'$ of $Y$ of size $|W|$. Recall that the candidate graph
  between~$Y'$ and~$W$ is the bipartite graph with edges $yw$ for $y\in Y'$
  and $w\in W$ with $w\in C(y)$, and that this is a subgraph of the
  underlying restriction graph between~$Y'$ and~$W$ with edges $yw$ with
  $w\in U(y)$. Observe, moreover, that this underlying restriction graph is
  isomorphic to the congestion graph $\AG(\Gamma,W,\cF)$ with
  $\cF=\big\{\psi\big(N_H(y)\big)\colon y\in Y'\big\}$, which has edges
  $wF$ with $w\in\comN_\Gamma(F)$. Now, since by~\ref{cpm:deg} each vertex
  in $Y'$ has at least $\mu(dp)^b|V_i|/4$ candidates in $V'_i$, which must
  lie in $W$, the number of edges in the candidate graph between $Y'$ and
  $W$ is at least
 \begin{align*}\tfrac{1}{4}\mu(dp)^b|V_i||Y'|&=
 \tfrac{1}{8}\mu(dp)^b|V_i||Y'|+\tfrac{1}{8}\mu(dp)^b|V_i||Y'|\\
 &> \frac{\mu d^b}{8\rho}p^b|Y'||W| +\frac{\mu
 d^b}{8 \kappa r_1} p^b n|Y'|\\
 &> 7p^b|Y'||W|+\rho p^b n|Y'|/r_1\,,
 \end{align*}
 where the second line comes from substituting $|V_i|=|X_i|>|Y'|/\rho=|W|/\rho$
 and using $|V_i|\ge n/(\kappa r_1)$, and
 the third follows from the bound on $\rho$. 
 Hence $e(\AG(\Gamma,W,\cF))>7p^b|Y'||W|+\rho p^b n|Y'|/r_1$, in
 contradiction to $\CON(\rho,r_1,\Delta)$. We conclude that $|W|\ge |Y|$ in this
 case.
 
 Second, if $\rho|X_i|<|Y|\le|X'_i|-\rho|X_i|=|V_i'|-\rho|V_i|$ and $|W|<|Y|$, then
 $|V_i'\setminus W|>\rho|V_i|$, so by~\ref{cpm:pseud} there are at most
 $\rho|X_i|$ vertices of $X'_i$ which do not have candidates in $V'_i\setminus W$. In particular there is a vertex of $Y$ with
 candidates in $V'_i\setminus W$, in contradiction to the definition of $W$. We
 conclude that $|W|\ge|Y|$ in this case as well.
 
 Finally, suppose $|Y|>|X'_i|-\rho|X_i|=|V_i'|-\rho|V_i|$. The vertices $V_i'\setminus W$
 are candidates only for vertices in $X'_i\setminus Y$, and each vertex in
 $V_i'\setminus W$ is a candidate for at least $\delta p^b|X_i|$ vertices in
 $X_i'\setminus Y$ by~\ref{cpm:cands}. If $|W|<|Y|$, then $|V_i'\setminus
 W|>|X'_i\setminus Y|$, and we can take a set $\tilde W\subset V_i'\setminus W$ of size $|X'_i\setminus Y|$. 
 Now the number of edges in the candidate graph
 between $\tilde W$ and $X'_i\setminus Y$ is at least
 \begin{align*}\delta p^b|X_i||\tilde W|
 &>\tfrac12\delta p^b\big(\tfrac{1}{\rho}|X'_i\setminus Y||\tilde W|+\tfrac{n}{\kappa r_1}|X'_i\setminus Y|\big)\\
 &>7p^b|X'_i\setminus Y||\tilde W|+\rho p^b n|X'_i\setminus Y|/r_1
 \end{align*}
 by essentially the same calculation as in the first case, using
 $|X_i|=|V_i|\geq n/(\kappa r_1)$.
 Taking $\cF'=\big\{\psi\big(N_H(x)\big)\colon x\in X'_i\setminus Y\big\}$,
 the number of edges in the corresponding congestion graph
 $\AG(\Gamma,\tilde W,\cF')$ is also at least this quantity, in contradiction to
 $\CON(\rho,r_1,\Delta)$.
 
 This completes the verification of Hall's condition, so there is a partial
 embedding $\psi'$ extending $\psi$ with $\dom(\psi')=\dom(\psi)\cup X'_i$
 in which the vertices $X'_i$ are embedded to $V'_i$. Since $\Xbuf$ is independent and all the vertices $N(\Xbuf)$ are
 embedded in $\psi$, it is trivially the case that $\psi'$ is a good partial embedding. 
\end{proof}

We are now ready to prove Lemma~\ref{lem:rg_image}. Summarising, the proof will
go as follows. We will apply Lemma~\ref{lem:matchreduce} to find partitions
of $G$ and $H$, and graphs $R$ and $R'$, satisfying the General Setup. We
choose, for each $i$ such that $\Xbuf_i$ is a clique buffer a set of
reserved cliques with vertices~$\Xc$. We can then apply the RGA lemma, the
queue embedding lemma, the buffer defect lemma, and the buffer embedding
lemma, in this order.

\begin{proof}[Proof of Lemma~\ref{lem:rg_image}]
 First we choose constants as follows. Given $\Delta$, $\DeltaRpbl$,
 $\Delta_J$ integers, $\alphabl$, $\zetabl$ and $d>0$, and $\kappabl>1$, we
 set $\vartheta=0$, $\DeltaRp=8(\Delta+\Delta_J)^{10}\DeltaRpbl$,
 $\alpha=\tfrac12\alphabl$, $\zeta=\tfrac12\zetabl$ and $\kappa=2\kappabl$.
 We now choose $\mu$, $\rho$ and $\eps'>0$ satisfying the conditions in
 Lemmas~\ref{lem:matchreduce}, \ref{lem:fewbad}, \ref{lem:rga},
 \ref{lem:queue}, \ref{lem:deletebad} and~\ref{lem:completematch}. For convenience we provide here sufficient choices:
\[
\mu<\frac{d^\Delta}{1320\kappa\DeltaRp},\text{  and  } \rho\le \frac{\mu^2d^{\Delta^2+1}}{250^{\Delta+1}\kappa\DeltaRp}, 
\text{  and  } \eps'\le \frac{\mu\zeta\rho d^\Delta}{32^{\Delta+2}2^{4/\rho}\kappa\Delta^2\DeltaRp}\,.
\]
 Now for input $\Delta$, $d$ and $\eps'$, Lemma~\ref{lem:det_Gnp} returns constants $\eps_{a,b}$ and $\eps>0$. Here we additionally require that $\eps<(\eps')^2$. 
 We let $\epsbl=\tfrac{1}{16}(\Delta+\Delta_J)^{-10}\eps$ and $\rhobl=\tfrac{1}{16}(\Delta+\Delta_J)^{-10}\rho$. 
 Now Lemma~\ref{lem:rg_image} returns $\epsbl$ and $\rhobl$. Given $\ronebl$ we let $r_1=8(\Delta+\Delta_J)^{10}\ronebl$. We choose $C$ sufficiently large for Lemma~\ref{lem:det_Gnp} with input $\Delta$, $d$, $\eps'$, $T=r_1$ and $\rho$.
  
 Given $p\ge C\big(\tfrac{\log n}{n}\big)^{1/\Delta}$, Lemma~\ref{lem:det_Gnp}
 states that a.a.s.\ $\Gamma=G(n,p)$ has properties
 \[\NS(\eps,r_1,\Delta)\,, \RI(\eps,(\eps_{a,b}),\eps',d,r_1,\Delta) \text{ and } \CON(\rho,r_1,\Delta)\,,\]
 respectively. From now on we will assume $\Gamma$ is an $n$-vertex graph which satisfies these three properties.
 
 Given a graph $\Rbl$ on $\rbl\le\ronebl$ vertices, and a spanning subgraph
 $\Rpbl$ with $\Delta(\Rpbl)\le\DeltaRpbl$, and graphs $H$ and
 $G\subset\Gamma$ with vertex partitions $\cXbl$ and $\cVbl$, families of
 image restrictions $\cIbl$ and of image restricting vertices $\cJ$, and a
 family of potential buffer vertices $\tcXbl$, suppose that the conditions
 of Lemma~\ref{lem:rg_image} are satisfied. Then
 Lemma~\ref{lem:matchreduce} gives (with~$\Delta$ in place of~$b$) a graph
 $R$ on $r\le r_1$ vertices, a spanning subgraph $R'$ with
 $\Delta(R')\le\DeltaRp$, and $\kappa$-balanced size-compatible partitions
 $\cX$ and $\cV$ of $H$ and $G$ respectively, each part having size at
 least $n/(\kappa r_1)$, together with a family $\tcX$ of potential buffer
 vertices and $\cI$ of image restrictions, subsets $\Xbuf_i$ of $\tX_i$ for
 each $i\in[r]$, and partitions
 $V_i=\Vmain_i\dcup\Vq_i\dcup\Vc_i\dcup\Vbuf_i$ for each $i\in[r]$ which
 satisfy the General Setup.
 
 Now for each $i\in[r]$ such that $\Xbuf_i$ is a clique buffer, in
 succession, we create a collection of reserved cliques~$\cK_i$ and a set
 $\Xc$ as follows. We choose $2\rho|X_i|$ vertices from $\tX_i$ which are
 in copies of $K_{\Delta+1}$ that do not contain image restricted vertices,
 vertices of $\Xbuf$, or vertices previously added to $\Xc$. We add these
 copies of $K_{\Delta+1}$ to $\cK_i$, and their vertices to $\Xc$. Observe
 that this is possible since~\ref{BUF:existcliques} guarantees that each
 $\tX_i$ contains at least $\tfrac{1}{2\Delta+4}\alpha|X_i|$ vertices in
 copies of $K_{\Delta+1}$ whose vertices are neither image restricted nor
 in $\Xbuf$ and, moreover, by~\ref{H:partition} the edges of these cliques lie
 along $R'$, so the number of these cliques which are chosen for some
 $\cK_j$ with $j<i$ is at most
 $\DeltaRp2\rho\kappa|X_i|<\tfrac{1}{4\Delta+8}\alpha|X_i|$ by choice of
 $\rho$. We conclude the following properties for these reserved cliques.
 \begin{enumerate}[label=\itmarab{RSC}]
  \item\label{BUF:resclique} The sets $\{\cK_i\}_{i\in[r]}$ are pairwise disjoint, and if $\Xbuf_i$
    is a clique buffer, then $\cK_i$ contains $2\rho|X_i|$ cliques, each with
one vertex in $X_i$, otherwise it
    is empty.
    \item\label{BUF:Rp} For each $i\in[r]$, $K\in\cK_i$ and $x,y\in K$ we have $x\in X_j$ and
    $y\in X_k$ where $jk\in R'$.
    \item \label{BUF:clique} For each $i$,  the set $\Xc_i$ of vertices 
    $x\in X_i$ contained in some reserved clique satisfies $|\Xc_i|\le 2\kappa(\DeltaRp+1)\rho|X_i|$.
 \end{enumerate}

\noindent Notice that~\ref{BUF:Rp} holds by~\ref{H:partition}, since the
first and second neighbours of vertices from $\tX_i$ go along the edges of
$R'$ by the definition of an $(\alpha,R')$-buffer. 

 We let $\Xmain=V(H)\setminus(\Xbuf\cup\Xc)$. We now begin the embedding of~$H$ into~$G$. 
 By Lemma~\ref{lem:rga}, there is a good partial embedding $\psiRGA$ with
 properties~\ref{rga:emb:nobufqueue}--\ref{rga:emb:nobad}. Letting
 $\Xq=\Xmain\setminus\dom(\psiRGA)$, by~\ref{rga:emb:smallqueue} we have
 $|\Xq_i|\le 2\rho|X_i|$ for each $i$, and by~\ref{rga:emb:place} we see
 that $\im(\psiRGA)$ is disjoint from each set $\Vq_i$, so the conditions of
 Lemma~\ref{lem:queue} are met. Feeding $\psiRGA$ into
 Lemma~\ref{lem:queue} we obtain a good partial embedding $\psiq$ extending
 $\psiRGA$ whose domain is $\Xmain$ and whose image is contained in
 $\Vmain\cup\Vq$. By~\ref{rga:emb:nobufqueue} all neighbours of all buffer
 vertices are in $\dom(\psiRGA)$, and therefore the candidate sets of all
 buffer vertices are the same with respect to $\psiRGA$ as with respect to
 $\psiq$. In particular $\psiq$ satisfies~\ref{rga:emb:main}
 and~\ref{rga:emb:nobad}.
 
 We now verify the conditions for Lemma~\ref{lem:deletebad}. Since $\psiq$
 extends $\psiRGA$ and the vertices $\im(\psiq)\setminus\im(\psiRGA)$ are
 embedded in the sets $\Vq_i$ it follows by~\ref{rga:emb:place} that
 $\im(\psiq)$ is disjoint from the sets $\bigcup_i\Vc_i$ and $\bigcup_i
 \Vbuf_i$. By~\ref{rga:emb:nobad}, if $\Xbuf_i$ is a degree-$b$ buffer for some $b$ but 
 is not a clique buffer, then all vertices in $V_i$ are candidate for at least
 $\mu(d^\Delta p/100)^b|X_i|$ vertices in $\Xbuf_i$ and $\cK_i$ is empty by \ref{BUF:resclique}.
 
 If on the other hand $\Xbuf_i$ is a clique buffer, then $|\cK_i|=2\rho|X_i|$. Now let
 $W$ be the set of vertices in $V_i$ which are candidates for fewer than
 $\mu(dp)^\Delta|X_i|/2$ vertices of $\Xbuf_i$. We will show that $|W|<\rho
 |V_i|$.
 If $|W|\ge\rho|V_i|$, then by~\ref{rga:emb:main} there are at most
 $\rho|X_i|$ vertices in $\Xbuf_i$ with fewer than $(dp)^\Delta |W|/2$ candidates in $W$. It follows that the average over $v\in W$ of the number of
 vertices of $\Xbuf_i$ for which $v$ is a candidate, is at least
 \[\frac{1}{|W|}\big(|\Xbuf_i|-\rho|X_i|\big)\frac{(dp)^\Delta
 |W|}{2}>\frac12\mu(dp)^\Delta|X_i|,\]
 where we used the facts $|\Xbuf_i|=4\mu|X_i|$ and $\rho<\mu$ in the inequality.
 Since at least one $v\in W$ attains at least the average, we have a
 contradiction to the definition of~$W$. 
 
 The
 conditions of Lemma~\ref{lem:deletebad} are thus satisfied, and by applying it
 we obtain a good partial embedding $\psigood$. Since $\psigood$ extends
 $\psiq$ and all sets $\Xc_i$ are contained in $\dom(\psigood)$
 by~\ref{FIN:place}, we conclude that the only vertices remaining
 unembedded are in $\Xbuf$.
 
 Finally, for $i\in[r]$, let $X'_i=\Xbuf_i\setminus\dom(\psigood)$ and let
 $\Xbuf_i$ be a degree-$b$ buffer for some $b$. Let
 $V'_i=V_i\setminus\im(\psigood)$. Because $|X_i|=|V_i|$ we have
 $|X'_i|=|V'_i|$. Because $\psigood$ satisfies~\ref{FIN:bigA}, we
 have~\ref{cpm:deg}. Because $\psigood$ satisfies~\ref{rga:emb:main}, in
 particular we have~\ref{cpm:pseud}. Finally, because $\psigood$
 satisfies~\ref{FIN:ManyCand} we have~\ref{cpm:cands} with
 $\delta=\mu(d^\Delta/100)^b$. Thus by Lemma~\ref{lem:completematch} there
 is an embedding $\psi'$ extending $\psigood$ which embeds~$H$ into~$G$.
\end{proof}

\section{Proof of the random graphs RGA lemma}
  \label{sec:rga}

  In this section we describe the random greedy algorithm (RGA) and prove
  that it produces a partial embedding which satisfies the assertions of
  the RGA lemma, Lemma~\ref{lem:rga}, with high probability.  This is Algorithm~\ref{alg:RGA}.  It embeds vertices~$x$ of~$H$
  sequentially, following an order $\tau$ given by Lemma~\ref{lem:tau} (a good vertex order for an RGA). In
  doing so, it builds up a sequence $\psi_0,\dots$ of good partial
  embeddings, and a \emph{queue} of vertices which it will not embed; we
  let $Q_t$ be the queue at time $t$ (i.e.\ corresponding to $\psi_t$).
  Recall that by $\Amain_t(x)$ we mean the set $\Amain(x)$ with reference
  to the partial embedding $\psi_t$. We let $B_t(x)$ denote the set of bad
  vertices (defined in Section~\ref{subsec:bad_vertices}) with respect
  to~$\psi_t$ and~$Q_t$. As mentioned in the proof overview
  (Section~\ref{sec:proof_overview}), to create $\psi_{t+1}$ from $\psi_t$
  we embed some vertex~$x$ uniformly at random into the set
  $\Amain_t(x)\setminus B_t(x)$, the set of available candidate vertices in
  $\Vmain$ minus the bad vertices; and we add~$y$ to the queue if
  the set $\Amain_t(y)\setminus B_t(y)$ gets small.
    
  Note that exactly $t$ vertices are embedded in $\psi_t$, though these
  vertices are not necessarily the first $t$ vertices of $\tau$ because
  vertices in $Q_t$ are skipped. The queue set~$Q_{t_\RGend}$ at the time
  $t_\RGend$ when the RGA terminates will then form the queue~$\Xq$.

  \begin{algorithm}[t]
    \caption{Random greedy algorithm for random graphs}\label{alg:RGA}
    \SetKwInOut{Input}{Input}
    \Input{$G\subseteq \Gamma$ and $H$ with partitions satisfying the General Setup; an ordering $\tau$ of $\Xmain$}
    $t:=0$ \; 
    $\psi_0:=\emptyset$ \;
    $Q_0:=\{x\in V(H):|I_x|<\tfrac12\mu(d-\eps)^{|J_x|}p^{|J_x|}|\Vmain(x)|\}$ \; 
    \Repeat{ $\dom(\psi_t)\cup Q_t=\Xmain$ }{
      let $x\in \Xmain\setminus (\dom(\psi_t)\cup Q_t)$ be the next vertex in
      the order $\tau$ \; 
      choose $v\in \Amain_t(x)\setminus B_t(x)$ uniformly at random \;
      $\psi_{t+1}:=\psi_t\cup\{x\to v\}$ \;
      $Q_{t+1}:=Q_t$ \;
      \ForAll {$y\in \Xmain\setminus \dom(\psi_{t+1})$}{
        \If{$(\big|\Amain_{t+1}(y)\big|<\tfrac12
        \mu(d-\eps')^{\pi^*_{t+1}(y)}p^{\pi^*_{t+1}(y)}|\Vmain(y)|)$ }
        {$Q_{t+1}:=Q_{t+1}\cup\{y\}$ \; } } 	
      $t:=t+1$ \;
    }
    $t_\RGend:=t \;$
  \end{algorithm}
   
  The proof that Algorithm~\ref{alg:RGA} a.a.s.\ produces a good partial
  embedding with the properties required in Lemma~\ref{lem:rga} is now
  quite straightforward: most of the work is to check that the conditions
  of the various lemmas on RGAs given in Section~\ref{sec:RGAlemmas} are
  met. The critical point is to show that certain invariants (see
  Claim~\ref{cl:rga:inv} below) are maintained.

\begin{proof}[Proof of Lemma~\ref{lem:rga}]
We require
\[\rho< \mu\le\frac{1}{100\kappa\DeltaRp}\,,\quad \eps'\le \frac{\mu\zeta d^\Delta\rho2^{-4/\rho}}{1000\kappa\Delta^24^\Delta\DeltaRp}\,,\quad\eps\le\frac{\eps'}{\kappa d}\]
and
\[ p\ge 10r_1\eps^{-1} \big(\tfrac{\log n}{n}\big)^{1/\Delta}\,,\quad n>\kappa r_1^2\,.\]

First, we apply Lemma~\ref{lem:tau} to obtain an ordering $\tau$ of the
vertices of $H$ with properties~\ref{cond:seg}--\ref{cond:ncl_ord} of this lemma.
With this ordering $\tau$ we run the random greedy algorithm, Algorithm~\ref{alg:RGA}, and in the following we show 
that a.a.s.\ 
it produces a good partial embedding $\psiRGA=\psi_{t_\RGend}$ satisfying
properties \ref{rga:emb:nobufqueue}--\ref{rga:emb:nobad} of Lemma~\ref{lem:rga}.

  We first show that some invariants are maintained (deterministically)
  throughout the algorithm. Recall that $\pi^*_t(x):=\pi_t(x)+|J_x|$, where
  $\pi_t(x)$ denotes the number of the embedded neighbours of $x$ at time
  $t$. 

  \begin{claim}\label{cl:rga:inv} The following hold at each time $t$ in Algorithm~\ref{alg:RGA}.
   \begin{enumerate}[label=\itmarab{INV}]
    \item \label{inv:gpe} $\psi_t$ is a good partial embedding.
    \item \label{inv:sizeA} For each $x\in\Xmain\setminus\dom(\psi_t)$, either $x\in Q_t$ or we have
    \[|\Amain_t(x)|\ge\tfrac12\mu(d-\eps')^{\pi^*_t(x)}p^{\pi^*_t(x)}|\Vmain(x)|\,.\]
    \item\label{inv:random} When we embed $x$ to create $\psi_{t+1}$, we do so
    uniformly at random into a set of size at least $\frac{1}{10}\mu\zeta
    (dp)^{\pi^*_{t}(x)} |V(x)|$.
   \end{enumerate}
  \end{claim}

  We remark that the bound given in~\ref{inv:random} is not sharp, but we
  give this relaxed bound in order to re-use the same invariants in the
  proofs of our other two blow-up lemmas.

  \begin{claimproof}[Proof of Claim~\ref{cl:rga:inv}]
    The invariant~\ref{inv:gpe} is maintained by definition of $A_t(x)$ and
    $B_t(x)$. Observe that at the end of each iteration of the repeat loop
    of Algorithm~\ref{alg:RGA} the
    queue is updated by adding precisely any vertices which would
    fail~\ref{inv:sizeA}, so that this invariant also holds.
    Finally, by choice of $\eps'$ and by~\ref{inv:sizeA}, if the vertex $x$
    is embedded to create $\psi_{t+1}$, then we have
    $|\Amain_{t}(x)|\ge\tfrac14\mu(dp)^{\pi^*_{t}(x)}|V(x)|$. Now, if all
    neighbours of $x$ are embedded under $\psi_{t}$ then by Lemma~\ref{lem:fewbad}\ref{fewbad:a}
    we have~\ref{inv:random}, while otherwise we apply
    Lemma~\ref{lem:fewbad}\ref{fewbad:b} with $D=\Delta$, for which the
    required assumptions on $\Gamma$ hold, to bound $|B_{t}(x)|$. We conclude that \ref{inv:random} holds.
  \end{claimproof}

  We now begin to verify the conclusions of Lemma~\ref{lem:rga}. Property \ref{rga:emb:place}, stating that all vertices in $X_i\cap
\dom(\psiRGA)$ are embedded to $\Vmain_i$, holds by our RGA.

\begin{claim}\label{cl:RGA_queue}
Property~\ref{rga:emb:nobufqueue} holds. Moreover, 
let $x$ be a buffer vertex and~$y$ its neighbour whose embedding creates $\psi_{t+1}$, then 
\begin{equation}\label{eq:rga:Amain}
\big|\Amain_t(y)\big|\ge\tfrac23(dp-\eps'p)^{\pi^*_t}|V(y)|\,.
\end{equation}
\end{claim}
\begin{claimproof}
We require $\mu\le 1/(100\kappa\DeltaRp)$.
To show~\ref{rga:emb:nobufqueue}, i.e., that all
neighbours of buffer vertices are in $\dom(\psiRGA)$, it is sufficient to
prove that none of the neighbours of buffer vertices are in $Q_{t_\RGend}$.
However, this is clear. Indeed, suppose that $y_1,\dots, y_{b}$ with $b\leq
\Delta$ are the neighbours of a buffer vertex $x$ appearing in this order
in $\tau$.  Suppose the embedding of~$y_1$ creates $\psi_{t_1+1}$. By~\ref{BUF:dist}, buffer vertices are at distance at least five in $H$, hence,
the neighbours of the vertices have distance at least three, and by~\ref{BUF:last} they are not image restricted. Therefore,
at time $t_1$ the available candidate set of each $y_j$ is
$\Amain_{t_1}(y_j)=\Vmain(y_j)\setminus \im(\psi_{t_1})$. The size of 
$\Amain_{t_1}(y_j)$  is by~\ref{BUF:sizebuf},~\ref{G:sizes} and choice of $\mu$ at least 
\[
|\Vmain(y_j)|-4\kappa\DeltaRp\mu|V(y_j)|=(1-3\mu-4\kappa\DeltaRp\mu)|V(y_j)|\ge \tfrac{2}{3}|V(y_j)|\,.
\]
It follows that $y_j$ is not added to $Q_t$ for any $t\le t_1$. Now the
vertices $y_1,\ldots,y_b$ are embedded consecutively by
Lemma~\ref{lem:tau}\ref{cond:noedge}, and since they are neighbours of $x$,
by~\ref{PtH:dist}, into distinct clusters of $G$.
  By Definition~\ref{def:bad_vertices} (of bad vertices with respect to $\psi$ and $Q$), 
  for each time $t'$ with $t_1\le t'\le t$, where $y_j$ is embedded to create $\psi_{t+1}$, we have
\[|\Amain_{t'}(y_j)|\ge \tfrac23(dp-\eps'p)^{\pi^*_{t'}(y_j)}|V(y_j)|\,,\]
which gives~\eqref{eq:rga:Amain} and that $y_j$ never enters $Q_t$, as desired.
\end{claimproof}

We next verify~\ref{rga:emb:smallqueue} and~\ref{rga:emb:main}. To this end
observe that because Algorithm~\ref{alg:RGA} preserves~\ref{inv:random},
the conditions of Lemma~\ref{lem:rga:welldistr} are met. Thus with
probability at least $1-2^{-n/(\kappa r_1)}$, the following event
$\ev_{\sublem{lem:rga:welldistr}}$ holds. For every $i\in[r]$ and $W\subset
V_i$ with $|W|\ge\rho|V_i|$, the number of vertices $x\in X_i\setminus
X^*_i$ such that for some time $t$ (at which $x$ is unembedded) we have
$\big|C_t(x)\cap W\big|<(dp-\eps'p)^{\pi^*_t(x)}|W|$ is at most
$\rho|X_i|$.

Suppose now that $\ev_{\sublem{lem:rga:welldistr}}$ holds. We first show
that then~\ref{rga:emb:smallqueue} holds. Set
$W:=\Vmain_i\setminus\im(\psi_{t_\RGend})$. By~\ref{BUF:sizebuf}
and~\ref{G:sizes}, we have $|\Xmain_i|\le (1-4\mu)|X_i|$ and
$|\Vmain_i|=(1-3\mu)|V_i|$. We conclude $|W|\ge\mu|V_i|$. Suppose that
$x\in Q_{t_\RGend}$. Then there is a first time $t$ at which $x\in Q_t$.
Since we have $\Amain_t(x)\supseteq C_t(x)\cap W$, by the construction of
$Q_t$ in Algorithm~\ref{alg:RGA}, we have
 \[|C_t(x)\cap W|<\tfrac12\mu(d-\eps')^{\pi^*_t(x)}p^{\pi^*_t(x)}|\Vmain(x)|<\big((d-\eps')p\big)^{\pi^*_t(x)}|W|\] 
 so that $x$ satisfies condition~\eqref{eq:rga:welldistr} of
 Lemma~\ref{lem:rga:welldistr}. Since $|W|>\rho|V_i|$, and because
 $\ev_{\sublem{lem:rga:welldistr}}$ holds, the number of $x\in \Xmain_i\cap
 Q_{t_\RGend}$ which are in $X_i\setminus X^*_i$ is at most $\rho|X_i|$.
 By~\ref{PtH:image}, we have $|X^*_i|\le\rho|X_i|$, so $|\Xmain_i\cap
 Q_{t_\RGend}|\le\rho|X_i|+|X^*_i|\le2\rho|X_i|$. By definition of
 Algorithm~\ref{alg:RGA}, the vertices $\Xmain_i\cap Q_{t_\RGend}$ are
 precisely the vertices of $\Xmain_i$ not in $\dom(\psi_{t_\RGend})$,
 giving~\ref{rga:emb:smallqueue} as desired.
 
 Property~\ref{rga:emb:main} follows from
 $\ev_{\sublem{lem:rga:welldistr}}$ by observing that $\Xbuf_i\subset X_i$
 and that when $\Xbuf_i$ is a degree-$b$ buffer we have
 $(dp-\eps'p)^{\pi^*_{t_\RGend}(x)}\ge (dp)^b/2$ by choice of~$\eps'$.

To complete the proof of Lemma~\ref{lem:rga} we now verify that a.a.s.\
after Algorithm~\ref{alg:RGA} finishes, also~\ref{rga:emb:nobad} is
satisfied. By~\ref{cond:seg} and~\ref{cond:ncl_ord}  of
Lemma~\ref{lem:tau}, the set of vertices in $N(\Xbuf)$ which are not
neighbours of clique buffer vertices is an initial segment of $T$ vertices
of $\tau$. We first show that the embedding of these~$T$ vertices cannot
fill up vertex neighbourhoods in~$G$. For this purpose we want to apply
Lemma~\ref{lem:large_nbs} to the
sequence of good partial embeddings produced by Algorithm~\ref{alg:RGA} on
this initial segment. Observe that by Claim~\ref{cl:RGA_queue}, all
vertices in this initial segment satisfy~\eqref{eq:rga:Amain} and none
enters the queue. By~\eqref{eq:rga:Amain}, Lemma~\ref{lem:fewbad}, and the
choice of $\eps'$, each such $y$, embedded to create $\psi_{t+1}$, is
embedded into a subset of $C_{t}(y)$ of size at least
\[\tfrac{2}{3}(dp-\eps'p)^{\pi^*_{t}}|V(y)|-20\Delta^2\eps'p^{\pi^*_t(y)}|V(y)|>\tfrac1{10}(dp)^{\pi^*_{t}(y)}|V(y)|\,.\]
Furthermore, for each such $y$ we have $\pi^*_{t}(y)\le \Delta-2$ by
Lemma~\ref{lem:tau}\ref{cond:noedge}. Finally, by~\ref{BUF:sizebuf} we have
$\big|\dom(\psi_T)\cap X_i\big|\le 4\mu\kappa\DeltaRp|X_i|$ for each
$i\in[r]$. This justifies that we can apply Lemma~\ref{lem:large_nbs} with
$B=\Delta$, and the result is that with probability at least $1-\exp(-\eps
p n/r_1)$, the following event $\ev_{\sublem{lem:large_nbs}}$ holds. For
each $v\in V_i$ and $j$ such that $ij\in E(R')$, we have
\begin{equation}\label{eq:rga:initial}
\big|N_G(v;\Vmain_j)\setminus\im(\psi_T)\big|\ge\tfrac12\deg_G(v;V_j)\,.
\end{equation}

We assume from now on that $\Xbuf_i$ is a non-clique degree-$b$ buffer, and
we fix a vertex $v\in V_i$. We would like to estimate the probability
that~\ref{rga:emb:nobad} fails for $v$. To that end, first fix
$x\in\Xbuf_i$, and let $y_1,\ldots,y_b$ be an enumeration of $N_H(x)$ in
the order $\tau$. We now justify that if $\ev_{\sublem{lem:large_nbs}}$
holds, then we are in a position to apply Lemma~\ref{lem:nonclique_buffer}
with $B=\Delta$. Recall that no vertices of $N(\Xbuf)$ enter the queue, by
Claim~\ref{cl:RGA_queue}, so that $\psi_{\tau(y_1)}$ is the good partial
embedding created by the embedding of $y_1$ (and so on). By~\ref{BUF:dist},
and since $\dom(\psi_{\tau(y_1)-1})\subset N(\Xbuf)$, no vertices at
distance two or less from $x$ in $H$ are embedded in $\psi_{\tau(y_1)-1}$.
Note that Lemma~\ref{lem:tau}\ref{cond:noedge} states that if $b=\Delta$
then $y_{b-1}y_b$ is not an edge of $H$. Algorithm~\ref{alg:RGA} thus
creates a sequence $\psi_{\tau(y_1)-1},\dots,\psi_{\tau(y_b)}$ of good
partial embeddings, and a sequence of queue sets
$Q_{\tau(y_1)-1},\dots,Q_{\tau(y_b)}$, matching the requirements of
Lemma~\ref{lem:nonclique_buffer}. Thus, by
Lemma~\ref{lem:nonclique_buffer}, the probability that $N_H(x)$ is embedded
to $N_G(v)$, conditioning on $\psi_{\tau(y_1)-1}$ and that
$\big|N_G(v;\Vmain_j)\setminus\im(\psi_{\tau(y_1)-1})\big|\ge\tfrac12\deg_G(v;V_j)$
for each $j$ such that $ij\in R'$ (which is justified by~\eqref{eq:rga:initial}), is at least $(d^\Delta p/10)^b$.

Now we use this to estimate the cumulative effect of all $x\in\Xbuf_i$. Let
$x_1,\dots,x_m$ be an enumeration of $\Xbuf_i$ according to the order on
$N(\Xbuf)$ given by $\tau$. Let $Y_1,\ldots,Y_m$ be Bernoulli random
variables with $Y_j=1$ if either $N_H(x_j)$ is embedded by
$\psi_{t_\RGend}$ to $N_G(v)$ or we witness a failure of
$\ev_{\sublem{lem:large_nbs}}$ before the first neighbour of $x_j$ is
embedded. Then we have just shown that $Y_j$ has probability at least
$(d^\Delta p/10)^b$ of being one, conditioned on the history up to, but not
including, the embedding of the first vertex of $N_H(x_j)$. This history
determines $Y_{j-1}$, so we can apply Lemma~\ref{lem:coupling}, with
$x=4\mu|X_i|(d^\Delta p/10)^b$ and $\delta=\tfrac12$, to conclude that
$Y_1+\dots+Y_m\ge 2\mu|X_i|(d^\Delta p/10)^b$ with probability at least
$1-\exp(-\mu|X_i|(d^\Delta p/10)^b/3)$. Taking the union bound over all
$v\in V(G)$ we see that with probability at least $1-n\exp(-\eps p^\Delta
n/r_1)>1-1/n$ (where the inequality is by choice of $p$ and since
$b\le\Delta$), either we witness a failure of
$\ev_{\sublem{lem:large_nbs}}$, or~\ref{rga:emb:nobad} holds.

Putting together the three probability bounds, we conclude that a.a.s.\ the events $\ev_{\sublem{lem:rga:welldistr}}$, $\ev_{\sublem{lem:large_nbs}}$ and~\ref{rga:emb:nobad} all hold, proving Lemma~\ref{lem:rga}.
\end{proof}
  
\section{Queue embedding}
\label{sec:queue}
In this section we prove Lemma~\ref{lem:queue}. The idea is as follows. For each $i\in[r]$ in succession, we embed $\Xq_i$ into $\Vq_i$, maintaining a good partial embedding. The way we do this is as follows. We need to embed each $x\in\Xq_i$ into $\Cq(x)\setminus B(x)$, where $B(x)$ is the set of bad vertices for $x$ with respect to the current good partial embedding and $Q=V(H)$. (This choice of $Q$ is 
made only so that we are able to apply Lemma~\ref{lem:fewbad} without having to verify the at this time pointless badness condition~\eqref{eq:bad_vertices}.) 
We therefore try to find a system of distinct representatives $v_x\in
\Cq(x)\setminus B(x)$ for each $x\in \Xq_i$, which we do by verifying
Hall's condition. In turn, we verify Hall's condition by showing that its
failure implies the existence of a dense spot in $G$, violating
$\CON(\rho,r_1,\Delta)$, which $\Gamma$ satisfies. At this point one might
be concerned that embedding $x'\in\Xq_i$ to $v_{x'}$ could change $B(x)$.
Observe, however, that since $x$ and $x'$ are at distance at least ten in $H$
by~\ref{PtH:dist} this does not occur.

\begin{proof}[Proof of Lemma~\ref{lem:queue}]
 We require
 \[\mu\le\frac16\,,\qquad\rho\le\frac{\mu\zeta d^\Delta}{200\kappa\Delta}\,, \qquad \eps'\le \frac{\mu\zeta d^\Delta}{1000\kappa4^\Delta\Delta^2}\qquad\text{and}\qquad\eps\le\frac{\eps'}{\kappa d}\,.\] 
 Let $\psi_0$ be a good partial embedding whose image is disjoint from $\Vq$, and suppose that for each $i\in[r]$ the set $\Xq_i$ has size at most $2\rho|X_i|$. We now define a sequence of good partial embeddings $\psi_1,\ldots,\psi_r$ with $\dom(\psi_r)=\dom(\psi_0)\cup\bigcup_{t=1}^r\Xq_t$ and $\psi_t(x)\in \Vq_t$ for each $x\in\Xq_t$ and $1\le t\le r$. We let $C_t(x)$ be the candidate set of $x$ with respect to $\psi_t$, and so on, and let $B_t(x)$ be the set of bad vertices for $x$ with respect to $\psi_t$ and $Q=V(H)$.
 
 Suppose that for some $1\le t\le r$ we have constructed $\psi_{t-1}$ as above. We let $(v_x)_{x\in\Xq_t}$ be a system of distinct representatives for the sets $\big(\Cq_{t-1}(x)\setminus B_{t-1}(x)\big)_{x\in\Xq_t}$, and we set $\psi_t=\psi_{t-1}\cup\{x\to v_x\colon x\in\Xq_t\}$. We need to prove that this system of distinct representatives exists, and that the resulting $\psi_t$ is a good partial embedding.
 
 To see that the system of distinct representatives exists, we verify Hall's condition. Let $X\subset\Xq_i$ be non-empty, and let $W=\bigcup_{x\in X} \Cq_{t-1}(x)\setminus B_{t-1}(x)$. Then we need to show $|W|\ge|X|$. Assume for a contradiction that $|W|<|X|$ holds. By averaging, there is $b\in\{0,\dots,\Delta\}$ such that we find a subset $X_b$ of (not necessarily all) vertices $x$ in $X$ with $\pi_{t-1}^*(x)=b$ of size exactly $\tfrac{1}{\Delta+1}|W|$. Now each $x\in X_b$ has $\Cq_{t-1}(x)\ge (1-\eps')\mu(dp-\eps'p)^{b-|J_x|}|I_x|$ by~\ref{GPE:sizeC}, and $|I_x|\ge\zeta(dp-\eps'p)^{|J_x|}|V_i|$ by~\ref{G:restr}. We conclude, by Lemma~\ref{lem:fewbad}, that
 \[\big|C_{t-1}(x)\setminus B_{t-1}(x)\big|\ge \tfrac12\mu\zeta(dp)^b|V_i|-20\Delta^2\eps'p^b|V_i|\ge\tfrac14\mu\zeta(dp)^b|V_i|\]
 by choice of $\eps'$. In particular, we see that for each $x\in X_b$ we have $\big|U_{t-1}(x)\cap W\big|\ge\tfrac14\mu\zeta(dp)^b|V_i|$. We therefore have
 \[
 \sum_{x\in X_b}\big|U_{t-1}(x)\cap W\big|\ge|X_b|\cdot \tfrac14\mu\zeta(dp)^b|V_i|\,.
 \]
Since we would like to resort to the congestion property $\CON(\rho,r_1,\Delta)$ we eventually need to pass to a subset $W'$ of $W$ of size  $|X_b|$. Picking uniformly at random a subset $W'$ of $W$ of size $|X_b|=\tfrac{|W|}{\Delta+1}$, we 
 have
 \[
  \Exp_{W'}\sum_{x\in X_b}\big|U_{t-1}(x)\cap W'\big|\ge\tfrac{|X_b|}{\Delta+1}\cdot \tfrac14\mu\zeta(dp)^b|V_i|\,,
 \]
 and so in particular there is a subset $W'$ of $W$ of size $|X_b|$ such that
 \begin{equation}\label{eq:queue:sumU}
   \sum_{x\in X_b}\big|U_{t-1}(x)\cap W'\big|\ge
   \tfrac{|X_b|}{\Delta+1}\cdot \tfrac14\mu\zeta(dp)^b|V_i|\,.
\end{equation}
 
 Recall that $\Pi_{t-1}(x)$ is the set of embedded neighbours of~$x$ in $\psi_{t-1}$.
 We now apply $\CON(\rho,r_1,\Delta)$ with
 $\mathcal{F}=\big\{\psi_{t-1}\big(\Pi_{t-1}(x)\cup J_x\big)\colon x\in
 X_b\big\}$
 and the set $W'$. Note that
 $|W'|=|X_b|=|\mathcal{F}|\le2\rho|X_i|/(\Delta+1)\le\rho|V(\Gamma)|$
 because the sets $\psi_{t-1}\big(\Pi_{t-1}(x)\cup J_x\big)$ are disjoint, so that we
 can do this. We conclude that the congestion graph
 $\AG(\Gamma,W',\mathcal{F})$ satisfies
 \begin{equation}\label{eq:queue:AG}
   e\big(\AG(\Gamma,W',\mathcal{F})\big)\le7p^b|W'||\mathcal{F}|+\rho p^b
   n|\mathcal{F}|/r_1\,.
 \end{equation}
 But the edges of this congestion graph are precisely the pairs $F_xu$ 
 such that $F_x=\psi_{t-1}\big(\Pi_{t-1}(x)\cup J_x\big)$
 for $x\in X_b$ and $u\in U_{t-1}(x)\cap W'$, so
 combining~\eqref{eq:queue:sumU} and~\eqref{eq:queue:AG} we have 
 \[\tfrac{1}{4(\Delta+1)}\mu\zeta (dp)^b|X_b||V_i|\le 14\rho p^b|X_b||V_i|+\rho\kappa p^b |X_b||V_i|\,,\]
 where we use $|X_b|=|W'|=|\mathcal{F}|\le 2\rho|V_i|$ and $|V_i|\ge n/(\kappa r_1)$. This is a contradiction by choice of $\rho$. We conclude that the desired system of distinct representatives exists.
 
 Now we show that $\psi_t$ is a good partial embedding. Since the $v_x$ are
 distinct, $\psi_t$ is injective. Since $\Cq_{t-1}(x)\subset I_x$ for each
 $x$, we have~\ref{GPE:rightplace} for each $x\in\Xq_i$. If~\ref{GPE:sizeC}
 or~\ref{GPE:sizeU} were to fail for some $y$, then $y\in N_H(x)$ for some
 $x\in\Xq_i$. Since by~\ref{PtH:dist}, vertices of $\Xq_i$ are at distance
 at least $10$ in $H$, this $x$ is unique. But by definition of
 $B_{t-1}(x)$ the vertex $v_x$ is not bad for $x$ with respect to
 $\psi_{t-1}$, i.e.\ this case does not occur. Finally, if~\ref{GPE:Ureg}
 fails for some $yz\in E(H)$, then again at least one of $y$ and $z$
 (possibly both) is a neighbour of some $x\in\Xq_i$. Again
 by~\ref{PtH:dist} this $x$ is unique, and again by definition of
 $B_{t-1}(x)$ the vertex~$v_x$ is not bad for $x$ with respect to $\psi_{t-1}$, so this case too does not occur. Thus $\psi_t$ is a good partial embedding as desired.
 
 By induction on $t$, the final $\psiq:=\psi_r$ is a good partial embedding satisfying the conclusion of Lemma~\ref{lem:queue}.
\end{proof}
  
  \section{Fixing buffer defects}
  \label{sec:fixbuffer}

  To prove Lemma~\ref{lem:deletebad} we need to embed the reserved
  cliques. The basic idea is as follows. If $\Xbuf_i$ is a clique buffer,
  and a vertex $v$ of $V_i$ is a candidate for too few vertices in
  $\Xbuf_i$ and not in the image of $\psi$ (we call such a vertex poor),
  then we will embed a reserved clique $K\in\cK_i$ on~$v$ and some further vertices
  of $\Vc$. Observe that, since~$K$ is a $(\Delta+1)$-clique, none of the
  vertices of~$K$ have embedded neighbours, hence~$v$ is in the available
  candidate set of the vertex of~$K$ in~$X_i$.

  The only difficulty now is that some of the poor vertices~$v$ may
  lie in $\Vbuf_i$, and we risk destroying the hard-earned property that
  every vertex in~$\Xbuf_i$ has many candidates. In order to deal with this,
  we will need to embed some vertices from~$\Xbuf_i$ as well, at which point
  further vertices of~$V_i$ may become poor and require embedding, and so
  on.

  Before we prove Lemma~\ref{lem:deletebad} we state two auxiliary lemmas. The
  first justifies that the process just described terminates without eating up too many
  vertices, that is, that we can find for each $i$ small subsets $P_i$ of $V_i$ and
  $D_i$ of $\Xbuf_i$ such that every vertex $x\in\Xbuf_i\setminus D_i$ has many candidates in $\Vbuf_i\setminus P_i$, and every vertex of
  $V_i\setminus P_i$ is a candidate for many vertices of $\Xbuf_i\setminus D_i$.
  \begin{lemma}\label{lem:identifybad}
    We assume the General Setup. Suppose that $\Gamma$ has
    $\CON(\rho,r_1,\Delta)$ and let~$i$ be fixed.  Suppose $\psi$ is a good
    partial embedding whose domain contains $N(\Xbuf)$, and that~$\Xbuf_i$
    is a clique buffer. Suppose furthermore that all but at most $\rho
    |V_i|$ vertices of $V_i$ are candidates for at least
    $\mu(dp)^\Delta|X_i|/2$ vertices of $\Xbuf_i$. Then there are subsets
    $P_i$ of $V_i$ and $D_i$ of $\Xbuf_i$ with the following properties.
    \begin{enumerate}[label=\itmarab{PD}]
      \item\label{identifybad:IB:small} We have $|P_i|<2\rho|V_i|$ and
       $|D_i|<2\rho|X_i|$.
      \item\label{identifybad:IB:goodX} Each vertex of $\Xbuf_i\setminus D_i$
       has at least $\mu(dp)^\Delta|V_i|/4$ candidates in $\Vbuf_i\setminus
       P_i$.
      \item\label{identifybad:IB:goodV} Each vertex of $V_i\setminus P_i$ is a
       candidate for at least $\mu(d p/100)^{\Delta}|X_i|$ of the vertices in
       $\Xbuf_i\setminus D_i$.
      \item\label{identifybad:IB:goodD} Each vertex in $D_i$ has at least
       $\mu(dp)^\Delta|V_i|/4$ candidates in $P_i\cap\Vbuf_i$.
    \end{enumerate}
  \end{lemma}
  
  The conclusion~\ref{identifybad:IB:goodD} will be used later to show that we can embed $D_i$ into $P_i\cap\Vbuf_i$.  
  
  \begin{proof}[Proof of Lemma~\ref{lem:identifybad}]
    We require
    \[\rho<\frac{\mu d^\Delta}{250\kappa}\,.\]
    We start with $D_i=\emptyset$ and $P_i$ being the set of vertices which
    fail~\ref{identifybad:IB:goodV}.  We sequentially add vertices to $D_i$
    and $P_i$ which fail~\ref{identifybad:IB:goodX}
    and~\ref{identifybad:IB:goodV} respectively (note that the property of
    failing either condition is monotone), until either there are no
    failing vertices to add to either set or one of the two sets reaches
    $2\rho|V_i|$ vertices. Observe that any vertex $x$ of $D_i$ has
    by~\ref{GPE:sizeC} at least $(1-\eps')\mu(d-\eps')^\Delta p^\Delta
    |V_i|$ candidates in $\Vbuf_i$. Since $x$ must have
    violated~\ref{identifybad:IB:goodX} at some point, by choice of $\eps'$
    it therefore has at least $\mu(dp)^\Delta|V_i|/4$ candidates in
    $P_i\cap\Vbuf_i$, establishing~\ref{identifybad:IB:goodD}.

    In the process of building $D_i$ and $P_i$ we call a vertex $v\in V_i\setminus P_i$ 
    \emph{poor} if $v$ is a candidate for less than $\mu(d p/100)^{\Delta}|X_i|$ vertices of $\Xbuf_i\setminus D_i$. 
    If $P_i$ reaches $2\rho|V_i|$ vertices first, let $P=P_i$
    and let $D$ a superset of $D_i$, both of size exactly $2\rho|V_i|$. Since, 
    by the assumption of the lemma,  at most
    $\rho|V_i|$ vertices in $V_i$ are poor with respect to $\Xbuf_i$, at 
    least $\rho|V_i|$ of the vertices in $P$ were added to $P_i$ because
    they became poor, that is, they are candidate for at least
    $\mu(dp)^\Delta|X_i|/2$ vertices of $\Xbuf_i$, but for less than
    $\mu(dp/100)^\Delta|X_i|$ of the vertices $\Xbuf_i\setminus D_i$. It
    follows that each is candidate for at least $\mu(dp)^\Delta|X_i|/4$
    vertices of $D_i$, and so there are at least
    $\rho|V_i|\mu(dp)^\Delta|X_i|/4$ edges in the candidate graph between $P$
    and $D$. Now
    \begin{align*}
      \tfrac{1}{4}\rho\mu(dp)^\Delta|X_i||V_i|&=
      \tfrac{1}{8}\rho\mu(dp)^\Delta|X_i||V_i|+\tfrac{1}{8}\rho\mu(dp)^\Delta|X_i||V_i|\\
      &> \frac{\mu d^\Delta}{32\rho}p^\Delta|P||D| +\frac{\mu
      d^\Delta}{16 \kappa r_1} p^\Delta n|D|\\
      &> 7p^\Delta|P||D|+\rho p^\Delta n|D|/r_1\,,
    \end{align*}
    where the first inequality comes from the sizes of $P$ and $D$, and the
    second from the choice of $\rho$ sufficiently small. The last line is in
    contradiction to the congestion condition $\CON(\rho,r_1,\Delta)$, since the candidate graph 
    between $P$ and $D$ is a subgraph of $\AG(\Gamma,P,\cF_D)$ with
    $\cF_D:=\{\psi(N_H(x))\colon x\in D\}$.
    
    If on the other hand~$D_i$ reaches
    $2\rho|V_i|$ vertices first, we define similarly $P$ a superset of $P_i$ and $D=D_i$, both of size $2\rho|V_i|$.  Each
    vertex of $D$ has at least $\mu(dp)^\Delta|V_i|/4$ candidates in $P$
    by~\ref{identifybad:IB:goodD}, and thus the candidate graph between $P$ and
    $D$ has at least $2\rho|X_i|\mu(dp)^\Delta|V_i|/4$ edges, which is larger
    than in the previous case and so also gives a contradiction to
    $\CON(\rho,r_1,\Delta)$.

    We conclude that the process terminates for lack of failing vertices, which establishes~\ref{identifybad:IB:small}.
  \end{proof}

  The second lemma simply gives a condition on a vertex $v\in V_i$ under which
  we can find a clique $K_{\Delta+1}$ containing $v$ in $G$ whose remaining
  vertices lie in $\Vc\setminus\im(\psi)$.

  \begin{lemma}\label{lem:findclique}
    We assume the General Setup. Suppose that $\Gamma$ has $\NS(\eps,r_1,\Delta)$,
    $\RI(\eps,(\eps_{a,b}),\eps',d,r_1,\Delta)$ and $\CON(\rho,r_1,\Delta)$. Let
    $i,j_1,\ldots,j_\Delta$ form a copy of $K_{\Delta+1}$ in $R'$ and $v$ be a
    vertex of $V_i$. Suppose that $\psi$ is any good partial embedding
    which embeds no vertex on~$v$, and is such that $v$ has at least
    \begin{equation}\label{eq:findclique:deg}
     \tfrac{\mu d}{8}\max\big(p|V_{j_k}|,\deg_\Gamma(v;V_{j_k})\big)
    \end{equation}
    neighbours in
    $\Vc_{j_k}\setminus\im(\psi)$ in $G$ for each $1\le k\le\Delta$.
    Suppose that there is a copy of $K_{\Delta+1}$ in $H$ using a vertex of
    $\tilde{X}_i$ and a vertex of each of $X_{j_1},\ldots,X_{j_\Delta}$. Then there are vertices $v_{k}\in\Vc_{j_k}\setminus\im(\psi)$ for $1\le k\le\Delta$, 
    such that $v,v_1,\ldots,v_\Delta$ form a copy of $K_{\Delta+1}$ in $G$.
  \end{lemma}
  \begin{proof}
    We require
    \[\eps<\eps'<\frac{(d/32)^\Delta\mu}{2\kappa\Delta^2}\,.\]
    We choose the vertices $v_1,\ldots,v_\Delta$ in succession. When we choose
    $v_s$, if $s<\Delta$ we require that for each $s+1\le k\le\Delta$ we have
    \begin{equation}\label{eq:degree_reg_Kc}
    |\comN_\Gamma(v,v_1,\ldots,v_s)\cap V_{j_k}|= (p\pm\eps p)^s
    \deg_\Gamma(v;V_{j_k})\text{ and }
    \end{equation}
    \begin{equation}\label{eq:degree_two_Kc}
    |\comN_G(v,v_1,\ldots,v_s)\cap \left(\Vc_{j_k}\setminus \im(\psi)\right)|\ge
    \big(\tfrac{dp}{4}\big)^{s}\tfrac{\mu
    d}{8}\max\big(p|V_{j_k}|,\deg_\Gamma(v;V_{j_k})\big)\,.
    \end{equation}
     If $s<\Delta-1$ we further require that for
    each $s+1\le k<k'\le\Delta$ the pair 
    \begin{equation}\label{eq:RI_Kc}
    \big(\comN_\Gamma(v,v_1,\ldots,v_s)\cap
    V_{j_k},\comN_\Gamma(v,v_1,\ldots,v_s)\cap V_{j_{k'}}\big)
    \end{equation}
     is $(\eps_{s,s},d,p)$-regular.
    
    Note that these conditions are satisfied, with $s=0$, before we have chosen
    any vertices. Indeed, the first is a tautology, the second is the assumption on $v$
    in the lemma statement, and the third is a statement that $V_i$ has
    two-sided inheritance with respect to $V_{j_k}$ and $V_{j_{k'}}$, which
    holds since we assumed there is a triangle of $X$ using one vertex of each
    of $\tX_i$, $X_{j_k}$ and $X_{j_{k'}}$ and
    by~\ref{H:partition} and~\ref{G:inh}.    
    
    Suppose we have chosen vertices $v_1$,\ldots, $v_{s-1}$ so far 
    and that we have $s\le \Delta-1$. 
    Because $\Gamma$ has $\NS(\eps,r_1,\Delta)$, at
    step $s$ the number of vertices $w$ failing condition~\eqref{eq:degree_reg_Kc} with $v_s=w$, is at most 
    \[\Delta\eps
    p^{\Delta-1}|V(\Gamma)|/r_1^2<\big(\tfrac{dp}{4}\big)^{s-1}\tfrac{\mu
    d}{32}\max\big(p|V_{j_k}|,\deg_\Gamma(v;V_{j_k})\big)\,.\] By
    $(\eps_{s-1,s-1},d,p)$-regularity of the pair
    \[\big(\comN_\Gamma(v,v_1,\ldots,v_{s-1})\cap
    V_{j_k},\comN_\Gamma(v,v_1,\ldots,v_{s-1})\cap V_{j_{k'}}\big)\,,\] 
    and using the upper bound from~\eqref{eq:degree_reg_Kc} on 
    $|\comN_\Gamma(v,v_1,\ldots,v_{s-1})\cap V_{j_k}|$, at most
    \[\Delta\eps'(p+\eps
    p)^{s-1}\deg_\Gamma(v;V_{j_k})<\big(\tfrac{dp}{4}\big)^{s-1}\tfrac{\mu
    d}{32}\max\big(p|V_{j_k}|,\deg_\Gamma(v;V_{j_k})\big)\] vertices 
    $w\in \Vc_{j_s}\setminus \im(\psi)$ fail condition~\eqref{eq:degree_two_Kc} with $v_s=w$.
    Finally, if $s\le\Delta-2$ and since $\Gamma$ has $\RI(\eps,(\eps_{a,b}),\eps',d,r_1,\Delta)$, at most
    \[\Delta^2\eps
    p^{\Delta-2}|V(\Gamma)|/r_1^2<\big(\tfrac{dp}{4}\big)^{s-1}\tfrac{\mu
    d}{32}\max\big(p|V_{j_k}|,\deg_\Gamma(v;V_{j_k})\big)\]
    vertices fail condition~\eqref{eq:RI_Kc}. Thus there are at least
    \[\big(\tfrac{dp}{4}\big)^{s-1}\tfrac{\mu
    d}{32}\max\big(p|V_{j_k}|,\deg_\Gamma(v;V_{j_k})\big)>0\] vertices 
    from $\comN_G(v,v_1,\ldots,v_{s-1})\cap \left(\Vc_{j_s}\setminus \im(\psi)\right)$ which
    satisfy all three conditions, and we can choose $v_s$ for each $1\le s\le\Delta$ as desired. 
    
    Furthermore observe that for $s=\Delta$ it suffices to choose 
    an arbitrary vertex from $\comN_G(v,v_1,\ldots,v_{\Delta-1})\cap \left(\Vc_{j_\Delta}\setminus \im(\psi)\right)$. 
    The lemma follows.
  \end{proof}

  We can now prove Lemma~\ref{lem:deletebad}. We construct the desired
  embedding in two steps. First, we embed each $D_i$ into $P_i\cap\Vbuf_i$,
  which by verifying Hall's condition. Let~$P'_i$ be those vertices
  of~$P_i$ not covered in this process. It then remains to cover all~$P'_i$
  using the reserved cliques, and embed any left-over reserved cliques
  at the end of this process. We do this sequentially by applying
  Lemma~\ref{lem:findclique} to find a destination for a reserved clique,
  either letting~$v$ be the next poor vertex or, if all of these are used, 
  letting~$v$ be some vertex of $\Vc$. Note that in this step, we
  need to maintain the degree condition~\eqref{eq:findclique:deg} of
  Lemma~\ref{lem:findclique}. We have this condition initially
  by~\ref{G:deg}, and we will see that to maintain it, it suffices to
  choose at each step a most dangerous vertex $v$, that is, one
  minimising the parameter $\mindeg$ we define in the proof.

  \begin{proof}[Proof of Lemma~\ref{lem:deletebad}]
    We require
    \[\rho<\min\Big(\frac{\mu
    d^\Delta}{250\kappa},\frac{\mu d}{64\kappa(\DeltaRp+1)}\Big) \quad\text{ and  }
    \eps<\eps'<\frac{d^\Delta}{32^\Delta\kappa\Delta\DeltaRp}\,,\]
    and assume that $\Gamma$ has $\NS(\eps,r_1,\Delta)$, $\RI(\eps,(\eps_{a,b}),\eps',d,r_1,\Delta)$ and
    $\CON(\rho,r_1,\Delta)$. 
   
    For each $i$ such that $\Xbuf_i$ is a clique buffer, we let $P_i$ and $D_i$
    be as given by Lemma~\ref{lem:identifybad}.
    If $\Xbuf_i$ is not a clique buffer, then we let $P_i=D_i=\emptyset$.

    Now for each $i$ in succession, we find a matching in the candidate graph
    between $D_i$ and $P_i\cap\Vbuf_i$. To do this we simply verify Hall's
    condition, using the congestion condition and~\ref{identifybad:IB:goodD}.
    Specifically, suppose we have a non-empty $W\subseteq D_i$, and $Z\subseteq
    P_i$ is the set of vertices in $P_i$ which are candidate for some $w\in W$.
    We need to verify that $|W|\le|Z|$. If this were false, then we could let
    $Z'$ be a superset of $Z$ of size $|W|$, and the number of edges in the
    candidate graph between $W$ and $Z'$ is by~\ref{identifybad:IB:goodD} at
    least $\mu(dp)^\Delta|V_i||W|/4$. Notice that $|D_i|$, $|P_i|<2\rho|V_i|$ 
    by~\ref{identifybad:IB:small}. Hence we have 
    \begin{align*}
      \tfrac{1}{4}\mu(dp)^\Delta|V_i||W|&=
      \tfrac{1}{8}\mu(dp)^\Delta|V_i||W|+\tfrac{1}{8}\mu(dp)^\Delta|V_i||W|\\
      &> \frac{\mu d^\Delta}{16\rho}p^\Delta|Z'||W| +\frac{\mu
      d^\Delta}{8 \kappa r_1} p^\Delta n|W|\\
      &> 7p^\Delta|Z'||W|+\rho p^\Delta n|W|/r_1\,,
    \end{align*}
    which is in contradiction to $\CON(\rho,r_1,\Delta)$, applied with $Z'$
    and the family $\cF=\{\psi(N_H(w)) \colon w\in W\}$ to the congestion
    graph $\AG(\Gamma,Z',\cF)$. We conclude that Hall's condition holds and
    hence the desired matching exists. We let $\psi'$ be the embedding
    obtained from~$\psi$ by embedding each vertex of each~$D_i$ to its
    matching partner in $P_i\cap\Vbuf_i$, and let $P'_i\subset P_i$ be the
    set of those vertices in~$P_i$ not embedded by~$\psi'$.  Observe that
    $\psi'$ is a good partial embedding since all neighbours of the buffer
    vertices are already embedded under $\psi$, and
    thus~\ref{GPE:rightplace}--\ref{GPE:Ureg} trivially hold for $\psi'$.

    As second step, we provide an algorithm, Algorithm~\ref{alg:delbad}, to embed the reserved cliques
    in~$\cK_i$ for each $i\in[r_1]$.  Given a vertex $v\in V_i$ and a
    partial embedding $\psi$ we define the parameter
    \[\mindeg_{\psi}(v):=
    \min_{j:ij\in R'}\frac{\deg_G\big(v;\Vc_j\setminus\im(\psi)\big)}
                                  {\max\big(p|V_j|,\deg_\Gamma(v;V_j)\big)}\,,\]
    which is used in Algorithm~\ref{alg:delbad}. 
        Observe that  
    we have initially $\mindeg_{\psi}(v)\ge \mu (d-\eps)/2$ for 
    all $v\in V_i$ with $\deg_{R'}(i)>0$,  which follows from~\ref{G:deg}.
    We claim that the algorithm runs correctly and the finally returned
    $\psigood$ is the desired good partial embedding. It remains to justify this
    claim. 

    \begin{algorithm}[t]
      \caption{Removing poor vertices}\label{alg:delbad}
      $t:=0$\;
      $\psi_0:=\psi'$\;
      \While{$\bigcup_i \cK_i\neq\emptyset$}{
        $I:=\{i\colon\text{there is an unembedded clique in }\cK_i\}$\;
        \eIf{$\bigcup_i P'_i\setminus\im(\psi_t)\neq\emptyset$}{
          choose $v\in \bigcup_i P'_i\setminus\im(\psi_t)$ minimising
          $\mindeg_{\psi_t}(v)$ \;
        }{
          choose $v\in \bigcup_{i\in I} \Vc_i\setminus\im(\psi_t)$ minimising $\mindeg_{\psi_t}(v)$ \;
        }
        set $i$ such that $v\in V_i$ \;
        choose an unembedded clique $K_H$ in $\cK_i$ \;
        choose a clique $K_G$ containing $v$ and one vertex
        from each $\Vc_j\setminus\im(\psi_t)$ such that $X_j$ intersects $K_H$ and $j\neq i$\;
        set $\psi_{t+1}:=\psi_t\cup\{K_H\to K_G\}$\; 
        $t:=t+1$\;
      }
     return $\psigood:=\psi_t$ \;
    \end{algorithm}
   
    Roughly speaking 
    each of the vertices from $P'_i$ can serve us as an image for each of the 
    $2\rho|X_i|$ cliques from $\cK_i$. 
    During the first  phase of the algorithm, that is, while the
    if-condition is true, we embed reserved cliques in these vertices,
    always choosing a vertex~$v$ minimising $\mindeg_\psi(v)$, which will guarantee
    successful completion. 
    Note that $|P'_i|<2\rho|V_i|$ holds by~\ref{identifybad:IB:small} and $|\cK_i|=2\rho|X_i|$, so 
    if the algorithm doesn't fail then all the vertices from 
    $\bigcup_i P'_i$ will be used. In the second phase, when the
    if-condition becomes false, we have to take 
    care of the remaining unembedded cliques in $\bigcup_i\cK_i$.
    Since the
    algorithm embeds only onto vertices $\bigcup P'_i\cup \Vc_i$, if it
    completes successfully, then the desired
    properties~\ref{FIN:bigA} and~\ref{FIN:ManyCand} are
    guaranteed for $\psigood$
    by~\ref{identifybad:IB:goodX}  and~\ref{identifybad:IB:goodV}.

    Therefore, it remains to show that the algorithm succeeds, that is,
    that there is no 
    failure while choosing  a clique $K_G$ in some iteration. 
    By Lemma~\ref{lem:findclique}, this is
    guaranteed if we can show that at each time $t$ the vertex $v\in V_i$ chosen
    has at least 
    \begin{equation}\label{eq:delbad:fail}
      \tfrac{\mu
        d}{8}\max\big(p|V_{j}|,\deg_\Gamma(v;V_{j})\big)
    \end{equation}
    neighbours in each set $\Vc_j\setminus \im(\psi_t)$ such that $ij\in
    R'$.
    
    By the assumption of the lemma, for each $i$,  the set $\Xc_i$ of vertices 
    $x\in X_i$ contained in some reserved clique from $\bigcup_{j}\cK_j$ satisfies 
    $|\Xc_i|\le 2\kappa(\DeltaRp+1)\rho|X_i|$. 
    It follows that for
    any $t$ and $j$ we have 
    \begin{equation}\label{eq:delbad:psi}
      |\im(\psi_t)\cap\Vc_j|\le 2\kappa(\DeltaRp+1)\rho|X_j|\,.
    \end{equation}

    \begin{claim}\label{cl:delbad:mindeg}
      For any $t$ at most $\DeltaRp\eps p^{\Delta-1}|V(\Gamma)|/r_1^2$
      vertices $v\in V_i$ satisfy \[\mindeg_{\psi_t}(v)<3\mu d/16\,.\]
    \end{claim}
    \begin{claimproof}
      Each vertex $v$ of $V_i$ has at least
      \[\mu (d-\eps)\max\big(p|V_j|,\deg_\Gamma(v;V_j)/2\big)>\tfrac{\mu
        d}{4}\max\big(p|V_j|,\deg_\Gamma(v;V_j)\big)\] neighbours in
      $\Vc_j$ by~\ref{G:deg}. 
      Moreover, it is a direct consequence of the
      neighbourhood size property $\NS(\eps,r_1,\Delta)$ of $\Gamma$ and~\eqref{eq:delbad:psi}
      that the number of vertices in
      $\Gamma$ which have more than $4\kappa(\DeltaRp+1)\rho p|X_j|$
      neighbours in $\im(\psi_t)\cap\Vc_j$ is at most $\eps
      p^{\Delta-1}|V(\Gamma)|/r_1^2$.
      Since $4\kappa(\DeltaRp+1)\rho p|X_j|<\mu
      dp|V_j|/16$, it follows that for any $t$ at most
      $\DeltaRp\eps p^{\Delta-1}|V(\Gamma)|/r_1^2$ vertices in $V_i$ have
      fewer than $\tfrac{3\mu
        d}{16}\max\big(p|V_j|,\deg_\Gamma(v;V_j)\big)$ neighbours in any
      set $\Vc_j\setminus\im(\psi_t)$ such that $ij\in R'$, that is, they have
      $\mindeg_{\psi_t}(v)<3\mu d/16$.
    \end{claimproof}

    In the remainder we show that if Algorithm~\ref{alg:delbad} chooses a vertex~$v$
    such that~\eqref{eq:delbad:fail} fails, then we obtain a contradiction to
    Claim~\ref{cl:delbad:mindeg}.
    Indeed, in each iteration at most one vertex is embedded in any given $\Vc_j$. So if
    there is a time~$t$ at which a vertex $v\in V_i$ is chosen with fewer 
    neighbours
    than the bound in~\eqref{eq:delbad:fail}
    in some set $\Vc_j\setminus\im(\psi_t)$ such that $ij\in R'$, then at each
    of the preceding $\mu dp|V_j|/16$ times $t'$ the vertex~$v$ had
    $\mindeg_{\psi_{t'}}(v)<3\mu d/16$. Hence each of the at least $\mu d p
    n/(16\kappa r_1)$ vertices $v'$ chosen at these times has
    \[\mindeg_{\psi_t}(v')\le\mindeg_{\psi_{t'}}(v')<3\mu d/16\,.\] 
    Since
    there are at most $r_1$ clusters $V_i$ in $G$, and since \[\mu d p
      n/(16\kappa r_1)>r_1\DeltaRp\eps p^{\Delta-1}|V(\Gamma)|/r_1^2\,,\] this 
    contradicts Claim~\ref{cl:delbad:mindeg}.
 \end{proof}
  
\chapter{Proof of the blow-up lemma for bijumbled graphs}
\label{chap:bij}

\section{The RGA lemma and the proof of the blow-up lemma}
\label{sec:jumbled}
In this section we prove Lemma~\ref{lem:psr_main}, conditional on the random greedy lemma (Lemma~\ref{lem:rga:psr}) which we prove in Section~\ref{sec:pseudoRGA}.

The proof is quite similar to that of Lemma~\ref{lem:rg_image}, and indeed we make use of the same General Setup. The difference is that we can no longer use property $\CON(\rho,r_1,\Delta)$ as
 a bijumbled graph $\Gamma$ need not satisfy it (see the discussion in Section~\ref{sec:pseudo}). Instead, we `replace' it with the lopsided neighbourhood size property $\LNS(\eps,r_1,\Delta)$.

In the next section we will use a somewhat different random greedy strategy than Algorithm~\ref{alg:RGA} to prove the following lemma. However, we will re-use the auxiliary lemmas proved in Section~\ref{sec:RGAlemmas} in its proof (recall that these lemmas are not dependent on Algorithm~\ref{alg:RGA}).

\begin{lemma}[bijumbled RGA lemma]
\label{lem:rga:psr}
  We assume the General Setup. Suppose further that $\Gamma$ has properties $\NS(\eps,r_1,\Delta+1)$, $\RI(\eps,(\eps_{a,b}),\eps',d,r_1,\Delta+1)$ and $\LNS(\eps,r_1,\Delta)$, and that at most $\rho p^\Delta|X_i|$ vertices in each $X_i$ are image restricted. Then there is a good partial embedding $\psiRGA$ such that the
  following hold for each~$i$. Let $b$ be such that $\Xbuf_i$ is a degree-$b$
  buffer.
  \begin{enumerate}[label=\itmarab{PRGA}]
    \item\label{rga:psr:place} Every vertex in
    $\Xmain$ is embedded to
    $\Vmain\cup \Vq$ by $\psiRGA$, and no vertex in $\Xbuf$ is embedded.
    \item\label{rga:psr:nobad} Every vertex in $V_i$ is a
    candidate for at least $\mu(d^\Delta p/100)^b|X_i|$ vertices of $\Xbuf_i$.
    \item\label{rga:psr:main} For every set $W\subset V_i$ of size at least
    $\rho|V_i|$, there are at most $\rho|X_i|$
    vertices in $\Xbuf_i$ with fewer than $\tfrac12(dp)^b|W|$ candidates in $W$.
    \item\label{rga:psr:Xsum} For all $x\in \Xbuf_i$, letting
    \[\qquad Y:=\big\{y\in\Xbuf_i:\big|\Ubuf(y)\cap \Cbuf(x)\big|>(p+\eps p)^{b}|\Cbuf(x)|\big\}\,,\]
    we have
    \begin{equation*}
     \sum_{y\in Y} \frac{\big|\Ubuf(y)\cap \Cbuf(x)\big|}{\big|\Cbuf(y)\big|}
     \le\frac{|\Cbuf(x)|}{20}\,.
    \end{equation*} 
    \item\label{rga:psr:Vsum} 
    For each $i$, for all but at most $\eps' p^\Delta |V_i|$ vertices $v\in V_i$, there are at most $\eps'p^\Delta|V_i|$ vertices $v'\in V_i$ such that
    \[\big|\{x\in\Xbuf_i:v,v'\in C(x)\}\big|>24\mu\Delta^2\big(20\mu^{-1}\zeta^{-1}d^{-\Delta}\big)^b p^{2b}|X_i|\,.\]
  \end{enumerate}
\end{lemma}

Comparing this lemma to Lemma~\ref{lem:rga} (RGA lemma), one notices that we demand more of the $\NS$ and $\RI$ pseudorandomness properties. The reason for doing this is that we can then make use of Lemmas~\ref{lem:large_nbs} and~\ref{lem:nonclique_buffer} for all buffer vertices, whether or not they are in copies of $K_{\Delta+1}$, and thus there is no exception for clique buffers in~\ref{rga:psr:nobad}. This simplifies our proof substantially, as we no longer need to fix buffer defects. We will explain later why this does not harm our eventual bijumbledness requirement.

Property~\ref{rga:psr:main} corresponds to~\ref{rga:emb:main}. Further, properties~\ref{rga:psr:Xsum} and~\ref{rga:psr:Vsum}
form the `no dense spots' property we require in order to
complete the embedding. In Lemma~\ref{lem:rga} no such property appears, but in random graphs the $\CON$ property guarantees that no dense spots exist. Somewhat
informally,~\ref{rga:psr:Xsum} asserts that typically candidate sets
intersect as if they are random sets, while
property~\ref{rga:psr:Vsum} states the same about intersections of the
neighbourhoods of vertices from $V(G)$ in the candidate graph.

 Observe that, in comparison to Lemma~\ref{lem:rga}, here we ask for an additional pseudorandomness property
$\LNS(\eps,r_1,\Delta)$, which we require to establish~\ref{rga:psr:Xsum} and~\ref{rga:psr:Vsum} and in order to embed all
vertices.

\subsection{Outline of the proof of Lemma~\ref{lem:psr_main}}
The proof of Lemma~\ref{lem:psr_main} now looks quite similar to the proof of Lemma~\ref{lem:rg_image} (the blow-up lemma for random graphs). Again, we use Lemma~\ref{lem:matchreduce} to obtain the General Setup. However, this time we do not choose any reserved cliques, whether or not we have buffer vertices in cliques. We note that although we could state a version of Lemma~\ref{lem:rga:psr} which asked only for $\NS(\eps,r_1,\Delta)$ and $\RI(\eps,(\eps_{a,b}),\eps',d,r_1,\Delta)$ rather than with $\Delta+1$, and which made exceptions for clique buffers, this would not improve our eventual bijumbledness requirement on $\Gamma$, except for $\Delta=2$, since for $\Delta\ge 3$ the requirement is determined by $\LNS(\eps,r_1,\Delta)$. We apply Lemma~\ref{lem:rga:psr} to obtain a good partial embedding whose domain is $\Xmain$ with the additional properties stated there. We then have only to embed $\Xbuf$. Again, we do this one part at a time, by verifying Hall's condition, and again we separate the verification into small, large and medium-sized subsets of $\Xbuf_i$. Again, the `medium-sized' case is dealt with quickly using~\ref{rga:psr:main}.

However, now we approach the `small' and `large' cases differently, since we no longer have access to the $\CON$ property. In the `small' case (see Claim~\ref{cl:psr:small}), our approach is the following. Given $X\subset\Xbuf_i$, we construct an auxiliary embedding of $X$ into $\Vbuf_i$ such that each $x$ is embedded into $C(x)$. Although this auxiliary embedding is \emph{not} part of the embedding we will finally give, its existence verifies Hall's condition for $X$. The way we will construct the embedding is simply to embed $X$ vertex by vertex, at each step choosing for $x$ uniformly at random a so far unused vertex of $\Cbuf(x)$. Since $\psiRGA$ has~\ref{GPE:sizeC} we know $\Cbuf(x)$ is always reasonably large, and we will show that with high probability it is never more than half covered. Let us briefly explain how this works. We assume that no candidate sets are more than half covered before reaching $x$. When we embedded some vertex $y$, the probability of embedding it to $C(x)$ was at most
\begin{equation}\label{eq:probxsketch}
 \frac{\big|\Ubuf(y)\cap \Cbuf(x)\big|}{\tfrac12|\Cbuf(y)|}
\end{equation}
since (by assumption) we embed $y$ into a set of size at least $\tfrac12|\Cbuf(y)|$, of which certainly at most $\big|\Ubuf(y)\cap \Cbuf(x)\big|$ vertices are in $\Cbuf(x)$. But~\ref{rga:psr:Xsum} gives us an upper bound for the sum of these probabilities over all $y\in\Xbuf_i$ such that~\eqref{eq:probxsketch} is `large', and it is easy to account for the $y\in\Xbuf_i$ such that~\eqref{eq:probxsketch} is not large. We can thus apply Lemma~\ref{lem:coupling} to conclude that it is unlikely that $\Cbuf(x)$ was more than half covered, and by the union bound over all $x\in\Xbuf_i$ we conclude that with high probability we never fail.

For the `large' case (see Claim~\ref{cl:psr:large}) we use a similar argument, but using~\ref{rga:psr:nobad} and~\ref{rga:psr:Vsum} instead of~\ref{GPE:sizeC} and~\ref{rga:psr:Xsum}.

\subsection{Proof of Lemma~\ref{lem:psr_main}}
Now we give the full details of the proof outlined above.
\begin{proof}[Proof of Lemma~\ref{lem:psr_main}]
 First we choose constants as follows. Given $\Delta$, $\DeltaRpbl$, $\Delta_J$ integers, $\alphabl$, $\zetabl$ and $d>0$, and $\kappabl>1$, we set $\vartheta=\Delta$, $\DeltaRp=8(\Delta+\Delta_J)^{10}\DeltaRpbl$, $\alpha=\tfrac12\alphabl$, $\zeta=\tfrac12\zetabl$ and $\kappa=2\kappabl$. We now choose
 \begin{align*}
  \mu&= \frac{\alpha d^\Delta}{20000\kappa \DeltaRp^4\Delta^{10}(\Delta+2)},
  &\rho&= \tfrac{\mu^2\zeta^2}{2000\Delta^2}\big(\tfrac{\mu \zeta d^{4\Delta+2}}{10^{6\Delta}}\big)^\Delta\,,\\
   &\text{and}\quad&\eps'&=\tfrac{\mu\zeta d^{\Delta^2}\rho}{10^{6\Delta}\Delta^3\kappa4^\Delta}2^{-4/\rho}\,.
 \end{align*}
 Now for input $\Delta$, $d$ and $\eps'$, Lemma~\ref{lem:det_jumbled} returns constants $\eps_{a,b}$ and $\eps_{\sublem{lem:det_jumbled}}>0$.
 We set
 \[\eps=  \min\big(\tfrac{2^{10\Delta}d(\eps')^2}{\kappa\DeltaRp},\eps_{\sublem{lem:det_jumbled}}\big)\,.\]
   We let $\epsbl=\tfrac{1}{16}(\Delta+\Delta_J)^{-10}\eps$ and $\rhobl=\tfrac{1}{16}(\Delta+\Delta_J)^{-10}\rho$.
 
 Now Lemma~\ref{lem:psr_main} returns $\epsbl$ and $\rhobl$. Given $\ronebl$ we let $r_1=8(\Delta+\Delta_J)^{10}\ronebl$. We choose $c$ sufficiently small for Lemma~\ref{lem:det_jumbled} with input $\Delta$, $d$, $\eps'$ and $T=r_1$.
 
  We require
 \[p\ge \tfrac{10^6(\Delta+\Delta_J)^{10}\kappa r_1}{\mu\zeta\rho d^\Delta \eps^{-2}} \big(\tfrac{\log n}{n}\big)^{1/(2\Delta)}\,,\]
 where the condition on the constant in front of the $\big(\tfrac{\log n}{n}\big)^{1/(2\Delta)}$-term comes from Lemma~\ref{lem:matchreduce}, while the term $\big(\tfrac{\log n}{n}\big)^{1/(2\Delta)}$ is required for Lemma~\ref{lem:rga:psr}. 
 We remark, however, that this lower bound is not a restriction since a somewhat stronger lower bound  on $p$ of the form $\Omega(n^{-1/2(t-1)})$ for any $(p,cp^tn)$-bijumbled graph follows by considering the inequality~\eqref{eq:expander_mixing} which defines bijumbledness for, say, sets $X=\{v\}$ and $Y=V(\Gamma)\setminus(\{v\}\cup N_\Gamma(v))$.
 
 Let $\Gamma$ be an $n$-vertex, $(p,cp^{\max(4,3\Delta/2+1/2)}n)$-bijumbled graph. Then Lemma~\ref{lem:det_jumbled} states that $\Gamma$ has the three properties $\NS(\eps,r_1,\Delta+1)$, $\RI(\eps,(\eps_{a,b}),\eps',d,r_1,\Delta+1)$ and $\LNS(\eps,r_1,\Delta)$. From now on we will simply assume $\Gamma$ is an $n$-vertex graph which satisfies these three properties. 
 
 Given a graph $\Rbl$ on $\rbl\le\ronebl$ vertices, a spanning subgraph $\Rpbl$ with $\Delta(\Rpbl)\le\DeltaRpbl$, and graphs $H$ and $G\subset\Gamma$ with vertex partitions $\cXbl$ and $\cVbl$, a family of potential buffer vertices $\tcXbl$, and a $(\rhobl p^\Delta,\zetabl,\Delta,\Delta_J)$-restriction pair $\cIbl$, $\cJ$, suppose that the conditions of Lemma~\ref{lem:psr_main} are satisfied. Then the good partitions lemma, Lemma~\ref{lem:matchreduce}, (with $b_{\sublem{lem:matchreduce}}=\Delta$) gives a graph $R$ on $r\le r_1$ vertices, a spanning subgraph $R'$ with $\Delta(R')\le\DeltaRp$, and $\kappa$-balanced size-compatible partitions $\cX$ and $\cV$ of $H$ and $G$ respectively, each part having size at least $n/(\kappa r_1)$, together with a family $\tcX$ of potential buffer vertices and $\cI$ of image restrictions, subsets $\Xbuf_i$ of $\tX_i$ for each $i\in[r]$, and partitions $V_i=\Vmain_i\dcup\Vq_i\dcup\Vc_i\dcup\Vbuf_i$ for each $i\in[r]$ which satisfy the General Setup. 
 
 We let $\Xmain=V(H)\setminus\Xbuf$. We now begin the embedding of $H$ into $G$. By Lemma~\ref{lem:rga:psr}, there is a good partial embedding $\psiRGA$ of $\Xmain$ into
 $\Vmain\cup\Vq$ with properties~\ref{rga:psr:place}--\ref{rga:psr:Vsum} as stated in that lemma.
 
 Our aim now is to complete the embedding of $H$ into $G$ by finding for each
 $i$ a matching in the available candidate graph between $V'_i:=V_i\setminus\im(\psiRGA)$ and $\Xbuf_i$. Let us assume that $\Xbuf_i$ is a degree-$b$ buffer. To show there is a matching, we will verify Hall's condition, so let $Y$ be a non-empty subset of $\Xbuf_i$. Let $U$ be the set of vertices in $V'_i$ which are candidate for at least one vertex of $Y$, then our aim is to show $|U|\ge|Y|$. We split this into three cases, the harder two of which are done in the following two claims.

\begin{claim}\label{cl:psr:small} If $0<|Y|\le\rho|X_i|$, then we have $|U|\ge|Y|$.
\end{claim}
\begin{claimproof} 
 By~\ref{BUF:last} no $y\in Y$ is image restricted, so because $\psiRGA$ is a good partial embedding, by~\ref{GPE:sizeC} for each $y\in Y$ we have 
 \[|\Cbuf(y)|\ge(1-\eps')\mu(d-\eps')^bp^b|V_i|\ge\tfrac12\mu (dp)^b|V_i|\,,\]
  where the second inequality is by choice of $\eps'$. Let $y_1,\ldots,y_{|Y|}$ be an enumeration of $Y$.
 Now for each $j=1,\ldots,|Y|$ we choose $v_j$ uniformly at random from $\Cbuf(y_j)\setminus \{v_1,\ldots,v_{j-1}\}$ if this is possible; if not, we say $v_j$ does not exist. Our aim is to show that this does not occur, in other words that we obtain $|Y|$ distinct vertices of $|U|$, verifying the claim.
 
 Let $\hist_0$ be the empty history, and for each $1\le j\le |Y|$, let $\hist_j$ be the history of this process up to and including the choice of $v_j$. We claim that a.a.s.\
\begin{equation}\label{eq:psr:smallY} 
  \big|\Cbuf(y_j)\setminus\{v_1,\ldots,v_{j-1}\}\big|\ge\tfrac12\big|\Cbuf(y_j)\big|
\end{equation} 
  holds for each $j$. For a given $j$ we define random variables $Z^{(j)}_\ell$, $\ell\in[j-1]$, as follows.
  We set $Z^{(j)}_\ell=1$ if $v_\ell\in \Cbuf(y_j)$ and $|\Cbuf(y_\ell)\setminus\{v_1,\ldots, v_{\ell-1}\}|\ge \tfrac{1}{2}|\Cbuf(y_\ell)|$ hold, and
  $Z^{(j)}_\ell=0$ otherwise. For any history $\hist_{\ell-1}$ of this process we can bound the
   conditional expectation of $Z^{(j)}_\ell$ by
  \begin{multline}\label{eq:psr:expZ}
    \Exp(Z^{(j)}_\ell|\hist_{\ell-1})=\Pr(Z^{(j)}_\ell=1|\hist_{\ell-1})\\
    \le \frac{\big|\Cbuf(y_\ell)\cap \Cbuf(y_j)\setminus\{v_1,\ldots,v_{\ell-1}\}\big|}{\big|\Cbuf(y_\ell)\setminus\{v_1,\ldots,v_{\ell-1}\}\big|}\le\frac{\big|\Ubuf(y_\ell)\cap\Cbuf(y_j)\big|}{\tfrac12\big|\Cbuf(y_\ell)\big|}\,.
  \end{multline}
  
Observe that $\sum_{\ell=1}^{j-1}Z^{(j)}_\ell$ is exactly the number of those  vertices $v_k$ among $v_1$, \ldots, $v_{\ell-1}$ which lie in $\Cbuf(y_j)$ and were chosen from a set of size at least  $\tfrac{1}{2}|\Cbuf(y_\ell)|$. 
To obtain an upper bound on the expectation of 
$\sum_{\ell=1}^{j-1}Z^{(j)}_\ell$, we sum up~\eqref{eq:psr:expZ} for all $\ell\in[j-1]$. We split this sum according to whether $\big|\Ubuf(y_\ell)\cap \Cbuf(y_j)\big|\le(p+\eps p)^b|\Cbuf(y_j)|$ or not.
 In the former case, each summand is by~\ref{GPE:sizeU} and~\ref{GPE:sizeC} at most
 \[\frac{2(p+\eps p)^b\mu (p+\eps p)^b|V_i|}{(1-\eps')\mu(dp-\eps'p)^b|V_i|}\le 4d^{-\Delta}p^b\,,\]
 where the inequality is by choice of $\eps$ and $\eps'$, while~\ref{rga:psr:Xsum} bounds the sum over the remaining terms. We get
 \begin{align*}
 \sum_{\ell=1}^{j-1}\Exp\big[Z^{(j)}_\ell\big|\hist_{\ell-1}\big]&\le\sum_{\ell=1}^{j-1}\frac{\big|\Ubuf(y_\ell)\cap\Cbuf(y_j)\big|}{\tfrac12\big|\Cbuf(y_\ell)\big|}\\
 &\le 4(j-1)d^{-\Delta} p^b+\frac{|\Cbuf(y_j)|}{10}\le \frac{|\Cbuf(y_j)|}{5}
 \end{align*}
 where the last inequality uses $j\le|Y|\le\rho|X_i|$,~\ref{GPE:sizeC} and the choice of $\rho$ and $\eps$. Using the sequential dependence lemma, Lemma~\ref{lem:coupling}, with $\delta=1$, we conclude that the probability of the event $\sum_{\ell=1}^{j-1}Z^{(j)}_\ell>\tfrac25|\Cbuf(y_j)|$ is at most $e^{-|\Cbuf(y_j)|/15}\le e^{-\mu d^b p^b|V_i|/30}$.
 
 Observe that the event$\sum_{\ell=1}^{j-1}Z^{(j)}_\ell>\tfrac25|\Cbuf(y_j)|$ contains the event that~\eqref{eq:psr:smallY} fails for $j$, given that it did not fail for smaller $j$. Thus, taking a union bound over the $|Y|\le n$ values of $j$, we see that the probability that~\eqref{eq:psr:smallY} fails at any stage is at most
 \[n\cdot  e^{-\mu d^b p^b|V_i|/30}\,,\]
 which is smaller than one for all sufficiently large $n$ by choice of $p$.
 
 We conclude that there is a positive probability of choosing $|Y|$ distinct vertices in $U$, so $|U|\ge|Y|$ as desired.
\end{claimproof}
 
As previously mentioned, although the proof of Claim~\ref{cl:psr:small} constructs an auxiliary embedding, this embedding is not part of the final embedding of $H$ into $G$, and exists only to verify the claim. The same goes for the following claim.
 
\begin{claim}\label{cl:psr:large} If $|Y|>|\Xbuf_i|-\rho|X_i|$, then $|U|\ge|Y|$.
\end{claim}
\begin{claimproof}
The vertices $V'_i\setminus U$ are candidates only for vertices of $\Xbuf_i\setminus Y$, and since $|V'_i|=|\Xbuf_i|$, what we need to show is that $|V'_i\setminus U|\le |\Xbuf_i\setminus Y|$ (as explained in detail in the proof of Lemma~\ref{lem:completematch}).
 
   Let $v\in V_i'$ be a vertex and let $C(v)$ denote those vertices $x$ from $\Xbuf_i$ 
  for which $v$ is a candidate, i.e.\ $C(v)$ denotes the neighbours of $v$ in the candidate graph between 
  $V'_i\setminus U$ and  $\Xbuf_i\setminus Y$. 
 By~\ref{rga:psr:nobad}, for each vertex $v$ of $V'_i\setminus U$ we have  
 $|C(v)|\ge\mu(d^\Delta p/100)^b|X_i|$. By~\ref{rga:psr:Vsum}, for all but at most $\eps'p^\Delta|V_i|$ vertices $v$ of $V'_i\setminus U$, we have
 \begin{equation}\label{eq:psr:rightVsum}
  \Big|\Big\{v'\in V_i\colon \big|\{x\in\Xbuf_i:v,v'\in C(x)\}\big|> 24\mu\Delta^2\big(\tfrac{20}{\mu\zeta d^{\Delta}}\big)^b p^{2b} |X_i|\Big\}\Big|\le\eps'p^\Delta|V_i|\,.
 \end{equation}
 We choose an ordering $v_1,\ldots,v_{|V'_i\setminus U|}$ of $V'_i\setminus U$ which puts the vertices $v$ failing~\eqref{eq:psr:rightVsum} first. We choose, for each $j=1,\ldots,|V'_i\setminus U|$, a vertex $x_j$ uniformly at random from the vertices of $\Xbuf_i\setminus\{x_1,\ldots,x_{j-1}\}$ for which $v_j$ is a candidate if this is possible; if not, we say $x_j$ does not exist. As in the previous claim, our aim is to show that a.a.s.\ all $x_j$ exist, which implies the claim.
 
 Let $\hist'_0$ be the empty history, and for each $1\le j\le |V'_i\setminus U|$, let $\hist'_j$ be the history of this process up to and including choosing $x_j$. We claim that, with positive probability, at each step $j$ we choose from a set of size at least $\tfrac{1}{2}|C(v_j)|\ge \tfrac12\mu(d^\Delta p/100)^b|X_i|$. This in particular implies that each $x_j$ exists.
 
   We introduce random variables $\tZ^{(j)}_i$ for $j\in[|V_i'\setminus U|]$ and $\ell\in[j-1]$ as follows. 
     We set $\tZ^{(j)}_\ell=1$ if $x_\ell\in C(v_j)$ and 
     $|C(v_\ell)\setminus\{x_1,\ldots, x_{\ell-1}\}|\ge \tfrac{1}{2}|C(v_\ell)|$ hold, and we set 
  $\tZ^{(j)}_\ell=0$ otherwise. Observe that $\sum_{\ell=1}^{j-1}\tZ^{(j)}_\ell$ is exactly the number of vertices $x_k$ among $x_1$, \ldots, $x_{j-1}$ which lie in 
$C(v_j)$ and which were chosen from a set of size at least $\tfrac{1}{2}|C(v_k)|$.
Thus, to prove the claim it suffices to show that with positive probability, for every $j$ we have 
\begin{equation}\label{eq:psr:numZ}
 \sum_{\ell=1}^{j-1}\tZ^{(j)}_\ell< \tfrac12\mu(d^\Delta p/100)^b|X_i|\,.
\end{equation}
   If $v_j$ is a vertex failing~\eqref{eq:psr:rightVsum}, then 
  $j\le\eps'p^\Delta|V_i|<\tfrac{1}{4}\mu(d^\Delta p/100)^b|X_i|$, 
  where the second inequality is by choice of $\eps'$. 
  Thus, even if all vertices $x_1,\ldots,x_{j-1}$ happen to be 
  candidates for  $v_j$, we still have~\eqref{eq:psr:numZ}. Next we assume that $v_j$
  satisfies~\eqref{eq:psr:rightVsum}. For each of the at most 
  $\eps'p^\Delta|X_i|$ vertices $v_\ell$ such that 
   \[\big|\{x\in\Xbuf_i:v_j,v_\ell\in C(x)\}\big|> 24\mu\Delta^2\big(\tfrac{20}{\mu\zeta d^{\Delta}}\big)^b p^{2b} |X_i|\,,\]
  the conditional expectation $\Exp(\tZ^{(j)}_\ell|\hist'_{\ell-1})$   can be bounded from above by $1$, 
  while for the remaining vertices, it is at most
 \[
 \frac{\big|\{x\in\Xbuf_i:v_j,v_\ell\in C(x)\}\big|}{\tfrac{1}{2}|C(v_\ell)|}\le\frac{24\mu\Delta^2\big(20\mu^{-1}\zeta^{-1}d^{-\Delta}\big)^b p^{2b}}{\tfrac12\mu(d^\Delta p/100)^b}\le 48\Delta^2\big(\tfrac{2000}{\mu\zeta d^{2\Delta}}\big)^\Delta p^b\,.
 \]
 This yields
 \[
 \sum_{\ell=1}^{j-1}\Exp\big[\tZ^{(j)}_\ell\big|\hist'_{\ell-1}\big]\le \eps'p^\Delta|X_i|+|V'_i\setminus U|\cdot 48\Delta^2\big(\tfrac{2000}{\mu\zeta d^{2\Delta}}\big)^\Delta p^b\le \tfrac{1}{4}\mu(d^\Delta p/100)^b|X_i|\,,
 \]
 where the last inequality is because $|V'_i\setminus U|\le \rho|X_i|$ and by choice of $\rho$ and $\eps'$. As before, we apply Lemma~\ref{lem:coupling} for each $j\in[|V'_i\setminus U|]$ with $\delta=1$ and $(\tZ^{(j)})_{\ell\in[j-1]}$. We thus may bound  the probability that $v_j$ is candidate for more than $\tfrac{1}{2}|C(v_i)|\ge\tfrac12\mu(d^\Delta p/100)^b|X_i|$ of the vertices $x_1,\ldots,x_{j-1}$ by at most $e^{-\mu(d^\Delta p/100)^b|X_i|/12}$. Taking a union bound over the at most $\rho|X_i|$ choices of $j$ we see that with probability at least $1-\rho|X_i|\cdot e^{-\mu(d^\Delta p/100)^b|X_i|/12}>0$ (where the inequality is by choice of $p$), at each step we choose from a set of size at least $\tfrac12\mu(d^\Delta p/100)^b|X_i|$ as claimed. In particular, we succeed in choosing $|V'_i\setminus U|$ distinct vertices of $\Xbuf_i\setminus Y$, so as desired we have $|U|\ge|Y|$.
\end{claimproof}
 
 The final case is to show that $\rho|X_i|<|Y|\le|\Xbuf_i|-\rho|X_i|=|V'_i|-\rho|V_i|$ and $|U|<|Y|$ is a contradiction. In this case, we have $|V_i'\setminus U|>\rho|V_i|$. By~\ref{rga:psr:main} there are at most
 $\rho|X_i|$ vertices of $\Xbuf_i$ with fewer than $\tfrac12(dp)^b|V_i'\setminus U|$
 candidates in $V_i'\setminus U$, so in particular there is a vertex of $Y$ with
 candidates in $V_i'\setminus U$, in contradiction to the definition of $U$.
  
 This completes the verification of Hall's condition, so we can extend
 $\psiRGA$ to an embedding $\psi$ of $H$ into $G$ as desired, completing the
 proof of Lemma~\ref{lem:psr_main}.
\end{proof}

\section{Proof of the bijumbled graphs RGA lemma}\label{sec:pseudoRGA}
The proof of Lemma~\ref{lem:rga:psr} is very similar to the proof of Lemma~\ref{lem:rga} (RGA lemma for random graphs), and indeed we can reuse the auxiliary lemmas of Section~\ref{sec:RGAlemmas}. We again use a random greedy algorithm, Algorithm~\ref{alg:RGA:psr}, which we show produces the desired embedding with high probability. The difference from Algorithm~\ref{alg:RGA} is that when at time $t$ we reach a vertex $x$ in the order $\tau$ (given again by Lemma~\ref{lem:tau}) which is in the queue $Q_{t}$, we do not skip it, but instead we embed it into $\Vq$. Thus, we always have $t=\tau(x)-1$.
 The embedding is done uniformly at random into the set $\Aq_{t}(x)\setminus B_{t}(x)$. We remind the reader that $B_t(x)$ 
is the set of bad vertices with respect to $\psi_t$ and $Q_t$, that is, the vertices $v$ such that the extension $\psi_{t}\cup\{x\to v\}$ is not a good partial embedding or there is an unembedded neighbour $y$ of $x$ not in $Q_{t}$ such that $\deg_G\big(v;\Amain_t(y)\big)<(d-\eps')p|\Amain_t(y)|$ holds (see Definition~\ref{def:bad_vertices}). 
 
 \begin{algorithm}[t]
    \caption{Random greedy algorithm for bijumbled graphs}\label{alg:RGA:psr}
    \SetKwInOut{Input}{Input}
    \Input{$G\subseteq \Gamma$ and $H$ with partitions satisfying the General Setup; an ordering $\tau$ on $\Xmain$}
    $t:=0$ \; 
    $\psi_0:=\emptyset$ \;
    $Q_0:=\{x\in V(H):|I_x|<\tfrac12\mu(d-\eps)^{|J_x|}p^{|J_x|}|\Vmain(x)|\}$ \; 
    \Repeat{ $\dom(\psi_t)=\Xmain$ }{
      let $x\in \Xmain\setminus \dom(\psi_t)$ be the next vertex in
      the order $\tau$ \;
      \If{$x\in Q_t$ \emph{and} $|\Aq_t(x)\setminus B_t(x)|<\tfrac{1}{10}\mu\zeta(dp)^{\pi^*_t(x)}|V(x)|$}{halt with failure \;}
      choose $v$ uniformly at random in $\begin{cases}
	\Amain_t(x)\setminus B_t(x) \quad & \text{if $x\not\in Q_t$} \\
	\Aq_t(x)\setminus B_t(x) & \text{if $x\in Q_t$} 
      \end{cases}$ \;
      $\psi_{t+1}:=\psi_t\cup\{x\to v\}$ \;
      $Q_{t+1}:=Q_t$ \;
      \ForAll {$y\in \Xmain\setminus \dom(\psi_{t+1})$}{
        \If{$(\big|\Amain_{t+1}(y)\big|<\tfrac12
        \mu(d-\eps')^{\pi^*_{t+1}(y)}p^{\pi^*_{t+1}(y)}|\Vmain(y)|)$ }
        {$Q_{t+1}:=Q_{t+1}\cup\{y\}$ \; } } 
      $t:=t+1$\;
    }
    $t_\RGend:=t$\;
 \end{algorithm}
 Observe that there is another important difference from Algorithm~\ref{alg:RGA}. Algorithm~\ref{alg:RGA:psr} can fail, and we will have to prove that with high probability it does not. This means proving that sets $\Aq_t(x)$ do not get small, or equivalently that sets $\Cq_t(x)$ are never substantially covered by $\im(\psi_t)$. However, there is also an important similarity: provided the algorithm has not yet halted with failure, it maintains the same invariants as Algorithm~\ref{alg:RGA} (see Claim~\ref{cl:psr:inv}).
 
 Observe that (in contrast to Algorithm~\ref{alg:RGA}) if Algorithm~\ref{alg:RGA:psr} has not yet failed at time $t$, then all of the first $t$ vertices in the order $\tau$ have been embedded; the vertices are not reordered. Thus $\pi^*_{\tau(x)-1}(x)$ is always equal to $\pitau(x)$, where 
 \[\pitau(x):=\big|\{y\in N_H(x):\tau(y)<\tau(x)\}\big|+|J_x|\,.\]
 Much of the work of proving Lemma~\ref{lem:rga:psr} is contained in the auxiliary lemmas in Section~\ref{sec:RGAlemmas} which we use again here. What remains is to justify that Algorithm~\ref{alg:RGA:psr} with high probability does not halt with failure and does give a good partial embedding with properties~\ref{rga:psr:Xsum} and~\ref{rga:psr:Vsum}.
 
 \begin{proof}[Proof of Lemma~\ref{lem:rga:psr}]
 We require
 \begin{align*}
  \mu&\le\tfrac{d^\Delta}{1320\kappa\DeltaRp}\,,
  &\rho&\le \tfrac{\mu^2\zeta^2 d^{2\Delta}}{2000}\,,\\
   \eps'&\le\tfrac{\mu\zeta d^\Delta\rho}{1000\kappa\Delta^2 \DeltaRp 4^\Delta}2^{-4/\rho}\,,
  &\eps&\le \min\big((\eps')^2,2^{-\Delta}\eps',\tfrac{\rho\mu\zeta d^\Delta}{20\kappa\Delta},\tfrac{\eps'}{\kappa\DeltaRp}\big)\,,\\
  p&\ge \tfrac{1000\Delta^4\kappa r_1^2}{\mu\zeta\rho d^\Delta\eps} \big(\tfrac{\log n}{n}\big)^{1/(2\Delta)}\qquad\text{and}
  &2&^{-\eps p n/(\kappa r_1)}r_1n<1\,.
 \end{align*} 
Recall that we assumed that $\Gamma$ possesses the three properties $\NS(\eps,r_1,\Delta+1)$, $\RI(\eps,(\eps_{a,b}),\eps',d,r_1,\Delta+1)$ and $\LNS(\eps,r_1,\Delta)$. Let $\tau$ be an order on $V(H)$ given by Lemma~\ref{lem:tau}. We run Algorithm~\ref{alg:RGA:psr}, which maintains the same invariants as Algorithm~\ref{alg:RGA}, as shown in the following claim.
 
\begin{claim}\label{cl:psr:inv}
 The following hold at each time $t$ in the running of Algorithm~\ref{alg:RGA:psr}.
 \begin{enumerate}[label=\itmarab{INV}]
   \item\label{psr:inv:gpe} $\psi_t$ is a good partial embedding,
   \item\label{psr:inv:sizeA} Either $|\Amain_t(x)|\ge\frac12\mu(d-\eps')^{\pi^*_t(x)}p^{\pi^*_t(x)}|\Vmain(x)|$
   or $x\in Q_t$.
   \item\label{psr:inv:random} When we embed $x$ to create $\psi_{t+1}$, we do so
   uniformly at random into a set of size at least $\frac{1}{10}\mu\zeta(dp)^{\pi^*_t(x)} |V(x)|$.
 \end{enumerate}
\end{claim}
\begin{claimproof} 
 Algorithm~\ref{alg:RGA:psr} maintains~\ref{psr:inv:gpe} and \ref{psr:inv:sizeA} by definition. When we embed $x$ to create $\psi_{t+1}$, either $x\in Q_t$ or not. In the former case, because Algorithm~\ref{alg:RGA:psr} has not failed we have~\ref{psr:inv:random}. In the latter case, in particular $x$ was not added to $Q_t$ at time $t-1$. Using Lemma~\ref{lem:fewbad} to bound the size of $B_t(x)$, and since $\tfrac12\mu(d-\eps')^bp^b-20\Delta^2\eps'p^b>\tfrac{1}{10}\mu\zeta d^bp^b$ by choice of $\eps'$ for each $0\le b\le\Delta$, also in this case Algorithm~\ref{alg:RGA:psr} maintains~\ref{psr:inv:random}.
\end{claimproof}

 As in the proof of Lemma~\ref{lem:rga}, the conditions for Lemma~\ref{lem:rga:welldistr} are met, so we conclude that with probability at least $1-r_12^{-n/(\kappa r_1)}$, at each time $t$ when Algorithm~\ref{alg:RGA:psr} is running and for each $i\in[r]$ we have $|Q_t\cap X_i|\le\rho|X_i|+|X^*_i|\le2\rho|X_i|$, where the final inequality is by~\ref{PtH:image}. In order to show that Algorithm~\ref{alg:RGA:psr} runs successfully, we need to show that the `halt with failure' line is a.a.s.\ never reached, i.e.\ that if $x\in Q_{\tau(x)-1}$ then $\big|\Aq_{\tau(x)-1}(x)\setminus B_{\tau(x)-1}(x)|$ is at least $\tfrac{1}{10}\mu\zeta(dp)^{\pi^*_{\tau(x)-1}(x)}|V_i|$. Since we know by Lemma~\ref{lem:fewbad} that $B_{\tau(x)-1}(x)$ is small and by~\ref{GPE:sizeC} that $\Cq_{\tau(x)-1}(x)$ is large, 
what we want is to show that $\im(\psi_{t-1})$ covers only a small fraction of $\Cq_{\tau(x)-1}(x)$. We will show this in Claim~\ref{cl:rga:nofill}. Before this, we argue in Claim~\ref{cl:psr:lnssum} that in the embedding no `dense spots' are created, which we need for Claim~\ref{cl:rga:nofill}.
 
 We now explain how the following Claim~\ref{cl:psr:lnssum} helps us to establish Claim~\ref{cl:rga:nofill}. Observe that we embed the vertices $Q_t\cap X_i$ into $\Vq_i$ in a random procedure which is very similar to the embedding strategy we used in the proof of Claim~\ref{cl:psr:small} (verification of Hall's condition for small sets of buffer vertices). However, here a complication is that we do not know `in advance' what $\Cq_{\tau(x)-1}(x)$ will be when embedding earlier vertices to $\Vq_i$. But we do know that it will be contained in the $\Gamma$-neighbourhood in $V_i$ of some collection of $\pitau(x)\le\Delta$ vertices, and that it will be large. Thus, it suffices to prove the following. For each $i$, each $1\le\ell\le\Delta$ and $v_1,\dots,v_\ell$ such that $\comN_\Gamma(v_1,\ldots,v_\ell;V_i)$ is large, $\im(\psi_t)$ never covers much of $\comN_\Gamma(v_1,\ldots,v_\ell;V_i)$. We can prove this statement by using the same analysis as in the proof of Lemma~\ref{lem:rg_image}, given a `sum condition' (in~\ref{lnssum:itm:a} below) similar to~\ref{rga:psr:Xsum}.  We also show that if Algorithm~\ref{alg:RGA:psr} does not halt with failure then it is likely to have~\ref{rga:psr:Xsum} (which is~\ref{lnssum:itm:b} below), since the arguments are similar.
  
 \begin{claim}\label{cl:psr:lnssum} Suppose that for each $i\in[r]$ and each time $t$ we have $|Q_t\cap X_i|\le2\rho|X_i|$. Then a.a.s.\ the following statements hold.
 \begin{enumerate}[label=\abc] 
 \item\label{lnssum:itm:a} For any $t$, any $1\le \ell \le\Delta$, any $i\in[r]$ and any vertices $v_1,\ldots,v_\ell$ such that
  \[\deg_\Gamma(v_1,\ldots,v_\ell;V_i)\ge\eps'p^\ell|V_i|\]
  we have
  \begin{equation}\label{eq:randmatch:sumVqLNS}
   \sum_{\substack{y\in\Xmain_i\cap Q_t\colon\\ \tau(y)\le t,J_y=\emptyset}} \frac{\big|U_{\tau(y)-1}(y)\cap \comN_\Gamma(v_1,\ldots,v_\ell;V_i)\big|}{\big|\Aq_{\tau(y)-1}(y)\setminus B_{\tau(y)-1}(y)\big|}\le\frac{\mu\zeta d^\Delta \deg_\Gamma(v_1,\ldots,v_\ell;V_i)}{20}\,.
  \end{equation}
  
 \item\label{lnssum:itm:b} For any $t$, for any $1\le \ell \le \Delta$, any $i\in[r]$ and any vertices $v_1,\ldots,v_\ell$ such that
  \[\deg_G(v_1,\ldots,v_\ell;\Vbuf_i)\ge(dp-\eps'p)^\ell|\Vbuf_i|\,,\]
  we define $X'_i$ to be the set
  \[\big\{y\in\Xbuf_i: \big|\Ubuf_t(y)\cap \comN_G(v_1,\ldots,v_\ell;\Vbuf_i)\big|> (p+\eps p)^{\pi^*_t(y)}\deg_G(v_1,\ldots,v_\ell;\Vbuf_i)\big\}\]
  and we have
  \begin{equation}\label{eq:randmatch:sumVbufLNS}
   \sum_{y\in X'_i} \frac{\big|\Ubuf_t(y)\cap \comN_G(v_1,\ldots,v_\ell;\Vbuf_i)\big|}{\big|\Cbuf_t(y)\big|}
   \le\frac{\deg_G(v_1,\ldots,v_\ell;\Vbuf_i)}{20}\,.
  \end{equation} 
 \end{enumerate}
 \end{claim}

 The idea of the proof of~\eqref{eq:randmatch:sumVqLNS} is as follows. The denominators on the left hand side of~\eqref{eq:randmatch:sumVqLNS} are by~\ref{psr:inv:random} never much smaller than they `should' be, so the main task is to show that the numerators do not tend to be too large. To show this, we consider the evolution of $U_t(y)\cap \comN_\Gamma(v_1,\dots,v_\ell;V_i)$ for some fixed $y$ as $t$ increases. Since $J_y=\emptyset$, at first this set has the size we expect, namely it is all of $\comN_\Gamma(v_1,\dots,v_\ell;V_i)$. Each time a neighbour of $y$ is embedded, we expect that the set size shrinks by a factor roughly $p$. If this is the case for each neighbour, the size at $t=\tau(y)-1$ is roughly a $p^{\pitau(y)}$-factor times its original size, which turns out to be a good enough bound for~\eqref{eq:randmatch:sumVqLNS}. If not, there is some first time when we embed a neighbour of $y$, say the $s$th neighbour, `badly', that is, the set size does not shrink by a factor roughly $p$. We say we fail at step $s$. At worst, it could be that the set size does not thereafter change, so that it stays roughly a $p^{s-1}$-factor times its original size. In this case we `lose' a $p^{\pitau(y)-s+1}$ factor. But the $\LNS$ property tells us that the probability of failing at step $s$ is less than $p^{\pitau(y)-s+1}$. Heuristically, this gets us back the lost factor; to make this rigorous, we apply Lemma~\ref{lem:coupling}.
 
 \begin{claimproof}[Proof of Claim~\ref{cl:psr:lnssum}]
  We require
  \[\rho\le \frac{\mu^2\zeta^2d^{2\Delta}}{2000}\,,\quad\eps\le\min\Big(2^{-\Delta}\eps',\frac{\rho\mu\zeta d^\Delta}{20\kappa\Delta}\Big)\quad\text{and}\quad p\ge\frac{20\Delta^4\kappa r^2_1}{\rho}\big(\tfrac{\log n}{n}\big)^{1/\Delta}\,.\]
 
  We start with~\eqref{eq:randmatch:sumVqLNS}. Since the sum in~\eqref{eq:randmatch:sumVqLNS} is monotonically increasing in $t$, and since the upper bound claimed in~\eqref{eq:randmatch:sumVqLNS} does not depend on $t$, it suffices to show that the bound holds when Algorithm~\ref{alg:RGA:psr} terminates (with or without failure).
  
  Given $1\le \ell\le \Delta $, let $v_1,\ldots,v_\ell$ be vertices of $\Gamma$ such that the set  $U:=\comN_\Gamma(v_1,\ldots,v_\ell;V_i)$ has size at least  $\eps'p^\ell|V_i|$. By~\ref{psr:inv:random} we have the lower bound
  \[\big|\Aq_{\tau(y)-1}(y)\setminus B_{\tau(y)-1}(y)\big|\ge \tfrac{1}{10}\mu\zeta(dp)^{\pitau(y)}|V_i|\]
  for the denominator in each summand of~\eqref{eq:randmatch:sumVqLNS}, and the difficulty is to upper bound the numerator. Consider the running of Algorithm~\ref{alg:RGA:psr}. At any time $t\le\tau(y)-1$, the set $\Uq_t(y)\cap U$ is the $\Gamma$-neighbourhood in $U$ of the $\pi^*_t(y)$ embedded vertices from $N_H(y)$, and in a truly random set we would thus expect to find that
  \begin{equation}\label{eq:lnssum:Uq}
   |U_t(y)\cap U|\le (p+\eps p)^{\pi^*_t(y)}|U|\,.
  \end{equation}
  If this inequality remains true up to $t=\tau(y)-1$, then since $|J_y|=0$ by assumption, we have $\pi^*_{\tau(y)-1}(y)=\pitau(y)$, and~\eqref{eq:lnssum:Uq} gives an upper bound good enough for the 
  summand $\frac{\big|U_{\tau(y)-1}(y)\cap \comN_\Gamma(v_1,\ldots,v_\ell;V_i)\big|}{\big|\Aq_{\tau(y)-1}(y)\setminus B_{\tau(y)-1}(y)\big|}$. 
  However, it is likely that some vertices will not satisfy~\eqref{eq:lnssum:Uq}, and we have to estimate their contribution to~\eqref{eq:randmatch:sumVqLNS}.
  
  We say that a vertex $y\in\Xmain_i$ \emph{fails at step $s$} if the vertex $z$ is the $s$th vertex of $N_H(y)$ in $\tau$ and $y$ satisfies~\eqref{eq:lnssum:Uq} for each $t\le\tau(z)-1$ but fails~\eqref{eq:lnssum:Uq} at $t=\tau(z)$. Observe that each vertex $y$ in the sum~\eqref{eq:randmatch:sumVqLNS} satisfies~\eqref{eq:lnssum:Uq} before any neighbour of $y$ is embedded because $J_y=\emptyset$.  Thus, if $y$ does not fail at any step $s$ with $1\le s\le\Delta$ then it satisfies~\eqref{eq:lnssum:Uq} at time $t=\tau(y)-1$, but if for some $1\le s\le\Delta$ it fails at step $s$, then $|U_{\tau(y)-1}(y)\cap U|\le(p+\eps p)^{s-1}|U|$. 
  
  Suppose now that $y$ fails at step $s$. Then the reason is that the vertex $z$, which is the $s$th neighbour of $y$ in $\tau$, is embedded to a vertex of $\Gamma$ with more neighbours in $U_{\tau(z)-1}(y)\cap U$ than allowed by~\eqref{eq:lnssum:Uq}. Let $W$ be a superset of $U_{\tau(z)-1}(y)\cap U$ of size $(p+\eps p)^{s-1}|U|$. By choice of $\eps$ we have $|W|\ge\eps p^{\ell+s-1}|V_i|$. Because $y$ fails at step $s$ we see that $z$ is embedded to a vertex $v$ of $\Gamma$ with $\deg_\Gamma(v;W)>(p+\eps p)^s|U|=(p+\eps p)|W|$. By $\LNS(\eps,r_1,\Delta)$ we know that the number of such vertices $v$ in $\Gamma$ is at most 
  $\eps p^{2\Delta-s}|V(\Gamma)|/r_1^2$ (since $\ell\le \Delta$). Therefore, by~\ref{psr:inv:random}, the probability of embedding $z$ (which has at least one unembedded neighbour and hence at most $\Delta-1$ embedded neighbours) to such a vertex $v$, conditioning on the history up to but not including embedding $z$, is at most
  \[\frac{\eps p^{2\Delta-s}|V(\Gamma)|/r_1^2}{\mu\zeta(dp)^{\Delta-1}|V_i|/10}\le \frac{\rho}{2\Delta} p^{\Delta-s+1}\le\frac{\rho}{2\Delta} p^{\pitau(y)-s+1}\,,\]
  where the first inequality is by choice of $\eps$ and the second since $\pitau(y)\le\Delta$.
  
  We can restate this as: the probability that $y$ fails at step $s=\pitau(y)-j+1$, conditioning on the history up to but not including the embedding of the vertex at step $s$, is at most $\tfrac{\rho}{2\Delta}p^j$ for each $1\le j\le \pitau(y)$. It follows that the expected number of vertices $y\in\Xmain_i$ which fail at step $\pitau(y)-j+1$ is at most $\tfrac{\rho}{2\Delta}p^j|X_i|$, and by Lemma~\ref{lem:coupling}, applied with $\delta=1$, the probability that more than $\tfrac{\rho}{\Delta}p^j|X_i|$ vertices $y\in\Xmain_i$ fail at step $\pitau(y)-j+1$ is at most $\exp\big(-\tfrac{\rho}{6\Delta}p^{j}|X_i|\big)$. Let $\cE$ be the event that, for each $1\le j\le\Delta$, at most $\tfrac{\rho}{\Delta}p^j|X_i|$ vertices $y$ of $\Xmain_i$ fail at step $\pitau(y)-j+1$. The probability that $\cE$ holds is thus at least
  \begin{equation}\label{eq:lnssum:Vqprob}
   1-\Delta\exp\big(-\tfrac{\rho}{6\Delta}p^{\Delta}|X_i|\big)\,.
  \end{equation}
  
Suppose $\cE$ occurs. Since a vertex $y$ failing at step $j$ satisfies 
\[|U_{\tau(y)-1}(y)\cap U|\le (p+\eps p)^{\pitau(y)-j}|U| \]
we have
\begin{multline}\label{eq:LNS:bigY}
   \sum_{\substack{y\in\Xmain_i: J_y=\emptyset,\\ |U_{\tau(y)-1}(y)\cap U|>(p+\eps p)^{\pitau(y)}|U|}} \frac{\big|U_{\tau(y)-1}(y)\cap U\big|}{\big|\Aq_{\tau(y)-1}(y)\setminus B_{\tau(y)-1}(y)\big|}\\
   \le\sum_{\substack{y\in\Xmain_i:\\ 
   y\text{ fails at step }s=\pitau(y)-j+1}} \frac{(p+\eps p)^{\pitau(y)-j}|U|}{\tfrac{1}{10}\mu\zeta(dp)^{\pitau(y)}|V_i|}
   \le  \sum_{j=1}^{\Delta} \frac{(p+\eps p)^{\pitau(y)-j}|U|\cdot \tfrac{\rho}{\Delta}p^j|X_i|}{\tfrac{1}{10}\mu\zeta(dp)^{\pitau(y)}|V_i|}\\
   \le \frac{20\rho|U|}{\mu\zeta d^{\Delta}}\le\frac{\mu\zeta d^\Delta|U|}{40}\,,
\end{multline}
  where the last two inequalities are by choice of $\eps$ and $\rho$ respectively.
  On the other hand, because $|\Xmain_i\cap Q_t|\le 2\rho|X_i|$, we have
  \begin{equation*}\begin{split}
   \sum_{\substack{y\in\Xmain_i\cap Q_t:\\ |U_{\tau(y)-1}(y)\cap U|\le(p+\eps p)^{\pitau(y)}|U|}} \frac{\big|U_{\tau(y)-1}(y)\cap U\big|}{\big|\Aq_{\tau(y)-1}(y)\setminus B_{\tau(y)-1}(y)\big|}\le\\
    \sum_{y\in\Xmain_i\cap Q_t} 
    \frac{(p+\eps p)^{\pitau(y)}|U|}{\tfrac{1}{10}\mu\zeta(dp)^{\pitau(y)}|V_i|}
   &\le  \frac{(1+\eps)^{\Delta}|U|\cdot 2\rho|X_i|}{\tfrac{1}{10}\mu\zeta d^{\Delta}|V_i|}\\
   &\le \frac{40\rho|U|}{\mu\zeta d^{\Delta}}\le\frac{\mu\zeta d^\Delta|U|}{40}\,,
  \end{split}\end{equation*}
  and putting these together we conclude~\eqref{eq:randmatch:sumVqLNS}.   
  
  It remains to bound the probability that for some vertices $v_1$, \ldots,
  $v_\ell$ and for some $i$ for which 
  \[\deg_\Gamma(v_1,\ldots,v_\ell;V_i)\ge \eps'p^\ell|V_i|\,,\]
  the inequality~\eqref{eq:randmatch:sumVqLNS} 
  is violated, in other words that $\cE$ fails for $v_1$, \ldots, $v_\ell$ and $i$. 
  There are  $r\le r_1$ choices of $i$ and at most $n+n^2+\dots+n^{\Delta}\le \Delta n^{\Delta}$ choices of $v_1,\ldots,v_\ell$, so by the union bound and~\eqref{eq:lnssum:Vqprob} the probability of such  a bad event occurring is at most
  \[r_1\Delta n^{\Delta}\cdot 2\Delta\exp\big(-\tfrac{\rho}{6\Delta\kappa r_1}p^{\Delta}n\big)\]
  which tends to zero as $n$ tends to infinity by choice of $p$.
  
  We now come to~\eqref{eq:randmatch:sumVbufLNS}. This inequality looks very much like~\eqref{eq:LNS:bigY} above, and its proof is almost identical to the proof of~\eqref{eq:LNS:bigY}. The differences are that we consider $G$-neighbourhoods in $\Vbuf_i$ not $\Gamma$-neighbourhoods in $V_i$ and that we have a lower bound on $\Cbuf_t(x)$ from~\ref{GPE:sizeC} rather than on $|\Aq_t(x)\setminus B_t(x)|$ from~\ref{psr:inv:random}. Since this lower bound is larger, the same constant choices work. We omit the details.
 \end{claimproof}
 
 We can now use Lemma~\ref{lem:coupling} to complete the proof that the sets $\Cq_t(x)$ do not get covered by $\im(\psi_t)$, much as in Claim~\ref{cl:psr:small} (verification of Hall's condition for small sets of buffer vertices).
 
 \begin{claim}\label{cl:rga:nofill} A.a.s.\ for each $x\in\Xmain$, at each time $t\le\tau(x)-1$ and before the termination of Algorithm~\ref{alg:RGA:psr}, we have $\big|\Cq_t(x)\cap\im(\psi_t)\big|<\tfrac12\big|\Cq_t(x)\big|$.
 \end{claim}
 \begin{claimproof}
  We require
  \[\eps\le \eps'\le 2^{-2\Delta}, \qquad \rho\le \tfrac{\mu \zeta d^\Delta}{30}\qquad\text{and}\qquad p\ge \tfrac{1000r_1}{\mu\zeta d^\Delta}\big(\tfrac{\log n}{n}\big)^{1/\Delta} \,.\] 
 
  Suppose that the conclusions of Claim~\ref{cl:psr:lnssum} hold (in particular~\eqref{eq:randmatch:sumVqLNS} holds). Given $i\in[r]$ and $x\in\Xmain_i$, if $t\le\tau(x)-1$ and Algorithm~\ref{alg:RGA:psr} has not terminated before time $t$, then by~\ref{GPE:sizeC} and~\ref{G:restr} we have $\big|\Cq_t(x)\big|\ge (1-\eps')\mu\zeta(dp-\eps'p)^{\pi^*_t(x)}|V_i|$. Since $\Cq_t(x)\subset U_t(x)= \comN_\Gamma(v_1,\ldots,v_\ell;V_i)$ for some vertices $v_1,\ldots,v_\ell$ with $\ell=\pi^*_t(x)\le\Delta$, we conclude by~\eqref{eq:randmatch:sumVqLNS} that
  \[\sum_{\substack{y\in\Xmain_i\cap Q_t\colon\\ \tau(y)\le t,J_y=\emptyset}} \frac{\big|\Uq_{\tau(y)-1}(y)\cap \Cq_t(x)\big|}{\big|\Aq_{\tau(y)-1}(y)\setminus B_{\tau(y)-1}(y)\big|}\le\frac{\mu\zeta d^\Delta\big|U_t(x)\big|}{20}\,.\]
  Observe that the summand in the above inequality is an upper bound for the probability that $y$ is embedded to $\Cq_t(x)$ for $y\in Q_t$, conditioning on the history up to but not including the embedding of $y$. Since the probability that $y$ is embedded to $\Cq_t(x)$ is zero if $y$ is not in $Q_t$, we are in a position to apply Lemma~\ref{lem:coupling}. This lemma, with $\delta=1$, tells us that the probability that more than $\tfrac1{10}\mu\zeta d^\Delta\big|U_t(x)\big|$ vertices $y$ of $\Xmain_i$ with $J_y=\emptyset$ are embedded to $\Cq_t(x)$ is at most $\exp\big(-\tfrac{1}{60}\mu\zeta d^\Delta|U_t(x)|\big)$. If this bad event does not occur, then by~\ref{GPE:sizeU},~\ref{GPE:sizeC} and~\ref{G:restr}, and because the number of vertices in $\Xmain_i$ with $J_y\neq\emptyset$ is by~\ref{PtH:image} at most $\rho p^\Delta|X_i|$, we have the desired statement:
  \[
  \big|\Cq_t(x)\cap\im(\psi_t)\big|< \tfrac{1}{10}\mu\zeta d^\Delta\big|U_t(x)\big|+\rho p^\Delta|X_i|\le\tfrac{1}{2}
  (1-\eps')\mu\zeta(dp-\eps'p)^{\pi^*_t(x)}|V_i|.
  \]
  
  The probability that the conclusions of Claim~\ref{cl:psr:lnssum} fail to hold, or that any of the above bad events occur, is at most $o(1)+r_1 n^2\cdot\exp\big(-\tfrac{1}{60r_1}\mu\zeta d^\Delta(p-\eps'p)^\Delta n\big)$, which tends to zero as $n$ tends to infinity by choice of $p$, completing the proof.
 \end{claimproof}
 
   If the conclusions of Lemma~\ref{lem:rga:welldistr}, Claim~\ref{cl:psr:lnssum} and 
   Claim~\ref{cl:rga:nofill} hold (which we think of as being good events), 
   then by~\ref{GPE:sizeC} and Lemma~\ref{lem:fewbad} we have
   \[\big|\Aq_{\tau(x)-1}(x)\setminus B_{\tau(x)-1}(x)\big|\ge\tfrac12\mu\zeta(dp-\eps'p)^{\pitau(x)}|V_i|-20\Delta^2\eps'p^{\pitau(x)}|V_i|\]
 and the right hand side is by choice of $\eps'$ at least $\tfrac{1}{10}\mu\zeta(dp)^{\pitau(x)}|V_i|$. In other words, the `halt with failure' line of Algorithm~\ref{alg:RGA:psr} is never reached. Since each of the three good events a.a.s.\ occurs, the algorithm a.a.s.\ completes successfully. Now~\ref{rga:psr:place} is guaranteed by successful completion of Algorithm~\ref{alg:RGA:psr}, and~\ref{rga:psr:main} is as~\ref{rga:emb:main} in the proof of Lemma~\ref{lem:rga} implied by the good event of Lemma~\ref{lem:rga:welldistr}, while~\ref{rga:psr:Xsum} is given by the good event of Claim~\ref{cl:psr:lnssum}, specifically taking~\eqref{eq:randmatch:sumVbufLNS} with $t=t_\RGend$ and vertices $\psi_t\big(N_H(x)\big)$ for each $x\in\Xbuf_i$.
  
 Next we establish that~\ref{rga:psr:nobad} a.a.s.\ holds. We apply Lemma~\ref{lem:large_nbs} with $B=\Delta+1$ and $T=\tau(z)$ where $z$ is the last vertex in $N(\Xbuf)$. Since $N(\Xbuf)$ forms the initial segment of $\tau$ by Lemma~\ref{lem:tau}\ref{cond:seg}, we see that all vertices $y$ embedded up to time $T$ have $\pitau(y)\le\Delta-1\le B-2$, and by identical logic as in Claim~\ref{cl:RGA_queue} none of these vertices enter the queue. Therefore, by~\ref{GPE:sizeC}, Lemma~\ref{lem:fewbad} and choice of $\eps'$, each $y$ from $N(\Xbuf)$  is embedded uniformly at random to a set of size at least $\tfrac{1}{10}(dp)^{\pitau(y)}|V(y)|$. Finally, by~\ref{BUF:sizebuf} we have $|N(\Xbuf)\cap X_i|\le4\kappa\DeltaRp\mu|X_i|$ for each $i\in[r]$. This justifies that the conditions of Lemma~\ref{lem:large_nbs} are met, so we conclude that a.a.s.\ the following event $\ev_{\sublem{lem:large_nbs}}$ occurs. For each $v\in V_i$ and $j$ such that $ij\in E(R')$ we have $\big|N_G(v;\Vmain_j)\setminus\im(\psi_T)\big|\ge\tfrac12\deg_G(v;\Vmain_j)$.
 
We are now in a position to apply Lemma~\ref{lem:nonclique_buffer}, again with $B=\Delta+1$. Observe that this time, if $\ev_{\sublem{lem:large_nbs}}$ occurs, it applies to any $x\in\Xbuf_i$ and $v\in V_i$, whether or not $x$ is in a clique buffer. The deduction that a.a.s.\ the embedding $\psi_{t_\RGend}$ has~\ref{rga:psr:nobad} follows exactly as in the proof of Lemma~\ref{lem:rga} (RGA lemma), and we do not repeat it.
 
 It remains to establish~\ref{rga:psr:Vsum}. For this we require the following claim, which shows that if vertices $v$ and $v'$ are common candidates for too many $x\in\Xbuf_i$ for~\ref{rga:psr:Vsum}, then the reason is that they have an exceptionally large  common $\Gamma$-neighbourhood in some cluster.
 
 \begin{claim}\label{cl:psr:Vsum} Asymptotically almost surely at the termination of Algorithm~\ref{alg:RGA:psr} the following holds. For each $i$ and pair of vertices $v,v'\in V_i$ such that $\deg_\Gamma(v,v';V_j)\le(p+\eps p)^2|V_j|$ whenever $ij\in R'$, if $b$ is such that $\Xbuf_i$ is a degree-$b$ buffer, then we have
  \begin{equation}\label{psr:Vsum:deseq}
   \big|\{x\in\Xbuf_i:v,v'\in C(x)\}\big|\le 24\mu\Delta^2\big(20\mu^{-1}\zeta^{-1}d^{-\Delta}\big)^b p^{2b}|X_i|\,.
  \end{equation}
 \end{claim}
 
To prove this we show that for each $x\in\Xbuf_i$, however, previous vertices are embedded, it is not too likely that $N_H(x)$ is embedded to $\comN_\Gamma(v,v')$, and apply Lemma~\ref{lem:coupling} and the union bound to deduce~\eqref{psr:Vsum:deseq} for all $i$, $v$, and $v'$. In turn, to prove the desired upper bound on the probability of embedding $N_H(x)$ to $\comN_\Gamma(v,v')$, we analyse the embedding of the vertices $N_H(x)=\{y_1,\dots,y_b\}$ one by one. We would expect that in each case roughly a $p^2$-fraction of $U(y_i)$ is contained in $\comN_\Gamma(v,v')$, and if this is the case at each step we obtain the desired upper bound. If not, the reason is that a previously embedded vertex---which must be one of the $y_j$ since we have not yet embedded any other neighbouring vertices---was `badly embedded'. Using the $\LNS$ property, we show this is an unlikely event and again obtain the desired upper bound. In fact, this is a slight oversimplification: we have to separate the cases that exactly one vertex is embedded badly (which implies we still have only at worst a probability $p$ of embedding future vertices of $N_H(x)$ to $\comN_\Gamma(v,v')$) or that more than one vertex is embedded badly, in which case we might have probability $1$ of embedding future vertices of $N_H(x)$ to $\comN_\Gamma(v,v')$, but this is counterbalanced by the unlikelihood of embedding two vertices badly.
 
 \begin{claimproof}[Proof of Claim~\ref{cl:psr:Vsum}]
  We require
  \[\eps\le\min\big( 2^{-\Delta},\tfrac{\mu\zeta d^\Delta}{10\Delta\kappa}\big)\quad\text{and}\quad p\ge 20\mu^{-1}\kappa r_1\big(\tfrac{\log n}{n}\big)^{1/(2\Delta)}\,.\]
  Let $v,v'\in V_i$ be such that $\deg_\Gamma(v,v';V_j)\le (p+\eps p)^2|V_j|$ for each $j$ such that $ij\in R'$. Suppose that $\Xbuf_i$ is a degree-$b$ buffer and $b\in[\Delta]$.
  
   Let $T=\tau(x)-|N_H(x)|-1$ be the time immediately before the first vertex of $N_H(x)$ is embedded. Since any two vertices of $\Xbuf_i$ are at distance at least ten in $H$ by~\ref{PtH:dist}, it is enough to show
   \begin{enumerate}[label=($\ast$)]
    \item\label{sillythingthatisntthere} for any one vertex $x\in\Xbuf_i$, by considering the embedding of $N_H(x)$, that the probability of $v,v'\in C(x)$, conditioning on the history up to time $T$, is at most $3\Delta^2\big(20\mu^{-1}\zeta^{-1}d^{-\Delta}\big)^bp^{2b}$.
   \end{enumerate}If we show this, then, since the vertices of $N_H(x)$ are embedded consecutively, by Lemma~\ref{lem:coupling}, applied with $\delta=1$ and using $|\Xbuf_i|=4\mu|X_i|$, we see that~\eqref{psr:Vsum:deseq} holds with probability at least $1-\exp\big(-4\mu\Delta^2\big(20\mu^{-1}\zeta^{-1}d^{-\Delta}\big)^bp^{2b}|X_i|\big)$, and taking a union bound over the at most $n^2$ choices of $v,v'$ we conclude that by choice of $p$ the conclusion of the claim a.a.s.\ holds as desired.
  
 We show that the probability of $v,v'\in C(x)$ is at most $3\Delta^2\big(20\mu^{-1}\zeta^{-1}d^{-\Delta}\big)^bp^{2b}$, conditioning on any $\psi_T$. Observe that $v,v'\in C(x)$ can occur only if $N_H(x)$ is embedded into $\comN_{\Gamma}(v,v')$. If $y\in N_H(x)\cap X_j$, then the probability that $y$ is embedded into $\comN_\Gamma(v,v')$ is, using~\ref{psr:inv:random}, at most
  \[\frac{\big|U_{\tau(y)-1}(y)\cap \comN_\Gamma(v,v')\big|}{\tfrac{1}{10}\mu\zeta(dp)^{\pitau(y)}|V_j|}\,.\]
  We would expect that the numerator is at most $(p+\eps p)^{2+\pitau(y)}|V(y)|$ for all  $y\in N_H(x)$. Of course, this does not always happen, since it may be the case that a neighbour $z$ of $y$ in $H$ with $\tau(z)<\tau(y)$ was embedded \emph{badly}, that is, to a vertex with more than 
  \begin{equation}\label{eq:badly_embed}
  (p+\eps p)\left|U_{\tau(z)-1}(y)\cap \comN_\Gamma(v,v')\right|
  \end{equation}
  neighbours in $U_{\tau(z)-1}(y)\cap \comN_\Gamma(v,v')$. 
  Since $\tau$ satisfies the condition of Lemma~\ref{lem:tau}\ref{cond:seg}, any such $z$ is in $N(\Xbuf)$. Since $x$ is at distance at most $2$ from $z$ in $H$, and at distance at least $5$ from any other vertex of $\Xbuf$ by~\ref{BUF:dist}, we see that $z\in N_H(x)$. Thus, if exactly $s$ many neighbours $z$ of $y$ with $\tau(z)<\tau(y)$ are embedded badly then 
  we see that with~\eqref{eq:badly_embed} the following inequality 
  \begin{equation}\label{eq:s_bad_vcs}
  \big|U_{\tau(y)-1}(y)\cap \comN_\Gamma(v,v')\big|\le (p+\eps p)^{\pitau(y)-s}|\comN_\Gamma(v,v')|\le (p+\eps p)^{\pitau(y)+2-s}|V(y)|
  \end{equation}
  holds.
   We now show how this implies~\ref{sillythingthatisntthere}. We separate three cases.
  
  First, no vertices in $N_H(x)$ are embedded badly. Then the probability that all $b$ vertices in $N_H(x)$ are embedded into $\comN_\Gamma(v,v')$, conditioning on $\psi_T$, is at most $\big(20\mu^{-1}\zeta^{-1}d^{-\Delta}p^2\big)^b$ by choice of $\eps$.
  
  Second, exactly one vertex in $N_H(x)$ is embedded badly. We let the vertices of $N_H(x)$ be $y_1,\ldots,y_b$ in order of $\tau$. Let us suppose that $y_\ell$ is the vertex embedded badly. Then the probability that the first $\ell-1$ vertices of $N_H(x)$ are embedded into $\comN_\Gamma(v,v')$, conditioning on $\psi_T$, is at most $\big(20\mu^{-1}\zeta^{-1}d^{-\Delta}p^2\big)^{\ell-1}$ by choice of $\eps$ (since these are not  badly embedded and conditional probabilities multiply). 
  
  We now estimate the probability of $y_\ell$ being embedded badly, conditioning on $\psi_T$ and on the embeddings of the previous $\ell-1$ vertices not being bad. Observe that, since there has been no previous bad embedding,  we have for each $\ell'>\ell$
  \[
  \big|U_{\tau(y_\ell)-1}(y_{\ell'})\cap \comN_\Gamma(v,v')\big|\le(p+\eps p)^{\pi_{\tau(y_\ell)-1}(y_{\ell'})+2}|V(y_{\ell'})|\,,
  \]
  and since $\pi_{\tau(y_\ell)-1}(y_{\ell'})\le\ell-1\le\Delta-2$, the right hand side is, by choice of $\eps$, at least $\eps p^{\Delta}n/r_1$. By $\LNS(\eps,r_1,\Delta)$ the number of vertices in $\Gamma$ with more than 
  $(p+\eps p)^{\pi_{\tau(y_\ell)-1}(y_{\ell'})+3}|V(y_{\ell'})|$ neighbours into $U_{\tau(y_\ell)-1}(y_{\ell'})\cap \comN_\Gamma(v,v')$
   is at most $\eps p^{2\Delta-1}n/r_1$. Since $\pitau(y_\ell)\le\ell-1$, by~\ref{psr:inv:random} the probability that $y_\ell$ is embedded to such a bad vertex is at most
  \[
  \frac{\Delta\eps p^{2\Delta-1}n/r_1}{\mu\zeta(dp)^{\ell-1}|V(y_\ell)|/10}\le p^{2\Delta-\ell}
  \]
  where the inequality is by choice of $\eps$. 
  
  Now, the probability that the last $b-\ell$ vertices are embedded into $\comN_\Gamma(v,v')$, conditioning on the previous embeddings, is at most $\big(20\mu^{-1}\zeta^{-1}d^{-\Delta}p\big)^{b-\ell}$ by choice of $\eps$. Indeed, 
  for each vertex $\ell'>\ell$ we have by~\eqref{eq:s_bad_vcs} that 
  \[
  \big|U_{\tau(y_{\ell'})-1}(y_{\ell'})\cap \comN_\Gamma(v,v')\big|\le (p+\eps p)^{\pitau(y_{\ell'})+1}|V(y_{\ell'})|
  \]
 since we condition on  exactly one neighbour of $y_{\ell'}$ being embedded badly. This gives a conditional probability of at most $20\mu^{-1}\zeta^{-1}d^{-\Delta}p$ that $v_{\ell'}$ is embedded to $\comN_\Gamma(v,v')$.   The conditional probabilities multiply, and taking the union bound over the choices of $\ell$, we see that the probability of this case occurring and all vertices of $N_H(x)$ being embedded to $\comN_\Gamma(v,v')$, conditioning on $\psi_T$, is at most 
 \begin{multline*}
\Delta\cdot\big(20\mu^{-1}\zeta^{-1}d^{-\Delta}p^2\big)^{\ell-1}\cdot p^{2\Delta-\ell}\cdot 
\big(20\mu^{-1}\zeta^{-1}d^{-\Delta}p\big)^{b-\ell}
\le\\ \Delta\big(20\mu^{-1}\zeta^{-1}d^{-\Delta}\big)^{b-1}p^{2\Delta+b-2}\le \Delta^2\big(20\mu^{-1}\zeta^{-1}d^{-\Delta}\big)^bp^{2b},
 \end{multline*}
where the last inequality follows since $\Delta\ge 2$ and $b\le\Delta$. 
  
  In the third case, at least two vertices in $N_H(x)$ are embedded badly. Suppose that the first two badly embedded vertices are the $j$th and $k$th vertices. Then the same logic as above, in particular the inequality~\eqref{eq:s_bad_vcs},  tells us that the probability that this case occurs and all vertices of $N_H(x)$ are embedded to $\comN_\Gamma(v,v')$ is at most
 \begin{multline*}
 \Delta^2\cdot \big(20\mu^{-1}\zeta^{-1}d^{-\Delta}p^2\big)^{j-1}\cdot 
 p^{2\Delta-j}\cdot \big(20\mu^{-1}\zeta^{-1}d^{-\Delta}p\big)^{k-j-1}\cdot p^{2\Delta-k}\le
 \\  
   \Delta^2\big(20\mu^{-1}\zeta^{-1}d^{-\Delta}\big)^{b-2}p^{4\Delta-3}.
  \end{multline*} 
  
  Putting these three cases together, and using the fact that $4\Delta-3\ge 2b$ since $\Delta\ge 2$ and $b\le\Delta$, we conclude that the probability that $v,v'\in C(x)$ is at most $3\Delta^2\big(20\mu^{-1}\zeta^{-1}d^{-\Delta}\big)^bp^{2b}$ as desired.
 \end{claimproof}
 
 Suppose now that the good event of Claim~\ref{cl:psr:Vsum} holds. For each $i\in[r]$, there may be exceptional vertices $v\in V_i$ which have more than $(p+\eps p)|V_j|$ neighbours in $V_j$ for some $j$ such that $ij\in R'$, but by $\NS(\eps,r_1,\Delta+1)$ there are at most $\DeltaRp\eps p^\Delta n/r_1\le\eps'p^\Delta|V_i|$ such vertices. If $v$ is not exceptional---that is, it has at most $(p+\eps)|V_j|$ neighbours in each $V_j$ with $ij\in R'$---then again there may be vertices $v'\in V_i$ which are exceptional for $v$, i.e.\ vertices $v'$ with more than $(p+\eps p)^2|V_j|$ common neighbours with $v$ in $V_j$ for some $j$ such that $ij\in R'$. But again by $\NS(\eps,r_1,\Delta+1)$ there are at most $\DeltaRp\eps p^\Delta n/r_1\le\eps'p^\Delta|V_i|$ such. Because the good event of Claim~\ref{cl:psr:Vsum} holds, for non-exceptional pairs of vertices $v,v'$ we have the bound given in~\ref{rga:psr:Vsum}. Because the good event of Claim~\ref{cl:psr:Vsum} holds a.a.s., property~\ref{rga:psr:Vsum} is a.a.s.\ obtained. This completes the proof of Lemma~\ref{lem:rga:psr}.
\end{proof}

 We note that it is Claim~\ref{cl:psr:lnssum} in this proof which is most responsible for our bijumbledness requirement on $\Gamma$ in the blow-up lemma for bijumbled graphs, Lemma~\ref{lem:psr_main}. Although the proof of this claim may seem quite wasteful---we assume that if one neighbour is embedded badly then potentially all future neighbours can be embedded badly without any further penalty, which seems unreasonable---we were not able to make it work with any weaker condition than $\LNS(\eps,r_1,\Delta)$ for general graphs $H$. Our analysis is in some sense tight for vertices $y$ which fail at step $\pitau(y)$, when the final neighbour $z$ of $y$ coming before $y$ in $\tau$ is embedded. What still seems unreasonable is that all the vertices in $\Gamma$ which have exceptionally high degree to $U_{\tau(z)-1}(y)\cap U$ turn out to be in $A_{\tau(z)-1}(z)$, which is the worst case we are effectively using in our proof. We expect that a more careful analysis, possibly involving some modification to Algorithm~\ref{alg:RGA:psr}, would allow one to improve on this and thus improve on the bijumbledness requirement of Lemma~\ref{lem:psr_main}.

\chapter{Improved bounds for degenerate graphs}
\label{chap:degen}
\section{The RGA lemma and the proof of the blow-up lemma}\label{sec:degen}
In this section we prove Lemma~\ref{lem:degen}. We have already seen most of the ideas in the proofs of Lemmas~\ref{lem:rg_image} and~\ref{lem:psr_main}. We use the same General Setup, and we continue to obtain it using Lemma~\ref{lem:matchreduce}. However, rather than defining an order $\tau$ putting the buffer vertices $\Xbuf$ at the end and their neighbours at the front, as we did in the proofs of Lemmas~\ref{lem:rga} and~\ref{lem:rga:psr}, we are supplied with an order $\tau$ which we will modify only in that we move $\Xbuf$ to the end of the order. In particular, the neighbours of buffer vertices need not appear early in $\tau$. The reason for this is that moving vertices in the order could result in substantially increasing $\pitau(x)$ for some vertices $x$. Moving the vertices of $\Xbuf$ to the end of the order can destroy boundedness of the order, but it turns out that the only vertices which cause boundedness to fail are those in $\Xbuf$, and in the proof the potentially failing conditions are not needed for these vertices. The following definition, which extends $(D,p,m)$-boundedness, accommodates this.
 \begin{definition}[$(D,p,m)$-bounded order]\label{def:Dpm_bdd_orderp} 
  Let~$H$ be a graph given with buffer sets $\tcX=\{\tX_i\}_{i\in[r]}$ and
  a restriction pair~$\cI=\{I_x\}_{x\in V(H)}$ and~$\cJ=\{J_x\}_{x\in V(H)}$.
  Let~$\tX=\bigcup_{i\in[r]}\tX_i$.
  Let~$\tau$ be an ordering of $V(H)$ and $X^e,\Xbuf\subset V(H)$.
  Then~$\tau$ is a \emph{$(D,p,m)$-bounded order} for~$H$, $\tcX$,
  $\cI$ and $\cJ$
  with \emph{exceptional set} $X^e$ and \emph{ignored set} $\Xbuf$ if the following conditions are
  satisfied for each $x\in V(H)$.
  \begin{enumerate}[label=\itmarabp{ORD}]
   \item\label{ordp:Dx} Define \[
     D_x:=\begin{cases}
       D-2 & \text{if there is $yz\in E(H)$ with $y,z\in N_H(x)$ and $\tau(y),\tau(z)>\tau(x)$}\\
       D-1 & \text{else if there is $y\in N_H(x)$ with $\tau(y)>\tau(x)$}\\
       D & \text{otherwise}\,.
     \end{cases}\]
     We have $\pitau(x)\le D_x$, and if $x\in N(\tX)\setminus\Xbuf$ even $\pitau(x)\le D_x-1$. Finally, if $x\in\tX$ we have $\deg(x)\le D$.
    \item\label{ordp:halfD} One of the following holds:
      \begin{itemize}
        \item $x\in X^e\cup \Xbuf$,
        \item $\pitau(x)\le \frac12 D$,
        \item $x$ is not image restricted and every neighbour~$y$ of~$x$
          with $\tau(y)<\tau(x)$ satisfies $\tau(x)-\tau(y)\le p^{\pitau(x)}m$.
      \end{itemize}
    \item\label{ordp:NtX} If $x\in N(\tX)\setminus \Xbuf$ then all but at most $D-1-\max_{z\not\in X^e}\pitau(z)$
      neighbours~$y$ of $x$ with $\tau(y)<\tau(x)$ satisfy
      $\tau(x)-\tau(y)\le p^D m$.
 \end{enumerate}
\end{definition}

 Observe that the only difference between Definitions~\ref{def:Dpm_bdd_order} and~\ref{def:Dpm_bdd_orderp} is the introduction of the ignored set $\Xbuf$, and conditions~\ref{ordp:halfD} and~\ref{ordp:NtX}, and part of~\ref{ordp:Dx}, are changed so that they are trivially satisfied by vertices of the ignored set $\Xbuf$. If the ignored set is empty, then Definition~\ref{def:Dpm_bdd_order} is recovered.

 We will use a modification of Algorithm~\ref{alg:RGA:psr} to perform this embedding. The modification consists of handling the exceptional vertices $X^e$ differently, which allows us to deal with a few vertices $x$ with $\pitau(x)$ significantly larger than normal; such vertices appear in applications. We can show that this algorithm succeeds in embedding all the vertices of $\Xmain$, in the order $\tau$, and that it returns a good partial embedding with the properties detailed in the lemma below with positive probability. Since we will have $\Xmain=V(H)\setminus \Xbuf$, we can complete the proof using Lemma~\ref{lem:completematch} to embed the buffer vertices in much the same way as in the proof of Lemma~\ref{lem:rg_image}.

Note that property $\LCON$, which appears in the following lemma, plays the same r\^ole as $\LNS$ in the proof of Lemma~\ref{lem:rga:psr}. It is easier to work with, and allows for linearly many image restrictions (but it is not in general true in bijumbled graphs).

\begin{lemma}[degenerate RGA lemma]
\label{lem:degenrga}
  We assume the General Setup. Suppose that an exceptional set $X^e$ with $|X^e|\le\tfrac12\eps p^{\max_{x\in X^e}\pitau(x)}n/r_1$ is given. Suppose that $D\ge 1$, and $\tau$ is a $\big(D,p,\tfrac12\eps n/r_1\big)$-bounded order for $H$, $\tcX$, $\cI$ and $\cJ$ with the exceptional set $X^e$ and ignored set $\Xbuf$. Suppose furthermore that the vertices $\Xbuf$ form the final segment of $\tau$. Suppose that $\Gamma$ has  properties $\NS(\eps,r_1,D)$, $\RI(\eps,(\eps_{a,b}),\eps',d,r_1,D)$ and $\LCON(\eps,r_1,D)$. Then there is a good partial embedding $\psiRGA$ of $H$ into $G$ with the following
  properties for each~$i\in[r]$. Let $b$ be such that $\Xbuf_i$ is a degree-$b$
  buffer.
  \begin{enumerate}[label=\itmarab{DRGA}]
    \item\label{degenrga:place} Every vertex in
    $\Xmain_i$ is embedded to
    $\Vmain_i\cup \Vq_i\cup \Vc_i$ by $\psiRGA$, and no vertex in $\Xbuf_i$ is embedded.
    \item\label{degenrga:main} For every set $W\subset V_i$ of size at least
    $\rho|V_i|$, there are at most $\rho|X_i|$
    vertices in $\Xbuf_i$ with fewer than $(dp)^b|W|/2$ candidates in $W$.
    \item\label{degenrga:nobad} Every vertex in $V_i$ is a
    candidate for at least \[4^{-10\Delta^3} 2^{-1000D^2\mu^{-1}\zeta^{-1}d^{-D}} d^{2D^2} \Delta^{-3}\mu p^b|X_i|\]
    vertices of $\Xbuf_i$.
  \end{enumerate}
\end{lemma}

We will prove this lemma in Section~\ref{sec:degenRGA}. Assuming Lemma~\ref{lem:degenrga}, we are in a position to prove Lemma~\ref{lem:degen}. This amounts to a verification that the conditions of Lemma~\ref{lem:degen} suffice to apply Lemmas~\ref{lem:det_Gnp},~\ref{lem:matchreduce},~\ref{lem:completematch} and~\ref{lem:degenrga}.

\begin{proof}[Proof of Lemma~\ref{lem:degen}]
 First we choose constants as follows. Given $\Delta$, $\DeltaRpbl$, $\Delta_J$ and $D$ integers, $\alphabl$, $\zetabl$ and $d>0$, and $\kappabl>1$, we set $\vartheta=0$, $\DeltaRp=8(\Delta+\Delta_J)^{10}\DeltaRpbl$, $\alpha=\tfrac12\alphabl$, $\zeta=\tfrac12\zetabl$ and $\kappa=2\kappabl$. We now choose
 \begin{multline*}
  \mu=\tfrac{\alpha}{20000\kappa\DeltaRp^4\Delta^{10}(\Delta+2)}\,,\quad\rho= \tfrac{\mu^2\zeta^2d^{3D^2}}{10000\kappa(\Delta+1)}4^{-2000\Delta^3\mu^{-1}\zeta^{-1}d^{-D}}\\
  \text{and}\quad\eps'= \tfrac{\mu\zeta d^{\Delta+1}}{1000\Delta^4\kappa4^\Delta}2^{-280(D+1)\mu^{-1}\zeta^{-1}d^{-D}}\,.\\
 \end{multline*}

  Now for input $D$ and $d,\eps'$, Lemma~\ref{lem:det_Gnp} returns constants $(\eps_{a,b})$ and $\eps_{\sublem{det_gnp:randri}}>0$. We set
  \[\eps=\min\big(\eps_{\sublem{det_gnp:randri}},\tfrac{d\eps'}{\kappa}\big)\,.\] 
  
   We let $\epsbl=\tfrac{1}{16}(\Delta+\Delta_J)^{-10}\eps$ and $\rhobl=\tfrac{1}{16}(\Delta+\Delta_J)^{-10}\rho$.
  Now Lemma~\ref{lem:degen} returns $\epsbl$ and $\rhobl$. Given $\ronebl$ we let $r_1=8(\Delta+\Delta_J)^{10}\ronebl$. We choose $C$ sufficiently large for Lemma~\ref{lem:det_Gnp} with input $D$, $d$, $\eps'$, $r_1$ and $\rho$, and such that
 \[C\ge\tfrac{10^8\cdot 4^{10\Delta^3}(\Delta+\Delta_J)^4\kappa r_1^2}{\eps^2\mu\zeta d^{\Delta+1}}2^{280(D+1)\mu^{-1}\zeta^{-1}d^{-D}}\,.\]
 
 Given $p\ge C\big(\tfrac{\log n}{n}\big)^{1/D}$, Lemma~\ref{lem:det_Gnp} states that a.a.s.\ $\Gamma=G(n,p)$ has the properties $\NS\big(\eps,r_1,D\big)$, $\RI\big(\eps,(\eps_{a,b}),\eps',d,r_1,D\big)$, $\CON(\rho,r_1,D)$ and $\LCON\big(\eps,r_1,D\big)$ respectively. From now on we will assume $\Gamma$ is an $n$-vertex graph which satisfies these four properties.
 
 Given a graph $\Rbl$ on $\rbl\le\ronebl$ vertices, and a spanning subgraph $\Rpbl$ with $\Delta(\Rpbl)\le\DeltaRpbl$, and graphs $H$ and $G\subset\Gamma$ with vertex partitions $\cXbl$ and $\cVbl$, families of image restrictions $\cIbl$ and of image restricting vertices $\cJ$, a family of potential buffer vertices $\tcXbl$, and an exceptional set $X^e$, suppose that the conditions of Lemma~\ref{lem:degen} are satisfied. Then Lemma~\ref{lem:matchreduce} gives a graph $R$ on $r\le r_1$ vertices, a spanning subgraph $R'$ with $\Delta(R')\le\DeltaRp$, and $\kappa$-balanced size-compatible partitions $\cX$ and $\cV$ of $H$ and $G$ respectively, each part having size at least $n/(\kappa r_1)$, together with a family $\tcX$ of potential buffer vertices and $\cI$ of image restrictions, subsets $\Xbuf_i$ of $\tX_i$ for each $i\in[r]$, and partitions $V_i=\Vmain_i\dcup\Vq_i\dcup\Vc_i\dcup\Vbuf_i$ for each $i\in[r]$ which satisfy the General Setup.
 
 We now modify the provided $(D,p,\epsbl n/\ronebl)$-bounded order $\taubl$ on $V(H)$ by moving the set $\Xbuf$ to the end of the order, obtaining a new order $\tau$. We claim $\tau$ is $(D,p,\epsbl n/\ronebl)$-bounded for $H$, $\tcXbl$, $\cIbl$ and $\cJ$ with exceptional set $X^e$ and ignored set $\Xbuf$. To see this, observe that if $x\not\in\Xbuf$ then $\pitau(x)\le\pi^{\tau^{BL}}(x)$. It is easy to check from this that vertices not in $\Xbuf$ satisfy all three conditions. If $x\in\Xbuf$ then $\pitau(x)=\deg(x)\le D$ by~\ref{ord:Dx}, since $\Xbuf\subset\tX$, and since $\Xbuf$ is an independent set by~\ref{BUF:dist} it follows that vertices of $\Xbuf$ satisfy~\ref{ordp:Dx}. Finally vertices in $\Xbuf$ trivially satisfy~\ref{ordp:halfD} and~\ref{ordp:NtX}.
 
 Since $\epsbl n/\ronebl=\tfrac12\eps n/r_1$, the order $\tau$ is also  $\big(D,p,\tfrac12\eps n/r_1\big)$-bounded for $H$, $\tcX$, $\cI$ and $\cJ$ with exceptional set $X^e$ and ignored set $\Xbuf$. To see this observe that the set $\cI$ and the specific partition $\tcX$ of $\tX$ do not play a r\^ole in Definition~\ref{def:Dpm_bdd_orderp} (the former comes together with the set $\cJ$, which does play a r\^ole, and the latter is important only in that it defines $\tX$).
 
 \smallskip
  
 We let $\Xmain=V(H)\setminus\Xbuf$. We now begin the embedding of $H$ into $G$. 
 By Lemma~\ref{lem:degenrga}, there is a good partial embedding $\psiRGA$ with the properties stated in that lemma. 
 By~\ref{GPE:sizeC}, condition~\ref{cpm:deg} of Lemma~\ref{lem:completematch} is satisfied, while 
  condition~\ref{cpm:pseud} holds by~\ref{degenrga:main} and~\ref{cpm:cands} follows from~\ref{degenrga:nobad}  with
 \[\delta=4^{-10\Delta^3} 2^{-1000D^2\mu^{-1}\zeta^{-1}d^{-D}} d^{2D^2} \Delta^{-3}\mu\,.\] 
Therefore, we can find  an embedding $\psi'$ extending $\psiRGA$ which embeds $\Xbuf_i$ to $V_i\setminus\im(\psiRGA)$. Repeating this for each $i\in[r]$, which we may do since $\Xbuf$ is an independent set in $H$, we obtain the desired embedding of $H$ into $G$.
\end{proof}

\section{Proof of the degenerate graph RGA lemma}
\label{sec:degenRGA}
 We prove Lemma~\ref{lem:degenrga} by analysing Algorithm~\ref{alg:RGA:deg}. The analysis is quite similar to what we saw before in the proofs of Lemmas~\ref{lem:rga} and~\ref{lem:rga:psr} (indeed, the main difference is that we are more careful to bound powers of $p$ using $D$ rather than just $\Delta$), so we will be brief and highlight the differences. The only difference between Algorithm~\ref{alg:RGA:psr} and Algorithm~\ref{alg:RGA:deg} is that vertices of $X^e$ are embedded into $\Vc$ rather than $\Vmain$ or $\Vq$.

  \begin{algorithm}[t]
    \caption{Random greedy algorithm for degenerate graphs}\label{alg:RGA:deg}
    \SetKwInOut{Input}{Input}
    \Input{$G\subseteq \Gamma$ and $H$ with partitions satisfying the General Setup; an ordering $\tau$ on $\Xmain$}
    $t:=0$ \; 
    $\psi_0:=\emptyset$ \; 
    $Q_0:=\{x\in V(H):|I_x|<\tfrac12\mu(d-\eps)^{|J_x|}p^{|J_x|}|\Vmain(x)|\}$ \; 
    \Repeat{ $\dom(\psi_t)=\Xmain$ }{
      let $x\in \Xmain\setminus \dom(\psi_t)$ be the next vertex in
      the order $\tau$ \;
      \If{$x\in Q_t\setminus X^e$ \emph{and} $|\Aq_t(x)\setminus B_t(x)|<\tfrac{1}{10}\mu\zeta(dp)^{\pi^*_t(x)}|V(x)|$}{halt with failure \;}   
      \If{$x\in Q_t\cap X^e$ \emph{and} $|\Ac_t(x)\setminus B_t(x)|<\tfrac{1}{10}\mu\zeta(dp)^{\pi^*_t(x)}|V(x)|$}{halt with failure \;}
      choose $v$ uniformly at random in $\begin{cases}
	\Amain_t(x)\setminus B_t(x) \quad & \text{if $x\not\in (Q_t\cup X^e)\,,$} \\
	\Aq_t(x)\setminus B_t(x) & \text{if $x\in Q_t\setminus X^e\,,$} \\
	\Ac_t(x)\setminus B_t(x) & \text{if $x\in X^e\,;$} 
      \end{cases}$ \\
      $\psi_{t+1}:=\psi_t\cup\{x\to v\}$ \;
      $Q_{t+1}:=Q_t$ \;
      \ForAll {$y\in \Xmain\setminus \dom(\psi_{t+1})$}{
        \If{$(\big|\Amain_{t+1}(y)\big|<\tfrac12
        \mu(d-\eps')^{\pi^*_{t+1}(y)}p^{\pi^*_{t+1}(y)}|\Vmain(y)|)$ }
        {$Q_{t+1}:=Q_{t+1}\cup\{y\}$ \; } } 
      $t:=t+1$\;
    }
    $t_\RGend:=t$\;
 \end{algorithm}

 We will see that~\ref{ordp:Dx} (see Definition~\ref{def:Dpm_bdd_orderp}
 of $(D,p,m)$-bounded order) is what we need to make
 Lemma~\ref{lem:fewbad} work with properties $\NS(\eps,r_1,D)$ and
 $\RI(\eps,(\eps_{a,b}),\eps',d,r_1,D)$, so allowing us to prove that
 $B_t(x)$ is always much smaller than $A_t(x)$. As in the proof of
 Lemma~\ref{lem:rga:psr}, our first task is to show that the algorithm
 a.a.s.\ completes successfully. Again by
 Lemma~\ref{lem:rga:welldistr} we can show that the queue remains
 small. We can also show, using~\ref{ordp:halfD}, that all vertices
 $x\in X_i\setminus X^e$ which enter the queue have
 $\pitau(x)\le D/2$. We embed the vertices of $X^e\cap X_i$ which
 enter the queue greedily into $\Vc_i$, and there are so few such
 vertices that this is guaranteed to succeed.  Property
 $\LCON(\eps,r_1,D)$ turns out to be what we need to verify that the
 queue embedding of the remaining vertices a.a.s.\ is successful,
 using the same strategy as in the proof of Lemma~\ref{lem:rga:psr}.
 
 At this stage we have~\ref{degenrga:place} simply because the
 algorithm completes, while \ref{degenrga:main} follows from
 Lemma~\ref{lem:rga:welldistr}. It remains to
 prove~\ref{degenrga:nobad}, which is where we need to
 use~\ref{ordp:NtX}. Here we deviate from the strategy we saw
 previously. We can no longer assume that neighbours of buffer
 vertices appear early in $\tau$, and thus we have to prove that for
 any $v\in V(G)$, even towards the end of the embedding, it is still
 reasonably likely that neighbours of buffer vertices are embedded to
 $N_G(v)$. We will see (Claim~\ref{cl:degen:nofill} below) that
 properties $\NS(\eps,r_1,D)$, $\RI(\eps,(\eps_{a,b}),\eps',d,r_1,D)$
 and $\LCON(\eps,r_1,D)$ allow us to show that a.a.s.\ given any
 $\ell\le D-\max_{x\in\Xmain\setminus X^e}\pitau(x)$ vertices
 $v_1,\ldots,v_\ell$ of $G$ such that
 $\comN_G(v_1,\ldots,v_\ell;\Vmain_i)$ is not small, the set
 $\comN_G(v_1,\ldots,v_\ell;\Vmain_i)$ is never completely filled by
 $\im(\psi_t)$. The idea of the proof remains similar to that of
 Lemma~\ref{lem:large_nbs}. We show, as there, that when the first
 small fraction of the vertices in the order $\tau$ are embedded, only
 at most half of $\comN_G(v_1,\ldots,v_\ell;\Vmain_i)$ is covered. But
 then we repeat this, showing that the next small fraction of $\tau$
 covers only at most half of what remains, and so on, so that when all
 vertices are embedded what remains is an exponentially small, but
 bounded away from zero, fraction of the original
 $\comN_G(v_1,\ldots,v_\ell;\Vmain_i)$. Once we have shown this,
 completing the proof that $N_H(x)$ is not too unlikely to be embedded
 to $N_G(v)$ for any $x\in\Xbuf$ and $v\in V(x)$ can be done along
 similar lines to the proof of Lemma~\ref{lem:nonclique_buffer}.
 
 \begin{proof}[Proof of Lemma~\ref{lem:degenrga}]
 We require
 \begin{align*}
  \mu\le\tfrac{1}{8}\,,\,\, \rho\le \frac{\mu^2\zeta^2d^{2D}}{10000(\Delta+1)}\,,\,\, \eps'\le \frac{\mu\zeta d^{\Delta+1}}{1000\Delta^2\kappa4^\Delta}2^{-280(D+1)\mu^{-1}\zeta^{-1}d^{-D}}\,,\,\, \eps\le\tfrac{d\eps'}{\kappa}\\
  \text{and}\quad p\ge \tfrac{10000\cdot 4^{10\Delta^3}\Delta^3\kappa r_1^2}{\eps\mu\zeta d^{\Delta+1}}2^{280(D+1)\mu^{-1}\zeta^{-1}d^{-D}}\big(\tfrac{\log n}{n}\big)^{1/D}\,.
 \end{align*}
 
 We run Algorithm~\ref{alg:RGA:deg}, using the order $\tau$ supplied to Lemma~\ref{lem:degenrga}. We claim that we can apply Lemma~\ref{lem:fewbad} to bound $B_{\tau(x)-1}(x)$. To see this, we verify properties~\ref{itm:fewbad:x}--\ref{itm:fewbad:xyyzxz} of Lemma~\ref{lem:fewbad}. We consider three cases.
 
 If $x$ has no neighbours after $\tau(x)$ in $\tau$, then Lemma~\ref{lem:fewbad} states that $B_{\tau(x)-1}(x)=\emptyset$ without requiring any property of $\Gamma$.
 
 If $x$ has a neighbour $y$ with $\tau(x)<\tau(y)$, then~\ref{ordp:Dx} states that $\pitau(x)\le D-1$, so in particular $\pi^*_{\tau(x)-1}(x)\le D-1$, satisfying~\ref{itm:fewbad:x}. Furthermore,~\ref{ordp:Dx} states that $\pitau(y)\le D$, so $\pi^*_{\tau(x)-1}(y)\le D-1$ since $x$ is not embedded, satisfying~\ref{itm:fewbad:xy}. Now suppose that $xy,yz\in E(H)$ for some unembedded $y$ and $z$. There are two cases to consider. First, if $\tau(y)<\tau(z)$ then~\ref{ordp:Dx} states that $\pitau(y)\le D-1$ and $\pitau(z)\le D$. Since $x$ and $y$ are unembedded at time ${\tau(x)-1}$, we thus have $\pi^*_{\tau(x)-1}(y)\le D-2$ and $\pi^*_{\tau(x)-1}(z)\le D-1$, as required by~\ref{itm:fewbad:xyyz}. Second, if $\tau(y)>\tau(z)$, then $\pi^*_{\tau(x)-1}(z)\le\pitau(z)\le D-1$ by~\ref{ordp:Dx}, while since $\pitau(y)\le D$ and both the vertices $x$ and $z$ are unembedded, we have $\pi^*_{\tau(x)-1}(y)\le D-2$, again as required by~\ref{itm:fewbad:xyyz}.
 
 Lastly, if there are unembedded vertices $y$ and $z$ such that $xy,yz,xz$ are all edges of $H$, then we claim $\pi_{\tau(x)-1}^*(x),\pi_{\tau(x)-1}^*(y),\pi_{\tau(x)-1}^*(z)\le D-2$. To see this, suppose $\tau(y)<\tau(z)$. By~\ref{ordp:Dx} we have $\pitau(x)\le D-2$, $\pitau(y)\le D-1$ and $\pitau(z)\le D$. Since $x$ and $y$ are unembedded, we have $\pi^*_{\tau(x)-1}(y)\le D-1-1=D-2$, and $\pi^*_{\tau(x)-1}(z)\le D-2$. Thus condition~\ref{itm:fewbad:xyyzxz} is satisfied.
 
 We conclude that we can apply Lemma~\ref{lem:fewbad}, and hence obtain $\big|B_{\tau(x)-1}(x)\big|\le 20\Delta^2\eps' p^{\pitau(x)}|V(x)|$ for each $x\in\Xmain$. 
 Next, we show that certain invariants are maintained during the running of Algorithm~\ref{alg:RGA:deg}. These are identical to the invariants seen in the previous two RGA lemmas, but are repeated here for the reader's convenience.
 
 \begin{claim}\label{cl:degen:inv}
 The following hold at each time $t$ in the running of Algorithm~\ref{alg:RGA:deg}.
 \begin{enumerate}[label=\itmarab{INV}]
   \item\label{degen:inv:gpe} $\psi_t$ is a good partial embedding,
   \item\label{degen:inv:sizeA} Either $|\Amain_t(x)|\ge\frac12\mu(d-\eps')^{\pi^*_t(x)}p^{\pi^*_t(x)}|\Vmain(x)|$
   or $x\in Q_t$.
   \item\label{degen:inv:random} When we embed $x$ to create $\psi_{t+1}$, we do so
   uniformly at random into a set of size at least $\frac{1}{10}\mu\zeta(dp)^{\pi^*_t(x)} |V(x)|$.
 \end{enumerate}
\end{claim}
\begin{claimproof} 
 Algorithm~\ref{alg:RGA:deg} maintains~\ref{degen:inv:gpe} and
 \ref{degen:inv:sizeA} by definition. Since
 \[\tfrac12\mu(d-\eps')^{\pitau(x)}p^{\pitau(x)}-20\Delta^2\eps'p^{\pitau(x)}>\tfrac{1}{10}\mu\zeta
 d^{\pitau(x)}p^{\pitau(x)}\]
 by choice of $\eps'$ for each $x\in\Xmain$, by Lemma~\ref{lem:fewbad} it also maintains~\ref{degen:inv:random}.
\end{claimproof}
  
 As in the previous proofs, the conditions of Lemma~\ref{lem:rga:welldistr} are met, so we conclude that with probability at least $1-2^{-n/(\kappa r_1)}r_1$, for each $i\in[r]$ and time $t$ when Algorithm~\ref{alg:RGA:deg} is running, we have $\big|Q_t\cap X_i\big|<\rho|X_i|+|X^*_i|\le 2\rho|X_i|$ by~\ref{PtH:image}.
 
 We now want to show that Algorithm~\ref{alg:RGA:deg} a.a.s.\ does not halt with failure. Since we know that $\big|B_{\tau(x)-1}(x)\big|<20\Delta^2\eps' p^{\pitau(x)}|V(x)|$, by~\ref{GPE:sizeC} and~\ref{G:restr} it is enough to show that a.a.s.\ we have
 \begin{align}
  \label{eq:qnotsmall}\big|\Aq_{\tau(x)-1}(x)\big|&\ge\tfrac12\big|\Cq_{\tau(x)-1}(x)\big|\quad\text{ for }\quad x\in\Xmain\cap Q_{\tau(x)-1}\setminus X^e\,,\quad\text{ and}\\
  \nonumber\big|\Ac_{\tau(x)-1}(x)\big|&\ge\tfrac12\big|\Cc_{\tau(x)-1}(x)\big|\quad\text{ for }\quad x\in \Xmain\cap Q_{\tau(x)-1}\cap X^e\,.
 \end{align}
 The latter of these is easy, since we embed only vertices of $X^e$ into $\Vc$, and $|X^e|\le \tfrac12\eps p^{\max_{x\in X^e}\pitau(x)}n/r_1$, which by~\ref{GPE:sizeC} and choice of $\eps$ is smaller than $\tfrac12\big|\Cc_{\tau(x)-1}(x)\big|$. In the next three claims we show~\eqref{eq:qnotsmall} holds.
 
 First, we show that we need only concern ourselves with $x$ such that $\pitau(x)\le D/2$.
 
 \begin{claim}\label{cl:queuedeg}
  For each $x\in\Xmain\setminus X^e$, if $\pitau(x)>D/2$ then there is no time $t$ such that $x\in Q_t$.
 \end{claim}
 \begin{claimproof}
  We require $\eps'\le \tfrac14d^D\mu$ and $\eps\le\kappa^{-1}\eps'$.
 
  Suppose $\pitau(x)>D/2$ and $x\not\in X^e$. Let $i$ be such that $x\in X_i$. By~\ref{ordp:halfD}, $x$ is not image restricted, and all $\pitau(x)$ neighbours of $x$ which precede $x$ in $\tau$ occur in the 
   $\tfrac12\eps p^{\pitau(x)} n/r_1\le \eps'p^{\pitau(x)}|X_i|$ places of $\tau$ before $x$. Because $x$ is not image restricted, at time $t\le\tau(x)-\eps'p^{\pitau(x)}|X_i|$ we have $\Cmain_{t}(x)=\Vmain_i$ for some $i$, and hence $\Amain_{t}(x)=\Vmain_i\setminus\im(\psi_{t})$. Because $|\Xmain_i|=|\Vmain_i|-\mu|V_i|$, it follows that $|\Amain_{t}(x)|\ge \mu|V_i|$, and in particular $x$ does not enter $Q_{t}$ at time $t$.
  
  Now suppose $t=\tau(x)-\eps'p^{\pitau(x)}|X_i|+s$ for some $0\le s\le \eps'p^{\pitau(x)}|X_i|-1$. We claim that 
  \[|\Amain_t(x)|\ge (dp-\eps'p)^{\pi^*_{t}(x)}\mu|V_i|-s\,.\]
  Indeed, for $s=0$ we have $\pi^*_t(x)=0$ and the statement $|\Amain_t(x)\ge\mu|V_i|$ was established above. Now suppose the statement holds for some $s\ge 0$. At time $t=\tau(x)-\eps'p^{\pitau(x)}|X_i|+s$ we embed a vertex $y$ to obtain $\psi_{t+1}$. If $y$ is neither a neighbour of $x$ nor in $X_i$, then $\Amain_t(x)=\Amain_{t+1}(x)$ and the statement holds for $s+1$. If $y$ is a neighbour of $x$, then we embed $y$ to a vertex of $A_t(y)\setminus B_t(y)$, and in particular to a vertex $v$ such that $\deg_G\big(v;\Amain_t(x)\big)\ge(d-\eps')p|\Amain_t(x)|$. We conclude that
  \[|\Amain_{t+1}(x)|\ge (d-\eps')p\big((dp-\eps'p)^{\pi^*_{t}(x)}\mu|V_i|-s\big)> (dp-\eps'p)^{\pi^*_{t+1}(x)}\mu|V_i|-s-1\,.\]
  as desired. Finally, if $y\in X_i$ then we have $\Amain_{t+1}(x)=\Amain_t(x)\setminus\{y\}$, and again 
  \[|\Amain_{t+1}(x)|\ge (dp-\eps'p)^{\pi^*_{t+1}(x)}\mu|V_i|-s-1\,.\]
  Thus for each $\tau(x)-\eps'p^{\pitau(x)}|X_i|\le t\le\tau(x)-1$ we have
  \[|\Amain_{t}(x)|\ge (dp-\eps'p)^{\pi^*_{t}(x)}\mu|V_i|-\eps'p^{\pitau(x)}|X_i|\ge\tfrac12(dp-\eps'p)^{\pi^*_{t}(x)}\mu|V_i|\,,\]
  where the final inequality is by choice of $\eps'$. It follows that $x$ does not satisfy the condition to enter $Q_t$. Since we embed $x$ at time $\tau(x)-1$, we conclude that as desired there is no time $t$ such that $x\in Q_t$.
 \end{claimproof}

 As in the proof of Lemma~\ref{lem:rga:psr} the next step is to establish a `sum condition'. In fact, we establish almost the same inequality as~\eqref{eq:randmatch:sumVqLNS} in Claim~\ref{cl:psr:lnssum} (the right hand side is identical, while the left hand side sum runs over all vertices outside $X^e$, including those $y$ with $J_y\neq\emptyset$), though we use a different analysis to do so.
 
 \begin{claim}\label{cl:lconsum}
  Suppose that $\Gamma$ has $\LCON(\eps,r_1,D)$. Then the following statement holds.
  
  For any $t$, any $1\le \ell \le D/2$, any $i\in[r]$ and any vertices $v_1,\ldots,v_\ell$ such that
   \[\deg_\Gamma(v_1,\ldots,v_\ell;V_i)\ge\eps'p^\ell|V_i|\]
   we have
   \begin{equation}\label{eq:sumVqLCON}
    \sum_{\substack{y\in\Xmain_i\cap Q_t\setminus X^e\colon\\ \tau(y)\le t}}\hspace{-1em} \frac{\big|U_{\tau(y)-1}(y)\cap \comN_\Gamma(v_1,\ldots,v_\ell;V_i)\big|}{\big|\Aq_{\tau(y)-1}(y)\setminus B_{\tau(y)-1}(y)\big|}\le\frac{\mu\zeta d^\Delta \deg_\Gamma(v_1,\ldots,v_\ell;V_i)}{20}\,.
   \end{equation}  
 \end{claim}

 The proof of this amounts to checking that the images under our embedding of the $N_H(y)$ for $y\in\Xmain_i\cap Q_t\setminus X^e$ form a family of sets to which we can apply the $\LCON$ property. This property gives us the desired inequality.
 
 \begin{claimproof}
  We require $\rho\le \tfrac{\mu^2\zeta^2 d^{D+\Delta}}{10000(\Delta+1)}$, $\eps'\le \mu\zeta(d/4)^D$ and $\eps\le\tfrac{\eps'}{\kappa}$. Since the left hand side of~\eqref{eq:sumVqLCON} is increasing in $t$, we may assume $t$ is the time at which Algorithm~\ref{alg:RGA:deg} ends.
 
  Given $1\le\ell\le D/2$ and $i\in[r]$, let $v_1,\ldots,v_\ell$ be vertices such that $U:=\comN_\Gamma(v_1,\ldots,v_\ell;V_i)$ has size at least $\eps'p^\ell|V_i|$. By choice of $\eps$ and $\eps'$, we have $|U|\ge\eps p^\ell n/r_1$. It follows that we can apply $\LCON(\eps,r_1,D)$ to bound the number of edges in $\AG(\Gamma,U,\mathcal{F})$ for any $1\le j\le D/2$ and family $\cF$ of pairwise disjoint $j$-sets in $V(\Gamma)\setminus U$. We set
  \[\mathcal{F}_j:=\big\{\Pi_{\tau(x)-1}(x)\cup J_x:x\in \Xmain_i\cap Q_t\setminus X^e, \pitau(x)=j,\tau(x)\le t\big\}\,.\]
  Observe that $\mathcal{F}_j$ is indeed a family of $j$-sets by definition of $\pitau(x)$, that these $j$-sets are pairwise disjoint because no vertex of $H$ has more than one neighbour in $X_i$ by~\ref{PtH:dist}, and element $F$ of $\cF_j$ intersects  $U\subset V_i$ because these sets $F$ are images  under $\psi$ of neighbours of vertices in $X_i$. Because $\eps|U|,|\Xmain_i\cap Q_t|\le 2\rho|X_i|$, by choice of $\eps$ and by $\LCON(\eps,r_1,D)$ we have
  \[e\big(\AG(\Gamma,U,\mathcal{F}_j)\big)\le 7p^j |U|(2\rho |X_i|)=14\rho p^j |U| |X_i|\,.\]
  By definition of $U_t(x)$ (see Section~\ref{sec:gpe}) we have the following equality
  \[e\big(\AG(\Gamma,U,\mathcal{F}_j)\big)=\sum_{x\in\Xmain_i\cap Q_t\setminus X^e:\pitau(x)=j,\tau(x)\le t} |U_{\tau(x)-1}(x)\cap U|\]
  and hence
  \[\sum_{x\in\Xmain_i\cap Q_t\setminus X^e:\pitau(x)=j,\tau(x)\le t}|U_{\tau(x)-1}(x)\cap U|\le 14\rho p^j|U||X_i|\]
  for each $1\le j\le D/2$. Observe that the $j=0$ case of the same inequality is trivially true (indeed with $14$ replaced by $2$, though we will not use this): each vertex of $\Xmain_i\cap Q_t$ contributes at most $|U|$ to the sum. Now by~\ref{degen:inv:random} we have
  \[
  \big|\Aq_{\tau(x)}(x)\setminus B_{\tau(x)}(x)\big|\ge\tfrac{1}{10}\mu\zeta(dp)^{\pitau(x)}|V_i|\]
  for each $x\in\Xmain_i\cap Q_t\setminus X^e$ with $\tau(x)\le t$. Finally, since any vertex $x$ in $\Xmain_i\cap Q_t\setminus X^e$ satisfies 
  $\pitau(x)\le D/2$ by Claim~\ref{cl:queuedeg}, we conclude that
  \begin{align*}
  \sum_{\substack{x\in\Xmain_i\cap Q_t\setminus X^e:\\ \tau(x)\le t}}\frac{|\Uq_{\tau(x)}(x)\cap U|}{|\Aq_{\tau(x)}(x)\setminus B_{\tau(x)}(x)\big|}&
  \le \sum_{j=0}^{D/2}\sum_{\substack{x\in\Xmain_i\cap Q_t\setminus X^e:\\ 
  \pitau(x)=j,\tau(x)\le t}}\frac{|\Uq_{\tau(x)}(x)\cap U|}{\tfrac{1}{10}\mu\zeta(dp)^{j}|V_i|}
  \\ \le \sum_{j=0}^{D/2}\frac{14\rho p^j|U||X_i|}{\tfrac{1}{10}\mu\zeta(dp)^{j}|V_i|}
  &\le\frac{14(1+D/2)\rho|U||X_i|}{\tfrac{1}{10}\mu\zeta d^{D/2}|V_i|}\le \tfrac{\mu\zeta d^\Delta}{20}|U|,
  \end{align*}
  where the final inequality is by choice of $\rho$, as desired.
 \end{claimproof}
 
 We can now complete the proof that candidate sets in $\Vq$ are not substantially covered by $\im(\psi_t)$ for any $t$.
 \begin{claim}\label{cl:rga:nofilldeg} A.a.s.\ for each $x\in\Xmain\setminus X^e$, at each time $t\le\tau(x)-1$ and before the termination of Algorithm~\ref{alg:RGA:deg}, we have $\big|\Cq_t(x)\cap\im(\psi_t)\big|<\tfrac12\big|\Cq_t(x)\big|$.
 \end{claim}
 \begin{claimproof}
 The proof of Claim~\ref{cl:rga:nofill}, if we  replace $\Xmain$ with $\Xmain\setminus X^e$, and replace~\eqref{eq:randmatch:sumVqLNS} with~\eqref{eq:sumVqLCON}, without handling image restricted vertices specially (since  condition~\eqref{eq:sumVqLCON} deals with them), gives this claim verbatim. 
 \end{claimproof}
 
 By the same logic as in the proof of Lemma~\ref{lem:rga:psr}, if the good event of Claim~\ref{cl:rga:nofilldeg} holds, it follows that candidate sets in $\Vq$ are not significantly covered by $\im(\psi_t)$, i.e.\ the inequality~\eqref{eq:qnotsmall} holds. This implies that Algorithm~\ref{alg:RGA:deg} a.a.s.\ completes successfully. Again, the successful running implies~\ref{degenrga:place}, and the good event of Lemma~\ref{lem:rga:welldistr} holding implies~\ref{degenrga:main}.
 
 We would like to emphasise at this point that we have established all
 the desired conclusions of Lemma~\ref{lem:degenrga}
 except~\ref{degenrga:nobad}, and we have so far not used
 condition~\ref{ordp:NtX}, and not used the second part of~\ref{ordp:Dx}
 (which states $\pitau(x)\le D_x-1$ for $x\in N(\tX)\setminus\Xbuf$ ). It follows
 that one can establish a version of Lemma~\ref{lem:degenrga}
 omitting~\ref{degenrga:nobad} given an order $\tau$ which need only
 satisfy the first part of~\ref{ordp:Dx} (which states
 $\pitau(x)\le D_x$ for all $x\in V(H)$ ) and~\ref{ordp:halfD}. We will
 return to this point in Section~\ref{sec:concl:degen}.
 
 We now turn to proving~\ref{degenrga:nobad} in the next two claims. First, in Claim~\ref{cl:degen:nofill}, we show that $G$-common neighbourhoods of at most $D-\max_{x\in\Xmain\setminus X^e}\pitau(x)$ 
 vertices which are large do not get completely filled at any time in the running of Algorithm~\ref{alg:RGA:deg}. Note that the upper bound on $\ell$ in Claim~\ref{cl:degen:nofill} comes from~\ref{ordp:NtX}. What the good event of Claim~\ref{cl:degen:nofill}, together with~\ref{ordp:NtX}, says is the following. If $v=v_1\in V_i$ is any vertex (which we eventually want to show is candidate for many vertices of $\Xbuf_i$) and $v_2,\dots,v_\ell$ are the embedded neighbours of some $x\in \Xbuf_i$, then the $G$-common neighbourhood of $v_1,\dots,v_\ell$ is not covered by $\im(\psi_t)$ provided $t\le\tau(x)-\tfrac12\eps p^Dn/r_1$. We will see that neighbours embedded after this time can be dealt with easily. Claim~\ref{cl:degen:nofill} has two parts,~\ref{degen:nofill:a} dealing with candidate sets in $\Vmain$ and~\ref{degen:nofill:b} dealing with candidate sets in $\Vq$. The reason for this is that neighbours of a buffer vertex $x$ may end up in the queue, and we need to show that even in this case $v$ is reasonably likely to be a candidate for $x$.
  This claim replaces Lemma~\ref{lem:large_nbs}, which we used in the proofs of the previous two RGA lemmas.
 
 \begin{claim}\label{cl:degen:nofill}
  Suppose that $\Gamma$ has $\LCON(\eps,r_1,D)$. Then a.a.s.\ the following holds at each time $t$. Given any $i\in[r]$, any $1\le \ell\le D-\max_{x\in\Xmain\setminus X^e}\pitau(x)$, and any vertices $v_1,\ldots,v_{\ell}$ of $G$, the following hold.
  \begin{enumerate}[label=\abc]
   \item\label{degen:nofill:a} If $\deg_G(v_1,\ldots,v_\ell;\Vmain_i)\ge (dp/2)^\ell|\Vmain_i|$, then
    \begin{equation}\label{eq:degenrga:bufnofill}\begin{split}
     \big|\comN_G(v_1,\ldots,v_\ell;\Vmain_i)&\setminus\im(\psi_t)\big|\ge\\
     (1-4\mu)&2^{-280(D+1)\mu^{-1}\zeta^{-1} d^{-D}}\deg_G(v_1,\ldots,v_\ell;\Vmain_i)\,.
    \end{split}\end{equation}
   \item\label{degen:nofill:b} If $\deg_G(v_1,\ldots,v_\ell;\Vq_i)\ge (dp/2)^\ell|\Vq_i|$, then
    \begin{equation}\label{eq:degenrga:queuenofill}
     \big|\comN_G(v_1,\ldots,v_\ell;\Vq_i)\setminus\im(\psi_t)\big|\ge\tfrac12\deg_G(v_1,\ldots,v_\ell;\Vq_i)\,.
    \end{equation}
   \end{enumerate}
 \end{claim}

 The proof of~\eqref{eq:degenrga:bufnofill}, as discussed at the beginning of this section, roughly amounts to repeating the argument used to prove Lemma~\ref{lem:large_nbs} several times. We show that, for some small constant $\eta$, each successive interval consisting of an $\eta$-fraction of the vertices $\Xmain_i$ is likely to cover less than half of whatever of $\comN_G(v_1,\dots,v_\ell;\Vmain_i)$ was uncovered before embedding that interval. The argument for each given interval is similar to that seen in Lemma~\ref{lem:large_nbs}, though here we take a short-cut by using property $\LCON$ rather than $\NS$, which simplifies the calculations. 
 
 \begin{claimproof}
  Let $h=\max_{x\in\Xmain\setminus X^e}\pitau(x)$. 
  We require
  \begin{align*}
   \mu&\le\tfrac18\,,\quad &\rho&\le\frac{\mu^2\zeta d^D}{1000(D+1)}\,,\\
   \eps&<\tfrac12 2^{-280(D+1)\mu^{-1}\zeta^{-1} d^{-D}}\big(\tfrac{d}{2}\big)^{D+1}\quad\text{and}\quad &p&\ge \tfrac{1000\Delta^2}{\eps\mu\zeta d^D}\big(\tfrac{\log n}{n}\big)^{1/D}\,.
  \end{align*}
  We first prove inequality~\eqref{eq:degenrga:bufnofill}. Set $\eta=\frac{\mu\zeta d^D}{280(D+1)}$.  
  Given $i\in[r]$ and $1\le\ell\le D-h$, suppose that $v_1,\ldots,v_\ell\in V(\Gamma)$ are such that $\comN_G(v_1,\ldots,v_\ell;\Vmain_i)\ge(dp/2)^\ell|\Vmain_i|$. We split the vertices of $\Xmain_i\setminus X^e$ into intervals $\Int_1,\ldots,\Int_{1/\eta}$ of equal size, with the first being the first $\eta |\Xmain_i\setminus X^e|$ vertices in the order $\tau$, and so on.
  
  We now aim to show the following.
  \begin{enumerate}[label=(\ddag)]
   \item\label{itm:degen:dag} For any fixed set $U\subseteq V_i$ of size at least $\eps p^\ell |V_i|$ and any $1\le j\le 1/\eta$, conditioning on the embedding up to the last vertex of $\Int_{j-1}$, with probability at least $1-n^{-D-1}$ there are at most $\tfrac12|U|$ vertices of $\Int_j$ are embedded to $U$.
  \end{enumerate}
  To that end, for each $x\in\Int_j$ let $\hist_{x,j}$ denote the history up to but not including the embedding of $x$. By definition of $h$, all vertices $x\in\Int_j$ have $\pitau(x)\le h$, so we split the vertices $x$ of $\Int_j$ up into $h+1$ classes according to $\pitau(x)$. We apply $\LCON(\eps,r_1,D)$ with the set $U$ and the family $\cF_s$ of $s$-sets given by the embedded images of $N_H(x)$ for $x\in\Int_j$ with $\pitau(x)=s$. Since $\eps|U|,|\cF_s|\le\eta|\Xmain_i|$, we  obtain the inequality
  \[\sum_{x\in\Int_j:\pitau(x)=s}\big|U_{\tau(x)-1}(x)\cap U\big|
   \le 7p^s |U| \eta |\Xmain_i|\,.\]
  Now each $x\in\Int_j$ with $\pitau(x)=s$ is by~\ref{degen:inv:random} embedded uniformly at random into a set of size at least $\tfrac{1}{10}\mu\zeta(dp)^s|V_i|$. We thus have
  \[\sum_{x\in\Int_j}\Pr\big(x\text{ is embedded to }U\big|\hist_{x,j}\big)
   \le (h+1)\frac{70\eta |U|}{\mu\zeta d^{h}}\le(D+1)\frac{70\eta |U|}{\mu\zeta d^{D}}\]
  and by Lemma~\ref{lem:coupling} with $\delta=1$ we see that the probability that more than
  \[140(D+1)\eta\mu^{-1}\zeta^{-1}d^{-D}|U|\le\tfrac12|U|\]
  of the vertices in $\Int_j$ are embedded to $U$ is at most
  \[\exp\left(-(D+1)\tfrac{70\eta |U|}{3\mu\zeta d^D}\right)\le n^{-D-1}\,,\]
  where the inequality is because $|U|\ge \eps p^{D-h}|V_i|$ and by choice of $p$. This completes the proof of~\ref{itm:degen:dag}.
  
  Let $U_0=\comN_G(v_1,\ldots,v_\ell;\Vmain_i)$, and $U_j=U_{j-1}\setminus\psi_{t_{j+1}}(\Int_j)$ for each $j\ge 1$. In the event that $|U_j|\ge\tfrac12|U_{j-1}|$ for each $j$, we have
  \[|U_j|\ge 2^{-1/\eta}(d/2)^\ell p^\ell |\Vmain_i|\ge\eps p^\ell |V_i|\]
  for each $j$, where the final inequality is by choice of $\eps$. In other words, each $U_{j-1}$ is large enough to play the r\^ole of $U$ in~\ref{itm:degen:dag}, so taking a union bound over the $1/\eta$ choices of $j$, we see that with probability at least $1-\eta^{-1}n^{-D-1}$ we have
  \[\big|\comN_G(v_1,\ldots,v_\ell;\Vmain_i)\setminus\im(\psi_{t_j})\big|\ge(1-4\mu)2^{1-j}\deg_G(v_1,\ldots,v_\ell;\Vmain_i)\]
  for each $1\le j\le 1/\eta+1$. Taking a union bound over the at most $r_1$ choices of $i\in[r]$ and $n+n^2+\dots+n^{D}$ choices of $v_1,\ldots,v_\ell$ we see that the desired event of the lemma holds with probability at least $1-r_1D\eta^{-1}n^{-1}$, which tends to one as $n$ tends to infinity. This completes the proof of~\eqref{eq:degenrga:bufnofill}.
  
  The proof of~\eqref{eq:degenrga:queuenofill} is simpler. Observe that $|\Xq_i|\le2\rho|X_i|<\eta|\Xmain_i|$, and $\deg_G(v_1,\ldots,v_\ell;\Vq_i)\ge(dp/2)^\ell|\Vq_i|\ge\eps p^\ell |V_i|$ so that the analysis above, with only one `interval' $\Xq_i\setminus X^e$, gives the desired result.
 \end{claimproof}
 
 Finally, we are in a position to prove~\ref{degenrga:nobad}, which we do in the following claim. The proof is quite similar to that of Lemma~\ref{lem:nonclique_buffer}, which amounts to showing that any vertex $v\in V_i$ is likely to be a candidate vertex for `reasonably' many of the buffer vertices $\Xbuf_i$. More precisely, our strategy is to fix a vertex $v\in V_i$ and show that the inequality~\eqref{eq:rga:bufdeg} below, which encapsulates these reasonable bounds, holds with sufficiently high probability to apply the union bound over all $v\in V(G)$. 
 To do this, we introduce the concept of an $(x\to v)$-buffer-friendly partial embedding for $x\in\Xbuf_i$. As in the proof of Lemma~\ref{lem:nonclique_buffer}, this is a good partial embedding which has a few extra properties that let us show that it is not too unlikely that $N_H(x)$ is embedded to $N_G(v)$. As there, the proof that it is not too unlikely that at each step Algorithm~\ref{alg:RGA:deg} maintains an $(x\to v)$-buffer-friendly partial embedding is mainly `bookkeeping' and long but not very hard. However, there is an important difference. It is no longer useful to simply multiply the conditional probabilities of maintaining an $(x\to v)$-buffer-friendly partial embedding in order to estimate the conditional probability that $N_H(x)$ is embedded to $N_G(v)$. This is because the vertices $N_H(x)$ are no longer embedded as a segment of $\tau$, so that Lemma~\ref{lem:coupling} is not applicable directly to these products of conditional probabilities. Nevertheless we show below that we can still apply Lemma~\ref{lem:coupling}, several times and with some extra care, to obtain the desired bounds.
 
 \begin{claim}\label{cl:degen:buffer}
  The following holds a.a.s.\ for each $i\in[r]$. Let $b$ be such that $\Xbuf_i$ is a degree-$b$ buffer. Then for each $v\in V_i$ we have 
   \begin{equation}\label{eq:rga:bufdeg}
    \big|\{x\in\Xbuf_i:v\in C_{t_\RGend}(x)\}\big|\ge 4^{-10\Delta^3}2^{-1000D^2\mu^{-1}\zeta^{-1}d^{-D}} d^{2D^2} \Delta^{-3}\mu p^b|X_i|\,.
   \end{equation}
 \end{claim}
 \begin{claimproof}
  We require
  \begin{align*}
   &\mu\le\tfrac18\,,\quad &\eps'&<\tfrac{\mu\zeta d^\Delta}{1000\Delta^2\kappa4^\Delta}\cdot 2^{-280(D+1)\mu^{-1}\zeta^{-1}d^{-D}}\,,\qquad\eps\le\eps'\,,\\
  &\text{and}\quad &p&>10000\cdot 4^{10\Delta^3}\Delta^3\kappa r_1\mu^{-1}d^{-\Delta-1}\cdot 2^{280(D+1)\mu^{-1}\zeta^{-1}d^{-D}}\big(\tfrac{\log n}{n}\big)^{1/D}\,.
  \end{align*}
 
  Let $\beta=(1-4\mu)2^{-280(D+1)\mu^{-1}\zeta^{-1}d^{-D}}$. It is convenient to consider only vertices of $\Xbuf_i$ which are far from $X^e$; since $X^e$ is very small, doing so does not exclude many vertices of $\Xbuf_i$.
  
  Given $x\in\Xbuf_i$ which is at distance greater than three from any vertex of $X^e$, and $v\in V_i$, we say that a good partial embedding $\psi$ is an \emph{$(x\to v)$-buffer-friendly partial embedding} if the following hold.
  \begin{enumerate}[label=\itmarab{BPE}]
   \item\label{BPE:inNG} For each $y\in\dom(\psi)\cap N_H(x)$ we have $\psi(y)\in N_G(v)$.
   \item\label{BPE:sizeU} For each unembedded $y\in N_H(x)$ we have
    \begin{align*}
     |U(y)\cap N_\Gamma(v)|&=(p\pm\eps'p)^{\pi^*(y)}\deg_\Gamma\big(v;V(y)\big)\,,\\
     |\Umain(y)\cap N_\Gamma(v)|&=(p\pm\eps'p)^{\pi^*(y)}\deg_\Gamma\big(v;\Vmain(y)\big)\quad\text{and}\\
     |\Uq(y)\cap N_\Gamma(v)|&=(p\pm\eps'p)^{\pi^*(y)}\deg_\Gamma\big(v;\Vq(y)\big)
    \end{align*}
    \item\label{BPE:sizeC} For each unembedded $y\in N_H(x)$ we have
    \begin{align*}
     |\Cq(y)\cap N_G(v)|&\ge(dp-\eps'p)^{\pi^*(y)}\deg_G\big(v;\Vq(y)\big)\quad\text{and}\\
     |\Cmain(y)\cap N_G(v)|&\ge (dp-\eps'p)^{\pi^*(y)}\deg_G\big(v;\Vmain(y)\big)
    \end{align*}
    \item\label{BPE:tworeg} For each unembedded $y,z\in N_H(x)$ with $yz\in E(H)$, the pair $\big(U(y)\cap N_\Gamma(v),U(z)\cap N_\Gamma(v)\big)$ is $(\eps_{\pi^*(y),\pi^*(z)},d,p)$-regular in $G$.
   \item\label{BPE:onereg} For each unembedded $y\in N_H(x)$ and $z\in N_H(y)$,
     we have that the pair $\big(U(y)\cap N_\Gamma(v),U(z)\big)$ is $(\eps_{\pi^*(y),\pi^*(z)},d,p)$-regular in $G$.
  \end{enumerate}
  
  Much as in the proof of Lemma~\ref{lem:nonclique_buffer}, the empty partial embedding $\psi_0$ is an $(x\to v)$-buffer-friendly partial embedding. Indeed,~\ref{BPE:inNG}--\ref{BPE:sizeC} are trivially true since by~\ref{BUF:last} neighbours of buffer vertices are not image restricted. For~\ref{BPE:tworeg} and~\ref{BPE:onereg}, since buffer vertices are by~\ref{BUF:last} at distance at least three from image restricted vertices the sets $U(y)$ and $U(z)$ are equal to $V(y)$ and $V(z)$ respectively, and the required regularity is thus given by~\ref{G:inh}, since $\Xbuf_i\subset\tX_i$.
  
  Our aim now is to give lower bounds on the probability that, given that at some time $t$ we currently have an $(x\to v)$-buffer-friendly partial embedding, Algorithm~\ref{alg:RGA:deg} embeds the next vertex in a way that maintains an $(x\to v)$-buffer-friendly partial embedding. To that end, given an $(x\to v)$-buffer-friendly partial embedding $\psi$ and an unembedded vertex $y$, we let $P(y)$ be the set of \emph{poor vertices for $y$}, namely those vertices $u\in C(y)$ such that either $\psi\cup\{y\to u\}$ is not an $(x\to v)$-buffer-friendly partial embedding, or such that there is an unembedded $z\in N_H(y)\cap N_H(x)$ such that either of the following conditions hold:
  \begin{equation}\label{eq:BPE_Py}
  \begin{split}
      \deg_G\big(u;\Amain(z)\cap N_G(v)\big)\le&(d-\eps')p\big|\Amain(z)\cap N_G(v)\big| \quad\text{or}\\
      \deg_G\big(u;\Aq(z)\cap N_G(v)\big)\le&(d-\eps')p\big|\Aq(z)\cap N_G(v)\big|\,. 
    \end{split}  
  \end{equation}
   Observe that if $y$ is at distance four or greater from $x$ in $H$, then $P(y)$ is always empty. This has two important consequences. First, we are about to talk about the probability of no vertices being embedded to poor vertices: this is really a condition on at most $3\Delta^3$ vertex embeddings. Second, it means that if $x,x'\in\Xbuf_i$, by~\ref{PtH:dist} the distance between $x$ and $x'$ is at least ten, so that any given vertex embedding affects at most one of having an $(x\to v)$-buffer-friendly partial embedding and having an $(x'\to v)$-buffer-friendly partial embedding.
  
  Before we try to estimate the probability of maintaining an $(x\to v)$-buffer-friendly partial embedding, we strengthen the conclusion of Claim~\ref{cl:degen:nofill} to cover the `last few vertices' before embedding $y\in N_H(x)$.
  \begin{fact}\label{fact:nofill}
   For a given $x\in\Xbuf_i$ at distance greater than three from $X^e$ and a fixed $v\in V(x)$, provided that up to time $t$ no vertex has been embedded to a poor vertex (with respect to an $(x\to v)$-buffer-friendly partial embedding), for each $y\in N_H(x)$ with $\tau(y)>t$, the following a.a.s.\ hold.
   \begin{equation}\label{eq:degen:sizeAs}
   \begin{split}
    \big|\Amain_t(y)\cap N_G(v)\big|&\ge\tfrac\beta2(dp/2)^{\pi^*_t(y)}\deg_G\big(v;\Vmain(y)\big)\quad\text{and}\\
    \big|\Aq_t(y)\cap N_G(v)\big|&\ge\tfrac14(dp/2)^{\pi^*_t(y)}\deg_G\big(v;\Vq(y)\big)\,.
   \end{split}
   \end{equation}
  \end{fact}
  
  We will see that Claim~\ref{cl:degen:nofill}, together with~\ref{ordp:NtX}, show that this fact holds provided $t\le\tau(y)-\tfrac12\eps p^D n/r_1$. The number of vertices embedded in the remaining $\tfrac12\eps p^D n/r_1$ steps is too small to significantly fill up either of these sets, so we need only show that embedding neighbours of $y$ does not adversely affect~\eqref{eq:degen:sizeAs}. This is guaranteed by condition~\eqref{eq:BPE_Py} in the definition of poor vertices.
  
  \begin{claimproof}
   Let $h=\max_{z\in\Xmain\setminus X^e}\pitau(z)$.
   Since $\Xbuf_i\subseteq \tX$ and $y\in N_H(x)$ we infer by~\ref{ordp:NtX} that all 
   but at most $s$ neighbours $z$ of $y$ 
   satisfy $\tau(y)-\tau(z)\le p^D \cdot\tfrac12\eps n/r_1$ for some $s\le D-1-h$ 
   (recall that $\tau$ is a $(D,p,\tfrac12\eps n/r_1)$-bounded order).
   Therefore, if $t\le \tau(y)-\tfrac12\eps p^{D} n/r_1$ then  $\pi^*_t(y)\le s\le  D-1-h$, 
   so that $\Pi^*_t(y)\cup\{v\}$ is by~\ref{BPE:sizeC},~\ref{G:deg} and~\ref{G:sizes} a set of $s+1\le D-h$ vertices in $G$ with at least $(dp/2)^{s+1}|\Vq(y)|$ common neighbours in $\Vq(y)$ and at least $(dp/2)^{s+1}|\Vmain(y)|$ common neighbours in $\Vmain(y)$.
    
   Therefore, by Claim~\ref{cl:degen:nofill}, we have
  \begin{align*}
   \big|\Amain_t(y)\cap N_G(v)\big|&\ge\beta(dp/2)^{\pi^*_t(y)}\deg_G\big(v;\Vmain(y)\big)\quad\text{and}\\
   \big|\Aq_t(y)\cap N_G(v)\big|&\ge\tfrac12(dp/2)^{\pi^*_t(y)}\deg_G\big(v;\Vq(y)\big)\,,
  \end{align*}
  which gives~\eqref{eq:degen:sizeAs} for this range of $t$. Now suppose that 
  \[
  t=\tau(y)-\eps p^{D} \cdot\tfrac12 n/r_1+\ell
  \]
   for some $0\le \ell<\tfrac12\eps p^{D} n/r_1$. We claim that
  \begin{equation}\label{eq:nosmall}\begin{split}
   \big|\Amain_t(y)\cap N_G(v)\big|&\ge\beta(dp/2)^{\pi^*_t(y)}\deg_G\big(v;\Vmain(y)\big)-\ell\quad\text{and}\\
   \big|\Aq_t(y)\cap N_G(v)\big|&\ge\tfrac12(dp/2)^{\pi^*_t(y)}\deg_G\big(v;\Vq(y)\big)-\ell\,.
  \end{split}\end{equation}
  Indeed, we have just established that~\eqref{eq:nosmall} holds for $\ell=0$. Given $\ell\ge 1$, suppose that~\eqref{eq:nosmall} holds for $\ell-1$, and consider the vertex $z$ embedded at time $\tau(z)=\tau(y)-\eps p^{D} \cdot\tfrac12 n/r_1+\ell$.
  
   If $z$ is not a neighbour of $y$, then its embedding decreases the sizes of the sets on the left hand side of~\eqref{eq:nosmall} by at most one, so the inequality continues to hold. If $z$ is a neighbour of $y$, then by~\eqref{eq:BPE_Py} in the definition of $P(z)$ and since, by the assumption of Fact~\ref{fact:nofill} no vertex has been embedded to a poor vertex, we have
  \begin{align*}
   \big|\Aq_{\tau(z)}(y)\cap N_G(v)\big|&\ge(d-\eps')p\big|\Aq_{\tau(z)-1}(y)\cap N_G(v)\big|\\
   &\ge\tfrac12(dp/2)^{\pi^*_{\tau(z)}(y)}\deg_G\big(v;\Vq(y)\big)-(d-\eps')p\ell
  \end{align*}
  where the second inequality is by choice of $\eps'$ and~\eqref{eq:nosmall}. A similar inequality holds for $\big|\Amain_{\tau(z)}(y)\cap N_G(v)\big|$. Thus again~\eqref{eq:nosmall} continues to hold. We conclude that~\eqref{eq:nosmall} holds for all $0\le \ell<\tfrac12\eps p^{D} n/r_1$.
  
  Now since $x\in\Xbuf$, we have $\tau(x)>\tau(y)$, so that by~\ref{ordp:Dx} we have $\pi^*_t(y)\le\pitau(y)\le D-1$. By~\ref{G:deg} and choice of $\eps$, we conclude that~\eqref{eq:degen:sizeAs} holds for all $t<\tau(y)$ as desired.
  \end{claimproof}
  
  We now continue with the proof of Claim~\ref{cl:degen:buffer} by estimating the probability of not choosing a poor vertex at a given step in Algorithm~\ref{alg:RGA:deg}. Recall that $B(y)$ is the set of bad vertices for a given good partial embedding and unembedded $y$, as defined in Definition~\ref{def:bad_vertices}.
  
  \begin{fact}\label{cl:degen:probs}
   Suppose that $\psi$ is a good partial embedding generated together with a queue $Q$ by Algorithm~\ref{alg:RGA:deg} which is also an $(x\to v)$-buffer-friendly partial embedding, where $x$ is at distance greater than three from any vertex of $X^e$, and~\eqref{eq:degen:sizeAs} holds for $\psi$ and all unembedded $y$. Given an unembedded $y$, suppose that $u$ is a vertex chosen uniformly at random in $\Amain(y)\setminus B(y)$, or in $\Aq(y)\setminus B(y)$, or in $\Ac(y)\setminus B(y)$, according to whether $y\not\in Q\cup X^e$, or $y\in Q\setminus X^e$, or $y\in X^e$ respectively.
   Then the following hold.
   \begin{enumerate}[label=\abc]
    \item\label{fact:nofill:a} If $\dist_H(x,y)>3$ then $\Pr\big(u\not\in P(y)\big)=1$.
    \item\label{fact:nofill:b} If $\dist_H(x,y)\in\{2,3\}$ then $\Pr\big(u\not\in P(y)\big)\ge\tfrac12$.
    \item\label{fact:nofill:c} If $\dist_H(x,y)=1$ then $\Pr\big(u\not\in P(y)\big)\ge \beta d^{D-1} 4^{-D}p$.
   \end{enumerate}
  \end{fact}

  The proof of this fact is quite similar to the analysis in Lemma~\ref{lem:nonclique_buffer}. We remark that the use we make of $\psi$ and $Q$ being generated by Algorithm~\ref{alg:RGA:deg} is that the invariants of Claim~\ref{cl:degen:inv} hold; we do not in the following proof perform any analysis of the probabilistic process generating $\psi$ and $Q$. 
  
  \begin{claimproof} 
   If $\dist_H(x,y)>3$ then $P(y)=\emptyset$, so~\ref{fact:nofill:a} is trivial. From now on we assume $\dist_H(x,y)\le 3$.
   
   By~\ref{degen:inv:random} we embed $y$ uniformly at random into a set whose size is at least $\tfrac{1}{10}\mu\zeta(dp)^{\pitau(y)}|V(y)|$. Thus in order to show~\ref{fact:nofill:b} we simply need to establish that $P(y)\cup B(y)$ is small compared to this set; this can be established using $\NS(\eps,r_1,D)$, $\RI(\eps,(\eps_{a,b}),\eps',d,r_1,D)$ and the buffer and good partial embedding properties, as in the proof of Lemma~\ref{lem:nonclique_buffer}. We need to use the fact, given by~\ref{ordp:Dx}, that $\pitau(z)\le D_z-1$ for each $z\in N(\tX)\setminus\Xbuf$, in other words that $\pitau(z)\le D-2$ for all $z\in N_H(x)$ and $\pitau(z)\le D-3$ if there is a triangle $zww'$ in $H$ with $\tau(w),\tau(w')>\tau(z)$. In order to establish that few vertices fail~\eqref{eq:BPE_Py}, we need to know that $\Amain(z)\cap N_G(v)$ and $\Aq(z)\cap N_G(v)$ are not small in comparison to $U(z)\cap N_\Gamma(v)$. This follows from~\eqref{eq:degen:sizeAs} since $z\in N_H(x)$. Since the proof is otherwise as in the proof of Lemma~\ref{lem:nonclique_buffer}, we omit further details.
   
   Finally we come to~\ref{fact:nofill:c}. We embed $y$ uniformly at random into a subset of $U(y)$, which by~\ref{GPE:sizeU} has size at most $2p^{\pitau(y)}|V(y)|$. If $xy\in H$, we need to show that it is not too unlikely that $y$ is embedded to $N_G(v)$. By~\eqref{eq:degen:sizeAs},~\ref{G:deg} and choice of $\beta$, we see that
   \[\big|\Amain(y)\cap N_G(v)\big|,\big|\Aq(y)\cap N_G(v)\big|\ge\tfrac{\beta}{4}(dp/2)^{\pitau(y)}\deg_G(v;V(y))\,.\]
   It is thus enough to show that at most half of these vertices are in $P(y)\cup B(y)$, which again we can do using $\NS(\eps,r_1,D)$, $\RI(\eps,(\eps_{a,b}),\eps',d,r_1,D)$ and the buffer and good partial embedding properties, as in the proof of Lemma~\ref{lem:nonclique_buffer}. Again, we use the fact $\pitau(z)\le D_z-1$ for $z\in N_H(x)$ in this verification. We omit the details.
  \end{claimproof}
  
  We now continue with the proof of Claim~\ref{cl:degen:buffer}. We fix $i$ and $v\in V_i$. Our goal is to show that with very high probability a reasonable fraction of the $x\in\Xbuf_i$ have $N_H(x)$ embedded to $N_G(v)$, since this immediately implies the claim. We have established in Fact~\ref{cl:degen:probs} that for any given vertex $x$ of $\Xbuf_i$ far from $X^e$, the probability that $N_H(x)$ is embedded to $N_G(v)$ is not too small. Fact~\ref{cl:degen:probs} states that this probability is more or less given by the embeddings of vertices at distance at most three from $x$, and distinct vertices of $\Xbuf_i$ are at distance at least ten by~\ref{PtH:dist}, so that no vertex is within distance three of two distinct $x,x'\in\Xbuf_i$. In the proof of Lemma~\ref{lem:nonclique_buffer} at this point we simply apply Lemma~\ref{lem:coupling} to complete the proof. However, here we cannot do this, because the sets $N_H(x)$ interleave each other in the order $\tau$. We have to be a bit more careful, as follows.
  
  
  For each $x\in\Xbuf_i$, let $M_x$ be the set of vertices of $H$ at distance one, two or three from $x$. We say $x$ \emph{survives at step $j$} if after the embedding of the $j$th vertex of $M_x$ in the order $\tau$, we still have an $(x\to v)$-buffer-friendly partial embedding. We will use Lemma~\ref{lem:coupling} to show that a reasonable fraction of $x\in\Xbuf_i$ survive at step $1$. We would like then to use Lemma~\ref{lem:coupling} again to show that (because a reasonable fraction of vertices survive at step $1$) a reasonable fraction of vertices survive at step $2$, and so on.
  
  In order to carry this out, it is convenient not in fact to look at all $x\in\Xbuf_i$, but only at a subset $Y$ in which the probability of surviving at each step $j$ (as given by Fact~\ref{cl:degen:probs}) does not depend on the particular vertex in $x\in Y$ but only on $j$. We now construct such a set $Y$. For any $x\in\Xbuf_i$ at distance greater than three from $X^e$, there are at most $\Delta^3+\Delta^2+\Delta<3\Delta^3$ vertices of $H$ at distance at most $3$ from $x$ in $H$. We can therefore associate to each $x\in\Xbuf_i$ at distance greater than three from $X^e$ a $0$--$1$ vector with at most $3\Delta^3$ non-zero entries, taking the value $0$ at place~$j$ if the $j$th vertex $y$ of $M_x$ (in the order $\tau$) is not a neighbour of $x$, and $1$ if $y$ is a neighbour of $x$. There are at most $3\Delta^3 2^{3\Delta^3}$ choices for this vector; we fix a most common choice $\mathbf{c}=(c_i)$, and let $Y$ be the subset of vertices of $\Xbuf_i$ at distance at least three from $X^e$ associated to this most common choice. By construction, by~\ref{BUF:sizebuf}, by choice of $\eps$ and since $|X^e|\le\tfrac12\eps p^D n/r_1$ and $|X_i|\ge n/\kappa r_1$, we have
  \begin{equation}\label{eq:degen:sizeY}
   |Y|\ge\tfrac16 \Delta^{-3} 2^{-3\Delta^3}|\Xbuf_i|\,.
  \end{equation}
  
  For each $x\in Y$ and $1\le j\le|\mathbf{c}|$ (where $|\mathbf{c}|$ is the number of 1-entries of $\mathbf{c}$), let $\hist_{x,j}$ be the history of Algorithm~\ref{alg:RGA:deg} up to the point immediately before embedding the $j$th vertex of $M_x$. For each $x\in Y$, we create a collection of Bernoulli random variables $A_{x,1}$ for $x\in Y$, which are set equal to one if either $x$ survives at step $1$ or~\eqref{eq:degen:sizeAs} has failed before embedding the first element of $M_x$, and zero otherwise. We identify in each $M_x$ the first (in $\tau$) element $w_{x,1}$, and sort the $A_{x,1}$ according to the order induced by $\tau$ on the $w_{x,1}$. In this order, the $A_{x,1}$ are sequentially dependent. We define
  \[s_1=\begin{cases}\tfrac12\beta d^{D+1}4^{-D}p|Y| & \text{if}\quad c_1=1\\ \tfrac14|Y| &\text{if}\quad c_1=0\end{cases}\,.\]
  Now for each $A_{x,1}$, either Fact~\ref{cl:degen:probs} gives us a lower bound on $\Exp[A_{x,1}|\hist_{x,1}]$, or~\eqref{eq:degen:sizeAs} has failed before we embed the first vertex of $M_x$, in which case $\Exp[A_{x,1}|\hist_{x,1}]=1$. We conclude, by definition of $s_1$, that
  \[\sum_{x\in Y}\Exp[A_{x,1}|\hist_{x,1}]\ge 2s_1\,.\]
  By Lemma~\ref{lem:coupling}, with $\delta=1/4$, with probability at least $1-e^{-s_1/24}$ we have $\sum_{x\in Y}A_{x,1}\ge s_1+1$. Observe that if~\eqref{eq:degen:sizeAs} does not fail for any $y$ or $t$ (which is a.a.s.\ the case by Fact~\ref{fact:nofill}) then $\cA_1$ holding says that at least $s_1+1$ vertices survive at step $1$.
  
  There remains a small technical difficulty in continuing. We would like to count vertices surviving at step $2$, and use Lemma~\ref{lem:coupling} and a lower bound on the sum of conditional expectations given by the product of $s_1$ and the probability bound from Fact~\ref{cl:degen:probs} to show that the number of vertices surviving at step $2$ is likely to be large. But it is possible that less than $s_1$ vertices survive at step $1$, so that this bound on the sum of conditional expectations does not hold almost surely (it only holds a.a.s.).
  
   To get around this problem, we use the following trick. Given $x$, immediately after embedding the first vertex of $M_x$ in Algorithm~\ref{alg:RGA:deg}, let $g$ be the number of vertices $x'$ in $Y$ such that $A_{x,1}$ is certainly equal to one, i.e.\ such that the first vertex of $M_{x'}$ was embedded to a vertex which is neither poor nor bad, so $x'$ survives at step $1$. Let $h$ be the number of vertices $x'$ in $Y$ such that the first vertex of $M_{x'}$ has not yet been embedded. If $g+h\le s_1$, we say $x$ is \emph{dangerous at step $1$} if the following holds. In other words, we say $x$ is dangerous at step $1$ if after embedding the first vertex of $M_x$ we already know that $\cA_1$ does not occur.
  
  We now define Bernoulli random variables $A_{x,2}$ for $x\in Y$, set equal to one if either $x$ survives at step $2$, or $x$ is dangerous at step $1$, or~\eqref{eq:degen:sizeAs} has failed for some $y$ prior to embedding the second vertex of $M_x$, and zero otherwise. The point of this definition is that it gives us an \emph{a priori} lower bound on the sum of conditional expectations of the $A_{x,2}$, as we require to apply Lemma~\ref{lem:coupling}, but nevertheless with very high probability the sum of the $A_{x,2}$ does simply count the number of vertices in $Y$ which survive at step $2$, because with very high probability~\eqref{eq:degen:sizeAs} does not fail and $\cA_1$ does occur, so no vertex is dangerous at step $1$. Again, this collection of random variables has a natural order given by the order $\tau$ on the second vertices of the $M_x$. Observe that we know whether $x$ is dangerous at step $1$ before we embed the second vertex of $M_x$. Thus in this order the random variables are sequentially dependent. This justifies that we can apply Lemma~\ref{lem:coupling} to estimate the sum of the $A_{x,2}$.
  
  We now say what exactly this application of Lemma~\ref{lem:coupling} gives us, and explain how to analyse steps $3,\dots,|\mathbf{c}|$ in the same way. We need to give integers $s_j$ for $2\le j\le |\mathbf{c}|$, which are our desired lower bounds on the number of vertices in $Y$ surviving at step $j$, to define the events $\cA_j$ and the concept of \emph{dangerous at step $j$} for $j\ge 2$, and define the random variables $A_{x,j}$ for $3\le j\le\mathbf{|c|}$.
  
  For each $2\le j\le|\mathbf{c}|$, we set
  \[
   s_j=\begin{cases}\tfrac12\beta d^{D+1}4^{-D}ps_{j-1} & \text{if}\quad c_j=1\\ \tfrac14s_{j-1} &\text{if}\quad c_j=0\end{cases}\,.
  \]
  We let the event $\sum_{x\in Y}A_{x,j}\ge s_j+1$ be $\cA_j$. We say $x\in Y$ is dangerous at step $j$ if immediately after embedding the $j$th vertex of $M_x$, the number of $x'\in Y$ such that $A_{x',j}$ is certainly equal to one, plus the number of $x'\in Y$ such that the $j$th vertex of $M_{x'}$ has not yet been embedded, is at most $s_j$. Again, this means that after embedding the $j$th vertex of $M_x$, we already know $\cA_j$ does not occur. Finally, for each $j\ge 3$ we define the Bernoulli random variables $A_{x,j}$ for $x\in Y$, set equal to one if $x$ survives at step $j$, or is dangerous at step $j-1$, or~\eqref{eq:degen:sizeAs} has failed before embedding the $j$th vertex of $M_x$, and zero otherwise.
  
  Before we continue, we observe that
  \begin{equation}\label{eq:degen:lowsc}
   s_1\ge s_2\ge\dots\ge s_{|\mathbf{c}|}\ge \big(\tfrac14\big)^{\Delta^2+\Delta^3}\big(\tfrac12\beta d^{D+1}4^{-D}p\big)^b|Y|\,,
  \end{equation}
  where $b$ is the degree of the buffer $\Xbuf_i$, i.e.\ $\deg_H(x)=b$ for each $x\in\Xbuf_i$; this number exists by~\ref{BUF:deg}. The reason is simply that there are at most $\Delta^2+\Delta^3$ vertices of $H$ at distance two or three in $H$ from any given $x\in Y$, and $b$ neighbours of $x$. In particular, by choice of $p$, by our lower bound on $|Y|$ and since $n$ is sufficiently large, we have $3s_j/2>s_j+1$ for each $1\le j\le|\mathbf{c}|$.  
  
  We now show that with very high probability $\cA_j$ occurs for each $1\le j\le|\mathbf{c}|$. Recall that by Fact~\ref{cl:degen:probs}, we have the lower bound $\sum_{x\in Y}\Exp[A_{x,1}|\hist_{x,1}]\ge 2s_1$. By Lemma~\ref{lem:coupling} with $\delta=1/4$, and the observation $3s_1/2>s_1+1$, the probability that $\cA_1$ fails is thus at most $e^{-s_1/24}$.
  
  Now for $j\ge 2$, for each $x\in Y$, either $x$ survived at step $j-1$ and~\eqref{eq:degen:sizeAs} has not failed before embedding the $j$th vertex of $M_x$ (in which case the expectation of $A_{x,j}$ conditioned on $\hist_{x,j}$ is at least the quantity given in Fact~\ref{cl:degen:probs}), or $x$ was dangerous at step $j-1$, or~\eqref{eq:degen:sizeAs} failed before embedding the $j$th vertex of $M_x$ (in which case we have $\Exp[A_{x,h}|\hist_{x,j}]=1$), or none of these occur (in which case we have $\Exp[A_{x,h}|\hist_{x,j}]\ge 0$). Furthermore, by definition of the $A_{x,j}$ and of the event $\cA_{j-1}$, we know (a priori, before beginning the embedding) that one of the first three cases occurs for at least $s_{j-1}$ of the vertices $Y$, so that we have $\sum_{x\in Y}\Exp[A_{x,j}|\hist_{x,j}]\ge 2s_j$ by definition of $s_j$. Applying Lemma~\ref{lem:coupling}, with $\delta=1/4$, and since $3s_j/2>s_j+1$, we conclude that the probability that $\cA_j$ fails is at most $e^{-s_j/24}$.
  
  By definition, if~\eqref{eq:degen:sizeAs} never fails and each event $\cA_1,\dots,\cA_{|\mathbf{c}|}$ occurs, then no vertex is dangerous at any step. Thus the number of vertices of $Y$ surviving at step $|\mathbf{c}|$ is at least $s_{|\mathbf{c}|}$. By definition, if $x$ survives at step $|\mathbf{c}|$ then $N_H(x)$ is embedded to $N_G(v)$, so that $v$ is a candidate for $x$. Since $s_{|\mathbf{c}|}$ is by~\eqref{eq:degen:lowsc},~\eqref{eq:degen:sizeY} and~\ref{BUF:sizebuf} at least
  \[ \big(\tfrac14\big)^{\Delta^2+\Delta^3}\big(\tfrac12\beta d^{D+1}4^{-D}p\big)^b\cdot \tfrac16 \Delta^{-3} 2^{-3\Delta^3}\cdot 4\mu|X_i|\]
  we conclude that~\eqref{eq:rga:bufdeg} holds for $v$, as desired.  
    The probability that any given one of the events $\cA_1,\dots,\cA_{|\mathbf{c}|}$ fails is at most $e^{-s_{|\mathbf{c}|}/24}<n^{-D-1}$ by~\eqref{eq:degen:lowsc},~\eqref{eq:degen:sizeY},~\ref{BUF:sizebuf}, since $|X_i|\ge n/\kappa r_1$ and by choice of $p$. Taking the union bound over the at most $3\Delta^3$ events $\cA_j$, and the at most $n$ choices of $v$, and since a.a.s.~\eqref{eq:degen:sizeAs} never fails by Fact~\ref{fact:nofill}, we conclude that a.a.s.~\eqref{eq:rga:bufdeg} holds for all $v\in V_i$ for each $i\in [r]$, completing the proof of Claim~\ref{cl:degen:buffer}.
   \end{claimproof}
  
  The good event of Claim~\ref{cl:degen:buffer} holding gives~\ref{degenrga:nobad}. Since a.a.s.\ the good events of each of the above claims and lemmas hold, this completes the proof of Lemma~\ref{lem:degenrga}. 
\end{proof}

\chapter{Proofs of applications}\label{chap:appproofs}
 In this chapter we prove the various theorems listed in Section~\ref{sec:appl}. We indicate when stronger results are proved (subsequently) elsewhere. To keep this section brief we often only sketch  proofs.
 
 \section{Universal graphs}
We begin by showing  universality of $G(n,p)$ for (almost) spanning graphs with bounded degree and bounded degeneracy. 

  \begin{proof}[Proof of Theorem~\ref{thm:degenuniv}]
    We first show~\ref{degenuniv:a}.
    Observe that any $d$-degenerate $n$-vertex graph has at most $dn$ edges
    and so contains at least $\tfrac{n}{2d+1}$ vertices of degree at most
    $2d$. We will apply Lemma~\ref{lem:degen} with input $\Delta$, with
    $\DeltaRp=8\Delta$, $D=2d+1$, $\alpha=\tfrac{\gamma}{10d}$, $\zeta=1$,
    density $\tfrac12$, and $\kappa=2$. Lemma~\ref{lem:degen} returns
    $\eps>0$, and for input $r_1=8\Delta$ also $C$. Choose $\eps^*\ll\eps$
    and suppose $C$ is also large enough for
    Lemma~\ref{lem:det_Gnp}\ref{det_gnp:randns} with input $\eps^*$,
    $8\Delta$ and $2$. Let $p\ge C\big(\tfrac{\log
      n}{n}\big)^{1/(2d+1)}$. Fix an equipartition $V_1,\ldots,V_{8\Delta}$
    of $[n]$. Now $G(n,p)$ a.a.s.\ satisfies the good event of
    Lemma~\ref{lem:degen}. By Lemma~\ref{lem:det_Gnp}\ref{det_gnp:randns}
    it a.a.s.\ has property $\NS(\eps^*,8\Delta,D)$. Finally, using
    Theorem~\ref{thm:chernoff} and the union bound it is easy to check that
    a.a.s.\ for each $i$ and $v\in V_i$, the vertex $v$ has
    $(1\pm\eps)p|V_j|$ neighbours in $V_j$ for each $j\neq i$. Fix a graph
    $\Gamma$ with each of these properties. Property
    $\NS(\eps^*,8\Delta,D)$ implies that any pair of disjoint subsets of
    $V(\Gamma)$, each of size at least $\tfrac{p n}{16\Delta}$, is
    $(\eps,\tfrac12,p)$-regular, so letting $R=R'=K_{8\Delta}$ we have an
    $(\eps,\tfrac12,p)$-regular $R$-partition which has super-regularity
    and one- and two-sided inheritance on $R'$.
   
    Given $H$ on $n$ vertices with degeneracy $d$ and maximum degree
    $\Delta$, let $X_1,\ldots,X_{8\Delta}$ be an equipartition of $V(H)$
    into independent sets, with each $X_i$ containing at least
    $\tfrac{1}{8\Delta(2d+1)}|X_i|$ vertices of degree at most $2d$. This
    equipartition exists by Lemma~\ref{lem:nicepartition}. We designate the
    vertices of degree at most $2d$ as potential buffer vertices. We do not
    image restrict any vertices. Then (perhaps after reordering) the $X_i$
    form an $R$-partition of $V(H)$ which is size-compatible with the
    $V_i$. We let $\tau$ be a degeneracy order on $V(H)$. We move all
    buffer vertices to the end of the ordering $\tau$. Observe that this
    slightly changed ordering satisfies
    conditions~\ref{ordp:Dx}--\ref{ordp:NtX} and thus is
    $(D,p,\eps^*n/r_1)$-bounded.  Then the conditions of
    Lemma~\ref{lem:degen} are satisfied, so $H\subseteq\Gamma$ as desired.
   
   To obtain \ref{degenuniv:b}, we work identically except that we add to each part of $V(H)$ isolated vertices to obtain a size-compatible $R$-partition, and we designate these isolated vertices as potential buffer vertices instead of the low-degree vertices in $H$. Then Lemma~\ref{lem:degen}, applied with $D=2d$, $p=C(\frac{\log n}{n})^{1/D}$ and all other parameters staying as before, gives the desired conclusion.
  \end{proof}

It is quite easy to use Lemma~\ref{lem:psr_main} to show that sufficiently pseudorandom graphs are universal for bounded-degree graphs.

  \begin{proof}[Sketch proof of Theorem~\ref{thm:pseudouniv}]
   Given $G$, we take a random equipartition of $V(G)$ into $\Delta+1$ clusters $V_1,\ldots,V_{\Delta+1}$. For each $v\in V(G)$ and $i$, the quantity $\deg(v;V_i)$ is hypergeometrically distributed with mean at least $\tfrac14p|V_i|$, so by Theorem~\ref{thm:chernoff} and the union bound we see that with positive probability we have $\deg(v;V_i)\ge\tfrac18p|V_i|$ for each $v$ and $i$. We fix such a partition, and let $R=R'=K_{\Delta+1}$. Using 
   the Hajnal-Szemer\'edi theorem (Theorem~\ref{thm:HajSze}) we can for any $n$-vertex graph $H$ with $\Delta(H)\le\Delta$ find a size-compatible $R$-partition of $V(H)$, and applying Lemma~\ref{lem:psr_main} with all vertices designated as potential buffer vertices, with $d=\tfrac{1}{16}$ and with no vertices image restricted, we see that $H\subseteq G$ as desired.
  \end{proof}
  
 \section{Partition universality}
 In this section we prove Theorems~\ref{thm:rpartuniv},~\ref{thm:rpartunivdegen} and~\ref{thm:folkmantype}. We need Ramsey's Theorem and Tur\'an's Theorem.  
The Ramsey number $r_k(K_t)$ is the smallest integer $n$ such that no matter how one colours the edges of $K_n$ with $k$ colours, there is a monochromatic copy of $K_t$. Ramsey~\cite{Ram30} proved these numbers exist, while the following quantitative statement is due to Erd\H{o}s and Szekeres~\cite{ErdSze35}.
\begin{theorem}[Erd\H{o}s and Szekeres~\cite{ErdSze35}]
For any $k$ and $t$ we have $r_k(K_t)\le k^{kt}$.
\end{theorem}

Tur\'an, generalising a result of Mantel~\cite{Man07}, proved the following.
\begin{theorem}[Tur\'an~\cite{Tur41}]
For any $r\ge 3$, any $n$-vertex $K_r$-free graph has at most $\frac{(r-2)n^2}{2(r-1)}+o(n^2)$ edges.
\end{theorem}

Finally we state the version of the sparse regularity lemma for many colours that we are going to apply. 
We also say that a graph $G$ with density $p$ is \emph{$(\eta,D)$-upper-uniform} if for all disjoint sets $U$ and $W$ of cardinality at least 
$\eta v(G)$ we have $e_G(U,W)\le D p|U||W|$.
\begin{lemma}[Sparse regularity lemma, coloured version~\cite{Kohayakawa97Szemeredi}]\label{lem:sparse_colour_RL}
For any real $D,\eps>0$, integers $k$ and $t_0$, there exist $\eta>0$ and $T$ such that any graph $G$ of edge density $p$ and on at least $t$ vertices, which is $(\eta,D)$-upper-uniform and  whose edges are coloured with $k$ colours admits a  partition of $V(G)$ into 
$V_1$, \ldots, $V_t$ with the following properties.
\begin{enumerate}[label=\abc]
\item $t_0\le t \le T$.
\item $\big| |V_i|-|V_j|\big|\le 1$ for all $i$, $j$.
\item all but at  most $\eps t^2$ pairs $(V_i,V_j)$ are $(\eps,d,p)$-fully-regular in each of the $k$ colours for some possibly different $d$.
\end{enumerate}
\end{lemma}
The assumption on $(\eta,D)$-upper uniformity is easily seen to be satisfied for any moderate $\eta$ and $D>1$ by our random graph 
$G(n,p)$ and by the bijumbled graphs that we will be considering. The partition asserted by 
Lemma~\ref{lem:sparse_colour_RL} is called $\eps$-regular. We define a coloured multigraph $R(\delta)$ on $[t]$ (associating each $i\in[t]$ with the class $V_i$) as follows. We put an edge $ij$ into $E({R})$ in colour $c$ if the pair $(V_i,V_j)$ is $(\eps,\delta,p)$-regular in colour $c$ in $G$.

  We now prove that $G(n,p)$ is a.a.s.\ $r$-partition universal for $\cH(n,d,\Delta)$ provided $p\ge C\big(\tfrac{\log n}{n}\big)^{1/(2d)}$, which improves on the result of Kohayakawa, R\"odl, Schacht and Szemer\'edi~\cite{ChvRand} for graphs with $d\le\Delta/2$. A proof of the result of~\cite{ChvRand} can be obtained along very similar lines.
  \begin{proof}[Proof of Theorem~\ref{thm:rpartuniv}]
   We apply Lemma~\ref{lem:degen} with input $\Delta$, with $\DeltaRp=0$, $D=2d$, $\alpha=\tfrac{1}{2}$, $\zeta=1$, density $\tfrac1{2d}$, and $\kappa=1$. Lemma~\ref{lem:degen} returns $\eps>0$, which we suppose is small enough for the application of Tur\'an's Theorem below. We let $r_1$ be large enough for the $k$-coloured sparse regularity lemma with input $\eps$ and also for the applications of Tur\'an's and Ramsey's Theorem below, and obtain $C\ge r_1$ from Lemma~\ref{lem:degen}. Given $p$, we generate $\Gamma=G(Cn,p)$, suppose it satisfies the good event of Lemma~\ref{lem:degen} and take any $k$-colouring of its edges. We apply the sparse regularity lemma for coloured graphs, Lemma~\ref{lem:sparse_colour_RL}, to this coloured graph. We obtain a coloured reduced graph with at most an $\eps$-fraction of pairs not forming edges, in which we find a $4^{\Delta+1}$-vertex clique  by Tur\'an's Theorem, and in that a monochromatic $(\Delta+1)$-vertex clique by Ramsey's Theorem. Thus we have $\Delta+1$ equal-sized clusters $V_1,\ldots,V_{\Delta+1}$, each of size at least $n$, such that there is a colour $c$ in which each pair of clusters is $(\eps,\tfrac{1}{2r},p)$-regular. Let $G$ be the subgraph of $\Gamma[V_1,\dots,V_{\Delta+1}]$ containing all edges of colour $c$, let $R=K_{\Delta+1}$ and $R'$ be the empty graph.
   
   Given $H\in\cH(n,\Delta)$, by Theorem~\ref{thm:HajSze} we can find an equipartition of $V(H)$ into $\Delta+1$ independent sets, each of size at most $\lceil n/(\Delta+1)\rceil<n/2$. We `pad' each set by adding at least $n/2$ isolated vertices to obtain a size-compatible $R$-partition $X_1,\ldots,X_{\Delta+1}$, and designate the isolated vertices as potential buffer vertices which come last in the degeneracy order $\tau$ of $H$. Then the result follows by Lemma~\ref{lem:degen}, with $\tau$ being a $(D,p,\eps n/r)$-bounded order on $V(H)$.
  \end{proof}
  
  \begin{proof}[Sketch proof of Theorem~\ref{thm:rpartunivdegen}]
   The same approach as for Theorem~\ref{thm:rpartuniv} works, replacing Lemma~\ref{lem:degen} with Lemma~\ref{lem:psr_main}. It is easy to check that the upper-uniformity condition of Lemma~\ref{lem:sparse_colour_RL} is satisfied for $\Gamma$ with the bijumbledness condition of the theorem.
  \end{proof}
  
  \begin{proof}[Sketch proof of Theorem~\ref{thm:folkmantype}]
   We follow the same approach as in the proof of Theorem~\ref{thm:rpartuniv}, with the exceptions that we use $p=C\big(\tfrac{\log n}{n}\big)^{1/\Delta}$, that we use Lemma~\ref{lem:rg_image} instead of Lemma~\ref{lem:degen}, and that after generating $\Gamma=G(n,p)$ which satisfies the good event of Lemma~\ref{lem:rg_image}, we form $G$ by deleting a minimum number of edges to remove all copies of $K_{2\Delta,2\Delta}$. The expected number of copies of $K_{2\Delta,2\Delta}$ in $G(n,p)$ is $O\big(p^{4\Delta^2}n^{4\Delta}\big)=o(pn^2)$, so by Markov's inequality we see that a.a.s.\ we delete $o(pn^2)$ edges of $\Gamma$ to form $G$. This $G$ is the claimed $r$-partition universal graph for $\cH(n,\Delta)$. The remainder of the proof of Theorem~\ref{thm:rpartuniv} proves the partition universality, since the number of edges deleted is too small to destroy regularity of any pair of clusters.
  \end{proof}

  \section{Maker-Breaker games}
   The proof of Theorem~\ref{thm:MakerBreaker} uses the result of Ferber, Krivelevich and Naves~\cite{FKN}, which informally says that if $p\ll b^{-1}$ is not too small, there is a (randomised) Maker strategy for the $1:b$ game on $K_n$ which gives Maker a subgraph of $G(n,p)$ with minimum degree very close to $pn$. More formally, they prove the following.
   \begin{theorem}\label{thm:FKNmaker} For any $\frac{\log n}{n}\ll p\ll b^{-1}$, there is a randomised Maker strategy which, for any fixed strategy of Breaker in the $1:b$ game on $K_n$, generates a graph~$\Gamma$ from the distribution $G(n,p)$, and a spanning subgraph $G$ of edges obtained by Maker which a.a.s.\ has minimum degree $(1-o(1))pn$.
   \end{theorem}
   This reduces the proof to showing a (far from optimal in terms of edge deletions) local resilience result for the graph classes we consider.
   
   \begin{proof}[Sketch proof of Theorem~\ref{thm:MakerBreaker}]
    We begin by sketching the proof of~\ref{thm:MakerBreaker:b}, i.e.\ that Maker obtains $\cH'(n,\Delta)$-universality. We take an equipartition of $[n]$ into $\Delta+1$ clusters $V_1$, \ldots, $V_{\Delta+1}$. We let $\eps>0$ be sufficiently small for Lemma~\ref{lem:rg_image}, with no vertices image restricted, and we set $\eps^*\ll\eps$. Suppose $C$ is large enough for Lemma~\ref{lem:rg_image}, and $p=C\big(\frac{\log n}{n}\big)^{1/\Delta}$. Let $c>0$ be sufficiently small, and suppose $b\le c\big(\tfrac{n}{\log n}\big)^{1/\Delta}$. Then a.a.s.\ in the $1:b$ game on $K_n$, Maker obtains a subgraph $G$ of $\Gamma=G(n,p)$ with $\delta(G)\ge(1-\eps^*)pn$ by Theorem~\ref{thm:FKNmaker}. It is easy to check, using Chernoff's inequality, that each vertex has $(1\pm2\eps^*)pn/(\Delta+1)$ neighbours in both $\Gamma$ and $G$ in each set $V_i$.
    
    Since the total number of edges of $\Gamma$ which are not in $G$ leaving any vertex is at most $4\eps^*pn$, the total number of such edges between any two vertex sets $X$ and $Y$ is at most $4\eps^*pn|X|$. An easy application of the Chernoff bound shows that $(X,V_i)$ is $\big(\eps^*,\tfrac34,p\big)$-regular in $\Gamma$ for each $V_i$ and disjoint vertex set $X$ with $|X|\ge pn/(4\Delta)$, and since $\eps^*\ll\eps$ we conclude that any such $(X,V_i)$ is also $\big(\eps,\tfrac12,p\big)$-regular in $G$.
    
    Letting $R=R'=K_{\Delta+1}$, we see that $G$ is $\big(\eps,\tfrac12,p\big)$-super-regular on $R'$ and has one-sided inheritance on $R'$. Given any graph $H\in\cH'(n,\Delta)$ we use Theorem~\ref{thm:HajSze} to find an equipartition of $V(H)$ into $\Delta+1$ independent sets $X_1,\ldots,X_{\Delta+1}$. Then the conditions of Lemma~\ref{lem:rg_image} are satisfied, so we have $H\subseteq G$ as desired.
    
    For~\ref{thm:MakerBreaker:a}, i.e.\ the almost-spanning $\cH(n,\Delta)$-universality game, we repeat the same argument, replacing $[n]$ with $[(1+\delta)n]$ and `padding' each equipartition class of $H$ with independent vertices to be size-compatible with $\cV$. We take these independent vertices to be the potential buffer vertices, so that two-sided inheritance is not needed.
    
    For~\ref{thm:MakerBreaker:c} and~\ref{thm:MakerBreaker:d}, the degeneracy statements, we replace Lemma~\ref{lem:rg_image} with Lemma~\ref{lem:degen}, taking respectively $D=2d$ for the almost-spanning universality and $D=2d+1$ for the spanning universality. Again, similar to Theorem~\ref{thm:rpartuniv}, it is easy to verify that the degeneracy order of $H\in\cH(n,d,\Delta)$ is     
 an appropriately bounded order in the sense of Definition~\ref{def:Dpm_bdd_order}. 
 
    In both cases, since we have shown that Maker has a randomised strategy which wins with positive probability against any strategy of Breaker, it follows that Breaker does not have a winning strategy. Since finite draw-free games are determined, we conclude that Maker does have a winning strategy.
   \end{proof}
   
   Note that, by an analysis similar to the proof of Theorem~\ref{thm:rpartuniv}, one can show that Maker also succeeds with this strategy in making a graph which is $r$-partition universal for $\cH(cn,\Delta)$ (respectively, for $\cH(cn,d,\Delta)$) for some small $c>0$, matching the density of the best known constructions.
  
  \section{Resilience for low-bandwidth graphs}\label{sec:resil}
   For this section we need the minimum degree form of the sparse regularity lemma, which we quote from~\cite{BKT}, the paper in which the bipartite case of Theorem~\ref{thm:sparsebandwidth} is proved. To state it we need to define two concepts. First, an \emph{$\eps$-equipartition} of a vertex set $V$ is a partition $V=V_0\dcup\dots\dcup V_r$ such that $|V_0|\le\eps|V|$ and $|V_1|=\dots=|V_r|$. Second, if $G$ is a graph with vertex set $V$, then the \emph{$(\eps,d,p)$-lower-regular reduced graph of $G$}, with respect to a given $\eps$-equipartition $V=V_0\dcup\dots\dcup V_r$, is the graph on $[r]$ with edges $ij$ corresponding to $(\eps,d,p)$-lower-regular pairs $(V_i,V_j)$ in $G$.
   
   \begin{lemma}[\cite{BKT}, Lemma~4.4]\label{lem:mindegSRL}
    For all $\vartheta\in[0,1]$, $\eps>0$ and every integer $r_0$, there exists $r_1\ge 1$ such that for all $d\in[0,1]$ the following holds a.a.s.\ for $\Gamma=G(n,p)$ if $\log^4 n/(pn)=o(1)$. Let $G=(V,E)$ be a spanning subgraph of $\Gamma$ with $\deg_G(v)\ge\vartheta\deg_\Gamma(v)$ for all $v\in V$. Then there is an $\eps$-equipartition of $G$ with $(\eps,d,p)$-lower-regular reduced graph $R$ of minimum degree $\delta(R)\ge(\vartheta-d-\eps)v(R)$, and $r_0\le v(R)\le r_1$.
   \end{lemma}
   
   Using this, we can sketch the proof of Theorem~\ref{thm:sparsebandwidth}. The strategy consists of modifying the argument in~\cite{BST}, the paper in which the Bandwidth Theorem was proved. We will not state formally the lemmas from that paper which we require; the reader not familiar with that argument will wish to read the following sketch in conjunction with Section~2 of~\cite{BST}, in which the lemmas are formally stated and their use outlined. The changes to their strategy we make are as follows. Their `Lemma for $G$' is replaced with Lemma~\ref{lem:mindegSRL} (since we do not need most of the properties of the `Lemma for $G$' in this setting), and Theorem~\ref{thm:BST} finds a `backbone graph' in the resulting $(\eps,d,p)$-reduced graph $R$. We can use their `Lemma for $H$' as written, and in this setting it gives a partition of $V(H)$ which is directly suitable to apply the blow-up lemma, Lemma~\ref{lem:rg_image}, yielding the desired embedding of $H$ into $G$.
   
   \begin{proof}[Sketch proof of Theorem~\ref{thm:sparsebandwidth}] 
    Given $\gamma>0$ and $\Delta$, we choose $d\ll\gamma$ and $r_0\gg\Delta$. We apply Lemma~\ref{lem:rg_image} with input $\Delta$, $\DeltaRp=\Delta_J=0$, $\alpha=\gamma/2$, $\zeta=1$, $d$ and $\kappa=2$. Lemma~\ref{lem:rg_image} returns $\rho,\eps>0$, of which we are only interested in $\eps$. We assume, without loss of generality, that $\eps\ll d$. We let $r_1$ be returned by Lemma~\ref{lem:mindegSRL} for input $\vartheta=\tfrac{\chi(H)-1}{\chi(H)}+\gamma$, $\eps$ and $r_0$. Finally, we choose $\beta\ll r_1^{-1}$, and let $C$ be returned by Lemma~\ref{lem:rg_image} for input $r_1$.
    
    Now, given $p\ge C\big(\tfrac{\log n}{n}\big)^{1/\Delta}$, we generate $\Gamma=G(n,p)$, and assume it satisfies the conditions of Lemmas~\ref{lem:rg_image} and~\ref{lem:mindegSRL} for the parameters given above. Let a graph $H$ on $(1-\gamma)n$ vertices with $\Delta(H)\le\Delta$ and $\bw(H)\le\beta n$ be given. Let $r=\chi(H)\le\Delta+1$. Let a spanning subgraph $G$ of $\Gamma$ with minimum degree $\vartheta p n=\big(\tfrac{r-1}{r}+\gamma\big)pn$ be given.
    
    We apply Lemma~\ref{lem:mindegSRL} to $G$, obtaining an $(\eps,d,p)$-reduced graph $R$, with $r_0\le v(R)\le r_1$, and minimum degree at least $\big(\tfrac{r-1}{r}+\tfrac{\gamma}{2}\big)v(R)$. By Theorem~\ref{thm:BST}, we can find in $R$ a spanning \emph{backbone graph}: This consists of a collection of vertex-disjoint copies of $K_r$, which come in a linear order, such that between one copy of $K_r$ and the next there is a copy of $K_{r,r}$ with a perfect matching removed. This graph is $r$-colourable and has maximum degree $3(r-1)\le 3\Delta$ and bandwidth at most $2r\le 2\Delta$, so that Theorem~\ref{thm:BST} is indeed applicable provided $r_0$ is sufficiently large compared to $\Delta$.
    
     It is quite easy to find a homomorphism from $H$ to the backbone graph. We simply divide $V(H)$ up into intervals in the bandwidth order, and map successive intervals of $V(H)$ to successive copies of $K_r$ in the backbone graph, choosing vertices of each copy of $K_r$ according to a fixed $r$-colouring of $H$. The r\^ole of the bandwidth restriction here is to ensure that edges of $H$ either lie within one interval, or go from one interval to the next, so that we only need edges in $R$ from one copy of $K_r$ in the backbone graph to the next in order to obtain a homomorphism. The point of fixing an $r$-colouring of $H$ is that the $i$th vertex in one copy of $K_r$ and that in the next are not adjacent in the backbone graph (since a perfect matching was removed from the $K_{r,r}$ between them) and we need to ensure that no edge of $H$ will be assigned to have one endpoint in each.
     
     Unfortunately, this is not quite enough: the colour classes of $H$ could be quite unbalanced, so that the homomorphism we have just described maps many more vertices of $H$ to some vertices of $R$ than others. In order to repair this, we need to `rebalance', which requires that each copy of $K_r$ in the backbone graph extends to a copy of $K_{r+1}$ using some other vertex of $R$ (which may be anywhere in the backbone graph). The Lemma for $H$ of~\cite{BST} now states that given a backbone graph whose $r$-cliques extend to copies of $K_{r+1}$, and $H$, there is a homomorphism $f$ from $V(H)$ to $V(R)$ in which each vertex of $R$ is the image of approximately the same number of vertices of $H$. Since $\eps\ll\gamma$, since $v(H)\le(1-\gamma)n$ and by choice of $\beta$ sufficiently small, the `approximately' in this statement in particular guarantees $\big|f^{-1}(i)\big|\le \big(1-\tfrac{\gamma}{2}\big)|V_i|$ for each $i\in V(R)$.
   
   It remains only to verify that the conditions of Lemma~\ref{lem:rg_image} can be met in order to find an embedding of $H$ into $G$. The idea is simple: we  `pad' $H$ by adding $n-v(H)$ isolated vertices, and give a partition $\cX=(X_i)_{i\in V(R)}$ of $V(H)$ in which $X_i$ consists of $f^{-1}(i)$ together with isolated vertices such that $|V_i|=|X_i|$. We designate the isolated vertices in each $X_i$ as the potential buffer vertices $\tX_i$. It then follows that the empty graph $R'$ on $V(R)$ with no edges, together with these potential buffer vertices, give us a $(\gamma/2,R')$-buffer for $(H,\cX)$. By construction, the partition $\cX$ is an $R$-partition of $H$. By definition, $\cV$ is an $(\eps,d,p)$-regular $R$-partition of $G$, and the super-regularity and inheritance properties required of $R'$ are satisfied vacuously. Finally, we do not image restrict any vertices of $H$, so that the restriction pair properties are satisfied vacuously. Thus Lemma~\ref{lem:rg_image} gives us the desired embedding of $H$ into $G$.
   \end{proof}
   
   We stress that the main difficulty in the proof of Theorem~\ref{thm:BST} is to obtain a spanning embedding; an almost-spanning embedding is much easier. It is similarly, and for similar reasons, much harder to prove Theorem~\ref{thm:ABET} than Theorem~\ref{thm:sparsebandwidth}. It is also worth noting that in the proof given in~\cite{BST}, there is a substantial amount of routine technical work to do in between obtaining size-compatible partitions of $H$ and $G$ and using the blow-up lemma to get an embedding of $H$ into $G$, which is encapsulated in the so-called partial embedding lemma. This work is necessary because the blow-up lemma of~\cite{KSS_bl} cannot be applied to the entire reduced graph. Our Lemma~\ref{lem:rg_image} can be applied to the entire reduced graph, and thus replaces both the partial embedding lemma and the blow-up lemma of~\cite{KSS_bl}.
   
  \section{Robustness of the Bandwidth Theorem}\label{sec:robust}
   As with the proof of Theorem~\ref{thm:sparsebandwidth}, the proof of Theorem~\ref{thm:robustbandwidth} amounts to modifying the proof of Theorem~\ref{thm:BST}. However, this time we need to rely on rather more of the machinery built up in~\cite{BST}. Again, the reader not familiar with the argument there will wish to read this sketch in conjunction with Section~2 of~\cite{BST}. 
  
   \begin{proof}[Sketch proof of Theorem~\ref{thm:robustbandwidth}]
    We choose constants as in~\cite[Proof of Theorem~2]{BST}, with the exception that we obtain $\eps'$ from Lemma~\ref{lem:rg_image} and $\eps$ from Theorem~\ref{thm:OneSideInherit} and Corollary~\ref{cor:TwoSideInherit} for input $\eps'$ rather than from the blow-up lemma of~\cite{KSS_bl} and the partial embedding lemma of~\cite{BST}. We then follow the proof given there up to the point at which in~\cite{BST} the first vertices of $H$ are embedded using the partial embedding lemma. Let us recap what this amounts to. We are given graphs $G$ and $H$ satisfying the conditions of Theorem~\ref{thm:robustbandwidth}. We apply the Lemma for $G$ of~\cite{BST}, which first returns a partition of $V(G)$ into parts $V'_{i,j}$, where $1\le i\le k$ and $1\le j\le\chi(H)$, with the following properties. Note that the Lemma for $G$ does not explicitly return this partition, but it is convenient for the explanation to mention its existence; it also does not explicitly give the upper bound on sizes of the parts, but this follows from the proof. First, $k\le K_0=K_0(\gamma,\chi(H))$ is bounded in terms of $\gamma$ and $\chi(H)$. Second, the $(\eps,d,1)$-reduced graph $R$ of this partition, whose vertex set is $[k]\times[\chi(H)]$ matching the indices of the partition, has minimum degree at least $\big(\tfrac{\chi(H)-1}{\chi(H)}+\tfrac{\gamma}{2}\big)v(R)$. Third, $R$ contains a spanning backbone graph: that is, if $|i-i'|\le 1$ and $j\neq j'$ then $(i,j)(i',j')\in E(R)$. Fourth, the parts $V'_{i,j}$ and $V'_{i,j'}$ differ in size by at most one, and each part has size between $(1-\eps)n/k\chi(H)$ and $2n/k\chi(H)$.
    
    Now we apply the Lemma for $H$ of~\cite{BST}, with the reduced graph $R$ and the integer partition of $n$ given by the $|V'_{i,j}|$. This gives us a homomorphism $f$ from $H$ to $R$, and a set of \emph{special vertices} $Z\subset V(H)$, with the following properties, which depend on a quantity $\xi$ satisfying $\beta\ll\xi\ll\eps,K_0^{-1}$. First, $|Z|\le k\chi(H)\xi n$. Second, for each $(i,j)\in V(R)$ we have $|f^{-1}(i,j)|=|V_{i,j}|\pm\xi n$. Third, if $uv\in H$ has neither endpoint in $Z$, then the first coordinates of $f(u)$ and $f(v)$ are equal, in other words $u$ and $v$ are mapped to vertices of the same clique in $R$. We set $X_{i,j}=f^{-1}(i,j)$ for each $(i,j)\in R$. By construction, the resulting partition $\cX$ is an $R$-partition of $H$. 
    
    Next, we return to the Lemma for $G$, which guarantees, given the sizes of the parts $|X_{i,j}|$ satisfying the above properties, a partition $\cV$ of $V(G)$ with parts $V_{i,j}$ for $(i,j)\in V(R)$ which has the following properties. First, $|V_{i,j}|=|X_{i,j}|$ for each $(i,j)\in V(R)$. Second, $R$ is an $(\eps,d,1)$-reduced graph for $G$ with respect to $\cV$. Third, $G$ is super-regular on the graph $R'$ whose edges are $(i,j)(i,j')$ for $i\in[k]$ and $1\le j<j'\le\chi(H)$, in other words on the $K_{\chi(H)}$-factor in the backbone graph.
    
    This is the point at which, in~\cite{BST}, the embedding of $H$ into $G$ begins. It is worth remarking that, because Lemma~\ref{lem:rg_image} applies to the entire reduced graph, we could complete their proof by simply verifying the conditions of Lemma~\ref{lem:rg_image} for $p=1$ (much as we do below), rather than needing the technical work of the partial embedding lemma.
    
    Recall that we wish to show that $H$ is a subgraph of $G_p$, where $p\ge C\big(\tfrac{\log n}{n}\big)^{1/\Delta}$ for some suitably large $C$. Observe that $E(G_p)$ is distributed as $E(G)\cap E(G(n,p))$. We now generate $\Gamma=G(n,p)$. Asymptotically almost surely, the good event of the blow-up lemma for random graphs, Lemma~\ref{lem:rg_image} occurs for $\Gamma$. We need to verify that $H$ and $G_p=G\cap\Gamma$ a.a.s.\ satisfy the conditions of Lemma~\ref{lem:rg_image}. We begin with $H$. Recall that $\cX$ is an $R$-partition of $H$. By choice of $\xi$, it is $4$-balanced. Also by choice of $\xi$, less than half of the vertices of any given $X_{i,j}$ are at distance two or less from $Z$. We let $\tX_{i,j}$ be the vertices of $X_{i,j}$ at distance three or more from $Z$. This gives us a $\big(\tfrac12,R'\big)$-buffer for $(H,\cX)$.
    
    We now need to show that a.a.s.\ $\cV$ is an $(\eps',d,p)$-regular $R$-partition of $G\cap\Gamma$, and that a.a.s.\ it is $(\eps',d,p)$-super-regular and has one- and two-sided inheritance on $R'$. The first of these is an easy consequence of Theorem~\ref{thm:chernoff} (Chernoff's inequality) and the fact that $\cV$ is an $(\eps,d,1)$-regular $R$-partition of $G$. Indeed, we can simply take the union bound over the at most $2^{2n}$ choices of pairs of subsets of $V(G)$ which we need to have density at least $(d-\eps')p$. Since $G$ has super-regularity on $R'$, again using Theorem~\ref{thm:chernoff} and taking the union bound over the choices of vertices in $V(G)$ and $V(R')$, we see that a.a.s.\ $\Gamma\cap G$ has super-regularity on $R'$. Next we show that a.a.s.\ if $v\in V_i$ and $ij,jk\in R'$, so $\big(N_{\Gamma\cap G}(v;V_j),V_k\big)$ is $(\eps',d,p)$-regular in $G$. Indeed, by Theorem~\ref{thm:OneSideInherit}, with $\beta=\tfrac14$, the probability that this fails is at most $2^{-pn}$, so that we can take a union bound over all choices of $i,j,k$ and $v\in V_i$. Similarly, using Corollary~\ref{cor:TwoSideInherit}, if also $ik\in R'$, a.a.s.\ $\big(N_{\Gamma\cap G}(v;V_j),N_{\Gamma\cap G}(v;V_k)\big)$ is $(\eps',d,p)$-regular in $G$.
    
    We do not image restrict any vertices of $H$, so the restriction pair condition of Lemma~\ref{lem:rg_image} is vacuously satisfied. Thus, applying Lemma~\ref{lem:rg_image}, we see $H\subset G\cap\Gamma$ as desired.
   \end{proof}
   
\chapter{Concluding remarks}
\label{chap:concl}

In this chapter we start by formulating a consequence of our random graphs
blow-up lemma for dense graphs for later reference.
We then discuss when our main results might be improved and when they are sharp (Section~\ref{sec:opt}). We explain how to obtain (randomised) algorithmic versions of our blow-up lemmas (Section~\ref{sec:alg}). We give a version of our random graphs blow-up lemma for directed graphs (Section~\ref{sec:dir}) and sketch how one might allow for coloured graph settings, and mention hypergraphs (Section~\ref{sec:hyp}). Finally, we give some open problems (Section~\ref{sec:open}).

\section{The dense case}

One of the limitations of the dense blow-up lemma of~\cite{KSS_bl} is that
it can only be applied to small parts of the reduced graph of a partition
given by the regularity lemma (because the regularity $\eps$
that the blow-up lemma requires depends on the number of clusters to which
it is applied). Hence, when using this lemma to embed spanning graphs one
usually applies the blow-up lemma several times to different small parts of
the reduced graph after setting up suitable connections. In order to appropriately
combine the connections with these blow-up lemma applications one uses
image restrictions.

Our blow-up lemmas can be applied
to the whole reduced graph. Since this can simplify blow-up lemma
applications also in the dense case, we state a dense version here that
is a direct consequence of Lemma~\ref{lem:rg_image}. For easier reference
we first repeat the relevant definitions needed for this lemma, which are simply
the $p=1$ case of the definitions introduced earlier.
We remark that in particular our definition of restriction pairs simplifies to
image restrictions (in the usual sense) in this case.

Let~$G$, $H$, $R$, $R'$ be graphs with $R'\subset R$. Let
$\cX=\{X_i\}_{i\in[r]}$ be a partition of $V(H)$ and
$\cV=\{V_i\}_{i\in[r]}$ be a partition of $V(G)$, let
$\tcX=\{\tX_i\}_{i\in[r]}$ be a family of subsets of $V(H)$, and let
$\cI=\{I_x\}_{x\in V(H)}$ be a collection of subsets of $V(G)$.
\begin{itemize}
\item We say that the partition~$\cX$ is \emph{$\kappa$-balanced} if there exists
$m\in\mathbb N$ such that we have $m\le|X_i|\le\kappa m$ for all $i,j\in[r]$.
The partitions~$\cX$ and~$\cV$ are \emph{size-compatible} if $|X_i|=|V_i|$ for all $i\in[r]$.
\item
We say $(H,\cX)$ is an \emph{$R$-partition} if each part of~$X$ is nonempty, and whenever
there are edges of~$H$ between~$X_i$ and~$X_j$ the pair~$ij$ is an edge of~$R$.
\item
The family $\tcX=\{\tX_i\}_{i\in[r]}$  is an
\emph{$(\alpha,R')$-buffer} for $(H,\cX)$ if for each $i\in[r]$ we have
$\tX_i\subset X_i$ and
 $|\tX_i|\ge\alpha|X_i|$, and for each $x\in\tX_i$ and $xy,yz\in E(H)$ with
 $x\in X_j$ and $z\in X_k$ we have $ij,jk\in E(R')$.
\item
$(G,\cV)$ is an \emph{$(\eps,d)$-regular $R$-partition} if for each
$ij\in E(R)$ the pair $(V_i,V_j)$ is \emph{$(\eps,d)$-regular}, that is,
$d(V'_i,V'_j)\ge d-\eps$ for all $V'_i\subset V_i$ with $|V'_i|\ge\eps|V_i|$ and
$V'_j\subset V_j$ with $|V'_j|\ge\eps|V_j|$.
\item
$(G,\cV)$ is \emph{$(\eps,d)$-super-regular on~$R'$} if for each
$ij\in E(R')$ the pair $(V_i,V_j)$ is \emph{$(\eps,d)$-super-regular}, that is, it
is $(\eps,d)$-regular and for every $u\in V_i$ we have $\deg_G(u;V_j)\ge(d-\eps)|V_j|$
and for every $u\in V_j$ we have $\deg_G(u;V_i)\ge(d-\eps)|V_i|$.
\item
We say that~$\cI$ is a family of
\emph{$(\rho,\zeta)$-image restrictions} if the
  following properties hold for each $i\in[r]$ and $x\in X_i$.
  \begin{enumerate}[label=\abc]
    \item The set $X_i^*\subset X_i$ of  \emph{image restricted}
      vertices in~$X_i$, that is, vertices such that $I_x\neq V_i$, has size $|X_i^*|\leq\rho|X_i|.$ 
    \item $|I_x|\ge\zeta|V_i|$.
  \end{enumerate}
\end{itemize}

\begin{lemma}[Dense blow-up lemma for the whole reduced graph]
  For all $\Delta,\DeltaRp,\kappa\ge 1$ and $\alpha,\zeta, d>0$
  there exists $\eps,\rho>0$ such that for all $r_1$ there is an $n_0$ such that
  for all $n\ge n_0$ the following holds.
  Let $R$ be a graph on $r\le r_1$ vertices and let $R'\subset R$ be a spanning
  subgraph with $\Delta(R')\leq \DeltaRp$.
  Let~$H$ and~$G$ be graphs with $\kappa$-balanced
  size-compatible vertex partitions
  $\cX=\{X_i\}_{i\in[r]}$ and $\cV=\{V_i\}_{i\in[r]}$, respectively, which have
  parts of size at least $m\ge n/(\kappa r_1)$.
  Let $\tcX=\{\tX_i\}_{i\in[r]}$ be a family of subsets of $V(H)$
  and let $\cI=\{I_x\}_{x\in V(H)}$ be a collection of subsets of $V(G)$.
  Suppose that
  \begin{enumerate}[label=\itmarab{B}]
  \item $\Delta(H)\leq \Delta$, $(H,\cX)$ is an $R$-partition, 
    and $\tcX$ is an $(\alpha,R')$-buffer for $(H,\cX)$,
  \item $(G,\cV)$ is an $(\eps,d)$-regular $R$-partition, which
    is $(\eps,d)$-super-regular on~$R'$,
  \item $\cI$ is a family of $(\rho,\zeta)$-image restrictions.
  \end{enumerate}
  Then there is an embedding $\psi\colon V(H)\to V(G)$ such that $\psi(x)\in
  I_x$ for all $x\in V(H)$.
\end{lemma}

This lemma simplifies even further if image restrictions are not necessary, which
is often the case in applications since this lemma can be applied to the reduced
graph as a whole.

Note that the same lemma is proved as the $p=1$ case of each of Lemmas~\ref{lem:degen} ($D$ can be taken equal to $2\Delta$ in this case, so that~\ref{dbulcon:4} becomes trivial) and~\ref{lem:psr_main}. This is useful to note in that the algorithms used for the embedding are slightly different, and one might be easier to analyse than another for future extensions.

 \section{Optimality of our main results}\label{sec:opt}
  In this section we discuss when our three blow-up lemmas can be improved. We expect that the statements of each lemma remain true for significantly smaller probabilities and weaker bijumbledness requirements in general, but we will observe that in some cases we cannot improve much, or to do so would require some additional conditions on the graphs into which we embed.
 
  \subsection{The blow-up lemma for random graphs}\label{sec:opt:random}
   In the case $\Delta=2$ and $H$ contains a triangle, our result is optimal in terms of $p$ up to the $\log$ factor, since it is well known that when $p=o(n^{-1/2})$ one can delete all triangles from $G(n,p)$ by removing only $o(pn^2)$ edges---so any `blow-up-type' statement will be false.
   
\smallskip   
   
   When $\Delta=3$, the statement of Lemma~\ref{lem:rg_image} is optimal up to the $\log$ factor. As we show in the next paragraph, in the event that $H$ is a spanning $K_4$-factor, for small $c>0$ and $p=cn^{-1/3}$, it is typically possible to find a subgraph $G$ of $G(4n,p)$ with the following properties. There is a partition $V(G)=V_1\cup\dots\cup V_4$ with each set of size $n$. Letting $R=R'=K_4$, the partition is an $(\eps,\tfrac12,p)$-regular $R$-partition, which is super-regular, with one- and two-sided inheritance on $R'$. However, there is a vertex of $G$ which is in no $K_4$, and thus there is no $K_4$-factor covering $G$.
   
   To construct $G$, we let $V_1,\ldots,V_4$ be an equipartition of $[4n]$, fix a vertex $v\in V_1$, reveal $G(4n,p)$, remove all edges within each part, and then remove from $N_G(v)$ all triangles by deleting a minimum number of edges. It is not hard to check using the Chernoff bound, Theorem~\ref{thm:chernoff}, that a.a.s.\ before this final step, we have an $(\eps,\tfrac34,p)$-regular $R$-partition with super-regularity and one- and two-sided inheritance on $R'$. It is similarly easy to check that a.a.s.\ before the final step the degree of each vertex in $V_i$ to any other $V_j$ is close to $pn$, and any pair of vertices has about $p^2 n$ common neighbours in other parts. Thus the expected number of triangles in $N_G(v)$ is $p^3(pn)^3=c^6n$, and the actual number is somewhat concentrated, so that a.a.s.\ in the final step we delete at most $2c^6n$ edges. This is far too small to destroy one-sided inheritance, and, for $c>0$ sufficiently small, too small to destroy two-sided inheritance. It remains to show that we have not removed too many edges from any one vertex. But the expected number of edges removed from $u\in N_G(v)$ is at most $p(p^2 n)^2=c^3p^2n$; the actual number is by Theorem~\ref{thm:chernoff} exponentially concentrated, and hence a.a.s.\ we remove $o(pn)$ edges from each $u$, so do not destroy super-regularity.
   
   Perhaps inserting an extra condition into the statement, such as insisting that all vertex neighbourhoods contain $\Omega(p^6n^3)$ triangles, would allow one to prove a blow-up lemma which allows for embedding a $K_4$-factor down to the natural limit $p=n^{-2/5}$. However, it would not always be possible to obtain such a condition. In~\cite{ABKNP} it is shown that Breaker wins the $K_4$-factor game on $K_n$ with a bias $13n^{1/3}$ (Lemma~\ref{lem:rg_image} is used to show that Maker wins when the bias is $c\big(\tfrac{n}{\log n}\big)^{1/3}$ for some small $c>0$). It follows that any `extra condition' is one which Maker cannot guarantee to obtain with bias $13n^{1/3}$.
   
\smallskip   
   
   Finally, for $\Delta\ge 4$ one might hope that the statement of our blow-up lemma remains true down to $p=Cn^{-2/(\Delta+3)}$, this being the point at which one can generalise the above construction (and therefore the point at which the statement provably does not guarantee a $K_{\Delta+1}$-factor). Perhaps, optimistically, one might hope that there are some natural extra conditions which even allow $p\ge Cn^{-2/(\Delta+2)}$, this being the point at which one can remove all copies of $K_{\Delta+1}$ by deleting a tiny fraction of the edges of $G(n,p)$.   
   
   However, we believe that improving upon our result is likely to be very challenging. Even in the (much simpler and well-studied) setting of trying to prove $\cH(n,\Delta)$-universality of $G(n,p)$, until very recently no improvement had been made over what one can obtain from Lemma~\ref{lem:rg_image}. Very recently Ferber and Nenadov~\cite{FNuniv} were able to obtain a small improvement. In the almost-spanning setting, Conlon, Ferber, Nenadov and {\v{S}}kori{\'c}~\cite{CFNS15} could improve more on what follows from Lemma~\ref{lem:rg_image}, showing that the random graph $G((1+\gamma)n,p)$ is a.a.s.\ $\cH(n,\Delta)$-universal when 
  $p=\omega(n^{-1/(\Delta-1)}\log^5 n)$ for $\Delta\ge3$. But the improvement in the exponent one would desire is of order $\Delta^{-1}$ not $\Delta^{-2}$.
   
  \subsection{The blow-up lemma for bijumbled graphs}\label{sec:opt:jumbled}
   We do not believe the $(p,cp^{\max(4,(3\Delta+1)/2})$-bijumbledness requirement in Lemma~\ref{lem:psr_main} is optimal. Most of the proof would work with $(p,cp^{\Delta+2}n)$-bijumbledness, but we were not able to find a way to avoid the $\LNS$ property which requires the stronger condition. Nevertheless we conjecture that $(p,cp^{\Delta+1}n)$-bijumbledness suffices (we expect that the extra factor of $p$ can be gained by using reserved cliques as in the proof of Lemma~\ref{lem:rg_image}). It is still not clear that this would be optimal. For $\Delta=2$ we need $(p,cp^2n)$-bijumbledness, since Alon~\cite{AlonConstr} constructed a $(p,Cp^2n)$-bijumbled graph without triangles. It is a believable conjecture (see for example Conlon, Fox and Zhao~\cite{CFZ}) that there are $(p,Cp^{\Delta}n)$-bijumbled graphs without $K_{\Delta+1}$ for every $\Delta$, in which case the same requirement would be necessary for a blow-up lemma. It is possible, however, (conjectured in~\cite{CFZ}, but the contrary is conjectured by Kohayakawa, R\"odl, Schacht and Skokan~\cite{KRSS}) that copies of $K_{\Delta+1}$ cannot be guaranteed in regular subgraphs of $(p,\beta)$-bijumbled graphs at this point, but instead that $\beta\le cp^{\Delta+1}n$ is required (it is proved in~\cite{CFZ} that at this point one can guarantee copies of $K_{\Delta+1}$).
   
   Note that Lemma~\ref{lem:psr_main} does not permit linearly many image restrictions. The reason for this is that our method for queue embedding requires that underlying restriction sets tend not to cluster, which is encapsulated in~\eqref{eq:randmatch:sumVqLNS} of Claim~\ref{cl:psr:lnssum}. But we can only establish the inequality~\eqref{eq:randmatch:sumVqLNS} by excluding image restricted vertices (and in fact the inequality would not necessarily hold if linearly many image restricted vertices were included). We can only image restrict about a $p^\Delta$-fraction of vertices in each cluster, as these are too few to cause problems with the queue embedding. In many applications this is not a problem (see for example~\cite{ABET}), but it could well cause a problem for some applications. We believe that it is possible to modify the proof strategy substantially in order to have some control over linearly many vertices of each $X_i$. The modification we would make is the following. We would permit the user of the blow-up lemma to specify $\rho|X_i|$ `pre-embedded' vertices in each part $X_i$ which are to be embedded first (before even the neighbours of buffer vertices). The user is then permitted to embed these vertices sequentially, subject to four conditions. First, the result must be a good partial embedding $\psi$. Second, each vertex must be embedded to a uniform random vertex from a set of size at least $\tfrac{1}{10}\mu\zeta(dp)^{\Delta-1}n/(\kappa r_1)$. Third, for each $jk\in R'$ and $v\in V_j$, at most $\tfrac{1}{100}\deg_G(v;V_k)$ neighbours of $v$ in $V_k$ may be in $\im(\psi)$. Fourth, the total number of vertices in $X_i$ with pre-embedded neighbours is at most $\rho|X_i|$. We would then follow the proof strategy of Lemma~\ref{lem:psr_main} to embed the remainder of $H$ into $G$, treating the neighbours of pre-embedded vertices as image restricted. We remark that this does not automatically resolve the problem with Claim~\ref{cl:psr:lnssum}, since the partition of $V(G)$ considered at this point is \emph{not} the same as the partition the user of the blow-up lemma supplies. But it is easy to modify Lemma~\ref{lem:matchreduce} (the good partitions lemma, which generates the partition considered by Claim~\ref{cl:psr:lnssum}) to show that all the sets under consideration are with high probability evenly distributed by the random equipartition, and then the proof of Claim~\ref{cl:psr:lnssum} does go through. Checking the full details of this approach, and for that matter using the resulting blow-up lemma, seems likely to be non-trivial, but it could potentially allow for stronger theorems.
   
  \subsection{The blow-up lemma to embed degenerate graphs}\label{sec:concl:degen}
   The definition of $(D,p,m)$-boundedness is rather complicated, but it allows us to prove as flexible a statement as we could, taking into account the structure of $H$ in order to work with lower probabilities. Observe that a $d$-degenerate $n$-vertex graph with bounded maximum degree, which Lemma~\ref{lem:degen} can handle with $p\approx n^{-1/(2d+1)}$, can contain almost $dn$ edges. This is comparable to a $2d$-regular graph for which Lemma~\ref{lem:rg_image} would require $p\approx n^{-1/2d}$. So even without saying anything about the structure of $H$ beyond its degeneracy, the result is almost as powerful as Lemma~\ref{lem:rg_image}.
   
    In the event that we only need an almost-spanning embedding, we can take the potential buffer vertices in each part to be isolated, and hence $D=2d$. We can then embed $d$-degenerate graphs with $p\approx n^{-1/(2d)}$, matching the performance of Lemma~\ref{lem:rg_image}. Finally, if $H$ is an $F$-factor we can take $D=d+3$ or better (depending on the structure of $F$), in which case the performance of Lemma~\ref{lem:degen} substantially improves, working with $p\approx n^{-1/(d+3)}$.
   
    In fact, we can improve
    Lemma~\ref{lem:degen} for $F$-factors. As stated in the proof of
    Lemma~\ref{lem:degenrga} (the degenerate RGA lemma) we require~\ref{ord:NtX} and
    the condition $\pitau(x)\le D_x-1$ for vertices $x$ adjacent to potential buffer vertices
    within~\ref{ord:Dx} only in order to prove~\ref{degenrga:nobad}. When
    embedding an $F$-factor (provided $R'$ is suitable, for example
    $R'=K_{|F|}$), we do not really need~\ref{degenrga:nobad},
    as~\ref{degenrga:main} shows that only a few vertices in each
    cluster of $G$ can fail~\ref{degenrga:nobad}, and we can use an
    argument similar to Lemma~\ref{lem:deletebad} to deal with
    them. This allows us to reduce the required $D$ by one, compared to the requirement of Lemma~\ref{lem:degen}.
   
   In particular, for each $s\le t$ one can prove a blow-up lemma which embeds a $K_{s,t}$-factor when $p\ge C\big(\tfrac{\log n}{n}\big)^{1/s}$. This is almost optimal in terms of $p$, since the $2$-density of $K_{s,t}$ is $\frac{st-1}{s+t-2}$, which approaches $s$ as $t$ becomes large; when $p$ is below $n^{-(s+t-2)/(st-1)}$, one can delete a very small fraction of the edges of $G(n,p)$ to destroy all copies of $K_{s,t}$.
   
 \section{Algorithmic embedding}\label{sec:alg}
  The proofs of our blow-up lemmas can be changed slightly to give polynomial-time randomised algorithms which with high probability construct the embeddings we prove exist. It is quite tedious to check the details, but we provide a sketch of this for the interested reader.
  
  The main change which needs to be made concerns the certification of sparse-regular pairs. The RI property (see Section~\ref{sec:pseudo}), which we require in all of our blow-up lemmas, guarantees that typical vertex neighbourhoods inherit regularity, but for an algorithm we need to be able to identify the set of vertices whose neighbourhoods inherit regularity in polynomial time.
  
   Alon, Duke, Lefmann, R\"odl and Yuster~\cite{ADLRY} showed that, in dense graphs, determining if a given bipartite graph is $\eps$-fully-regular is co-NP-complete, but that there is a polynomial-time algorithm which either certifies $\eps$-full-regularity, or returns a witness to the failure of $\eps'$-full-regularity, for some $\eps'$ which may be much smaller than $\eps$ but does not depend on the number of vertices in the regular pair. For sparse graphs, a corresponding polynomial-time certification algorithm was given by Alon, Coja-Oghlan, H\`an, Kang, R\"odl and Schacht~\cite{ACOHKRS} which either certifies $(\eps,d,p)$-full-regularity or returns a witness to the failure of $(\eps',d,p)$-full-regularity.
  
  Unfortunately, we do not know of any such algorithm in the literature for lower-regular pairs, so we now sketch a certification algorithm, using the results of~\cite{ACOHKRS}, for lower-regularity which works in subgraphs of random graphs. Given as input $\eps$, $d$, $p$ and a bipartite graph, which must be bounded (the definition is in~\cite{ACOHKRS}, but the reader does not need to know it), we apply the algorithmic sparse regularity lemma of~\cite{ACOHKRS} with regularity parameter $\tfrac1{100}\eps^3$. This returns a partition of each side of the bipartite graph into approximately equal numbers of parts, which approximately equipartition each side, even if the bipartite graph itself is very unbalanced (we may want, for example, to know whether a bipartite graph with parts of size $pn$ and $p^2n$ respectively is lower-regular). We choose $0<\eps'<\tfrac{1}{100}\eps^3$ such that any given part of the partition contains at least an $\eps'$-fraction of its side. Now if any pair of parts in this partition has density less than $(d-\eps')p$, it is a witness to a failure of $(\eps',d,p)$-lower-regularity. If not, we claim the bipartite graph is $(\eps,d,p)$-lower-regular; the choice of regularity parameter in the use of sparse regularity ensures that there are too few irregular pairs to seriously affect densities between large sets. We note that the requirement of boundedness is needed for the algorithmic sparse regularity lemma, and that this boundedness holds a.a.s.\ in subgraphs of typical random graphs (and is implied by the $\NS$ property which we require in any case). In contrast to the certification algorithm for sparse full-regularity, the dependency of $\eps'$ on $\eps$ here is very poor: there is a tower-type relationship, which appears iterated in the constant dependencies of the algorithmic versions of Lemmas~\ref{lem:rg_image} and~\ref{lem:degen}.
  
  For either random or bijumbled ambient graphs $\Gamma$, given a certification algorithm and inheritance lemmas, one can prove a variant of the $\RI$ property in which not only do typical vertex neighbourhoods inherit (either version of) regularity, but they do so certifiably. We follow the proofs more or less as in Section~\ref{sec:pseudo}, except that at each step, where we need certifiable $(\eps,d,p)$-regularity for some $\eps$, we obtain $\eps'$ from the certification algorithm and then let $\eps''$ be returned by our inheritance lemmas for input $\eps'$. Now a typical vertex neighbourhood is $(\eps',d,p)$-regular, so that the certification algorithm will certify it to be $(\eps,d,p)$-regular as desired.
 
  We now sketch how one can use this to obtain algorithmic versions of our blow-up lemmas.
  
  It is necessary to check that the algorithm implicit in the proof of the
  lemma obtaining good partitions of~$H$ and~$G$ (Lemma~\ref{lem:matchreduce}) is a randomised polynomial time algorithm. This follows since the Hajnal-Szemer\'edi theorem (Theorem~\ref{thm:HajSze}) has an algorithmic version~\cite{AlgHajSze}, and since the proof of Lemma~\ref{lem:nicepartition} can then easily be made constructive. Furthermore, a failure of the randomised partitioning to produce a good $G$-partition can be detected in polynomial time.
  
  It is further necessary to check that each of our RGA algorithms (Algorithms~\ref{alg:RGA},~\ref{alg:RGA:psr} and~\ref{alg:RGA:deg}) can be carried out in polynomial time. This amounts to checking that the various sets that appear in these algorithms can be constructed in polynomial time. For most of these sets, this is obviously possible. However, in order to construct the bad sets $B_t(x)$ we need the certifiable regularity inheritance discussed above.
  
  In proving Lemma~\ref{lem:rg_image} we give (implicitly) an algorithm for embedding queue vertices. Again, to run this algorithm we need to construct the bad sets $B(x)$ and this requires certifiable regularity inheritance. We also use the fact that the bipartite matching problem can be solved in polynomial time (see Kuhn~\cite{KuhnMatch}). For fixing buffer defects in polynomial time, Lemma~\ref{lem:identifybad} gives sets $P_i$ and $D_i$, and the proof constructs these sets in polynomial time, while our algorithm for fixing buffer defects, Algorithm~\ref{alg:delbad}, which uses the $P_i$ and $D_i$,  is trivially polynomial time.
  
  Finally, each of our blow-up lemmas is completed by embedding the buffer vertices. This amounts to a bipartite matching problem, and can be solved in polynomial time.
  
  It seems reasonable to believe that it is possible to derandomise our RGA
  algorithms and good partitions algorithm, which would yield polynomial time algorithms for constructing each of the claimed embeddings. Certainly Koml\'os, S\'ark\"ozy and Szemer\'edi~\cite{KSS_blalg} were able to derandomise their original (RGA-based) proof of the dense blow-up lemma. However, we did not attempt to check whether their methods suffice in our case.

 \section{Directed graphs}\label{sec:dir}
  Although our blow-up lemmas as written apply to undirected graphs, we can also apply them to subdigraphs of random directed graphs (or bijumbled directed graphs). The random directed graph $\overrightarrow{G}(n,p)$ is obtained by choosing, for each ordered pair of vertices in a vertex set of size $n$, to put an arc independently with probability $p$. We give for illustration the directed statement corresponding to Lemma~\ref{lem:rg_image}. In order to state this, we define the \emph{undirection} of a digraph $\overrightarrow{G}$ to be the graph $G$ with $uv\in E(G)$ if and only if either $\overrightarrow{uv}$ or $\overrightarrow{vu}$ is an arc of $\overrightarrow{G}$. The terms which we defined for undirected graphs (such as $R$-partition, buffer, and so on) are taken as applying to the undirections of the digraphs considered. The exception is that when we talk about an $\overrightarrow{R}$-partition of $\overrightarrow{G}$ or $\overrightarrow{H}$ we require, in addition to the conditions for the undirections of digraphs, that all arcs of $\overrightarrow{G}$ or $\overrightarrow{H}$ go in the direction specified by the arcs of $\overrightarrow{R}$. In the following lemma we work with $\overrightarrow{G}(n,q)$ but use $p=2q-q^2$ as the parameter for regularity. The reason for this is that the undirection of $\overrightarrow{G}(n,q)$ is $G(n,p)$.
  
  \begin{lemma}[Blow-up lemma for random directed graphs]\label{lem:diblowup}
	  For all $\Delta$, $\DeltaRp$, $\Delta_J$, $\alpha,\zeta, d>0$, $\kappa>1$
	  there exist $\eps,\rho>0$ such that for all $r_1$ there is a $C$ such that for
	  $q>C(\log n/n)^{1/\Delta}$ the following holds. Let $p=2q-q^2$. The random directed graph $\overrightarrow{\Gamma}=\overrightarrow{G}(n,q)$ asymptotically
	  almost surely satisfies the following.
	   
	  Let $\overrightarrow{R}$ be a digraph on $r\le r_1$ vertices without cycles of length $2$ and let $R'$ be a subgraph of the undirection of $\overrightarrow{R}$ with $\Delta(R')\leq \DeltaRp$.
	  Let $\overrightarrow{H}$ and $\overrightarrow{G}\subset \overrightarrow{\Gamma}$ be digraphs given with $\kappa$-balanced,
	  size-compatible vertex partitions 
	  $\cX=\{X_i\}_{i\in[r]}$ and $\cV=\{V_i\}_{i\in[r]}$ with parts of size at
	  least $m\ge n/(\kappa r_1)$. 
	  Let $\cI=\{I_x\}_{x\in V(\overrightarrow{H})}$ be a family of image restrictions, and
	  $\cJ=\{J_x\}_{x\in  V(\overrightarrow{H})}$  be a family of restricting vertices.
	  Suppose that
	  \begin{enumerate}[label=\itmarab{BUL}]
	  \item The undirection of $\overrightarrow{H}$ has maximum degree at most $\Delta$, $(\overrightarrow{H},\cX)$ is an $\overrightarrow{R}$-partition,
	    and $\tcX=\{\tX_i\}_{i\in[r]}$ is an
	    $(\alpha,R')$-buffer for $\overrightarrow{H}$,
	\item $(\overrightarrow{G},\cV)$ is an $(\eps,d,p)$-regular $\overrightarrow{R}$-partition, which
    is $(\eps,d,p)$-super-regular on~$R'$, 
    has one-sided inheritance on~$R'$,
    and two-sided inheritance on~$R'$ for $\tcX$, 
	  \item $\cI$ and $\cJ$ form
	    a $(\rho,\zeta,\Delta,\Delta_J)$-restriction pair.
	  \end{enumerate}
	  Then there is an embedding $\psi\colon V(\overrightarrow{H})\to V(\overrightarrow{G})$ such that $\psi(x)\in
	  I_x$ for each $x\in \overrightarrow{H}$.
	\end{lemma}
	
	This lemma is a corollary of Lemma~\ref{lem:rg_image}. We simply work with the undirections of the given digraphs, which satisfy the conditions of Lemma~\ref{lem:rg_image}, and observe that an embedding $\psi$ of the undirection of $\overrightarrow{H}$ into the undirection of $\overrightarrow{G}$ such that $\psi(x)\in V_i$ for each $x\in X_i$ is by definition of an $\overrightarrow{R}$-partition automatically an embedding of $\overrightarrow{H}$ into $\overrightarrow{G}$.
	
	With rather more work, we believe we could allow $\overrightarrow{R}$ to contain $2$-cycles. More generally, we could allow $G$ to be coloured from a palette of at most $\kappa$ colours, define $R$ to be the multicoloured graph (with edges permitted to have several colours) corresponding to relatively dense sparse-regular pairs in the given colour, and supply a coloured graph $H$ with a coloured graph homomorphism to $R$, which we require to be embedded with edges going to edges of $G$ with the correct colour. We believe that such coloured versions of all three of our blow-up lemmas can be proved, following the strategies given in this paper. However, to do so requires appropriate modifications to several definitions and recalculation of various parameters. We see no reason why this should cause difficulty, but we did not check the details.

 \section{Hypergraphs}\label{sec:hyp}
  It seems likely that the techniques developed in this paper will be very helpful for proving a blow-up lemma for uniform hypergraphs which works relative to sparse random or pseudorandom (appropriately defined) hypergraphs. In the dense case, Keevash~\cite{HypBlow} proved a hypergraph blow-up lemma. However, it has a serious limitation, in that it only allows for image restriction of single vertices. In many applications one needs image restrictions of vertex tuples of size up to $k-1$ in $k$-uniform hypergraphs.

 \section{Open problems}\label{sec:open}
  Beyond the question of improving on our main results (as discussed in Section~\ref{sec:opt}), we would like to pose the following problems.  
  
  \begin{problem} Is it true that for each $r,\Delta\ge 2$ and $n$ there exists a $K_{\Delta+1,\Delta+1}$-free graph $G$ which is $r$-partition universal for $\cH(n,\Delta)$, with $v(G)=O(n)$?
  \end{problem}
  We have seen that the answer is `yes' if $\Delta+1$ is replaced by $2\Delta$ (Theorem~\ref{thm:folkmantype}), and trivially the answer is `no' if $\Delta+1$ is replaced by $\Delta$. It is not clear that $G$ should be constructed randomly (which is how Theorem~\ref{thm:folkmantype} is proved) in order to obtain an affirmative answer to this problem. But it is also not clear how to construct sparse graphs with strong Ramsey properties any other way.
  
  \smallskip
  
    Before trying to improve the random graph blow-up lemmas in this paper, Lemmas~\ref{lem:rg_image} and~\ref{lem:degen}, we should at least know how to embed large subgraphs in the random graph itself in a robust way. 
  
  \begin{problem}
   For what $p$ is $G(n,p)$ typically $\cH(n,\Delta)$-universal, or typically $\cH(n,d,\Delta)$-universal? Does the answer change substantially if $G(n,p)$ is replaced with $G(Cn,p)$ for $C$ large, or for $C$ slightly larger than one?
  \end{problem}
So far, all the methods used to attack problems of this type construct a single embedding of some $H$ into $G(n,p)$ step by step, and rely on there being many ways to continue the embedding after each step.  
  In this paper, each step is a single vertex embedding. In order for there to be many ways to embed a vertex with $\Delta$ embedded neighbours, we certainly need $p^\Delta n\gg 1$, and Lemma~\ref{lem:rg_image} gets within a log-factor of this bound. Recent papers, in particular~\cite{CFNS15} and~\cite{FNuniv}, improve slightly on this bound by embedding more than one vertex at a time. But it seems reasonable to conjecture that if $p\ge Cn^{-2/(\Delta+1)}\log^{1/\binom{\Delta+1}{2}}$ then $G(n,p)$ should typically be $\cH(n,\Delta)$-universal. For large $\Delta$, this is much sparser than anything we can currently handle, and it seems likely that being more clever with the step-by-step embedding methods currently used will not prove this conjecture. Assuming the conjecture is correct, it can make at most a log-factor difference if $G(n,p)$ is replaced by $G(Cn,p)$ for $C$ slightly larger than $1$, or for $C$ large. However, it should be significantly easier to prove universality in the latter two cases.
  
  For $\cH(n,d,\Delta)$-universality of $G(Cn,p)$ with $C$ slightly
  larger than $1$, an almost optimal bound on $p$ is obtained
  in~\cite[Theorem~3.6]{nenadov16:_ramsey}. This problem is easier: it
  is necessary for $p^d n$ to be large in order for a vertex-by-vertex
  embedding method to work, but it is easy to see that this is also a
  lower bound for $\cH(n,d,\Delta)$-universality. It would still be
  nice to remove the extra log-factors
  from~~\cite[Theorem~3.6]{nenadov16:_ramsey}.  What seems much harder
  is to prove a similar spanning universality result.
  
  \begin{problem}
   For what $\beta$ are $(p,\beta)$-bijumbled $n$-vertex graphs $G$ with minimum degree $\tfrac12 p n$ always $\cH(n,\Delta)$-universal? Does the answer change if we allow $G$ to have $Cn$ vertices for $C$ large, or for $C$ slightly larger than one?
  \end{problem}
  This problem is a first step towards improving Lemma~\ref{lem:psr_main}. It would also be interesting to replace bijumbledness with one of the other standard notions of pseudorandomness. However, in this setting we do not really know how large $\beta$ can be: we do not have good constructions of graphs which are strongly bijumbled but are not $\cH(n,\Delta)$-universal, except for Alon's construction~\cite{AlonConstr} for $\Delta=2$.
  
  \begin{problem}
   Do there exist graphs $G$ which are $r$-partition universal for $\cH(n,\Delta)$ with only $O(n^{2-2/\Delta})$ edges?
  \end{problem}
Alon and Capalbo~\cite{AlCap} showed that for $r=1$ (i.e.\ just universality) the answer is yes and that this is best possible. This is the best known lower bound for larger $r$, and it would be interesting to know if it is correct. Here a simple random construction (which is used in~\cite{KRSS} to obtain the best known upper bound on the number of edges required) cannot work: a random graph with this many edges will not contain a $K_{\Delta+1}$-factor at all, let alone have the Ramsey property for it.  
  
  \begin{problem}
   For what bias $b$ can Maker win the $\cH(n,\Delta)$-universality game on $K_n$, or on $K_{Cn}$ for $C$ large, or for $C$ slightly larger than one?
  \end{problem}
  Although there exist universal graphs with $n^{2-2/\Delta}$ edges (Alon and Capalbo~\cite{AlCap}), Maker certainly cannot make them with a bias $\Omega\big(n^{2/\Delta}\big)$, since Maker requires $b=O\big(n^{2/(\Delta+2)}\big)$ in order to make just one copy of $K_{\Delta+1}$.

\bibliographystyle{amsplain}  
\bibliography{extracted-revised}

\begin{dajauthors}
\begin{authorinfo}[pa]    
  Peter Allen\\
  Department of Mathematics\\
  London School of Economics\\
  Houghton Street\\
  London WC2A 2AE, U.K.\\
  p\imagedot{}d\imagedot{}allen\imageat{}lse\imagedot{}ac\imagedot{}uk
\end{authorinfo}
\begin{authorinfo}[jb]
  Julia B\"ottcher\\
  Department of Mathematics\\
  London School of Economics\\
  Houghton Street\\
  London WC2A 2AE, U.K.\\
  j\imagedot{}boettcher\imageat{}lse\imagedot{}ac\imagedot{}uk
\end{authorinfo}
\begin{authorinfo}[hiep]
  Hi\d{\^{e}}p H\`an\\
  Departamento de Matem\'atica y Ciencia de la Computaci\'on\\
  Universidad de Santiago de Chile\\
  Las Sophoras 173\\
  Estaci\'on Central, Santiago, Chile\\
  hiep\imagedot{}han\imageat{}usach\imagedot{}cl
\end{authorinfo}
\begin{authorinfo}[yk]
  Yoshiharu Kohayakawa\\
  Instituto de Matem\'atica e Estat\'{\i}stica\\
  Universidade de S\~ao Paulo\\
  Rua do Mat\~ao 1010\\
  05508--090~S\~ao Paulo, Brazil\\
  yoshi\imageat{}ime\imagedot{}usp\imagedot{}br
\end{authorinfo}
\begin{authorinfo}[yp]
  Yury Person\\
  Institut f\"ur Mathematik\\
  Technische Universit\"at Ilmenau\\
  98684 Ilmenau, Germany \\
  yury\imagedot{}person\imageat{}tu-ilmenau\imagedot{}de
\end{authorinfo}
\end{dajauthors}

\end{document}